\documentclass[11pt]{article}

\usepackage[margin=1in]{geometry}
\usepackage[T1]{fontenc}
\usepackage[utf8]{inputenc}
\usepackage{hyperref}
\usepackage{amsmath,amsthm}
\usepackage{lmodern,microtype}
\usepackage{amssymb,mathtools,bm,mathrsfs}
\usepackage{booktabs,longtable,array,enumitem}
\usepackage{xcolor}
\usepackage{caption}

\usepackage{graphicx}
\usepackage{tikz,pgfplots}
\usetikzlibrary{arrows.meta,positioning,calc}
\pgfplotsset{compat=1.18}
\hypersetup{hidelinks,hypertexnames=false,unicode=true,
  pdftitle={Chebyshev-Exact Acceleration under Hessian Variation, II: Computer-Assisted Variational Lower Bound},
  pdfauthor={Dmitry Pasechnyuk-Vilensky; Martin Tak\'{a}\v{c}},
  pdfsubject={Variational lower bounds for the sensitivity of terminal-exact accelerated recurrences},
  pdfkeywords={Chebyshev acceleration, Jacobi matrices, Birman--Schwinger operators, variational bounds, validated numerics},
  pdflang={en-US}}
\ifdefined\pdfgentounicode
  \input{glyphtounicode}
\fi
\usepackage[nameinlink,capitalize]{cleveref}
\microtypesetup{protrusion=false}

\AtBeginDocument{%
  \setlength{\abovedisplayskip}{7pt plus 2pt minus 3pt}%
  \setlength{\belowdisplayskip}{7pt plus 2pt minus 3pt}%
  \setlength{\abovedisplayshortskip}{4pt plus 2pt minus 2pt}%
  \setlength{\belowdisplayshortskip}{4pt plus 2pt minus 2pt}%
}

\usepackage[compact]{titlesec}

\titlespacing*{\section}
  {0pt}{1.55ex plus 0.35ex minus 0.25ex}{0.75ex plus 0.15ex}

\titlespacing*{\subsection}
  {0pt}{1.25ex plus 0.30ex minus 0.20ex}{0.55ex plus 0.12ex}

\titlespacing*{\subsubsection}
  {0pt}{1.00ex plus 0.25ex minus 0.15ex}{0.45ex plus 0.10ex}

\titleformat{\section}
  {\normalfont\large\bfseries}{\thesection.}{0.5em}{}

\titleformat{\subsection}
  {\normalfont\normalsize\bfseries}{\thesubsection.}{0.5em}{}

\titleformat{\subsubsection}
  {\normalfont\normalsize\itshape}{\thesubsubsection.}{0.5em}{}

\makeatletter
\def\thm@space@setup{%
  \thm@preskip=6pt plus 2pt minus 2pt
  \thm@postskip=6pt plus 2pt minus 2pt
}
\makeatother

\makeatletter
\renewcommand{\maketitle}{%
  \begin{center}
    {\Large\bfseries \@title\par}
    \vskip 0.45em
    {\normalsize
      \lineskip .5em
      \begin{tabular}[t]{c}
        \@author
      \end{tabular}\par}
    \vskip 0.35em
    {\small \@date\par}
  \end{center}
  \par\vskip 0.35em
  \@thanks
  \vskip 0.7em
}
\makeatother

\renewenvironment{abstract}
  {\small\begin{center}\bfseries Abstract\end{center}\vspace{-0.6em}\small}
  {\par\vspace{0.6em}}

\allowdisplaybreaks[2]
\setlist[enumerate]{label=\textup{(\roman*)},leftmargin=2.5em}
\newtheorem{theorem}{Theorem}[section]
\newtheorem{lemma}{Lemma}
\newtheorem{proposition}{Proposition}
\newtheorem{corollary}{Corollary}
\theoremstyle{definition}
\newtheorem{definition}{Definition}

\Crefname{algorithm}{Algorithm}{Algorithms}
\Crefname{lemma}{Lemma}{Lemmas}
\newtheorem{hypothesis}{Hypothesis}
\crefname{hypothesis}{hypothesis}{hypotheses}
\Crefname{hypothesis}{Hypothesis}{Hypotheses}
\newcommand{\R}{\mathbb R}

\newcommand{\dd}{\,\mathrm d}
\newcommand{\tr}{\operatorname{tr}}
\newcommand{\supp}{\operatorname{supp}}

\newcommand{\diag}{\operatorname{diag}}
\newcommand{\Span}{\operatorname{span}}
\newcommand{\HS}{\mathrm{HS}}

\newcommand{\E}{\mathcal E}
\newcommand{\HH}{\mathcal H}
\newcommand{\Mc}{\mathcal M_c}

\newcommand{\X}{\mathcal X}
\newcommand{\cs}{c_{\sin}}
\newcommand{\epsN}{\epsilon_N^\star}
\newcommand{\ip}[2]{\langle #1,#2\rangle}

\newcommand{\ind}{\mathbf1}
\newcommand{\dyad}[2]{#1\cdot2^{#2}}
\title{Chebyshev-Exact Acceleration under Hessian Variation,\\II: Computer-Assisted Variational Lower Bound
}
\author{Dmitry Pasechnyuk-Vilensky\thanks{Mohamed bin Zayed University of Artificial Intelligence, Abu Dhabi, UAE. \texttt{dmivilensky1@gmail.com}}, Martin Tak\'{a}\v{c}\thanks{Mohamed bin Zayed University of Artificial Intelligence, Abu Dhabi, UAE.}}
\date{}
\newcommand{\num}[1]{\ensuremath{\csname computed/#1\endcsname}}
\expandafter\def\csname computed/feasible/central_energy/hi24\endcsname{1.142692108398734554420122}
\expandafter\def\csname computed/feasible/contraction/hi24\endcsname{2.533293143397725\cdot10^{-9}}
\expandafter\def\csname computed/feasible/degree\endcsname{16}
\expandafter\def\csname computed/feasible/directions\endcsname{0,1,2,3,6,7}
\expandafter\def\csname computed/feasible/energy_ceiling/hi24\endcsname{1.142692108398735477232678}
\expandafter\def\csname computed/feasible/radius/exact\endcsname{\dyad{27684846429260304880501991746477461659874142439}{-207}}
\expandafter\def\csname computed/feasible/residual/hi24\endcsname{6.729813370879936\cdot10^{-17}}
\expandafter\def\csname computed/reference/H_signal_0/hi15\endcsname{-0.001269156586438}
\expandafter\def\csname computed/reference/H_signal_0/lo15\endcsname{-0.001269156586442}
\expandafter\def\csname computed/reference/H_signal_1/hi15\endcsname{-0.002545408246816}
\expandafter\def\csname computed/reference/H_signal_1/lo15\endcsname{-0.002545408246820}
\expandafter\def\csname computed/reference/S_signal/hi15\endcsname{0.000049913515800}
\expandafter\def\csname computed/reference/S_signal/lo15\endcsname{0.000049913515799}
\expandafter\def\csname computed/reference/b/exact\endcsname{0.11085839213923356}
\expandafter\def\csname computed/reference/base/hi15\endcsname{1.142663674319440}
\expandafter\def\csname computed/reference/base/lo15\endcsname{1.142663674318292}
\expandafter\def\csname computed/reference/beta/exact\endcsname{2704.261804447028}
\expandafter\def\csname computed/reference/gamma/exact\endcsname{1351.3424903216799}
\expandafter\def\csname computed/reference/mass_1/hi15\endcsname{0.999855703345566}
\expandafter\def\csname computed/reference/mass_1/lo15\endcsname{0.999855703345564}
\expandafter\def\csname computed/reference/mass_2/hi15\endcsname{1.000000000819152}
\expandafter\def\csname computed/reference/mass_2/lo15\endcsname{1.000000000819150}
\expandafter\def\csname computed/reference/pzero\endcsname{0.11104619364409615}
\expandafter\def\csname computed/reference/terms\endcsname{5}
\expandafter\def\csname computed/certificate/bound/hi24\endcsname{1.142683536244580089104558}
\expandafter\def\csname computed/certificate/bound/lo24\endcsname{1.142683536243431440994557}
\expandafter\def\csname computed/certificate/price/hi24\endcsname{0.000017926991598184194252}
\expandafter\def\csname computed/certificate/price/lo24\endcsname{0.000017926991598178728941}
\expandafter\def\csname computed/sine/A\endcsname{1+a^{2}+a^{6}+a^{12}}
\expandafter\def\csname computed/sine/B\endcsname{1-2a^{1}+2a^{4}-2a^{9}}
\expandafter\def\csname computed/sine/J\endcsname{3}
\expandafter\def\csname computed/sine/degree\endcsname{126}
\expandafter\def\csname computed/sine/energy/hi25\endcsname{1.1426921798157914393593577}
\expandafter\def\csname computed/sine/energy/lo25\endcsname{1.1426921798157914393588862}
\expandafter\def\csname computed/sine/tail/hi25\endcsname{2.35701467\cdot10^{-22}}
\begin{document}
\maketitle

\begin{abstract}
We study two-step recurrences with positive step sizes and nonnegative
momentum coefficients whose terminal residual is the Chebyshev minimax
polynomial on a prescribed spectral interval $[\mu,L]$. These recurrences
have the same terminal error on a fixed quadratic, but their first-order
responses to stepwise Hessian perturbations can differ. We prove that the
sine--Jacobi recurrence of Part~I is asymptotically optimal for this
response norm within a factor of $1.000004$, uniformly over $0<\mu<L$.
Every terminal-exact recurrence satisfies
$A_N\ge(2\sqrt{c_{\mathrm{cert}}}-o(1))N^{3/2}\epsN/(L-\mu)$ with
$c_{\mathrm{cert}}>1.14268353624$, whereas the sine--Jacobi recurrence
attains the constant $2\sqrt{\cs}$ with $\cs<1.14269218$.

The proof reduces the gain to the diagonal energy of a Jacobi Green
matrix. A centering lift represents this matrix by an exponential kernel;
packing its diagonal into a density on the line preserves the energy and
controls the change in the Birman--Schwinger spectrum. A calibrated
identity gives a lower bound using the ground parameter and one Fredholm
determinant. Two further determinants and a trace observation improve
the bound through a pointwise inequality on an explicit compact state
region. One rational certificate verifies this inequality by interval
arithmetic. A second certificate constructs a density satisfying the
selected observations and both normalizations with energy below the sine
value. Finally, no $C^1$ function of finitely many observations of the
specified kinds gives a locally sharp lower bound at the sine density.
\end{abstract}

\section{Introduction}\label{sec:introduction}

\subsection{The question}

For a quadratic objective with Hessian $H$, $\mu I\preceq H\preceq LI$,
a first-order method with one gradient per iteration produces errors
$p_N(H)e_0$, where $p_N$ is a polynomial of degree $N$ with $p_N(0)=1$.
The smallest worst-case terminal error over the spectral interval is
\[
 \epsN=\frac1{T_N\bigl((L+\mu)/(L-\mu)\bigr)},
\]
attained by the shifted Chebyshev polynomial $T_N$ \cite[\S2.3]{DST}.
Many methods produce this terminal polynomial: a two-step recurrence
$p_{j+1}(\lambda)=(\alpha_j-\gamma_j\lambda)p_j(\lambda)-\beta_jp_{j-1}(\lambda)$
with $\gamma_j>0$, $\beta_j\ge0$, and $\alpha_j=1+\beta_j$ retains freedom
in its intermediate polynomials $p_1,\dots,p_{N-1}$ even when $p_N$ is
fixed. On a fixed quadratic, every such recurrence produces the same
terminal error $p_N(H)e_0$.

Their responses differ when the matrix at step $j$ changes from $H$ to
$H+B_j$. Hessian variation along a nonquadratic trajectory motivates this
sensitivity question. Part~I
\cite{PartI} measured this difference by the {Hessian-drift gain}
$A_N$, the operator norm of the first-order response of the terminal
residual to independent perturbations $B_0,\dots,B_{N-1}$
(\Cref{def:methods}), and constructed the {sine--Jacobi} recurrence,
for which
\[
 A_N=(2\sqrt{\cs}+o(1))\frac{N^{3/2}\epsN}{L-\mu},\qquad
 \cs=1.1426921798\ldots .
\]
Every terminal-exact recurrence has $A_N\ge2N^{3/2}\epsN/(L-\mu)$
(\Cref{thm:gain} and \eqref{eq:trace-N2}), so the constant inside the square root lies in
$[1,\cs]$ for the best recurrence. This paper shows that the sine--Jacobi
value is optimal to within a relative error of $7.6\cdot10^{-6}$ in this
constant.

\subsection{Main results}

Let $\mathfrak E_N(\mu,L)$ denote the terminal-exact recurrences
(\Cref{def:methods}), and let $\overline{\mathcal Y}_N$ be the set of real symmetric
tridiagonal $N\times N$ matrices with nonnegative off-diagonal entries and
characteristic polynomial $2^{1-N}T_N$. Every terminal-exact recurrence
corresponds to exactly one $J\in\overline{\mathcal Y}_N$ (\Cref{thm:jacobi}). We call
$G=(I-J)^{-1}$ the {Green matrix} of $J$ and $\sum_iG_{ii}^2$ its
{Green energy}.

\begin{theorem}[Green energy of the Chebyshev fiber]\label{thm:energy}
There is an explicit constant $c_{\mathrm{cert}}>1.14268353624$
(\Cref{cor:numerical}) such that
\[
 \liminf_{N\to\infty}\ \inf_{J\in\overline{\mathcal Y}_N}
 \frac1{N^3}\sum_{i=1}^N[(I-J)^{-1}]_{ii}^2\ \ge\ c_{\mathrm{cert}} .
\]
For the sine--Jacobi matrices $J_N^\star$ of \Cref{thm:sine-finite},
\[
 \lim_{N\to\infty}\frac1{N^3}\sum_{i=1}^N[(I-J_N^\star)^{-1}]_{ii}^2=\cs,
 \qquad \cs<1.14269218 .
\]
\end{theorem}

\begin{theorem}[Near-optimality of the sine--Jacobi method]
\label{thm:main-centered}
For every $\varepsilon>0$ there is $N_0$ such that, for all
$0<\mu<L$, $N\ge N_0$ and $\mathcal R\in\mathfrak E_N(\mu,L)$,
\begin{equation}\label{eq:main-A-lower}
 A_N(\mathcal R)\ge(2\sqrt{c_{\mathrm{cert}}}-\varepsilon)
 \frac{N^{3/2}\epsN}{L-\mu},
 \qquad
 A_N(\mathcal R_N^\star)\le(2\sqrt{\cs}+\varepsilon)
 \frac{N^{3/2}\epsN}{L-\mu},
\end{equation}
where $\mathcal R_N^\star$ is the sine--Jacobi recurrence. Consequently
\begin{equation}\label{eq:main-centered-ratio}
 \limsup_{N\to\infty}\ \sup_{0<\mu<L}
 \frac{A_N(\mathcal R_N^\star)}
 {\inf_{\mathcal R\in\mathfrak E_N(\mu,L)}A_N(\mathcal R)}
 \le\sqrt{\frac{\cs}{c_{\mathrm{cert}}}}<1.000004 .
\end{equation}
\end{theorem}

Since the infimum over a subclass is at least the infimum over
$\mathfrak E_N(\mu,L)$, the same bound holds with the infimum restricted
to the irreducible or to the centered recurrences; the sine--Jacobi
recurrence is centered.

\subsection{Structure of the argument}\label{sec:intro-map}

The proof has five steps, shown in \Cref{fig:map}. Three propositions are
computer-assisted, \Cref{ca:transport,ca:sine,ca:primal-second}; each asserts
that explicitly defined numbers lie in printed intervals or that an
explicit continuous function is negative on an explicit compact set. All
other statements are proved analytically. Printed decimal and dyadic
certificate data are exact rationals; displayed approximations are
identified as such.

\begin{figure}[ht]
\centering
\begin{tikzpicture}[
 box/.style={draw=black!55,rounded corners=2pt,align=center,
 text width=5.7cm,minimum height=10mm,inner sep=5pt,font=\small},
 cert/.style={box,fill=blue!8},arr/.style={-{Stealth[length=2mm]},thick},
 node distance=8mm and 9mm]
\node[box] (rec) {Terminal-exact recurrences\\Jacobi Green energy (step 1)};
\node[box,right=of rec] (sin) {Sine--Jacobi matrices\\Legendre kernel and sine density};
\node[box,below=of rec] (pack) {Packed densities\\Uniform spectral transfer (step 2)};
\node[box,below=of pack] (base) {One-determinant calibration\\$\underline{\mathsf e}\ge\mathcal B_*$ (step 3)};
\node[box,below=of base] (dual) {Additional observations and state potential\\Pointwise dual criterion (step 3)};
\node[cert,below=of dual] (lower) {Verified dual certificate\\$c_{\mathrm{cert}}>1.14268353624$ (step 4)};
\node[cert,below=of sin] (upper) {Verified sine energy\\$c_{\sin}<1.14269218$ (step 4)};
\node[box,below=of upper] (ratio) {Asymptotic gain ratio\\$<1.000004$};
\node[cert,below=of ratio] (primal) {Verified feasible density\\$E(m_{\mathrm{feas}})<c_{\sin}$ (step 5)};
\node[box,below=of primal] (limit) {Finite observations\\No locally sharp $C^1$ bound (step 5)};
\draw[arr] (rec)--(pack);
\draw[arr] (pack)--(base);
\draw[arr] (base)--(dual);
\draw[arr] (dual)--(lower);
\draw[arr] (sin)--(upper);
\draw[arr] (upper)--(ratio);
\draw[arr] (lower.east)--++(3mm,0)|-(ratio.west);
\draw[arr] (dual.east)--(primal.west);
\draw[arr] (sin.east)--++(4mm,0)|-(limit.east);
\end{tikzpicture}
\caption{The structure of the proof. Shaded boxes are the
computer-assisted propositions. Arrows indicate dependencies; step numbers
refer to
\Cref{sec:intro-map}.}
\label{fig:map}
\end{figure}

{Step 1: from recurrences to Green energies} (\Cref{sec:reduction-gain}).
A terminal-exact recurrence is the same object as a Jacobi matrix $J$ whose
characteristic polynomial is a multiple of $T_N$, and its gain is at least
$2\epsN(L-\mu)^{-1}$ times $(\sum_iG_{ii}^2)^{1/2}$, $G=(I-J)^{-1}$, with
equality for the sine--Jacobi matrices. The problem becomes: minimize the
Green energy $\E_N(J)=\sum_iG_{ii}^2$ over the fiber $\overline{\mathcal Y}_N$, that is,
bound $\underline{\mathsf e}=\liminf_N\inf_{J\in\overline{\mathcal Y}_N}N^{-3}\E_N(J)$ from below; then
$\mathscr G(\mu,L)=\liminf_N\inf_{\mathcal R}\frac{(L-\mu)A_N(\mathcal R)}{N^{3/2}\epsN}\ge2\sqrt{\underline{\mathsf e}}$.

{Step 2: from matrices to densities} (\Cref{sec:reduction-density}).
For irreducible $J$, a centering lift writes the Green matrix as an
exponential kernel, $G_{ij}=\sqrt{G_{ii}G_{jj}}\,e^{-|y_i-y_j|}$, for points $y_i$ on
the line; reducible $J$ are treated block by block.
Spreading the diagonal over intervals of length $1/G_{ii}$ gives a density
$m_J$ whose energy $\int m_J^{3/2}$ equals $N^{-3}\E_N(J)$ and whose
Birman--Schwinger operator $K_{m_J}$, with kernel
$\sqrt{m(x)}e^{-|x-y|}\sqrt{m(y)}$, has almost the spectrum of $G/N^2$. That
spectrum is fixed by terminal exactness and converges to
$\Lambda_j=2\pi^{-2}(j+\frac12)^{-2}$. These constraints motivate the continuum isospectral problem
\begin{equation}\label{eq:limit-problem}
 \text{minimize } E(m)=\int_\R m^{3/2}
 \quad\text{subject to}\quad
 \int m=\int\sqrt m=1,\quad \lambda_j(K_m)=\Lambda_j\ (j\ge0),
\end{equation}
and the Green diagonals of the sine--Jacobi matrices converge to a
profile whose density $m_*$ on the line satisfies all constraints and has
$E(m_*)=\cs$. We keep five spectral observations
of $m$: the ground parameter $\kappa(m)=2/\lambda_0(K_m)$, three Fredholm
determinants $H_m(s)=\log\det(I+\frac s2K_m)$, and the second trace
$\tr K_m^2$. A lower bound that depends continuously on these observations transfers
back to the matrices uniformly in $J$. The proof uses this direction of
transfer and the explicit sine sequence; it requires neither a recovery
sequence for every continuum density nor equality of the discrete and
continuum infima.

{Step 3: a lower bound by integration along the line}
(\Cref{sec:mechanism}). Along a density, the ground state, the
determinants and the second trace are all governed by first-order
differential equations: the logarithmic derivative of the ground state
gives a coordinate $q\in[0,2]$ with endpoint limits $0$ and $2$, and each
determinant is the integral of a Riccati {filter}. A tangent-line
identity for $n^3$, calibrated by one determinant, writes $E(m)$ as an
explicit function of the observations plus a nonnegative integral. The
one-determinant estimate already gives an asymptotic gain ratio below
$1.000013$. The remaining observations are compared with those of a fixed comparison density
by exact identities. Adding these identities with multipliers, together
with exact derivatives, leaves an integrand that depends only on the
amplitude $n=\sqrt m$ and a four-dimensional local state; the amplitude is
eliminated in closed form, and the lower bound holds as soon as an
explicit function is nonpositive on a compact state region. The data that
make this work form a {dual certificate}.

{Step 4: one certificate} (\Cref{sec:certificates}). We construct one
certificate $\Theta_{\mathrm{cert}}$ by numerical optimization and verify
the pointwise inequality by interval arithmetic
(\Cref{ca:transport}). Its value is $c_{\mathrm{cert}}$. Together with steps 1
and 2 this gives the lower bounds of \Cref{thm:energy,thm:main-centered};
the upper bounds come from the sine density, whose energy is enclosed by a
rapidly converging series (\Cref{ca:sine}).

{Step 5: what finitely many observations can give}
(\Cref{sec:feasible,sec:structural-limits}). A primal certificate
produces a density that satisfies the five observations and both
normalizations and has energy below $\cs$ (\Cref{cor:primal}); hence no
dual certificate built on these five observations, with
$\varsigma=\varsigma_{\mathrm{cal}}$, reaches $\cs$. In a different direction, no
continuously differentiable function of finitely many ground, trace and
determinant observations that equals $E$ at $m_*$ is a lower bound for $E$
near $m_*$ (\Cref{thm:no-touching}).

The appendices contain all proofs, organized along the same steps:
\Cref{app:reduction} for steps 1--2, \Cref{app:sine} for the sine
reference, \Cref{app:states,app:calibration,app:integration} for step 3,
\Cref{app:verification} for step 4, and \Cref{app:primal,app:obstruction} for
step 5.

\subsection{Related work}

Lower bounds of other types concern the iteration complexity of oblivious
first-order methods \cite{ArjevaniShamir} and rate--variance tradeoffs
under additive noise \cite{MRJ}; there the terminal polynomial is not
fixed. The energy $\int m^{3/2}$ is the functional on the right of the
one-dimensional Lieb--Thirring inequality
$\sum_j|e_j|\le L_{1,1}\int V_+^{3/2}$ for the negative eigenvalues $e_j$ of
$-\partial_x^2-V$ \cite{FGL}; in \eqref{eq:limit-problem} the whole
Birman--Schwinger spectrum at one energy is prescribed instead of a moment
of it. The lower bound of step 3 integrates a pointwise inequality along
the line, as in the pointwise dual relaxations of integral variational
problems \cite{CBFG}; sharpness of such relaxations is studied in
\cite{FantuzziTobasco}.

\section{From recurrences to a variational problem on the line}\label{sec:reduction}

The recurrence determines a Jacobi matrix, whose Green diagonal bounds
the gain. We then represent the Green matrix on the line and spread its
diagonal into a density, preserving the energy while controlling the
spectrum. The sine family identifies the reference value for the
resulting variational problem. Proofs are in \Cref{app:reduction,app:sine}.

Throughout, $N\ge2$, $0<\mu<L$, and
\[
 \Delta=L-\mu,\qquad z_0=\frac{L+\mu}{\Delta},\qquad
 z(\lambda)=z_0-\frac{2\lambda}\Delta,\qquad \epsN=\frac1{T_N(z_0)},
\]
where $T_N(\cos\theta)=\cos N\theta$. The affine map $z$ sends $[\mu,L]$
onto $[-1,1]$, and $z(0)=z_0>1$.

\subsection{Recurrences, Jacobi matrices and the gain}\label{sec:reduction-gain}

\begin{definition}[Terminal-exact recurrences and their gain]
\label{def:methods}
A recurrence $\mathcal R$ of length $N$ is a list of coefficients
$(\alpha_j,\beta_j,\gamma_j)_{0\le j<N}$ with
\[
 \gamma_j>0,\qquad \beta_0=0,\qquad \beta_j\ge0,\qquad
 \alpha_j=1+\beta_j .
\]
Its residual polynomials are $p_{-1}=0$, $p_0=1$, and
$p_{j+1}(\lambda)=(\alpha_j-\gamma_j\lambda)p_j(\lambda)-\beta_jp_{j-1}(\lambda)$.
The recurrence is {terminal-exact} on $[\mu,L]$ if
$p_N(\lambda)=\epsN T_N(z(\lambda))$ identically; these recurrences form
the class $\mathfrak E_N(\mu,L)$. Its {irreducible} subclass has
$\beta_j>0$ for $1\le j<N$, and its {centered} subclass consists of
the irreducible recurrences with $\alpha_j=(L+\mu)\gamma_j/2$ for all $j$.

For real $p\times p$ matrices $H_0,\dots,H_{N-1}$, the recurrence acts by
\[
 \mathbf p_{-1}=0,\qquad \mathbf p_0=I,\qquad
 \mathbf p_{j+1}=(\alpha_jI-\gamma_jH_j)\mathbf p_j-\beta_j\mathbf p_{j-1};
\]
with $H_j=H$ for all $j$ this gives $\mathbf p_N=p_N(H)$. For symmetric
$H$ with $\mu I\preceq H\preceq LI$ and real $p\times p$ matrices
$B=(B_0,\dots,B_{N-1})$, let $\mathsf D_{\mathcal R,H}[B]$ be the
derivative of $\mathbf p_N$ at $(H,\dots,H)$ in the direction $B$. The
{Hessian-drift gain} is
\[
 A_N(\mathcal R)=
 \sup_{\substack{p\ge1,\ H=H^T\\\mu I\preceq H\preceq LI}}
 \ \sup_{B\ne0}
 \frac{\|\mathsf D_{\mathcal R,H}[B]\|_{\HS}}
 {(\sum_j\|B_j\|_{\HS}^2)^{1/2}},
 \qquad
 \mathscr G(\mu,L)=\liminf_{N\to\infty}\ \inf_{\mathcal R\in\mathfrak E_N(\mu,L)}
 \frac{\Delta A_N(\mathcal R)}{N^{3/2}\epsN}.
\]
\end{definition}

The normalization $\alpha_j=1+\beta_j$ says $p_j(0)=1$ for all $j$, that
is, the recurrence is a first-order method: for
$f(x)=\frac12(x-x_*)^TH(x-x_*)$ it reads
$x_{j+1}=x_j-\gamma_j\nabla f(x_j)+\beta_j(x_j-x_{j-1})$, a heavy-ball method with
step-dependent coefficients, and $x_j-x_*=p_j(H)(x_0-x_*)$. If the gradient
at step $j$ is computed with the Hessian $H_j$, the errors are
$\mathbf p_j(x_0-x_*)$ with $\mathbf p_j$ as above. The gain is the largest
amplification, to first order, of independent Hessian errors at the $N$
iterations into the terminal residual; all terminal-exact recurrences
have the same terminal residual $p_N(H)$, but not the same gain.

A recurrence is encoded by a Jacobi matrix. Let $\overline{\mathcal Y}_N$ be the set of
real symmetric tridiagonal $N\times N$ matrices with nonnegative
off-diagonal entries and characteristic polynomial $2^{1-N}T_N$, let
$\mathcal Y_N\subset\overline{\mathcal Y}_N$ consist of those with positive off-diagonal
entries, and $\X_N\subset\mathcal Y_N$ of those with zero diagonal. The roots
$\cos((j-\frac12)\pi/N)$ of $T_N$ lie in $(-1,1)$, so $\|J\|<1$ for
$J\in\overline{\mathcal Y}_N$, and we put
\[
 G=(I-J)^{-1},\qquad d_i=G_{ii},\qquad
 \E_N(J)=\sum_{i=1}^Nd_i^2,\qquad
 \underline{\mathsf e}=\liminf_{N\to\infty}\ \inf_{J\in\overline{\mathcal Y}_N}N^{-3}\E_N(J).
\]
We call $G$ the {Green matrix} and $\E_N(J)$ the {Green energy}
of $J$. Write $J_{[a:b]}$ for the principal submatrix on the indices
$a,\dots,b$, and $P_j(z)=\det(zI-J_{[1:j]})$ for $0\le j\le N$, with the
determinant of an empty matrix equal to one.

\begin{theorem}[Recurrences and Jacobi matrices]\label{thm:jacobi}
For $J\in\overline{\mathcal Y}_N$ with diagonal $b_i$ and off-diagonal entries
$J_{j,j+1}=\sqrt{v_j}$, one has $P_j(z_0)>0$ for $0\le j\le N$, and the
coefficients
\[
 \gamma_j=\frac2\Delta\frac{P_j(z_0)}{P_{j+1}(z_0)},\qquad
 \alpha_j=(z_0-b_{j+1})\frac{P_j(z_0)}{P_{j+1}(z_0)},\qquad
 \beta_j=v_j\frac{P_{j-1}(z_0)}{P_{j+1}(z_0)}\quad(\beta_0=0)
\]
define a recurrence $\mathcal R_J\in\mathfrak E_N(\mu,L)$ with residual
polynomials $p_j(\lambda)=P_j(z(\lambda))/P_j(z_0)$. The map
$J\mapsto\mathcal R_J$ is a bijection
$\overline{\mathcal Y}_N\to\mathfrak E_N(\mu,L)$. It maps $\mathcal Y_N$ onto the irreducible
recurrences and $\X_N$ onto the centered ones. All six classes are
nonempty.
\end{theorem}

The residual polynomials of a recurrence are thus the normalized leading
principal characteristic polynomials of a Jacobi matrix, and terminal
exactness fixes only the characteristic polynomial of the whole matrix.
In these terms the gain is bounded below by the Euclidean norm of the
diagonal of the Green matrix.

\begin{theorem}[The gain of a Jacobi matrix]\label{thm:gain}
For $J\in\overline{\mathcal Y}_N$, with $\widehat P_j(z)=\det(zI-J_{[j+2:N]})$ for $0\le j<N$,
\begin{equation}\label{eq:full-gain}
 A_N(\mathcal R_J)
 =\frac2{\Delta P_N(z_0)}
 \max_{x,y\in[-1,1]}
       \Bigl(\sum_{j=0}^{N-1}\widehat P_j(x)^2P_j(y)^2\Bigr)^{1/2}
 \ \ge\ \frac{2\epsN}{\Delta}\sqrt{\E_N(J)},
\end{equation}
and the maximum is attained in dimension $p=2$. Consequently
$\mathscr G(\mu,L)\ge2\sqrt{\underline{\mathsf e}}$.
\end{theorem}

\paragraph{Exact normalization.}
Since $P_N(1)/P_N(z_0)=\epsN$, \eqref{eq:full-gain} gives
\begin{equation}\label{eq:gain-normalized}
 \frac{\Delta A_N(\mathcal R_J)}{N^{3/2}\epsN}
 =\frac{2}{N^{3/2}P_N(1)}
 \max_{x,y\in[-1,1]}
 \Bigl(\sum_{j=0}^{N-1}\widehat P_j(x)^2P_j(y)^2\Bigr)^{1/2}.
\end{equation}
The right-hand side and the Jacobi fiber are independent of $\mu,L$.
The normalized optimal gain is therefore independent of the spectral
interval for every $N$, which also explains the uniformity in
\Cref{thm:main-centered}.

\begin{corollary}[Symmetric perturbations]\label{cor:symmetric-gain}
Let $A_N^{\rm sym}$ be the gain of \Cref{def:methods} with every $B_j$
required to be symmetric. For every $J\in\overline{\mathcal Y}_N$,
\[
 \frac{2\epsN}{\Delta}\sqrt{\E_N(J)}
 \le A_N^{\rm sym}(\mathcal R_J)\le A_N(\mathcal R_J).
\]
Both inequalities are equalities for $J_N^\star$. Consequently,
\Cref{thm:main-centered} holds with $A_N$ replaced by $A_N^{\rm sym}$.
\end{corollary}
\begin{proof}
The upper bound restricts the perturbation class. For the lower bound,
take dimension one and $H=\mu$, so $Z=1$ in \Cref{lem:gain-kernel}.
Choose scalar $B_j$ proportional to $\widehat P_j(1)P_j(1)$, with
$\sum_jB_j^2=1$. The response norm is
$2\epsN\Delta^{-1}\sqrt{\E_N(J)}$ by the cofactor formula for $G$.
\Cref{thm:sine-finite} gives equality for the sine matrix; its upper
bound and \Cref{thm:energy} then give the asymptotic assertion.
\end{proof}

The gain is a first-order sensitivity norm for independently chosen
stepwise perturbations. Relating it to a nonlinear optimization trajectory
requires additional control of the perturbations generated by that trajectory.

The lower bound in \eqref{eq:full-gain} is the value at $x=y=1$: there
$\widehat P_j(1)P_j(1)/P_N(1)$ is the cofactor expression of $G_{j+1,j+1}$.
The trace of $G$ does not depend on $J$: $\tr G=N^2$ on the whole fiber
(\Cref{lem:trace}). The Cauchy--Schwarz inequality therefore gives
\begin{equation}\label{eq:trace-N2}
 \E_N(J)\ge\frac{(\tr G)^2}N=N^3\qquad(J\in\overline{\mathcal Y}_N),
\end{equation}
with equality only for a constant diagonal. This is the elementary bound
$\underline{\mathsf e}\ge1$. The question is how unevenly the fixed trace $N^2$ must be
distributed along the diagonal of $G$. The bounds below use the whole
spectrum of $G$, which is also fixed on the fiber: $\overline{\mathcal Y}_N$ consists of
the Jacobi matrices with eigenvalues $\cos((j-\frac12)\pi/N)$, and $G$ is the
resolvent of $J$ at the point $z=1$, the image of $\lambda=\mu$, just outside
the spectrum.

\subsection{Green matrices as exponential kernels, and densities}
\label{sec:reduction-density}

The eigenvalues of $G$ encode terminal exactness, while its diagonal
determines the energy. An exponential-kernel representation places both
quantities on the line.

\begin{proposition}[Exponential representation]\label{cor:centering-packing}
Let $J$ be a real symmetric tridiagonal $M\times M$ matrix with positive
off-diagonal entries and $\|J\|<1$, $G=(I-J)^{-1}$, $d_i=G_{ii}$. There
are points $y_1<\dots<y_M$ and bounded closed intervals $I_i\ni y_i$
with pairwise disjoint interiors and $|I_i|=1/d_i$ such that
$G_{ij}=\sqrt{d_id_j}\,e^{-|y_i-y_j|}$.
\end{proposition}

For a Jacobi matrix with zero diagonal and positive off-diagonal entries,
the Green matrix factorizes as $\sqrt{d_id_j}\,e^{-|x_i-x_j|}$ with explicit
points determined by the products $J_{i,i+1}G_{i,i+1}$ (\Cref{lem:green}).
A general $J$ is reduced to this case by a {centering lift}: with
the bidiagonal Cholesky factor $\mathsf B\mathsf B^T=(I+J)/2$, the
$2M\times2M$ Jacobi matrix $\widehat J$ with zero diagonal that alternates
the entries of $\mathsf B$ has $2(I-J)^{-1}$ as the block of
$(I-\widehat J)^{-1}$ on the odd indices (\Cref{thm:centering-transfer}). The
points of $\widehat J$ with even indices separate those with odd indices
far enough to place disjoint intervals of length $1/d_i$.

The proposition suggests the following objects. For a measurable $m\ge0$
on $\R$ write $m=n^2$ and
\[
 E(m)=\int_\R m^{3/2},\qquad
 \Mc=\Bigl\{m\in L^\infty:\ m\ge0,\ \supp m\text{ compact},\
 \int m=\int\sqrt m=1\Bigr\}.
\]
For $0\le m\in L^1$, let $K_m$ be the integral operator on $L^2(\R)$ with
kernel
\[
 K_m(x,y)=\sqrt{m(x)}\,e^{-|x-y|}\sqrt{m(y)} .
\]
It is positive and trace class with $\tr K_m=\int m$ (\Cref{lem:features});
its eigenvalues $\lambda_0(K_m)\ge\lambda_1(K_m)\ge\dots\ge0$ are listed with
multiplicity and padded with zeros. Since $\frac12e^{-|x-y|}$ is the kernel
of $(-\partial_x^2+1)^{-1}$, for $\kappa>0$
\[
 \frac\kappa2K_m=(\kappa m)^{1/2}(-\partial_x^2+1)^{-1}(\kappa m)^{1/2}
\]
is the Birman--Schwinger operator of the Schr\"odinger operator
$-\partial_x^2+1-\kappa m$ \cite{SimonTrace}. We use three kinds of spectral
observations of $m$:
\begin{itemize}
\item the {ground parameter} $\kappa(m)=2/\lambda_0(K_m)$, defined when $m$
is not zero almost everywhere; by the Birman--Schwinger principle it is
the largest $\kappa$ for which $-\partial_x^2+1-\kappa m\ge0$
(\Cref{thm:ground});
\item the {determinants}
\[
 H_m(s)=\log\det\Bigl(I+\frac s2K_m\Bigr)
 =\sum_{j\ge0}\log\Bigl(1+\frac s2\lambda_j(K_m)\Bigr)\qquad(s\ge0),
\]
where, for a positive trace-class operator $A$,
$\det(I+A)=\prod_j(1+\lambda_j(A))$ is the Fredholm determinant
\cite[Chapter~3]{SimonTrace}; $H_m(s)$ is the logarithm of the
perturbation determinant of $-\partial_x^2+1+sm$ relative to
$-\partial_x^2+1$;
\item the {second trace} $\tr K_m^2=\iint m(x)m(y)e^{-2|x-y|}\dd x\dd y$.
\end{itemize}

\begin{theorem}[Packing]\label{thm:packing}
For every $J\in\overline{\mathcal Y}_N$ there are bounded closed intervals $I_i$ with
disjoint interiors and lengths $|I_i|=1/d_i$ such that the density
\[
 m_J=\sum_{i=1}^N\frac{d_i^2}{N^2}\ind_{I_i}
\]
belongs to $\Mc$, satisfies $E(m_J)=N^{-3}\E_N(J)$, and
\begin{equation}\label{eq:block-spectral-error}
 \sum_{j\ge0}\bigl|\lambda_j(K_{m_J})-\lambda_j(G/N^2)\bigr|
 \le\frac2{\sqrt N}+\frac1N .
\end{equation}
\end{theorem}

For each $J\in\overline{\mathcal Y}_N$ we fix one choice of intervals as in the proof of
\Cref{thm:packing}, and $m_J$ denotes the resulting density. The block
heights are chosen so that the energy is exact and both normalizations
hold: $\int m_J=\tr G/N^2=1$ and $\int\sqrt{m_J}=1$. For irreducible $J$, the
point masses $(d_i/N^2)\delta_{y_i}$ at the points of
\Cref{cor:centering-packing} have exactly the nonzero spectrum of $G/N^2$
(\Cref{lem:features}(iii)); the error in \eqref{eq:block-spectral-error} comes
from spreading them into blocks of width $1/d_i$ (the term $2/\sqrt N$) and,
for reducible $J$, from the coupling of the separated irreducible blocks
(the term $1/N$). The spectrum of
$G/N^2$ is the same for every $J\in\overline{\mathcal Y}_N$, namely
$\{(N^2(1-\cos((j-\frac12)\pi/N)))^{-1}\}_{j=1}^N$ (\Cref{lem:trace}), and it
converges.

\begin{theorem}[Uniform spectral transfer]\label{thm:transfer}
Put
\[
 k=\frac{\pi^2}4,\qquad \Lambda_j=\frac2{\pi^2(j+\frac12)^2}\quad(j\ge0),
 \qquad H_*(s)=\log\cosh\sqrt s .
\]
Then $\sum_j\Lambda_j=1$, $\sum_j\Lambda_j^2=\frac23$,
$\sum_j\log(1+s\Lambda_j/2)=H_*(s)$, and for every $J\in\overline{\mathcal Y}_N$
\begin{equation}\label{eq:deltaN}
 \sum_{j\ge0}|\lambda_j(K_{m_J})-\Lambda_j|
 \le\delta_N:=\frac2{\sqrt N}+\frac1N+\frac8{\pi^2(2N-1)} .
\end{equation}
Consequently, uniformly in $J\in\overline{\mathcal Y}_N$ and $s\ge0$,
\[
 |H_{m_J}(s)-H_*(s)|\le\frac s2\delta_N,\qquad
 \bigl|\tr K_{m_J}^2-\tfrac23\bigr|\le2\delta_N,\qquad
 |\kappa(m_J)-k|\le\frac{4\delta_N}{\Lambda_0^2}\quad
 (\delta_N<\Lambda_0/2).
\]
\end{theorem}

Thus every packed density has, up to $O(N^{-1/2})$ uniformly in $J$, the
spectrum $(\Lambda_j)$, and its energy is the normalized Green energy.
The associated continuum constraints are those of \eqref{eq:limit-problem}.
By the Birman--Schwinger
principle its constraints are a spectral condition on a weighted
Schr\"odinger problem: if $K_mf=\lambda f$ with $\lambda>0$, then
$u=\int e^{-|\cdot-y|}\sqrt{m(y)}f(y)\dd y$ solves $-u''+u=\frac2\lambda mu$, and
conversely, if $u\in H^1(\R)$ solves $-u''+u=\mu mu$ with $mu\ne0$, then
$f=\sqrt m\,u$ satisfies $K_mf=\frac2\mu f$. So \eqref{eq:limit-problem}
minimizes $\int m^{3/2}$ over the densities for which $-u''+u=\mu mu$ has
exactly the eigenvalues $\mu_j=\pi^2(j+\frac12)^2$, $j\ge0$, all simple.

\subsection{The sine density}\label{sec:reference}

\begin{theorem}[The sine--Jacobi matrices]\label{thm:sine-finite}

Let $J_N^\star$ be the symmetric tridiagonal matrix with zero diagonal
and off-diagonal entries
\[
 J_{i,i+1}^\star=\frac{\sin(i\pi/N)}
 {2\sqrt{\sin((i-\frac12)\pi/N)\,\sin((i+\frac12)\pi/N)}}
 \qquad(1\le i<N).
\]
Then $J_N^\star\in\X_N$, and the recurrence
$\mathcal R_N^\star=\mathcal R_{J_N^\star}$ is centered and satisfies
\begin{equation}\label{eq:sine-gain}
 A_N(\mathcal R_N^\star)=\frac{2\epsN}{\Delta}\sqrt{\E_N(J_N^\star)}
 \qquad(0<\mu<L).
\end{equation}
The rescaled Green diagonal $D_N(t)=N^{-1}d_i$ on
$[(i-1)/N,i/N)$ converges in $L^2(0,1)$ to
\[
 \phi_*(t)=\frac4{\pi^2}\sin(\pi t)\,
 \mathbf K\bigl(\sin^2\tfrac{\pi t}2\bigr)\mathbf K\bigl(\cos^2\tfrac{\pi t}2\bigr),
 \qquad
 \mathbf K(s)=\int_0^{\pi/2}\frac{\dd\theta}{\sqrt{1-s\sin^2\theta}},
\]
so that $N^{-3}\E_N(J_N^\star)\to\cs:=\int_0^1\phi_*^2$ and
$\mathscr G(\mu,L)\le2\sqrt{\cs}$.
\end{theorem}

The limiting Green matrix is the kernel $G_*$ of the no-logarithm
self-adjoint realization of $-\partial_t^2-\pi^2/(4\sin^2\pi t)$ on $(0,1)$,
defined precisely in \Cref{lem:legendre}. Its
eigenfunctions are $\sqrt{\sin\pi t}\,\mathsf P_j(\cos\pi t)$, with the Legendre
polynomials $\mathsf P_j$ and eigenvalues $\pi^2(j+\frac12)^2=2/\Lambda_j$
(\Cref{lem:legendre}), and $\phi_*(t)=2G_*(t,t)$. The change of variables
$\dd x=\dd t/\phi_*(t)$, the continuum analogue of the block widths $1/d_i$,
gives the density of the sine--Jacobi family (\Cref{fig:sine}).

\begin{definition}[The sine density]\label{def:sine}
The function $\phi_*$ is continuous and positive on $(0,1)$. Let
$X_*:(0,1)\to\R$ be the strictly increasing function with $X_*'=1/\phi_*$
and $X_*(\frac12)=0$. Put $m_*(x)=\phi_*(t)^2$ if $x=X_*(t)$ for some
$t\in(0,1)$, $m_*(x)=0$ otherwise, and $n_*=\sqrt{m_*}$.
\end{definition}

\begin{proposition}[Properties of the sine density]\label{prop:target}
The map $X_*$ is an increasing bijection of $(0,1)$ onto $\R$ with
$X_*(1-t)=-X_*(t)$, and
\[
 \int m_*=\int\sqrt{m_*}=1,\qquad E(m_*)=\cs,\qquad
 \lambda_j(K_{m_*})=\Lambda_j\quad(j\ge0).
\]
With $\Sigma_1(a)=\sum_{j\ge0}a^{j(j+1)}$ and
$\Sigma_2(a)=1+2\sum_{j\ge1}(-1)^ja^{j^2}$,
\begin{equation}\label{eq:sine-theta}
 n_*(x)=8\varpi\,e^{-\pi\varpi/2}\,\Sigma_1(e^{-\pi\varpi})^2\,
 \Sigma_2(e^{-\pi\varpi})^2,\qquad \varpi=e^{2|x|} .
\end{equation}
In particular $n_*(x)=\exp(-\frac\pi2e^{2|x|}+O(|x|))$ as
$|x|\to\infty$.
\end{proposition}

\begin{figure}[ht]
\centering
\begin{tikzpicture}
\begin{axis}[width=.48\linewidth,height=5.1cm,grid=major,grid style={black!10},tick label style={font=\scriptsize},label style={font=\small},legend style={font=\scriptsize,draw=none,fill=white,at={(.5,1.02)},anchor=south,legend columns=3},xmin=-1.4,xmax=1.4,ymin=0,ymax=2.2,xlabel={$x$},ylabel={$m(x)$}]
\addplot[orange!85!black,thin] coordinates {
(-0.67801695,0) (-0.67801695,0.29665255) (-0.37201464,0.29665255) (-0.37201464,0) (-0.28898667,0) (-0.28898667,1.225086) (-0.13840731,1.225086) (-0.13840731,0) (-0.12866237,0) (-0.12866237,1.8184678) (-0.0050687866,1.8184678) (-0.0050687866,0) (-0.003537878,0) (-0.003537878,1.8184678) (0.12005571,1.8184678) (0.12005571,0) (0.12025918,0) (0.12025918,1.225086) (0.27083855,1.225086) (0.27083855,0) (0.37201464,0) (0.37201464,0.29665255) (0.67801695,0.29665255) (0.67801695,0)
};
\addlegendentry{$m_{J_6^\star}$}
\addplot[blue!65,thin] coordinates {
(-0.76030298,0) (-0.76030298,0.038134805) (-0.54693561,0.038134805) (-0.54693561,0) (-0.49052834,0) (-0.49052834,0.21176943) (-0.39998487,0.21176943) (-0.39998487,0) (-0.39175084,0) (-0.39175084,0.43668913) (-0.3286983,0.43668913) (-0.3286983,0) (-0.32687053,0) (-0.32687053,0.67928215) (-0.27631558,0.67928215) (-0.27631558,0) (-0.2756255,0) (-0.2756255,0.9216581) (-0.23222409,0.9216581) (-0.23222409,0) (-0.23189273,0) (-0.23189273,1.1520638) (-0.19307318,1.1520638) (-0.19307318,0) (-0.19289119,0) (-0.19289119,1.3620657) (-0.15718942,1.3620657) (-0.15718942,0) (-0.15708092,0) (-0.15708092,1.5453618) (-0.12356326,1.5453618) (-0.12356326,0) (-0.12349524,0) (-0.12349524,1.6971871) (-0.091511892,1.6971871) (-0.091511892,0) (-0.091468035,0) (-0.091468035,1.8139844) (-0.060531478,1.8139844) (-0.060531478,0) (-0.06050294,0) (-0.06050294,1.893212) (-0.030220621,1.893212) (-0.030220621,0) (-0.030202225,0) (-0.030202225,1.9332287) (-0.00023495769,1.9332287) (-0.00023495769,0) (-0.00022344903,0) (-0.00022344903,1.9332287) (0.029743818,1.9332287) (0.029743818,0) (0.029750664,0) (0.029750664,1.893212) (0.060032983,1.893212) (0.060032983,0) (0.060036889,0) (0.060036889,1.8139844) (0.090973447,1.8139844) (0.090973447,0) (0.090976029,0) (0.090976029,1.6971871) (0.12295938,1.6971871) (0.12295938,0) (0.12296258,0) (0.12296258,1.5453618) (0.15648023,1.5453618) (0.15648023,0) (0.15648702,0) (0.15648702,1.3620657) (0.19218879,1.3620657) (0.19218879,0) (0.19220462,0) (0.19220462,1.1520638) (0.23102417,1.1520638) (0.23102417,0) (0.23106074,0) (0.23106074,0.9216581) (0.27446215,0.9216581) (0.27446215,0) (0.27454835,0) (0.27454835,0.67928215) (0.3251033,0.67928215) (0.3251033,0) (0.32532566,0) (0.32532566,0.43668913) (0.3883782,0.43668913) (0.3883782,0) (0.38908574,0) (0.38908574,0.21176943) (0.47962921,0.21176943) (0.47962921,0) (0.54693561,0) (0.54693561,0.038134805) (0.76030298,0.038134805) (0.76030298,0)
};
\addlegendentry{$m_{J_{24}^\star}$}
\addplot[black,thick] coordinates {
(-1.35,7.0764601e-17) (-1.335,2.6531389e-16) (-1.32,9.5492925e-16) (-1.305,3.3034958e-15) (-1.29,1.0997051e-14) (-1.275,3.5267305e-14) (-1.26,1.090789e-13) (-1.245,3.257216e-13) (-1.23,9.4002788e-13) (-1.215,2.6245874e-12) (-1.2,7.0963053e-12) (-1.185,1.8598021e-11) (-1.17,4.7289471e-11) (-1.155,1.1676527e-10) (-1.14,2.8021491e-10) (-1.125,6.5412796e-10) (-1.11,1.4865618e-09) (-1.095,3.2915186e-09) (-1.08,7.1061923e-09) (-1.065,1.4970254e-08) (-1.05,3.0795529e-08) (-1.035,6.1903878e-08) (-1.02,1.2167914e-07) (-1.005,2.3402922e-07) (-0.99,4.4071604e-07) (-0.975,8.1311752e-07) (-0.96,1.4706737e-06) (-0.945,2.6091738e-06) (-0.93,4.5431899e-06) (-0.915,7.7683756e-06) (-0.9,1.3051011e-05) (-0.885,2.1554051e-05) (-0.87,3.5010964e-05) (-0.855,5.5960644e-05) (-0.84,8.8058559e-05) (-0.825,0.00013648073) (-0.81,0.00020843791) (-0.795,0.00031381712) (-0.78,0.00046596625) (-0.765,0.00068263429) (-0.75,0.00098707501) (-0.735,0.0014093152) (-0.72,0.0019875793) (-0.705,0.0027698532) (-0.69,0.0038155561) (-0.675,0.0051972778) (-0.66,0.007002525) (-0.645,0.009335407) (-0.63,0.01231818) (-0.615,0.016092557) (-0.6,0.020820691) (-0.585,0.026685718) (-0.57,0.033891773) (-0.555,0.042663372) (-0.54,0.053244076) (-0.525,0.065894361) (-0.51,0.080888648) (-0.495,0.098511438) (-0.48,0.11905256) (-0.465,0.14280155) (-0.45,0.17004118) (-0.435,0.20104024) (-0.42,0.23604569) (-0.405,0.27527422) (-0.39,0.31890361) (-0.375,0.36706376) (-0.36,0.41982789) (-0.345,0.4772041) (-0.33,0.53912735) (-0.315,0.60545249) (-0.3,0.67594832) (-0.285,0.75029315) (-0.27,0.82807208) (-0.255,0.90877633) (-0.24,0.9918047) (-0.225,1.0764676) (-0.21,1.1619936) (-0.195,1.2475384) (-0.18,1.3321967) (-0.165,1.4150162) (-0.15,1.4950144) (-0.135,1.5711964) (-0.12,1.6425751) (-0.105,1.7081918) (-0.09,1.7671374) (-0.075,1.818573) (-0.06,1.8617503) (-0.045,1.8960291) (-0.03,1.920894) (-0.015,1.935967) (0,1.9410172) (0.015,1.935967) (0.03,1.920894) (0.045,1.8960291) (0.06,1.8617503) (0.075,1.818573) (0.09,1.7671374) (0.105,1.7081918) (0.12,1.6425751) (0.135,1.5711964) (0.15,1.4950144) (0.165,1.4150162) (0.18,1.3321967) (0.195,1.2475384) (0.21,1.1619936) (0.225,1.0764676) (0.24,0.9918047) (0.255,0.90877633) (0.27,0.82807208) (0.285,0.75029315) (0.3,0.67594832) (0.315,0.60545249) (0.33,0.53912735) (0.345,0.4772041) (0.36,0.41982789) (0.375,0.36706376) (0.39,0.31890361) (0.405,0.27527422) (0.42,0.23604569) (0.435,0.20104024) (0.45,0.17004118) (0.465,0.14280155) (0.48,0.11905256) (0.495,0.098511438) (0.51,0.080888648) (0.525,0.065894361) (0.54,0.053244076) (0.555,0.042663372) (0.57,0.033891773) (0.585,0.026685718) (0.6,0.020820691) (0.615,0.016092557) (0.63,0.01231818) (0.645,0.009335407) (0.66,0.007002525) (0.675,0.0051972778) (0.69,0.0038155561) (0.705,0.0027698532) (0.72,0.0019875793) (0.735,0.0014093152) (0.75,0.00098707501) (0.765,0.00068263429) (0.78,0.00046596625) (0.795,0.00031381712) (0.81,0.00020843791) (0.825,0.00013648073) (0.84,8.8058559e-05) (0.855,5.5960644e-05) (0.87,3.5010964e-05) (0.885,2.1554051e-05) (0.9,1.3051011e-05) (0.915,7.7683756e-06) (0.93,4.5431899e-06) (0.945,2.6091738e-06) (0.96,1.4706737e-06) (0.975,8.1311752e-07) (0.99,4.4071604e-07) (1.005,2.3402922e-07) (1.02,1.2167914e-07) (1.035,6.1903878e-08) (1.05,3.0795529e-08) (1.065,1.4970254e-08) (1.08,7.1061923e-09) (1.095,3.2915186e-09) (1.11,1.4865618e-09) (1.125,6.5412796e-10) (1.14,2.8021491e-10) (1.155,1.1676527e-10) (1.17,4.7289471e-11) (1.185,1.8598021e-11) (1.2,7.0963053e-12) (1.215,2.6245874e-12) (1.23,9.4002788e-13) (1.245,3.257216e-13) (1.26,1.090789e-13) (1.275,3.5267305e-14) (1.29,1.0997051e-14) (1.305,3.3034958e-15) (1.32,9.5492925e-16) (1.335,2.6531389e-16) (1.35,7.0764601e-17)
};
\addlegendentry{$m_*$}
\end{axis}
\end{tikzpicture}\hfill
\begin{tikzpicture}
\begin{axis}[width=.48\linewidth,height=5.1cm,grid=major,grid style={black!10},tick label style={font=\scriptsize},label style={font=\small},legend style={font=\scriptsize,draw=none,fill=white,at={(.5,1.02)},anchor=south,legend columns=3},xmin=0,xmax=1,ymin=0,ymax=1.55,xlabel={$t$},ylabel={$D_N(t)$}]
\addplot[orange!85!black,thin] coordinates {
(0,0.5446582) (0.16666667,0.5446582) (0.16666667,1.106836) (0.33333333,1.106836) (0.33333333,1.3485058) (0.5,1.3485058) (0.5,1.3485058) (0.66666667,1.3485058) (0.66666667,1.106836) (0.83333333,1.106836) (0.83333333,0.5446582) (1,0.5446582)
};
\addlegendentry{$N=6$}
\addplot[blue!65,thin] coordinates {
(0,0.19528135) (0.041666667,0.19528135) (0.041666667,0.46018412) (0.083333333,0.46018412) (0.083333333,0.66082458) (0.125,0.66082458) (0.125,0.82418575) (0.16666667,0.82418575) (0.16666667,0.96003026) (0.20833333,0.96003026) (0.20833333,1.0733423) (0.25,1.0733423) (0.25,1.1670757) (0.29166667,1.1670757) (0.29166667,1.2431258) (0.33333333,1.2431258) (0.33333333,1.3027613) (0.375,1.3027613) (0.375,1.3468424) (0.41666667,1.3468424) (0.41666667,1.3759404) (0.45833333,1.3759404) (0.45833333,1.390406) (0.5,1.390406) (0.5,1.390406) (0.54166667,1.390406) (0.54166667,1.3759404) (0.58333333,1.3759404) (0.58333333,1.3468424) (0.625,1.3468424) (0.625,1.3027613) (0.66666667,1.3027613) (0.66666667,1.2431258) (0.70833333,1.2431258) (0.70833333,1.1670757) (0.75,1.1670757) (0.75,1.0733423) (0.79166667,1.0733423) (0.79166667,0.96003026) (0.83333333,0.96003026) (0.83333333,0.82418575) (0.875,0.82418575) (0.875,0.66082458) (0.91666667,0.66082458) (0.91666667,0.46018412) (0.95833333,0.46018412) (0.95833333,0.19528135) (1,0.19528135)
};
\addlegendentry{$N=24$}
\addplot[black,thick] coordinates {
(0,0) (0.00625,0.075122261) (0.0125,0.13290911) (0.01875,0.1841434) (0.025,0.2311137) (0.03125,0.27490494) (0.0375,0.31615473) (0.04375,0.35528332) (0.05,0.39258857) (0.05625,0.42829247) (0.0625,0.46256676) (0.06875,0.49554815) (0.075,0.52734794) (0.08125,0.55805844) (0.0875,0.58775737) (0.09375,0.61651104) (0.1,0.64437661) (0.10625,0.67140384) (0.1125,0.69763642) (0.11875,0.72311295) (0.125,0.74786779) (0.13125,0.7719317) (0.1375,0.79533231) (0.14375,0.81809461) (0.15,0.84024128) (0.15625,0.86179297) (0.1625,0.88276854) (0.16875,0.90318529) (0.175,0.92305914) (0.18125,0.94240474) (0.1875,0.96123565) (0.19375,0.97956441) (0.2,0.99740268) (0.20625,1.0147613) (0.2125,1.0316503) (0.21875,1.0480792) (0.225,1.0640567) (0.23125,1.0795911) (0.2375,1.0946901) (0.24375,1.1093608) (0.25,1.1236102) (0.25625,1.1374445) (0.2625,1.1508698) (0.26875,1.1638918) (0.275,1.1765157) (0.28125,1.1887466) (0.2875,1.2005892) (0.29375,1.2120479) (0.3,1.2231271) (0.30625,1.2338306) (0.3125,1.2441622) (0.31875,1.2541254) (0.325,1.2637237) (0.33125,1.2729602) (0.3375,1.2818378) (0.34375,1.2903593) (0.35,1.2985275) (0.35625,1.3063448) (0.3625,1.3138136) (0.36875,1.3209361) (0.375,1.3277143) (0.38125,1.3341503) (0.3875,1.3402459) (0.39375,1.3460027) (0.4,1.3514225) (0.40625,1.3565066) (0.4125,1.3612565) (0.41875,1.3656734) (0.425,1.3697585) (0.43125,1.3735129) (0.4375,1.3769377) (0.44375,1.3800336) (0.45,1.3828015) (0.45625,1.3852421) (0.4625,1.387356) (0.46875,1.3891438) (0.475,1.390606) (0.48125,1.3917428) (0.4875,1.3925546) (0.49375,1.3930416) (0.5,1.3932039) (0.50625,1.3930416) (0.5125,1.3925546) (0.51875,1.3917428) (0.525,1.390606) (0.53125,1.3891438) (0.5375,1.387356) (0.54375,1.3852421) (0.55,1.3828015) (0.55625,1.3800336) (0.5625,1.3769377) (0.56875,1.3735129) (0.575,1.3697585) (0.58125,1.3656734) (0.5875,1.3612565) (0.59375,1.3565066) (0.6,1.3514225) (0.60625,1.3460027) (0.6125,1.3402459) (0.61875,1.3341503) (0.625,1.3277143) (0.63125,1.3209361) (0.6375,1.3138136) (0.64375,1.3063448) (0.65,1.2985275) (0.65625,1.2903593) (0.6625,1.2818378) (0.66875,1.2729602) (0.675,1.2637237) (0.68125,1.2541254) (0.6875,1.2441622) (0.69375,1.2338306) (0.7,1.2231271) (0.70625,1.2120479) (0.7125,1.2005892) (0.71875,1.1887466) (0.725,1.1765157) (0.73125,1.1638918) (0.7375,1.1508698) (0.74375,1.1374445) (0.75,1.1236102) (0.75625,1.1093608) (0.7625,1.0946901) (0.76875,1.0795911) (0.775,1.0640567) (0.78125,1.0480792) (0.7875,1.0316503) (0.79375,1.0147613) (0.8,0.99740268) (0.80625,0.97956441) (0.8125,0.96123565) (0.81875,0.94240474) (0.825,0.92305914) (0.83125,0.90318529) (0.8375,0.88276854) (0.84375,0.86179297) (0.85,0.84024128) (0.85625,0.81809461) (0.8625,0.79533231) (0.86875,0.7719317) (0.875,0.74786779) (0.88125,0.72311295) (0.8875,0.69763642) (0.89375,0.67140384) (0.9,0.64437661) (0.90625,0.61651104) (0.9125,0.58775737) (0.91875,0.55805844) (0.925,0.52734794) (0.93125,0.49554815) (0.9375,0.46256676) (0.94375,0.42829247) (0.95,0.39258857) (0.95625,0.35528332) (0.9625,0.31615473) (0.96875,0.27490494) (0.975,0.2311137) (0.98125,0.1841434) (0.9875,0.13290911) (0.99375,0.075122261) (1,0)
};
\addlegendentry{$\phi_*$}
\end{axis}
\end{tikzpicture}
\caption{Left: the packed densities $m_J$ of \Cref{thm:packing} for
$J=J_6^\star$ and $J=J_{24}^\star$, translated to center the extreme
kernel points, and the sine density $m_*$.
Right: the rescaled Green diagonals $D_N$ of $J_N^\star$ and their limit
$\phi_*$.}
\label{fig:sine}
\end{figure}

Thus $m_*$ satisfies every constraint of \eqref{eq:limit-problem} and has
energy $\cs$. To bound other feasible energies from below, we retain a
finite set of spectral observations.

\subsection{The relaxation}\label{sec:relaxation}

We keep the ground parameter, three determinants and the second trace.
Let $\mathcal M^{1}_c$ be the set of bounded compactly supported $m\ge0$ with
$\int m=1$. Fix the scales $s_0=32$ and $s_1=64$ (chosen in
\Cref{sec:construction}).

\begin{definition}[Observations and the relaxed class]\label{def:classes}
For $\varsigma>0$ and $m\in\mathcal M^1_c$ with ground parameter $\kappa=\kappa(m)$,
the {observation vector} is
\[
 \mathbf o_\varsigma(m)=\bigl(H_m(s_0),\,H_m(s_1),\,H_m(\varsigma\kappa),\,
 \tr K_m^2\bigr),
\]
and its target is
$\mathbf o_\varsigma^*=(H_*(s_0),H_*(s_1),H_*(\varsigma k),\frac23)$. The relaxed
class is
\[
 \mathcal A_{\mathrm{obs}}(\varsigma)=\{m\in\mathcal M^{1}_c:\ \kappa(m)=k,\
 \mathbf o_\varsigma(m)=\mathbf o_\varsigma^*\}.
\]
\end{definition}

By \Cref{prop:target}, $m_*$ has $\kappa(m_*)=k$ and the target observations,
although it is not compactly supported. The root mass $\int\sqrt m$ is
not part of the relaxation; the lower bound does not use it. The third
determinant is taken at a scale proportional to $\kappa(m)$, because it
enters through a filter of scale $\varsigma\kappa$ in the calibration of
\Cref{sec:calibration}; the number $\varsigma$ is chosen in
\Cref{sec:construction}. By \Cref{thm:transfer}, the packed densities $m_J$
satisfy these conditions up to $O(\delta_N)$, uniformly in $J$; a lower
bound for $E$ on $\mathcal A_{\mathrm{obs}}(\varsigma)$ that is stable under such
perturbations is a lower bound for $\underline{\mathsf e}$ (\Cref{thm:dual}).

\section{A lower bound by integration along the line}\label{sec:mechanism}

The ground state and the determinant filters turn spectral data into
local differential identities. We first combine these identities into a
calibration that bounds the energy using one determinant. The remaining
observations then strengthen the estimate: comparison identities have
prescribed integrals, while a state potential contributes a total
derivative. After adding these terms, we minimize explicitly over the
amplitude $n=\sqrt m$. The resulting inequality depends only on a state
in a compact region and can be verified by a finite computation.
Proofs are in \Cref{app:states,app:calibration,app:integration}.

\subsection{The ground coordinate}\label{sec:ground-coordinate}

\begin{theorem}[Ground state]\label{thm:ground}
Let $m\ge0$ be integrable and not zero, and $\kappa=\kappa(m)$. The
eigenvalue $\lambda_0(K_m)$ is simple, and there is $u>0$,
$u\in H^1(\R)\cap C^1(\R)$ with $u'$ locally absolutely continuous and
$|u'|\le u$, such that $-u''+u=\kappa mu$ almost everywhere. For every
$f\in H^1(\R)$,
\begin{equation}\label{eq:ground-square}
 \int(f'^2+f^2)-\kappa\int mf^2=\int\Bigl(f'-\frac{u'}uf\Bigr)^2,
\end{equation}
whose right-hand side vanishes exactly for multiples of $u$. Every positive
$\tilde u\in H^1(\R)\cap C^1(\R)$ with $\tilde u'$ locally absolutely continuous
and $-\tilde u''+\tilde u=\kappa m\tilde u$ almost everywhere is a positive
multiple of $u$. The {ground coordinate} $q=1-u'/u$, which therefore
depends only on $m$, is locally absolutely continuous and satisfies
\[
 0\le q\le2,\qquad q'=\kappa m-Q(q)\ \text{a.e.},\qquad Q(q)=q(2-q),
 \qquad q(-\infty)=0,\quad q(+\infty)=2,
\]
$Q(q),q'\in L^1(\R)$ and $\int Q(q)\dd x=\kappa\int m-2$. If
$\supp m\subset[x_-,x_+]$, then $q=0$ on $(-\infty,x_-]$ and $q=2$ on
$[x_+,\infty)$; if moreover $m$ is bounded, $q$ is Lipschitz. In particular,
$\kappa(m)=\max\{\kappa'>0:\int(f'^2+f^2)\ge\kappa'\int mf^2\ \text{for all }f\in H^1(\R)\}$.
\end{theorem}

By \eqref{eq:ground-square}, $\kappa(m)$ is the largest coupling for
which $-\partial_x^2+1-\kappa m\ge0$, and $u$ is its zero-energy ground
state. The coordinate $q$ has endpoint limits $0$ and $2$, but need not
be monotone: where $m=0$, $q'=-Q(q)\le0$. The argument uses the chain
rule along $q$, not a change of variables through an inverse. We use
$q^{-1}$ only for the reference, comparison, and prescribed-speed
densities, whose strict monotonicity is established separately.

First, {every function of $q$ has a known integral against $q'$}:
for a bounded $\eta$ that is continuous outside a finite set,
\begin{equation}\label{eq:price-identity}
 \int_\R\eta(q)\,q'\dd x=\int_0^2\eta(p)\dd p
\end{equation}
for every density (proof of \Cref{prop:certificate-mechanism}). Since
$q'=\kappa n^2-Q(q)$, this turns $\int\eta(q)n^2$ into a constant plus
$\kappa^{-1}\int\eta(q)Q(q)$. A lower bound may therefore pay for a
pointwise defect by a term $\eta(q)q'$; we call $\eta$ a {price}.

Second, the integral of $Q(q)$ is fixed: $\int Q(q)\dd x=\kappa\int m-2=\kappa-2$ for
$\int m=1$ (\Cref{thm:ground}). We call this the {mass identity}; its
multiplier $\tau$ enters every certificate.

\subsection{Filters and the spectral observations}\label{sec:filters}

\begin{proposition}[Filters]\label{lem:filter}
Let $m\ge0$ be integrable and not zero, with $M=\int m$, ground
parameter $\kappa$ and ground coordinate $q$.
\begin{enumerate}
\item The linear filter $r_\circ(x)=\kappa\int_{-\infty}^xe^{-2(x-y)}m(y)\dd y$ is
absolutely continuous, $r_\circ'=\kappa m-2r_\circ$, $r_\circ(\pm\infty)=0$,
$\|r_\circ'\|_1\le2\kappa M$, and $0\le r_\circ\le q$. The function $h=q-r_\circ$
satisfies $h'=q^2-2h$, $0\le h\le q$, and
\[
 S(m):=\int_\R r_\circ^2=\frac{\kappa^2}4\tr K_m^2 .
\]
More generally, for $c>0$ and $0\le\tilde m\in L^1(\R)$,
\[
 \int_\R\Bigl(c\int_{-\infty}^xe^{-2(x-y)}\tilde m(y)\dd y\Bigr)^2\dd x
 =\frac{c^2}4\iint e^{-2|y-z|}\tilde m(y)\tilde m(z)\dd y\dd z .
\]
\item For $s\ge0$ there is a unique bounded locally absolutely continuous
solution $r_s$ of $r_s'=sm-2r_s-r_s^2$ with $r_s(x)\to0$ as $x\to-\infty$.
It satisfies
\[
 0\le r_s\le\frac s\kappa r_\circ\le\frac s\kappa q,\qquad r_s(+\infty)=0,\qquad
 \int_\R r_s=H_m(s),\qquad \|r_s'\|_1\le sM+(2+\|r_s\|_\infty)H_m(s).
\]
The function $V_s=\exp\int_{-\infty}^xr_s$ is the unique bounded solution of
$V_s(x)=1+\frac s2\int_{-\infty}^x(1-e^{-2(x-y)})m(y)V_s(y)\dd y$; it is $C^1$,
$V_s'$ is locally absolutely continuous, $V_s''+2V_s'=smV_s$ almost
everywhere, $V_s(-\infty)=1$ and $V_s(+\infty)=\exp H_m(s)$.
\end{enumerate}
\end{proposition}

The function $y=e^xV_s$ is the solution of $-y''+(1+sm)y=0$ that is
asymptotic to $e^x$ at $-\infty$ (equal to it to the left of the support if $m$
has compact support), and $r_s=y'/y-1$; to the right of the support,
$y=e^{H_m(s)}e^x+\beta e^{-x}$. The determinant is therefore the
coefficient of the growing solution, as expressed by the Fredholm--Jost
identity. Its logarithm is the integral of the {filter} $r_s$. Filters are causal: $r_s(x)$ depends only on $m$ on $(-\infty,x]$.
The linear filter $r_\circ$ is the limit $\kappa r_s/s$ as $s\downarrow0$, and it
carries the second trace. Every determinant observation, and the second
trace, is thus an integral of a local quantity that solves a first-order
equation driven by $n^2$.

\subsection{The tangent identity and the calibration}\label{sec:calibration}

Two auxiliary densities appear below. The {reference density} $m_\circ$,
built from a profile $b_\circ$, serves only to construct the potential $g_\circ$
of the calibration (\Cref{thm:contact}). The {comparison density}
$\bar m_\kappa$, whose amplitude is the tangent point $\nu_\kappa$, is the base point
of the comparisons (\Cref{sec:comparisons}). At a self-consistent
reference they coincide up to translation (\Cref{prop:exact-reference});
for $\Theta_{\mathrm{cert}}$ they differ slightly.

For $\nu\ge0$ and $n\ge0$,
\begin{equation}\label{eq:cubic-loss}
 n^3=\tfrac32\nu n^2-\tfrac12\nu^3+\rho(n,\nu),\qquad
 \rho(n,\nu)=(n-\nu)^2\bigl(n+\tfrac\nu2\bigr)\ge0,
\end{equation}
with $\rho(n,\nu)=0$ only at $n=\nu$: the function $n^3=(n^2)^{3/2}$ of $n^2$
lies above its tangent at $\nu^2$. Choose the tangent point as a function
$\nu=\nu(q)$ of the ground coordinate. By \eqref{eq:price-identity} with
$\eta=\nu$ and $q'=\kappa n^2-Q(q)$, integration of \eqref{eq:cubic-loss} gives
\[
 E(m)=\frac3{2\kappa}\int_0^2\nu\dd q
 +\int\Bigl[\frac3{2\kappa}\nu(q)Q(q)-\frac12\nu(q)^3\Bigr]\dd x+\int\rho(n,\nu(q))\dd x .
\]
The middle integrand is a function of $q$ alone. Its integral is
controlled if it is a combination of $Q(q)$, whose integral is fixed by the
mass identity, and of a function $\Gamma(q)\ge0$ whose integral is bounded
through the calibration filter at scale $\varsigma\kappa$. We therefore make it equal to
$-\frac b\kappa Q(q)-\omega\Gamma(q)$, which determines $\nu$ as the root of a cubic.

Fix $\varsigma>0$ and put $f_\varsigma=r_{\varsigma\kappa}/\varsigma$ and $h_\varsigma=q-f_\varsigma$. By
\Cref{thm:ground,lem:filter},
\[
 f_\varsigma'=\kappa n^2-2f_\varsigma-\varsigma f_\varsigma^2,\qquad
 h_\varsigma'=q^2+\varsigma f_\varsigma^2-2h_\varsigma,\qquad 0\le h_\varsigma\le q,
\]
$f_\varsigma(\pm\infty)=0$ and $\int f_\varsigma=H_m(\varsigma\kappa)/\varsigma$; integrating the
equation for $f_\varsigma$ over $\R$ gives, for $\int m=1$,
\begin{equation}\label{eq:base-observable}
 S_\varsigma(m):=\int f_\varsigma^2
 =\frac\kappa\varsigma-\frac2{\varsigma^2}H_m(\varsigma\kappa).
\end{equation}
Thus the third determinant fixes $\int f_\varsigma^2$. A function $g(h)$
will define a potential whose derivative converts this integrand into a
function of $q$ alone, up to a nonnegative defect.

\begin{definition}[Calibration]\label{def:base}
A {calibration} consists of $\varsigma>0$, $b>0$, $\omega\ge0$, and
a continuous $g:[0,2]\to\R$ with $0\le g(h)\le2-h$. Put
\[
 \mathcal J_\varsigma(h)=-\int_0^h\frac{g}{2+\varsigma g},\qquad
 F_\varsigma(q,h)=\frac{2(q-h)^2-(q^2-2h)g(h)}{2+\varsigma g(h)},\qquad
 \Gamma_\varsigma(q)=\min_{0\le h\le q}F_\varsigma(q,h).
\]
For $\kappa>0$, the {comparison amplitude} $\nu_\kappa(q)$ is the
positive root of
\begin{equation}\label{eq:base-root}
 \kappa\nu^3-3Q(q)\nu-2bQ(q)-2\kappa\omega\Gamma_\varsigma(q)=0
 \quad(0<q<2),\qquad \nu_\kappa(0)=\nu_\kappa(2)=0,
\end{equation}
$d_\kappa=\kappa\nu_\kappa^2-Q$, and the {base value} of a density
$m$ with $\kappa(m)=\kappa$ is
\begin{equation}\label{eq:base-value}
 \mathcal B_\kappa(m)=-b+\frac{2b}\kappa+\frac3{2\kappa}\int_0^2\nu_\kappa\dd q
 -\omega \mathcal J_\varsigma(2)-\omega S_\varsigma(m).
\end{equation}
\end{definition}

\begin{proposition}[Exact base identity]\label{lem:base-identity}
For a calibration, $\Gamma_\varsigma$ is continuous,
$0\le\Gamma_\varsigma(q)\le\min\{q^2,(2-q)^2\}$, the root $\nu_\kappa(q)$ is unique
and jointly continuous in $(\kappa,q)\in(0,\infty)\times[0,2]$, and
$d_\kappa>2Q$ on $(0,2)$. Along every density, almost everywhere,
\begin{equation}\label{eq:calibration-pointwise}
 f_\varsigma^2+\bigl(\mathcal J_\varsigma(h_\varsigma)\bigr)'=F_\varsigma(q,h_\varsigma).
\end{equation}
Every $m\ge0$ with $\int m=1$, $E(m)<\infty$ and $\kappa(m)=\kappa$
satisfies
\begin{equation}\label{eq:base-deficit}
 E(m)=\mathcal B_\kappa(m)+\int_\R\Bigl[\rho\bigl(n,\nu_\kappa(q)\bigr)
 +\omega\,\mathcal T\Bigr]\dd x,\qquad
 \mathcal T=F_\varsigma(q,h_\varsigma)-\Gamma_\varsigma(q)\ge0,
\end{equation}
where both integrands are nonnegative and integrable.
\end{proposition}

The base value depends on $m$ only through $\kappa(m)$ and the observation
$H_m(\varsigma\kappa)$, so it is constant on $\mathcal A_{\mathrm{obs}}(\varsigma)$, and
\eqref{eq:base-deficit} is already a lower bound. The cubic
\eqref{eq:base-root} makes the part of the integrand that depends on $q$
alone a combination of $Q$ and $\Gamma_\varsigma$; the remainder $\rho(n,\nu_\kappa(q))$
vanishes exactly where $n=\nu_\kappa(q)$. The bound is exact for a density along which
$n=\nu_\kappa(q)$ and $h_\varsigma$ minimizes $F_\varsigma(q,\cdot)$ at every point. The
potential $g$ is chosen so that the second condition holds along a
reference density.

For a positive, nondecreasing, smooth function $b_\circ$ on $[0,1]$, let the
{reference amplitude} $\nu_\circ(q)$ be the positive root of
$k\nu^3-3Q(q)\nu-2b_\circ(Q(q))Q(q)=0$ for $0<q<2$, $\nu_\circ(0)=\nu_\circ(2)=0$, put
$d_\circ=k\nu_\circ^2-Q$ and $X_\circ(q)=\int_0^qd_\circ^{-1}$, and let the
{reference density} $m_\circ$ be $\nu_\circ(q)^2$ at $x=X_\circ(q)$ and zero
outside $[0,X_\circ(2)]$. For $\varsigma>0$ let $f_\circ=r_{\varsigma k}/\varsigma$ be
the calibration filter of $m_\circ$ and
$\check f_\circ(x)=f_\circ(X_\circ(2)-x)$ its mirror image.

\begin{theorem}[Contact potential]\label{thm:contact}
Let $b_\circ$ be smooth, positive and nondecreasing on $[0,1]$, and
$\varsigma>0$. Then $X_\circ(2)<\infty$, $m_\circ$ is continuous, bounded,
symmetric about $X_\circ(2)/2$, has ground parameter $k$ and ground
coordinate $q_\circ$ with $q_\circ=X_\circ^{-1}$ on $[0,X_\circ(2)]$. The function
$h_c=q_\circ-f_\circ$ increases strictly on $[0,\infty)$ from $h_c(0)=0$ to
$h_c(+\infty)=2$, and $g_\circ(h_c(x))=f_\circ(x)+\check f_\circ(x)$, $g_\circ(2)=0$,
defines a continuous $g_\circ$ with $0\le g_\circ(h)\le2-h$. For $g=g_\circ$
and every $0<q<2$, the function $F_\varsigma(q,\cdot)$ has a unique minimizer
on $[0,2]$, namely $h_c(x)$ with $q_\circ(x)=q$, and
\begin{equation}\label{eq:contact}
 \Gamma_\varsigma(q)=
 \frac{Q(q)(f_\circ+\check f_\circ)-2f_\circ\check f_\circ}
 {2+\varsigma(f_\circ+\check f_\circ)}\quad\text{at }q_\circ=q,
 \qquad \Gamma_\varsigma(2-q)=\Gamma_\varsigma(q).
\end{equation}
\end{theorem}

The contact potential is the sum of the calibration filter of $m_\circ$ and
of its mirror image, written as a function of $h_c$. Along $m_\circ$ the
filter state $h_\varsigma=h_c$ is the minimizer, so $\mathcal T\equiv0$ on
$m_\circ$. If moreover $m_\circ$ has the comparison amplitude, also $\rho\equiv0$;
this is a fixed-point condition on $b_\circ$, and at a fixed point the base
value solves the relaxation that keeps only one determinant.

\begin{proposition}[Exactness at a self-consistent reference]
\label{prop:exact-reference}
Let $b_\circ$ and $\varsigma$ be as in \Cref{thm:contact}, and let $b>0$,
$\omega\ge0$ satisfy
\begin{equation}\label{eq:self-consistent}
 b_\circ(Q(q))\,Q(q)=b\,Q(q)+k\omega\Gamma_\varsigma(q)\qquad(0<q<2)
\end{equation}
for the calibration $(\varsigma,b,\omega,g_\circ)$. Then $\nu_\circ=\nu_k$, and
$m_\circ(x)=\bar m_k(x-X_\circ(1))$, where $\bar m_k$ is the comparison density of the
calibration (\Cref{def:local-profile}). If moreover $\int m_\circ=1$, then $E(m_\circ)=\mathcal B_k(m_\circ)$, and
\[
 E(m_\circ)=\min\bigl\{E(m):\ m\in\mathcal M^1_c,\ \kappa(m)=k,\
 H_m(\varsigma k)=H_{m_\circ}(\varsigma k)\bigr\}.
\]
\end{proposition}

The calibration of the certificate in \Cref{sec:certificates} is computed
as an approximate solution of \eqref{eq:self-consistent} with
$H_{m_\circ}(\varsigma k)=H_*(\varsigma k)$, and $\varsigma$ is chosen by a local search for the
largest resulting value (\Cref{sec:construction}). The remaining observations
then raise the bound further.

\begin{proposition}[A one-determinant lower bound]\label{prop:baseline}
For any calibration of \Cref{def:base}, put
\begin{equation}\label{eq:baseline-value}
 \mathcal B_*=-b+\frac{2b}k+\frac3{2k}\int_0^2\nu_k\dd q
 -\omega\mathcal J_\varsigma(2)
 -\omega\Bigl(\frac k\varsigma-\frac{2H_*(\varsigma k)}{\varsigma^2}\Bigr).
\end{equation}
Every integrable $m\ge0$ with $\int m=1$, $\kappa(m)=k$ and
$H_m(\varsigma k)=H_*(\varsigma k)$ satisfies $E(m)\ge\mathcal B_*$.
Moreover, $\underline{\mathsf e}\ge\mathcal B_*$.
\end{proposition}
\begin{proof}
If $E(m)<\infty$, the first assertion follows from
\eqref{eq:base-deficit}; otherwise it is immediate. For the packed
densities, \Cref{thm:transfer} gives $\kappa(m_J)\to k$ uniformly in $J$.
Since $\int m_J=1$, $0\le\partial_sH_{m_J}(s)\le1/2$, and hence
\[
 |H_{m_J}(\varsigma\kappa(m_J))-H_*(\varsigma k)|
 \le\frac\varsigma2|\kappa(m_J)-k|+\frac{\varsigma k}2\delta_N\to0.
\]
Joint continuity of $\nu_\kappa(q)$ on compact parameter sets
(\Cref{lem:base-identity}) gives
$\mathcal B_{\kappa(m_J)}(m_J)\to\mathcal B_*$ uniformly. Apply
\eqref{eq:base-deficit} and $E(m_J)=N^{-3}\E_N(J)$.
\end{proof}

This bound needs no state potential, comparison multipliers, or cell
prices. The additional observations will improve it by controlling part
of the nonnegative remainder in \eqref{eq:base-deficit}.

\subsection{The state and its range}\label{sec:state-region}

The base identity uses one determinant. To exploit the remaining
observations, we need the joint range of their local filters. These
filters are not independent. At a point $x$ with
$q(x)>0$, let $\mathcal K_x$ be the operator on $L^2(-\infty,x)$ with kernel
$\sqrt{m(t)}G_x(t,y)\sqrt{m(y)}$, where $G_x(t,y)=\frac12(e^{-|t-y|}-e^{t+y-2x})$ is
the Green function of $-\partial^2+1$ on $(-\infty,x)$ with a Dirichlet
condition at $x$, and let $v(t)=\sqrt{m(t)}\,e^{t-x}$, $t<x$.
Then $r_s(x)=s\ip v{(I+s\mathcal K_x)^{-1}v}$ for $s\ge0$,
$q(x)=\kappa\ip v{(I-\kappa\mathcal K_x)^{-1}v}$ and $r_\circ(x)=\kappa\|v\|^2$, so all
filters are integrals against the spectral measure of $\mathcal K_x$ at $v$.
After normalization this reads as follows.

\begin{theorem}[One measure for all filters]\label{thm:common-measure}
Let $m$ be as in \Cref{lem:filter}. At every $x$ with $q(x)>0$ there is
a probability measure $\mu_x$ on $[0,1]$ such that
\begin{equation}\label{eq:common}
 \frac{\kappa r_s(x)}{sq(x)}
   =\int_0^1\frac{1-z}{1+(s/\kappa)z}\dd\mu_x(z)\quad(s>0),\qquad
 \frac{r_\circ(x)}{q(x)}=\int_0^1(1-z)\dd\mu_x(z).
\end{equation}
If $q(x)=0$, then $r_s(x)=r_\circ(x)=0$ for all $s$.
\end{theorem}

We use the filters at the two scales $s_0=32$, $s_1=64$ of the
determinant observations and the linear filter.

\begin{definition}[States]\label{def:states}
The {state} of $m$ at $x$ is $Y=(q,u_0,u_1,h)$ with
$u_i=r_{s_i}/s_i$ and $h=q-r_\circ$. By \Cref{thm:ground,lem:filter} it is
absolutely continuous on compact intervals, $Y'$ is integrable,
\begin{equation}\label{eq:state-flow}
 Y'=\mathsf b(Y)n^2-\mathsf d(Y),\qquad
 \mathsf b=(\kappa,1,1,0),\qquad
 \mathsf d=\bigl(Q(q),\,2u_0+s_0u_0^2,\,2u_1+s_1u_1^2,\,2h-q^2\bigr),
\end{equation}
and $Y\to Y_-=(0,0,0,0)$ at $-\infty$, $Y\to Y_+=(2,0,0,2)$ at
$+\infty$.
\end{definition}

The normalized coordinates $U_i=\kappa r_{s_i}/(s_iq)$ and
$\vartheta=1-r_\circ/q$ are the three integrals \eqref{eq:common} of one
probability measure on $[0,1]$. Their joint range is a moment space, the
set of values of three integrals over all probability measures, and it
can be computed exactly (\Cref{prop:three-filter}). After the change of
variables $w=\kappa(1-z)/(\kappa+s_0z)$, $U_0$ is the mean of $w$ and $U_1$ the mean of
$f(w)=(\kappa+s_0)w/(\kappa+s_1-(s_1-s_0)w)$, a convex function with $f(0)=0$ and
$f(1)=1$; hence $f(U_0)\le U_1\le U_0$, by Jensen's inequality and by
comparison with the chord. For given $(U_0,U_1)$, $\vartheta$ fills an interval
$[\vartheta_-,\vartheta_+]$ whose end points are attained by two-point measures. We
parametrize the pairs by $(U,T)\in[0,1]^2$: $U_0=U$, $U_1=(1-T)f(U)+TU$.

\begin{proposition}[The state region]\label{prop:state-region}
Let $\kappa>0$. There are continuous functions $V$ and
$0\le\vartheta_-\le\vartheta_+\le1$ on $[0,1]^2$, given explicitly in
\Cref{lem:range}, such that the set
\[
 \Omega_\kappa=\{Y_-\}\cup\Bigl\{\Bigl(q,\frac{qU}\kappa,\frac{qV(U,T)}\kappa,q\vartheta\Bigr):
 0<q\le2,\ (U,T)\in[0,1]^2,\ \vartheta_-(U,T)\le\vartheta\le\vartheta_+(U,T)\Bigr\}
\]
is compact, and the state of every $m$ as in \Cref{lem:filter} with
$\kappa(m)=\kappa$ lies in $\Omega_\kappa$ at every point. Every point of
$\Omega_\kappa$ with $q>0$ corresponds, through \eqref{eq:common}, to some
probability measure on $[0,1]$.
\end{proposition}

The last statement says that $\Omega_\kappa$ is the smallest region that can
be obtained from \Cref{thm:common-measure} alone. The pointwise inequality
of a certificate is required on all of $\Omega_\kappa$.

\subsection{Comparisons for the remaining observations}\label{sec:comparisons}

The tangent point $\nu_\kappa(q)$ defines a density with exactly this
amplitude. Every bounded speed $D$ on $(0,2)$ with $1/D$ integrable
defines a density with a prescribed ground parameter and ground
coordinate: place the amplitude $\sqrt{(D+Q)/\kappa}$ at the point $x$ where
$q=X^{-1}(x)$, $X'=1/D$ (\Cref{thm:primal-map}). With $D=d_\kappa$ the
amplitude is $\nu_\kappa$.

\begin{definition}[Comparison density]\label{def:local-profile}
For a calibration and $\kappa>0$, let $X_\kappa(q)=\int_1^qd_\kappa^{-1}$ and
$\bar m_\kappa(X_\kappa(q))=\nu_\kappa(q)^2$ on $(X_\kappa(0),X_\kappa(2))$,
$\bar m_\kappa=0$ elsewhere ($X_\kappa(0)$ and $X_\kappa(2)$ are finite by
\Cref{lem:local-profile}). Its state at $X_\kappa(q)$ is denoted by
$\bar Y(q)=(q,\bar u_0,\bar u_1,\bar h)$. Its {right filters} are the
functions of $q$
\[
 R_s(q)=\check r_s(-X_\kappa(q)),\qquad R_\circ(q)=\check r_\circ(-X_\kappa(q)),
\]
where $\check r_s,\check r_\circ$ are the filters of \Cref{lem:filter} of the
reflected density $\bar m_\kappa(-x)$.
\end{definition}

The right filters are the anticausal counterparts of the filters: they
see the comparison density to the right of a point. By
\Cref{lem:local-profile,lem:right-filters}, $\bar m_\kappa$ is bounded,
compactly supported, has ground parameter $\kappa$ and ground coordinate $q$
at $X_\kappa(q)$, and
\[
 d_\kappa\frac{\dd R_s}{\dd q}=-s\nu_\kappa^2+2R_s+R_s^2,\qquad
 d_\kappa\frac{\dd R_\circ}{\dd q}=-\kappa\nu_\kappa^2+2R_\circ,\qquad
 R_s(2)=R_\circ(2)=0 .
\]

\begin{theorem}[Signed comparisons]\label{thm:comparisons}

Let $m\ge0$ be integrable and not zero, with $E(m)<\infty$,
$\kappa(m)=\kappa$, ground coordinate $q$ and filters $r_s,r_\circ$. Put
\[
 a_s(q,r)=-\frac{sQ(q)+\kappa rR_s(q)}{d_\kappa(q)\bigl(2+r+R_s(q)\bigr)},
 \qquad
 a_\circ(q,r)=\frac{\kappa\bigl[2rR_\circ(q)-Q(q)(r+R_\circ(q))\bigr]}{2d_\kappa(q)}
\]
for $0<q<2$, and $a_s=a_\circ=0$ at $q\in\{0,2\}$. Then
$-sQ/(2d_\kappa)\le a_s(q,r_s)\le0$, $-\kappa Q/d_\kappa\le a_\circ(q,r_\circ)\le0$, and
\[
 H_m(s)-H_{\bar m_\kappa}(s)=\int a_s(q,r_s)\bigl(n^2-\nu_\kappa(q)^2\bigr)\dd x,
 \qquad
 S(m)-S(\bar m_\kappa)=\int a_\circ(q,r_\circ)\bigl(n^2-\nu_\kappa(q)^2\bigr)\dd x ,
\]
with integrable integrands.
\end{theorem}

Each identity is the integral of the derivative of an explicit function
of $(q,r_s)$ along $m$ (\Cref{lem:comparison-general}): the derivative is
$-r_s$ plus a multiple of the amplitude defect $n^2-\nu_\kappa(q)^2$, whose
coefficient combines the filter of $m$ (from the left) with the right
filter of $\bar m_\kappa$ (from the right). The coefficients are bounded by
multiples of $\nu_\kappa$ (\Cref{lem:local-profile}), so the integrands are
integrable whenever $E(m)<\infty$. Since the observations are equality
constraints, their multipliers may have either sign. For $\Theta_{\mathrm{cert}}$
the comparison density is not in the relaxed class: its observations
differ from the targets (\Cref{tab:reference-enclosures}), and these
differences enter the final bound.

\subsection{The residual and the pointwise criterion}\label{sec:criterion}

Let $\xi_0,\xi_1,\chi,\tau$ be real numbers and put
\[
 a=\xi_0a_{s_0}+\xi_1a_{s_1}+\chi a_\circ ,
\]
evaluated at $(q,r_{s_0})$, $(q,r_{s_1})$ and $(q,r_\circ)$, that is, at
$r_{s_i}=s_iu_i$ and $r_\circ=q-h$. Let $\Psi\in C^1(\R^4)$ with
$\Psi(Y_-)=\Psi(Y_+)$, and put $\mathsf b_\Psi=\nabla\Psi\cdot\mathsf b$ and
$\mathsf d_\Psi=\nabla\Psi\cdot\mathsf d$, so that $(\Psi\circ Y)'=\mathsf b_\Psi n^2-\mathsf d_\Psi$
along every density. Let $\eta$ be a bounded function on $[0,2]$ that is
continuous outside a finite set. The {residual} at a state
$Y\in\Omega_\kappa$ and an amplitude $n\ge0$ is
\begin{equation}\label{eq:residual}
 \mathscr R_\kappa(Y,n)=\rho(n,\nu_\kappa)-a\,(n^2-\nu_\kappa^2)
 +\frac{\eta(q)}\kappa(\kappa n^2-Q)+\frac\tau\kappa Q
 +\mathsf b_\Psi n^2-\mathsf d_\Psi ,
\end{equation}
with $\nu_\kappa=\nu_\kappa(q)$ and $Q=Q(q)$. Each term after $\rho$ has a known
integral along every density with $\kappa(m)=\kappa$ and $\int m=1$: the
comparison identities, the price identity \eqref{eq:price-identity}, the
mass identity, and $\int(\Psi\circ Y)'=0$. The potential $\Psi$ is a null
Lagrangian; it does not change the bound, but it changes the integrand.

\begin{proposition}[Integrating the residual]\label{prop:certificate-mechanism}
Suppose that $\mathscr R_\kappa(Y,n)\ge0$ for all $Y\in\Omega_\kappa$
with $0<q<2$ and all $n\ge0$, and that $\mathscr R_\kappa(Y,0)\ge0$ for
$Y\in\Omega_\kappa$ with $q\in\{0,2\}$. Then every $m\in\mathcal M^{1}_c$ with
$\kappa(m)=\kappa$ satisfies
\[
 E(m)\ge\mathcal B_\kappa(m)+\sum_{i=0}^1\xi_i\bigl[H_m(s_i)-H_{\bar m_\kappa}(s_i)\bigr]
 +\chi\bigl[S(m)-S(\bar m_\kappa)\bigr]-\mathcal P_\kappa,
 \qquad
 \mathcal P_\kappa=\frac1\kappa\int_0^2\eta+\frac{\tau(\kappa-2)}\kappa .
\]
\end{proposition}

The amplitude enters the residual only through $n^2$ and $n^3$, and the
infimum over $n\ge0$ is explicit.

Explicitly, for $\nu>0$ and $y\in\R$,
\[
 \mathcal C(\nu,y):=\max_{n\ge0}\bigl\{y(n^2-\nu^2)-\rho(n,\nu)\bigr\}=\nu^3\psi(y/\nu),
 \qquad\psi(x)=-x-\tfrac12+\tfrac4{27}\bigl(x+\tfrac32\bigr)_+^3,
\]
and $\psi$ is convex and twice continuously differentiable (\Cref{lem:conjugate}).

\begin{theorem}[Pointwise criterion]\label{thm:master}
With the data above, $\mathscr R_\kappa(Y,n)\ge0$ for all $n\ge0$ if
and only if
\begin{equation}\label{eq:master}
 \mathcal C\bigl(\nu_\kappa,\ a-\eta-\mathsf b_\Psi\bigr)
 \le\frac{\eta d_\kappa+\tau Q}\kappa+\mathsf b_\Psi\nu_\kappa^2-\mathsf d_\Psi .
\end{equation}
For every $F\in C^1(\R)$ with $F(0)=F(2)$, the replacement
$\Psi\mapsto\Psi-F(q)$, $\eta\mapsto\eta+\kappa F'(q)$ changes neither
the two sides of \eqref{eq:master} nor $\mathcal P_\kappa$. If
$\eta=-\tau Q/d_\kappa+\nu_\kappa e$ on $(0,2)$ for a bounded function $e$
that is continuous outside a finite set, then the right side of
\eqref{eq:master} is
$\nu_\kappa e\,d_\kappa/\kappa+\mathsf b_\Psi\nu_\kappa^2-\mathsf d_\Psi$ and
\begin{equation}\label{eq:centered-price}
 \mathcal P_\kappa=\frac1\kappa\int_0^2\nu_\kappa e\dd q
 +\tau\Bigl(1-\int\bar m_\kappa\Bigr).
\end{equation}
\end{theorem}

The second statement is a gauge freedom: a potential that depends on $q$
alone can be traded for a price. We use it to normalize the potential so
that it vanishes on the comparison trajectory, $\Psi=\Psi_{\rm raw}-\Psi_{\rm raw}(\bar Y(q))$.
The last statement fixes the form of the price: the part $-\tau Q/d_\kappa$
absorbs the mass multiplier, and the free part $\nu_\kappa e$ vanishes where
$\nu_\kappa$ does, at $q=0,2$, so that the criterion stays regular there. The
cost of the price is $\frac1\kappa\int_0^2\nu_\kappa e\dd q$. The lower bound is now
reduced to the inequality \eqref{eq:master} on the compact set $\Omega_\kappa$,
in which the amplitude no longer appears.

\subsection{Dual certificates}\label{sec:dual-certificates}

Near the ends of the transition the comparison amplitude vanishes like
$Q^{1/3}$. The two ends are therefore parametrized by $t\in[0,1]$ through
\[
 q_{\mathrm L}(t)=t^3,\qquad q_{\mathrm R}(t)=2-t^3,
\]
the left and the right {side}; on each side the comparison amplitude
vanishes linearly in $t$, and we write
$\nu_\kappa(q_{\mathrm s}(t))=t\,\hat\nu_\kappa(t)$.

\begin{definition}[Dual certificate]\label{def:dual}
A dual certificate $\Theta$ consists of
\begin{enumerate}
\item a calibration (\Cref{def:base}) with rational $\varsigma,b>0$,
$\omega\ge0$, and $g=g_\circ$ the contact potential
(\Cref{thm:contact}) of a reference profile
\[
 b_\circ(Q)=p_0+\sum_jp_j\frac{Q}{Q+\zeta_j}
 \qquad(\text{finitely many rational }p_0,p_j,\zeta_j>0);
\]
\item a potential
\begin{equation}\label{eq:transport-potential}
 \Psi_{\rm raw}=(2-q)^2\bigl[P_h(q)h+qP_0(q)u_0+qP_1(q)u_1
 +P_{00}(q)u_0^2+P_{01}(q)u_0u_1+P_{11}(q)u_1^2\bigr]
\end{equation}
with polynomials $P_\bullet$ with rational coefficients;
\item rational multipliers $\xi_0,\xi_1,\chi,\tau$;
\item for each side $\mathrm s\in\{\mathrm L,\mathrm R\}$, a partition of
$[0,1]$ into closed intervals $\mathsf C$ ({cells}) and a rational
affine function $e_{\mathsf C}$ on each cell (the {price}).
\end{enumerate}
\end{definition}

The reference profile is a positive combination of nondecreasing
functions of $Q$, as \Cref{thm:contact} requires. The factor $(2-q)^2$ and
the extra factor $q$ in the linear terms of $\Psi_{\rm raw}$ make the
normalized criterion below continuous up to the ends $t=0$ of both sides,
and $\mathscr R_\kappa(Y,0)=0$ at $q\in\{0,2\}$ (\Cref{lem:regular-criterion}(ii),(iv)).

For $\kappa>0$, the certificate uses the potential
$\Psi=\Psi_{\rm raw}-\Psi_{\rm raw}(\bar Y(q))$ and the price function
\[
 \eta=-\tau Q/d_\kappa+\nu_\kappa e,\qquad
 e(q_{\mathrm s}(t))=e_{\mathsf C}(t)\quad(t\in\mathsf C),
\]
where at a common end point of two cells either price may be used. With
$q=q_{\mathrm s}(t)$, the coordinates $(U,T)$ of \Cref{prop:state-region}, and
$\vartheta$ at one of its ends $\vartheta_\sigma$, $\sigma\in\{-,+\}$, the criterion \eqref{eq:master} divided by $\nu_\kappa^3$ takes the form
\begin{equation}\label{eq:closed-criterion-main}
 \psi\bigl(\mathfrak b_\sigma-e\bigr)-\mathfrak c_\sigma-e(1-\mathfrak l)\le0
\end{equation}
with explicit functions $\mathfrak b_\sigma,\mathfrak c_\sigma,\mathfrak l$ of $(t,U,T)$ and $\kappa$
that are continuous on $[0,1]^3\times(0,\infty)$ (\Cref{lem:regular-criterion}).
The function $\mathfrak l$ depends only on $(t,\kappa)$, and $\mathfrak b_\sigma,\mathfrak c_\sigma$
are linear in the multipliers and in $\Psi_{\rm raw}$. Since the left side
of \eqref{eq:closed-criterion-main} is convex in $\vartheta$, the two ends
$\vartheta_\pm$ suffice.

\begin{definition}[Admissibility and value]\label{def:admissible}

A dual certificate is {admissible at $\kappa$} if, for every side,
cell $\mathsf C$ and end $\sigma$,
\begin{equation}\label{eq:witness}
 \max_{(t,U,T)\in\mathsf C\times[0,1]^2}
 \Bigl[\psi\bigl(\mathfrak b_\sigma-e_{\mathsf C}\bigr)-\mathfrak c_\sigma
 -e_{\mathsf C}(1-\mathfrak l)\Bigr]<0 .
\end{equation}
Its {value} at a density $m$ with $\kappa(m)=\kappa$ is
\begin{equation}\label{eq:dual-value}
 \ell_\kappa(\Theta;m)=\mathcal B_\kappa(m)
 +\sum_{i=0}^1\xi_i\bigl[H_m(s_i)-H_{\bar m_\kappa}(s_i)\bigr]
 +\chi\bigl[S(m)-S(\bar m_\kappa)\bigr]-\mathcal P_\kappa,
\end{equation}
with
\begin{equation}\label{eq:closed-price-main}
 \mathcal P_\kappa=\sum_{\mathrm s,\mathsf C}\int_{\mathsf C}
 \frac{3t^3\hat\nu_\kappa(t)}{\kappa}e_{\mathsf C}(t)\dd t
 +\tau\Bigl(1-\int\bar m_\kappa\Bigr).
\end{equation}
Its {target value} is the value at $\kappa=k$ with the target
observations:
\begin{equation}\label{eq:dual-target-value}
 \ell_*(\Theta)=\mathcal B_*
 +\sum_{i=0}^1\xi_i\,\delta H_i+\chi\,\delta S
 -\Pi-\tau\Bigl(1-\int\bar m_k\Bigr),
\end{equation}
where $\mathcal B_*$ is defined in \eqref{eq:baseline-value} and
\[
 \Pi=\sum_{\mathrm s,\mathsf C}\int_{\mathsf C}
 \frac{3t^3\hat\nu_k(t)}{k}e_{\mathsf C}(t)\dd t,
\]
\[
 \delta H_i=H_*(s_i)-H_{\bar m_k}(s_i),\qquad
 \delta S=\frac{k^2}6-S(\bar m_k).
\]
\end{definition}

On each side $|\dd q|=3t^2\dd t$ and $\nu_\kappa=t\hat\nu_\kappa$, so the sum in
\eqref{eq:closed-price-main} is the price cost $\frac1\kappa\int_0^2\nu_\kappa e\dd q$
of \eqref{eq:centered-price}. The target value has four parts: the base
value $\mathcal B_*$ of the calibration on $\mathcal A_{\mathrm{obs}}(\varsigma)$, the terms
$\xi_i\delta H_i$ and $\chi\delta S$, of either sign, from the mismatch between the
observations of the comparison density and the targets, the cost $\Pi$
of the price, and a mass correction.

\begin{theorem}[Dual certificates bound the energy]\label{thm:dual}
Let $\Theta$ be admissible at $k$. Then
\begin{enumerate}
\item every $m\in\mathcal A_{\mathrm{obs}}(\varsigma)$ satisfies $E(m)\ge\ell_*(\Theta)$;
\item for every $\varepsilon>0$ there is $N_0$ such that
$N^{-3}\E_N(J)\ge\ell_*(\Theta)-\varepsilon$ for all $N\ge N_0$ and
$J\in\overline{\mathcal Y}_N$. Hence $\underline{\mathsf e}\ge\ell_*(\Theta)$ and, if
$\ell_*(\Theta)\ge0$, $\mathscr G(\mu,L)\ge2\sqrt{\ell_*(\Theta)}$ for all
$0<\mu<L$.
\end{enumerate}
\end{theorem}

Part (i) combines \Cref{prop:certificate-mechanism,thm:master}. Part (ii)
needs more, because the packed densities $m_J$ satisfy the constraints
only approximately and have $\kappa(m_J)\ne k$. Admissibility is a strict
inequality between continuous functions on a compact set, so it
persists for $\kappa$ near $k$, and every term of $\ell_\kappa(\Theta;m)$ is
continuous in $\kappa$ and in the observations of $m$
(\Cref{thm:certificate-transfer}). By \Cref{thm:transfer}, $\ell_{\kappa(m_J)}(\Theta;m_J)$
is then within $o(1)$ of $\ell_*(\Theta)$, uniformly in $J$.

\section{The certificate and the constants}\label{sec:certificates}

The numerical search chooses the rational data of
$\Theta_{\mathrm{cert}}$; verification establishes the inequalities needed
by \Cref{thm:dual}. We describe the search to explain the choice of
observations and potential, then state the verified bounds from which
the main results follow. Proofs and data are in \Cref{app:verification}.

\subsection{Construction of the certificate}\label{sec:construction}

The certificate was built in four stages, by the search programs of
\cite{Code}; each stage fixes part of the data of \Cref{def:dual} and
passes it on.

{Stage 1: one determinant.} By \Cref{prop:exact-reference}, a
reference profile $b_\circ$ that satisfies the fixed-point equation
\eqref{eq:self-consistent} gives a calibration whose base value is the
minimum of $E$ over the class with $\int m=1$, $\kappa(m)=k$ and one
prescribed determinant $H_m(\varsigma k)$. For fixed $\varsigma$, the fixed point with
$\int m_\circ=1$ and $H_{m_\circ}(\varsigma k)=H_*(\varsigma k)$ is computed numerically on
the left half $t\in[0,1]$, $q=t^3$, of the symmetric reference density.
The unknowns are the normalized contact trajectory $h_c/q_\circ$, the mirror
filter $\check f_\circ$, the accumulated mass and the accumulated
$\int(f_\circ^2+\check f_\circ^2)$ on the left half, as functions of $t$, and the two
numbers $b,\omega$, which act as Lagrange multipliers of the two
constraints. The amplitude is the root of the cubic \eqref{eq:base-root}
with $\Gamma_\varsigma$ given by \eqref{eq:contact} in terms of $f_\circ,\check f_\circ$. The
boundary conditions are regularity of $h_c/q_\circ$ at the singular end $t=0$,
the symmetry condition $\check f_\circ=f_\circ$ at $t=1$, mass $\frac12$ on the left
half, and $\int f_\circ^2=\frac k\varsigma-\frac2{\varsigma^2}H_*(\varsigma k)$, which includes the explicit
contribution of the decaying filter to the right of the support. The
boundary value problem is solved by collocation. The value
$E(m_\circ(\varsigma))$ is then maximized locally over $\varsigma$, first along $\varsigma=2^i$ and
then by a one-dimensional search; the largest value found is at
$\varsigma_{\mathrm{cal}}=2704.26\ldots$ in \eqref{eq:cert-scalars}.

{Stage 2: an exact reference profile.} The profile
$b_\circ(Q)=b+k\omega\Gamma_\varsigma(q)/Q$ of the fixed point is approximated by a
positive rational function $p_0+\sum_{j=1}^Jp_jQ/(Q+\zeta_j)$, fitted by
least squares in the logarithms of the coefficients. Each term is
nondecreasing in $Q$, as \Cref{thm:contact} requires. For the fitted
profile, the contact potential $g_\circ$ is fixed, and $(b,\omega)$ are
recomputed to maximize the base value $\mathcal B_*$, which for fixed $g_\circ$
is an explicit function of $(b,\omega)$. The number of terms $J$ is increased,
the fit with the largest base value is kept, and the search stops when
the base value agrees with the energy of the fixed point of stage 1 to
$10^{-12}$; $\Theta_{\mathrm{cert}}$ has $J=5$ and the values of \Cref{tab:cert-data}. The resulting base value is
$\mathcal B_*=1.14266367431\ldots$ (\Cref{tab:reference-enclosures}): a lower
bound for $E$ on $\mathcal A_{\mathrm{obs}}(\varsigma_{\mathrm{cal}})$ that uses only $\kappa=k$,
$\int m=1$ and $H_m(\varsigma_{\mathrm{cal}}k)$.

{Stage 3: multipliers and potential.} The comparison density $\bar m_k$
of this calibration has the observations of the target only
approximately: $H_{\bar m_k}(s)$, $S(\bar m_k)$ and $\int\bar m_k$ differ from
$H_*(s)$, $k^2/6$ and $1$. The multipliers $\xi_0,\xi_1,\chi,\tau$ turn these
differences into the terms $\xi_i\delta H_i$, $\chi\delta S$ and the mass
correction of \eqref{eq:dual-target-value}, and the potential $\Psi_{\rm raw}$
is needed to make the pointwise criterion hold at the price that these
terms can afford. For fixed multipliers and potential, let $e(t)$ be the
smallest price for which \eqref{eq:closed-criterion-main} holds at all
sampled states $(U,T,\sigma)$ over a sample point $t$. At each fixed state, the
left side of \eqref{eq:closed-criterion-main} is jointly convex in the multipliers, the
coefficients of $\Psi_{\rm raw}$ and $e$, because $\mathfrak b_\sigma,\mathfrak c_\sigma$ are linear
in the data and $\psi$ is convex; so the set of data and prices satisfying
the sampled inequalities is convex, and $e(t)$ is a convex function of the
data on the convex set of data for which a price exists at every sample. The sampled target value, in which $\Pi$ is a sum of these prices
with positive weights, is therefore concave. It was maximized by
sequential quadratic programming over the four multipliers and the
coefficients of the polynomials $P_\bullet$, first with degree at most two
on $64$ Gauss nodes in $t$ per side and a $33\times33$ grid in $(U,T)$, then with
degree at most three on $80$ nodes and a $41\times41$ grid, both grids refined
toward the comparison trajectory, starting from the certificate without
the second trace. The scales $s_0<s_1$ were selected among powers of two up to
$2^8$, with half-step refinement, by the same sampled value without the
second trace; the best pair is $(32,64)$.

{Stage 4: cells, prices and rectangles.} Each side is subdivided
into cells. On a cell, the affine price is obtained by Newton's method at
the cell center: its value makes vanish an upper bound, over a cover of
$[0,1]^2$ by rectangles, for the left side of \eqref{eq:witness} at the cell
center, its slope is the derivative of that value along the active state,
and a small reserve is added. The square $[0,1]^2$ is
then bisected into rectangles until an upper bound for the left side of
\eqref{eq:witness} on each rectangle (\Cref{lem:box-bound}) is negative; if
this fails within a work budget, the cell is bisected. (The production
program can also scale the multipliers, the potential and the prices
down by a common factor when no price makes the criterion hold at a cell
center; for $\Theta_{\mathrm{cert}}$ no scaling occurred.) The result has $8\,988$
cells and $6\,055\,942$ rectangles.

\Cref{tab:ledger} shows the parts of the resulting value. The base value
of stage 2 lies $2.85\cdot10^{-5}$ below $\cs$; the remaining observations
reduce this gap by a factor of about $3.3$.

\begin{table}[ht]
\centering
\caption{The parts of $c_{\mathrm{cert}}=\ell_*(\Theta_{\mathrm{cert}})$ in
\eqref{eq:dual-target-value}; small terms to five significant digits,
totals to eight.}
\label{tab:ledger}
\small
\begin{tabular}{@{}lr@{}}
\toprule
term&value\\\midrule
base value $\mathcal B_*$ ($\kappa$, $\int m$, $H(\varsigma_{\mathrm{cal}}k)$)&$1.1426637$\\
$\xi_0\,\delta H_0$ (determinant at $s_0=32$)&$1.9265\cdot10^{-4}$\\
$\xi_1\,\delta H_1$ (determinant at $s_1=64$)&$-1.5267\cdot10^{-4}$\\
$\chi\,\delta S$ (second trace)&$-2.1932\cdot10^{-6}$\\
$-\tau(1-\int\bar m_k)$ (mass)&$-4.6886\cdot10^{-10}$\\
$-\Pi$ (price)&$-1.7927\cdot10^{-5}$\\\midrule
$c_{\mathrm{cert}}$&$1.1426835$\\
$\cs$&$1.1426922$\\
\bottomrule
\end{tabular}
\end{table}

\subsection{Verification}\label{sec:verification-summary}

The certificate $\Theta_{\mathrm{cert}}$ has a reference profile with five terms,
the calibration scalars
\begin{equation}\label{eq:cert-scalars}
 \varsigma_{\mathrm{cal}}=2704.261804447028,\qquad
 b=0.11085839213923356,\qquad
 \omega=1351.3424903216799,
\end{equation}
the profile coefficients, multipliers and potential coefficients of
\Cref{tab:cert-data,tab:transport-polynomials}, and $8\,988$ cells with
affine prices.

\begin{proposition}[Computed values for $\Theta_{\mathrm{cert}}$]
\label{ca:transport}
\begin{enumerate}
\item $\Theta_{\mathrm{cert}}$ is admissible at $k$.
\item The quantities in \Cref{tab:reference-enclosures} lie in the stated
intervals.
\end{enumerate}
\end{proposition}

The computation (\Cref{app:verification}) encloses, in interval arithmetic
with outward rounding, the solutions of the ordinary differential
equations that define the reference density, the contact potential, the
comparison density and its filters, as functions of $t$ on both sides;
it bounds the left side of \eqref{eq:witness} on each rectangle by
second-order Taylor bounds and checks that the rectangles of each cell
cover $[0,1]^2$; and it encloses the integrals in
\Cref{tab:reference-enclosures}. All enclosure rules are lemmas of
\Cref{sec:enclosures}.

\begin{table}[ht]
\centering
\caption{Enclosures for $\Theta_{\mathrm{cert}}$ (\Cref{ca:transport}).}
\label{tab:reference-enclosures}
\small
\begin{tabular}{@{}lll@{}}
\toprule
quantity&lower endpoint&upper endpoint\\\midrule
$\mathcal B_*$&$1.142663674318292$&$1.142663674319440$\\
$\bar M:=\int\bar m_k$&$1.000000000819150$&$1.000000000819152$\\
$\delta H_0$&$-0.001269156586442$&$-0.001269156586438$\\
$\delta H_1$&$-0.002545408246820$&$-0.002545408246816$\\
$\delta S$&$0.000049913515799$&$0.000049913515800$\\
$10^5\,\Pi$&$1.7926991598178728941$&$1.7926991598184194252$\\
\bottomrule
\end{tabular}
\end{table}

\begin{corollary}[The certified constant]\label{cor:numerical}
Put $c_{\mathrm{cert}}=\ell_*(\Theta_{\mathrm{cert}})$. With the endpoints of
\Cref{tab:reference-enclosures} denoted by superscripts,
\[
 c_{\mathrm{cert}}\ge c^-:=\mathcal B_*^-+\xi_0\,\delta H_0^++\xi_1\,\delta H_1^-
 +\chi\,\delta S^++\tau(\bar M^+-1)-\Pi^+=1.142683536243430795\ldots,
\]
\[
 c_{\mathrm{cert}}\le c^+:=\mathcal B_*^++\xi_0\,\delta H_0^-+\xi_1\,\delta H_1^+
 +\chi\,\delta S^-+\tau(\bar M^--1)-\Pi^-=1.142683536244580836\ldots .
\]
Assembling the original rational interval endpoints before rounding the
entries of \Cref{tab:reference-enclosures} gives the sharper enclosure
\begin{equation}\label{eq:archived-bound}
 \num{certificate/bound/lo24}\le c_{\mathrm{cert}}
 \le\num{certificate/bound/hi24}.
\end{equation}
In particular $c_{\mathrm{cert}}>1.14268353624$, $c^+-c^-<1.2\cdot10^{-12}$, and
\[
 \inf_{\mathcal A_{\mathrm{obs}}(\varsigma_{\mathrm{cal}})}E\ge c_{\mathrm{cert}},\qquad
 \underline{\mathsf e}\ge c_{\mathrm{cert}},\qquad
 \mathscr G(\mu,L)\ge2\sqrt{c_{\mathrm{cert}}}\quad(0<\mu<L).
\]
\end{corollary}

\subsection{The sine energy}\label{sec:sine-energy}

By the theta form \eqref{eq:sine-theta} and the symmetry of $n_*$,
\[
 \cs=E(m_*)=512\int_1^\infty\varpi^2e^{-3\pi\varpi/2}\,
 \Sigma_1(e^{-\pi\varpi})^6\Sigma_2(e^{-\pi\varpi})^6\dd\varpi .
\]
Truncating $\Sigma_1,\Sigma_2$ after $K+1$ terms gives a polynomial in
$a=e^{-\pi\varpi}$, whose integral is an explicit finite sum $E_K$, and the
tails give an explicit bound $R_K$ (\Cref{lem:sine-sum}).

\begin{proposition}[The sine energy]\label{ca:sine}
Let
\[
 (1+a^2+a^6+a^{12})^6(1-2a+2a^4-2a^9)^6=\sum_{n=0}^{126}C_na^n\in\mathbb Z[a],
\]
and let $E_3$ and $R_3$ be the quantities of \Cref{lem:sine-sum} with
$K=3$. Then
\[
 \begin{gathered}
 1.142692179815791439359121928\le E_3\le1.142692179815791439359121929,\\
 0<R_3<2.35701467\cdot10^{-22}.
 \end{gathered}
\]
and consequently
\[
 1.1426921798157914393588862<\cs<1.1426921798157914393593577 .
\]
\end{proposition}

\Cref{thm:energy,thm:main-centered} follow from \Cref{cor:numerical},
\Cref{ca:sine}, \Cref{thm:sine-finite} and \Cref{thm:jacobi,thm:gain}; the proofs
are in \Cref{app:main-proofs}.

\paragraph{The one-determinant estimate.}
For the same calibration data, \Cref{prop:baseline} and the enclosure
of $\mathcal B_*$ in \Cref{tab:reference-enclosures} give
\[
 \underline{\mathsf e}\ge\mathcal B_*>1.14266367431829,
 \qquad
 \limsup_{N\to\infty}\sup_{0<\mu<L}
 \frac{A_N(\mathcal R_N^\star)}
 {\inf_{\mathcal R\in\mathfrak E_N(\mu,L)}A_N(\mathcal R)}
 \le\sqrt{\frac{\cs}{\mathcal B_*}}<1.000013.
\]
The ratio follows from \Cref{thm:gain,thm:sine-finite}; its numerical
bound uses only $\cs<1.14269218$ and the displayed lower endpoint for
$\mathcal B_*$. The full certificate improves this factor to
$1.000004$ by using the additional observations.

\section{A feasible density below the sine energy}\label{sec:feasible}

The lower certificate leaves a gap $\cs-c_{\mathrm{cert}}<10^{-5}$.
That gap may reflect either the chosen certificate or the relaxation
itself. We now show that the relaxation admits a density in
$\mathcal A_{\mathrm{obs}}(\varsigma_{\mathrm{cal}})\cap\Mc$ with energy strictly
below $\cs$. Thus even an optimal certificate based on these observations
cannot attain the sine value. Proofs are in \Cref{app:primal}.

Prescribing the speed of the ground coordinate enforces the ground
parameter exactly, leaving the other observations as finite-dimensional
constraints.

\begin{theorem}[Densities with a prescribed ground parameter]
\label{thm:primal-map}
Let $\kappa>0$ and let $D:(0,2)\to(0,\infty)$ be measurable with
$D\in L^\infty$ and $1/D\in L^1(0,2)$. Put
$X(q)=\int_1^qD^{-1}$ and let $m$ equal $(D(q)+Q(q))/\kappa$ at
$x=X(q)$ and zero outside $[X(0),X(2)]$. Then $m$ is bounded, compactly
supported, has ground parameter $\kappa$ and ground coordinate
$X^{-1}$ on its support, and
\begin{equation}\label{eq:primal-masses}
 \int m^{j/2}=\int_0^2\Bigl(\frac{D+Q}\kappa\Bigr)^{j/2}\frac{\dd q}D
 \qquad(j=1,2,3).
\end{equation}
\end{theorem}

The ground coordinate of such an $m$ satisfies $q'=D(q)$, and the
function $u=\exp\int(1-q)$ is a positive solution of $-u''+u=\kappa mu$; by the
identity \eqref{eq:ground-square} for this $u$, $\kappa(m)=\kappa$ and $u$ is the ground
state (proof of \Cref{thm:primal-map}). The remaining
observations are then solved for by Newton's method in finitely many
parameters of $D$.

\begin{definition}[Primal certificate]\label{def:primal}
A primal certificate consists of rationals $c_0,\dots,c_{16}$, a matrix
$C\in\mathbb Q^{6\times6}$ and a rational $\varrho>0$. Put
\[
 \phi_0(t)=\sum_{j=0}^{16}c_jT_j(2t-1),\qquad
 \phi_i(t)=T_{j_i}(2t-1),\quad (j_1,\dots,j_6)=(0,1,2,3,6,7),
\]
and, for $a\in\R^6$, $\phi_a=\phi_0+\sum_ia_i\phi_i$. Let $m_a$ be the
density of \Cref{thm:primal-map} with $\kappa=k$ and the symmetric speed
\[
 D_a(t^3)=t^2e^{\phi_a(t)}\quad(0<t\le1),\qquad D_a(2-q)=D_a(q),
\]
and
\[
 F(a)=\Bigl(\int m_a-1,\ \int\sqrt{m_a}-1,\
 \bigl(H_{m_a}(s)-H_*(s)\bigr)_{s=s_0,s_1,\varsigma_{\mathrm{cal}}k},\
 S(m_a)-\frac{k^2}6\Bigr).
\]
\end{definition}

The factor $t^2$ matches the endpoint behavior
$d_\kappa\sim Q^{2/3}$ of the comparison densities. The logarithm of
the normalized speed, $\log(D_a(t^3)/t^2)=\phi_a(t)$, is a polynomial;
the coordinate derivative $\dd x/\dd t=3e^{-\phi_a(t)}$ is therefore regular
at $t=0$. The six components of $F$ are the two normalizations
and the four observations of the relaxation; $\kappa=k$ holds by
construction.

\begin{theorem}[Primal certificates define feasible densities]\label{thm:primal}
Suppose that $\varepsilon_F\ge\|CF(0)\|_\infty$ and
$\theta_F\ge\sup_{\|a\|_\infty\le\varrho}\|I-C\,DF(a)\|_\infty$ satisfy
$\theta_F<1$ and $\varepsilon_F+\theta_F\varrho<\varrho$. Then $F$ has
exactly one zero $a_*$ with $\|a_*\|_\infty\le\varrho$, the iterates
$a^{(0)}=0$, $a^{(j+1)}=a^{(j)}-CF(a^{(j)})$ converge to it, and
$m_{\mathrm{feas}}:=m_{a_*}$ belongs to $\mathcal A_{\mathrm{obs}}(\varsigma_{\mathrm{cal}})\cap\Mc$
and satisfies
\[
 E(m_{\mathrm{feas}})\le E(m_0)e^{6\varrho}.
\]
\end{theorem}

{Construction.} The search minimizes $E(m_\phi)$ numerically over all
$17$ Chebyshev coefficients of $\phi$, first with the five equations obtained
by omitting $S$, then with all six equations $F=0$. It starts from a fit to
the logarithm of the normalized speed of the comparison density $\bar m_k$
of $\Theta_{\mathrm{cert}}$ and uses the first numerical candidate to initialize
the second search. Six coordinates are selected for a chord Newton
correction with an approximate inverse $C$ of the Jacobian. The resulting
exact dyadic center and radius $\varrho\approx1.346\cdot10^{-16}$ are fixed
inputs to the subsequent contraction verification (\Cref{app:primal});
global optimality of the numerical search is not required.

\begin{proposition}[Computed values for the primal certificate]
\label{ca:primal-second}
The residuals $F_i(0)$ lie in the intervals of \Cref{tab:primal-values};
for every $\|a\|_\infty\le\varrho$ the derivative satisfies
$|\partial_jF_i(a)-g_{ij}|\le10^{-11}|g_{ij}|$ with $g_{ij}$ from
\Cref{tab:primal-values}; and $E(m_0)\le\num{feasible/central_energy/hi24}$.
The unrounded interval data also give
\[
 \begin{aligned}
 \|CF(0)\|_\infty&\le\num{feasible/residual/hi24},\\
 \sup_{\|a\|_\infty\le\varrho}\|I-CDF(a)\|_\infty
 &\le\num{feasible/contraction/hi24},
 \end{aligned}
\]
and $E(m_0)e^{6\varrho}\le\num{feasible/energy_ceiling/hi24}$.
\end{proposition}

\begin{corollary}[A feasible density below the sine energy]\label{cor:primal}
The primal certificate satisfies the hypotheses of \Cref{thm:primal} with
$\varepsilon_F=6.7299\cdot10^{-17}$ and $\theta_F=1.818\cdot10^{-7}$, and
\[
 \begin{gathered}
 c_{\mathrm{cert}}\le\inf_{\mathcal A_{\mathrm{obs}}(\varsigma_{\mathrm{cal}})\cap\Mc}E
 \le E(m_{\mathrm{feas}})\le U_{\mathrm{feas}}<\cs,\\
 U_{\mathrm{feas}}:=\num{feasible/energy_ceiling/hi24}.
 \end{gathered}
\]
\end{corollary}

The value of the relaxation therefore lies in
$[c_{\mathrm{cert}},U_{\mathrm{feas}}]$: its width is less than
$8.572156\cdot10^{-6}$, and the relaxation is more than
$7.1417\cdot10^{-8}$ below $\cs$. By
\Cref{thm:dual}(i), every dual certificate with $\varsigma=\varsigma_{\mathrm{cal}}$ that is
admissible at $k$ has $\ell_*(\Theta)\le E(m_{\mathrm{feas}})<\cs$. Closing the gap
requires further observations.

\section{No locally sharp differentiable bound at the sine density}\label{sec:structural-limits}

\Cref{cor:primal} shows that the five observations of
$\mathcal A_{\mathrm{obs}}(\varsigma_{\mathrm{cal}})$ do not determine $\cs$. The next theorem
shows that no finite set of observations of the same kinds does, in a
local and differentiable sense. The proof is in \Cref{app:obstruction}.

\begin{theorem}[No locally sharp $C^1$ bound from finitely many observations]
\label{thm:no-touching}
Fix scales $0<\sigma_1<\dots<\sigma_r$ and put
\[
 \mathbf O(m)=\bigl(\kappa(m),\ \tr K_m^2,\ (H_m(\sigma_j),H_m'(\sigma_j))_{j=1}^r\bigr),
 \qquad H_m'(s)=\partial_sH_m(s).
\]
Let $U$ be a neighborhood of $\mathbf O(m_*)$ and $\Phi\in C^1(U)$ with
$\Phi(\mathbf O(m_*))=E(m_*)$. Then for every $\varepsilon>0$ there is
$v\in C_c^\infty(\R)$ with $\|v\|_\infty<\varepsilon$ such that
$m=m_*+v$ satisfies $m\ge0$, $\int m=\int\sqrt m=1$, $\mathbf O(m)\in U$ and
$E(m)<\Phi(\mathbf O(m))$. The same holds if some coordinates of $\mathbf O$ are
omitted, and if the scales $\sigma_j$ are continuously differentiable
functions of $\kappa(m)$ with $\sigma_j(\kappa(m_*))>0$.
\end{theorem}

The conclusion concerns $C^1$ functions of the listed observations near
their values at $m_*$. It does not address arbitrary nonsmooth bounds
or assert a lower-energy perturbation on every exact observation level
set. Exact feasibility below $\cs$ for the observations used here is
established separately by \Cref{cor:primal}.

The proof compares decay rates. If $E\ge\Phi\circ\mathbf O$ near
$m_*$ with equality at $m_*$, then $m_*$ minimizes $E-\Phi\circ\mathbf O$ under the
two normalizations, and the first-order condition is a pointwise
identity
\[
 \tfrac32n_*=\alpha+\frac\zeta{2n_*}+\mathfrak g ,
\]
with Lagrange multipliers $\alpha,\zeta$ and a combination $\mathfrak g$ of the
gradients of the observations. These gradients are the squared ground
state, the diagonals $G_s(x,x)$ of the Green functions of
$-\partial_x^2+1+sm_*$, their derivatives in $s$, and the convolution
$e^{-2|\cdot|}*m_*$ (\Cref{lem:gradients}). All of them behave like
$c_0+c_1e^{-2x}+o(n_*)$ as $x\to\infty$, while $n_*$ decays like
$\exp(-\frac\pi2e^{2x})$ (\Cref{prop:target}). Letting $x\to\infty$ in the identity
multiplied by $n_*$, then in the identity itself, and then in the identity
multiplied by $e^{2x}$, forces $\zeta=0$, $\alpha+c_0=0$ and $c_1=0$; then
$\frac32n_*=o(n_*)$, which is impossible.

\section{Discussion}\label{sec:discussion}

By \Cref{cor:numerical,thm:sine-finite},
$2\sqrt{c_{\mathrm{cert}}}\le\mathscr G(\mu,L)\le2\sqrt{\cs}$ for all $0<\mu<L$, with
$\sqrt{\cs/c_{\mathrm{cert}}}<1.000003782131<1.000004$; in terms of the Green energy,
$c_{\mathrm{cert}}\le\underline{\mathsf e}\le\cs$ and $(\cs-c_{\mathrm{cert}})/\cs<7.6\cdot10^{-6}$.

\begin{hypothesis}[Optimality of the sine--Jacobi method]
\label{hyp:sine-lower}
$\underline{\mathsf e}=\cs$. Consequently $\mathscr G(\mu,L)=2\sqrt{\cs}$ for all
$0<\mu<L$.
\end{hypothesis}

The consequence follows from \Cref{thm:gain,thm:sine-finite}. By
\Cref{cor:primal}, the five observations of the relaxation leave a gap of
more than $7.1417\cdot10^{-8}$ below $\cs$, and by \Cref{thm:no-touching} no
continuously differentiable function of finitely many ground, trace and
determinant observations that equals $\cs$ at $m_*$ is a lower bound for $E$
near $m_*$. The spectral constraints of \eqref{eq:limit-problem} are
equivalent to $H_m(s)=H_*(s)$ for all $s\ge0$, since the zeros of the entire
function $s\mapsto\det(I+\frac s2K_m)$ are the numbers $-2/\lambda_j(K_m)$. Since $m_{\mathrm{feas}}$
also satisfies $\int\sqrt{m_{\mathrm{feas}}}=1$, adding the root mass to the
relaxation does not close the gap of \Cref{cor:primal}. By
\Cref{thm:common-measure}, the filters at all scales at a point are
integrals against one probability measure $\mu_x$; whether the relaxation
that keeps $H_m(s)$ for all $s\ge0$ admits a dual certificate of the type of
\Cref{sec:mechanism} whose state is the whole measure $\mu_x$ is open.

\paragraph{Use of AI tools.}
AI tools assisted the authors in writing and debugging code, comparing
computational approaches, and selecting numerical formulations with
more transparent and better-controlled error propagation. They also
provided numerous suggestions on analytical methods and on the structure
of proofs at the stages leading to certificates with the required
properties. The authors reviewed these suggestions and take
responsibility for the mathematical arguments, the implementation,
the verification of the certificates, and the final content of the paper.
The rigorous numerical claims rest on the analytical error bounds and
certificate-verification procedures described in the paper.

\appendix

\section{Recurrences, Green matrices and packing}\label{app:reduction}

This appendix proves the results of \Cref{sec:reduction-gain,sec:reduction-density}:
the correspondence between recurrences and Jacobi matrices and the
formula for the gain, the exponential representation of Green
matrices, and the packing and transfer theorems.

\subsection{Recurrences, Jacobi matrices and the gain}\label{sec:finite}

The continuant recurrence identifies residual polynomials with normalized
leading minors and expresses the terminal response through leading and
trailing minors. These two identities give the Jacobi correspondence
and the exact gain formula. We begin with the spectral quantities that
are fixed throughout the fiber.

For a real symmetric tridiagonal $J$ with diagonal $b_i$ and
$J_{j,j+1}=\sqrt{v_j}$, expansion of the determinant along the last row gives
\begin{equation}\label{eq:continuant}
 P_{-1}=0,\qquad P_0=1,\qquad
 P_{j+1}(z)=(z-b_{j+1})P_j(z)-v_jP_{j-1}(z)
 \quad(0\le j<N),\ v_0:=0,
\end{equation}
and the same recurrence holds for $\widehat P_{l,j}(z)=\det(zI-J_{[l+2:j]})$
in $j\ge l+1$, with $\widehat P_{l,l}=0$, $\widehat P_{l,l+1}=1$; note
$\widehat P_{l,N}=\widehat P_l$.

\begin{lemma}[Spectral identities of the fiber]\label{lem:trace}
Let $J\in\overline{\mathcal Y}_N$ and $G=(I-J)^{-1}$. Then $G$ is positive definite
with eigenvalues $(1-\xi_j)^{-1}$, $\xi_j=\cos((j-\frac12)\pi/N)$,
$1\le j\le N$, and
\[
 \tr G=N^2,\qquad \tr(G^2)=\frac{2N^4+N^2}3 .
\]
The eigenvalues of $G/N^2$, in decreasing order, are
\[
 r_{N,j}=\frac1{2N^2\sin^2((2j+1)\pi/(4N))}\qquad(0\le j<N).
\]
\end{lemma}

\begin{proof}\label{proof:lem:trace}
The characteristic polynomial of $J$ is $P=2^{1-N}T_N=\prod_j(z-\xi_j)$
with the distinct roots $\xi_j\in(-1,1)$ of $T_N$, so $G$ has the stated
eigenvalues, and
$\tr G=\sum_j(1-\xi_j)^{-1}=P'(1)/P(1)$,
$\tr G^2=\sum_j(1-\xi_j)^{-2}=(P'/P)^2(1)-P''(1)/P(1)$. The equation
$(1-x^2)T_N''-xT_N'+N^2T_N=0$ at $x=1$ gives $T_N'(1)=N^2T_N(1)=N^2$, and its
derivative $(1-x^2)T_N'''-3xT_N''+(N^2-1)T_N'=0$ at $x=1$ gives
$T_N''(1)=N^2(N^2-1)/3$. Hence $\tr G=N^2$ and
$\tr G^2=N^4-N^2(N^2-1)/3$. Finally
$1-\xi_{j+1}=2\sin^2((2j+1)\pi/(4N))$.
\end{proof}

\begin{proof}[Proof of \Cref{thm:jacobi}]\label{proof:thm:jacobi}
Since $\|J\|=\cos\frac\pi{2N}<1<z_0$, every $z_0I-J_{[1:j]}$ is positive
definite, so $P_j(z_0)>0$. Evaluating \eqref{eq:continuant} at $z_0$ and
dividing by $P_{j+1}(z_0)$ gives $\alpha_j-\beta_j=1$; the formulas also give $\gamma_j>0$,
$\beta_j\ge0$ with equality exactly when $v_j=0$, and $\alpha_0=1$,
$\beta_0=0$. Evaluating \eqref{eq:continuant} at $z(\lambda)$, with
$z(\lambda)-b_{j+1}=(z_0-b_{j+1})-2\lambda/\Delta$, and dividing by $P_{j+1}(z_0)$
shows that $P_j(z(\lambda))/P_j(z_0)$ satisfies the residual recurrence of
$\mathcal R_J$; hence it is $p_j$, and $p_N(\lambda)=T_N(z(\lambda))/T_N(z_0)$. So
$\mathcal R_J\in\mathfrak E_N(\mu,L)$.

Conversely, let $\mathcal R\in\mathfrak E_N(\mu,L)$, put $a_j=\Delta\gamma_j/2$,
$c_j=\prod_{i<j}a_i^{-1}$, $b_{j+1}=z_0-\alpha_j/a_j$ and
$v_j=\beta_j/(a_ja_{j-1})\ge0$. Then $\widetilde P_j(z)=c_jp_j(\Delta(z_0-z)/2)$
satisfies \eqref{eq:continuant} with these coefficients, so it is the
$j$-th leading minor of the Jacobi matrix $J$ with diagonal $b_i$ and
off-diagonal entries $\sqrt{v_j}$. Terminal exactness makes
$\widetilde P_N$ a multiple of $T_N$; being monic, $\widetilde P_N=2^{1-N}T_N$,
so $J\in\overline{\mathcal Y}_N$. Since $\alpha_j-\beta_j=1$, $p_j(0)=1$ and
$\widetilde P_j(z_0)=c_j$; inserting this into the formulas of the theorem
returns $(\alpha_j,\beta_j,\gamma_j)$. Inserting the formulas of the theorem
into the definitions of $b_{j+1},v_j$ returns the entries of $J$, the
off-diagonal ones by the nonnegative square root. The two maps are
therefore inverse to each other.

Since $\beta_j>0$ exactly when $v_j>0$, and $b_{j+1}=0$ exactly when
$\alpha_j=z_0a_j=\frac{L+\mu}2\gamma_j$, the bijection maps $\mathcal Y_N$ onto the
irreducible and $\X_N$ onto the centered recurrences. The zero-diagonal
matrix with $v_1=\frac12$ and $v_j=\frac14$ ($j\ge2$) has $P_1=T_1$ and
$P_j=2^{1-j}T_j$ ($j\ge2$) by \eqref{eq:continuant} and the Chebyshev
recurrence, so it lies in $\X_N$, and all six classes are nonempty.
\end{proof}

\begin{lemma}[Derivative of the terminal residual]\label{lem:gain-kernel}
For $J\in\overline{\mathcal Y}_N$, every $p\ge1$, every symmetric $p\times p$ matrix $H$
with $\mu I\preceq H\preceq LI$, and real $p\times p$ matrices $B=(B_0,\dots,B_{N-1})$,
with $Z=z_0I-2H/\Delta$,
\begin{equation}\label{eq:gain-kernel}
 \mathsf D_{\mathcal R_J,H}[B]
 =-\frac2{\Delta P_N(z_0)}\sum_{j=0}^{N-1}\widehat P_j(Z)B_jP_j(Z).
\end{equation}
\end{lemma}

\begin{proof}[Proof of \Cref{lem:gain-kernel}]\label{proof:lem:gain-kernel}
The iterates are polynomials in the entries of $H_0,\dots,H_{N-1}$. Let
$U_j$ be the derivative of $\mathbf p_j$ at $(H,\dots,H)$ in the direction
$B$. Since $\mathbf p_j=p_j(H)=P_j(Z)/P_j(z_0)$ at the base point,
$U_{-1}=U_0=0$ and
$U_{j+1}=(\alpha_jI-\gamma_jH)U_j-\beta_jU_{j-1}-\gamma_jB_jp_j(H)$. By the formulas
of \Cref{thm:jacobi}, $P_{j+1}(z_0)(\alpha_jI-\gamma_jH)=P_j(z_0)(Z-b_{j+1}I)$,
$P_{j+1}(z_0)\beta_j=v_jP_{j-1}(z_0)$ and $P_{j+1}(z_0)\gamma_j=2P_j(z_0)/\Delta$; so
$W_j=P_j(z_0)U_j$ satisfies
\[
 W_{j+1}=(Z-b_{j+1}I)W_j-v_jW_{j-1}-\frac2\Delta B_jP_j(Z),\qquad W_{-1}=W_0=0 .
\]
The solution of the homogeneous recurrence with $W_l=0$, $W_{l+1}=X$ is
$W_j=\widehat P_{l,j}(Z)X$, by the continuant recurrence of $\widehat P_{l,j}$.
By linearity, $W_N=-\frac2\Delta\sum_l\widehat P_{l,N}(Z)B_lP_l(Z)$, and
$\mathsf D[B]=U_N=W_N/P_N(z_0)$ is \eqref{eq:gain-kernel}.
\end{proof}

\begin{proof}[Proof of \Cref{thm:gain}]\label{proof:thm:gain}
Put $\bar c=2/(\Delta P_N(z_0))$, $c_l(x,y)=\widehat P_l(x)P_l(y)$ and
$C_*^2=\max_{[-1,1]^2}\sum_lc_l^2$. Diagonalize $Z=O\diag(z_i)O^T$ with
$z_i\in[-1,1]$ and put $\widetilde B_l=O^TB_lO$. By \eqref{eq:gain-kernel} the
$(i,i')$ entry of $O^T\mathsf D[B]O$ is $-\bar c\sum_lc_l(z_i,z_{i'})(\widetilde B_l)_{ii'}$,
and the Cauchy--Schwarz inequality in $l$ gives
$\|\mathsf D[B]\|_{\HS}^2\le\bar c^2C_*^2\sum_l\|B_l\|_{\HS}^2$. For a
maximizer $(x_*,y_*)$, $p=2$, $Z=\diag(x_*,y_*)$, that is,
$H=\frac\Delta2(z_0I-Z)$, and
$B_l=c_l(x_*,y_*)E_{12}/C_*$, with $E_{12}$ the matrix unit, give
$\mathsf D[B]=-\bar cC_*E_{12}$ and $\sum_l\|B_l\|_{\HS}^2=1$; this proves the
equality in \eqref{eq:full-gain} and attainment in dimension two.

Deleting row and column $l+1$ from $I-J$ leaves the blocks
$I-J_{[1:l]}$ and $I-J_{[l+2:N]}$, so $d_{l+1}=P_l(1)\widehat P_l(1)/P_N(1)$.
With $P_N(1)/P_N(z_0)=T_N(1)/T_N(z_0)=\epsN$, the value of the maximum at
$(1,1)$ gives the inequality in \eqref{eq:full-gain}. Taking the infimum over
the fiber, which is in bijection with $\mathfrak E_N$, and the lower limit in
$N$, gives $\mathscr G(\mu,L)\ge2\sqrt{\underline{\mathsf e}}$, since the square root is
continuous and increasing.
\end{proof}

\subsection{Green matrices as exponential kernels}\label{app:centering}

This subsection proves \Cref{cor:centering-packing}. A Jacobi matrix with
zero diagonal has an exponential Green matrix (\Cref{lem:green}); the
centering lift reduces a general Jacobi matrix to this case
(\Cref{thm:centering-transfer}); and the points of the lift leave room for
the intervals.

\begin{lemma}[Green factorization]\label{lem:green}
Let $J$ be a real symmetric tridiagonal $M\times M$ matrix, $M\ge2$,
with zero diagonal, positive off-diagonal entries and $\|J\|<1$, and let
$G=(I-J)^{-1}$, $d_i=G_{ii}$. Put $t_0=t_M=0$ and
$t_i=J_{i,i+1}G_{i,i+1}$ for $1\le i<M$. Then $t_i>0$, $d_i=1+t_{i-1}+t_i$,
and with
\[
 x_1=0,\qquad x_{i+1}-x_i=\tfrac12\log(1+1/t_i)\quad(1\le i<M),
\]
one has
\[
 G_{ij}=\sqrt{d_id_j}\,e^{-|x_i-x_j|}\qquad(1\le i,j\le M).
\]
\end{lemma}

\begin{proof}[Proof of \Cref{lem:green}]\label{proof:lem:green}
Put $A=I-J\succ0$, $a_i=J_{i,i+1}>0$, $\theta_i=\det A_{[1:i]}$ and
$\varphi_i=\det A_{[i:M]}$, with $\theta_0=\varphi_{M+1}=1$; all are positive. The
inverse of a tridiagonal matrix is
\[
 G_{ij}=\frac{\theta_{i-1}\varphi_{j+1}}{\theta_M}\prod_{r=i}^{j-1}a_r
 \qquad(i\le j)
\]
\cite{Meurant}, and the Laplace expansion of $\det A$ along its first $i$
rows gives $\theta_i\varphi_{i+1}-a_i^2\theta_{i-1}\varphi_{i+2}=\theta_M$ $(1\le i<M)$. Hence $G_{ij}>0$,
$G_{ij}G_{jl}=G_{il}d_j$ for $i\le j\le l$, and
$d_id_{i+1}-G_{i,i+1}^2=\theta_{i-1}\varphi_{i+2}/\theta_M=G_{i,i+1}/a_i$. The diagonal of
$AG=I$ reads $d_i=1+t_{i-1}+t_i$, and
$G_{i,i+1}^2/(d_id_{i+1})=t_i/(1+t_i)=e^{-2(x_{i+1}-x_i)}$. The quotients
$c_{ij}=G_{ij}/\sqrt{d_id_j}>0$ satisfy $c_{il}=c_{ij}c_{jl}$ for $i\le j\le l$,
so $c_{ij}=\prod_{r=i}^{j-1}c_{r,r+1}=e^{-(x_j-x_i)}$.
\end{proof}

\begin{theorem}[The centering lift]\label{thm:centering-transfer}
Let $J$ be a real symmetric tridiagonal $M\times M$ matrix with positive
off-diagonal entries and $\|J\|<1$. There is a unique lower bidiagonal
matrix $\mathsf B$ with positive diagonal and subdiagonal such that
$\mathsf B\mathsf B^T=(I+J)/2$. Define the $2M\times2M$ matrix $\widehat J$
on the index sets $\mathrm o=(1,3,\dots,2M-1)$ and
$\mathrm e=(2,4,\dots,2M)$ by
\[
 \widehat J_{\mathrm o,\mathrm e}=\mathsf B,\qquad
 \widehat J_{\mathrm e,\mathrm o}=\mathsf B^T,\qquad
 \widehat J_{\mathrm o,\mathrm o}=\widehat J_{\mathrm e,\mathrm e}=0.
\]
Then $\widehat J$ is tridiagonal with zero diagonal and positive
off-diagonal entries, $\|\widehat J\|<1$,
\begin{equation}\label{eq:lift-green-block}
 \bigl[(I-\widehat J)^{-1}\bigr]_{\mathrm o,\mathrm o}=2(I-J)^{-1},
\end{equation}
and $\det(xI-\widehat J)=2^{-M}\det((2x^2-1)I-J)$. For $M=N$ the map
$J\mapsto\widehat J$ is a bijection $\mathcal Y_N\to\X_{2N}$ with inverse
$J=2\mathsf B\mathsf B^T-I$.
\end{theorem}

\begin{proof}[Proof of \Cref{thm:centering-transfer}]
\label{proof:thm:centering-transfer}
The matrix $\Sigma=(I+J)/2$ satisfies $0\prec\Sigma\prec I$ and is tridiagonal;
its Cholesky factor $\mathsf B$ is the unique lower triangular matrix with
positive diagonal and $\mathsf B\mathsf B^T=\Sigma$, it is bidiagonal, and
$\mathsf B_{i+1,i}=\Sigma_{i,i+1}/\mathsf B_{ii}>0$ \cite[\S\S4.2--4.3]{GolubVanLoan}. In
the interlaced order the entries of $\widehat J$ above the diagonal are
$\mathsf B_{ii}$ and $\mathsf B_{i+1,i}$, so $\widehat J$ is a Jacobi matrix with
zero diagonal and positive off-diagonal entries. Grouping the odd
indices before the even ones, $\widehat J$ becomes
$\mathscr J=\bigl(\begin{smallmatrix}0&\mathsf B\\\mathsf B^T&0\end{smallmatrix}\bigr)$,
and $\mathscr J^2=\diag(\mathsf B\mathsf B^T,\mathsf B^T\mathsf B)$ has the
eigenvalues of $\Sigma$ in $(0,1)$; so $\|\widehat J\|<1$. The Schur complement
of the lower right block of $xI-\mathscr J$ gives
$\det(xI-\widehat J)=x^M\det(xI-x^{-1}\Sigma)=\det(x^2I-\Sigma)$ for $x\ne0$, and by
continuity for all $x$. The system $(I-\mathscr J)(r,s)=(f,0)$ reads
$r-\mathsf Bs=f$, $s=\mathsf B^Tr$, so $(I-\Sigma)r=f$ and $r=2(I-J)^{-1}f$, which
is \eqref{eq:lift-green-block}.

For $J\in\mathcal Y_N$, $T_N(2x^2-1)=T_{2N}(x)$ gives
$\det(xI-\widehat J)=2^{1-2N}T_{2N}(x)$, so $\widehat J\in\X_{2N}$. Conversely, for
$C\in\X_{2N}$ the block $\mathsf B=C_{\mathrm o,\mathrm e}$ is lower bidiagonal with
positive diagonal and subdiagonal, $J=2\mathsf B\mathsf B^T-I$ is a Jacobi
matrix with positive off-diagonal entries, and the same determinant
identity gives $\det(yI-J)=2^{1-N}T_N(y)$ for $y\ge-1$, hence for all $y$;
so $J\in\mathcal Y_N$. By uniqueness of the Cholesky factor the two maps are
inverse to each other.
\end{proof}

\begin{proof}[Proof of \Cref{cor:centering-packing}]\label{proof:cor:centering-packing}
For $M=1$, take $y_1=0$ and $I_1=[-1/d_1,0]$. Assume $M\ge2$ and
apply \Cref{lem:green} to the lift $\widehat J$, with points
$x_1<\dots<x_{2M}$, numbers $t_r$ and diagonal $D_r=1+t_{r-1}+t_r$ of
$(I-\widehat J)^{-1}$, and put $y_i=x_{2i-1}$. By \eqref{eq:lift-green-block},
$D_{2i-1}=2d_i$ and $2G_{ij}=\sqrt{(2d_i)(2d_j)}\,e^{-|y_i-y_j|}$. Since
$\frac12\log(1+1/t)=\operatorname{artanh}(1/(2t+1))\ge1/(2t+1)$,
\[
 x_{2i}-x_{2i-2}\ge\frac1{2t_{2i-2}+1}+\frac1{2t_{2i-1}+1}\ge\frac4{2D_{2i-1}}=\frac1{d_i}
 \qquad(2\le i\le M-1),
\]
by the inequality $\frac1a+\frac1b\ge\frac4{a+b}$. Put $I_1=[y_1-1/d_1,y_1]$,
$I_M=[y_M,y_M+1/d_M]$ (for $M\ge2$), and for $2\le i\le M-1$ let $I_i$ be the
interval of length $1/d_i$ in $[x_{2i-2},x_{2i}]$ whose left end point is
$\max\{x_{2i-2},y_i-1/d_i\}$; it lies in $[x_{2i-2},x_{2i}]$ and contains $y_i$,
because $x_{2i-2}<y_i<x_{2i}$ and $x_{2i}-x_{2i-2}\ge1/d_i$. These intervals lie in
the consecutive ranges $(-\infty,x_1]$, $[x_{2i-2},x_{2i}]$ and $[x_{2M-1},\infty)$,
which have pairwise disjoint interiors.
\end{proof}

\subsection{Feature operators, packing and spectral transfer}\label{sec:embedding}

The densities are compared with atomic measures through operators on
$H^1(\R)$: the kernel $e^{-|x-y|}$ is the Gram kernel of the functions
$v_x=e^{-|\cdot-x|}/\sqrt2$, so $K_m$ and its atomic analogues are Gram
operators of weighted families $\{v_x\}$ (\Cref{lem:features}). Moving mass
by a small distance changes these operators little in trace norm, which
gives the packing bound; the spectrum of $G/N^2$ is then compared with
$(\Lambda_j)$.

Let $\HH=H^1(\R)$ be the real Sobolev space with
$\ip uv_\HH=\int_\R(u'v'+uv)$, and put
\[
 \mathsf k_x(t)=\tfrac12e^{-|t-x|},\qquad v_x=\sqrt2\,\mathsf k_x,\qquad
 \mathcal P_x=v_x\otimes v_x,\qquad
 \mathcal K_\sigma=\int_\R\mathcal P_x\dd\sigma(x)
\]
for a finite signed Borel measure $\sigma$, where $(a\otimes b)u=\ip ub_\HH a$;
for real $h\in L^1(\R)$ write $\mathcal K_h=\mathcal K_{h(x)\dd x}$. Trace norms
are denoted $\|\cdot\|_1$. Eigenvalues of positive compact operators are
listed in nonincreasing order with multiplicity and padded with zeros.

\begin{lemma}[Feature operators]\label{lem:features}
\begin{enumerate}
\item $\HH$ is a separable Hilbert space; every $u\in\HH$ has a continuous
representative with $u(x)\to0$ as $|x|\to\infty$; and
$\ip u{\mathsf k_x}_\HH=u(x)$, $\ip{v_x}{v_y}_\HH=e^{-|x-y|}$,
$\|\mathcal P_x-\mathcal P_y\|_1=2\sqrt{1-e^{-2|x-y|}}$.
\item The integral defining $\mathcal K_\sigma$ exists as a Bochner integral
in the trace class, $\|\mathcal K_\sigma\|_1\le|\sigma|(\R)$ and
$\tr\mathcal K_\sigma=\sigma(\R)$; for $\sigma\ge0$, $\mathcal K_\sigma$ is
positive.
\item For $0\le m\in L^1(\R)$, $K_m$ is positive and trace class, its
nonzero eigenvalues coincide with those of $\mathcal K_m$ with
multiplicity, and $\tr K_m=\int m$. For $\sigma=\sum_{i=1}^Mw_i\delta_{y_i}$
with $w_i\ge0$, the nonzero eigenvalues of $\mathcal K_\sigma$ coincide with
those of the matrix $(\sqrt{w_iw_j}\,e^{-|y_i-y_j|})_{i,j=1}^M$.
\item For positive trace-class operators $A,B$ on a separable real Hilbert
space,
\begin{equation}\label{eq:trace-lipschitz}
 \sum_{j\ge0}|\lambda_j(A)-\lambda_j(B)|\le\|A-B\|_1 .
\end{equation}
\item If $\pi$ is a finite positive Borel measure on $\R^2$ of total mass
$M_\pi$ whose marginals are $\sigma$ and $\sigma'$, then
\begin{equation}\label{eq:coupling-spectrum}
 \|\mathcal K_\sigma-\mathcal K_{\sigma'}\|_1
 \le2\sqrt{2M_\pi\int|x-y|\dd\pi(x,y)} .
\end{equation}
\end{enumerate}
\end{lemma}

\begin{proof}\label{proof:lem:features}
(i) The Sobolev facts follow from \cite[Proposition~8.1 and Theorem~8.2]{Brezis}.
In particular, $(u^2)'=2uu'\in L^1$, so $u^2$ has limits at both
infinities; these limits vanish because $u\in L^2$. The distributional
identity $(-\partial_t^2+1)\mathsf k_x=\delta_x$ gives
$\ip u{\mathsf k_x}_\HH=u(x)$, first for smooth compactly supported $u$
and then by density in $H^1$. Thus $\ip{v_x}{v_y}_\HH=e^{-|x-y|}$.
For unit vectors $v,w$ with $r=\ip vw$, the operator
$v\otimes v-w\otimes w$ has nonzero eigenvalues
$\pm\sqrt{1-r^2}$ (diagonalize on $\Span\{v,w\}$). This proves the
trace-norm formula, including the case $v=w$.

(ii) The map $x\mapsto\mathcal P_x$ is continuous in trace norm, has a
separable range, and satisfies $\|\mathcal P_x\|_1=1$. It is therefore
Bochner integrable against every finite signed measure, with
$\|\mathcal K_\sigma\|_1\le|\sigma|(\R)$. Positivity follows by integrating
$\ip{\mathcal P_xu}u=|\ip u{v_x}|^2$. Continuity of the trace on the
trace class gives $\tr\mathcal K_\sigma=\int\tr\mathcal P_x\dd\sigma=\sigma(\R)$.

(iii) Define $A_m:L^2(\R)\to\HH$ by
$A_mf=\int v_x\sqrt{m(x)}f(x)\dd x$. The Bochner integral exists and
$\|A_mf\|_\HH\le\|m\|_1^{1/2}\|f\|_2$. Part (i) gives
\[
 A_m^*u=\sqrt{2m}\,u,\qquad A_mA_m^*=\mathcal K_m,\qquad A_m^*A_m=K_m.
\]
For an orthonormal basis $(e_j)$ of $\HH$, Tonelli's theorem yields
\[
 \|A_m^*\|_{\HS}^2
 =\sum_j\int m(x)|\ip{e_j}{v_x}_\HH|^2\dd x
 =\int m(x)\|v_x\|_\HH^2\dd x=\int m.
\]
Thus $A_m$ is Hilbert--Schmidt, both products are positive trace class,
and their nonzero eigenvalues coincide with multiplicity: $A_m$ and
$\lambda^{-1}A_m^*$ identify the eigenspaces at $\lambda>0$.
For an atomic measure, use instead
$Ca=\sum_i\sqrt{w_i}a_iv_{y_i}$; then $CC^*=\mathcal K_\sigma$ and
$(C^*C)_{ij}=\sqrt{w_iw_j}e^{-|y_i-y_j|}$.

(iv) Write $A-B=D_+-D_-$ for its positive and negative parts. The min--max
principle gives eigenvalue monotonicity under operator order. Since
$A\preceq B+D_+$,
\[
 \sum_j(\lambda_j(A)-\lambda_j(B))_+
 \le\sum_j[\lambda_j(B+D_+)-\lambda_j(B)]=\tr D_+.
\]
Exchange $A$ and $B$ and add; $\tr D_++\tr D_-=\|A-B\|_1$ proves
\eqref{eq:trace-lipschitz}. Finite-dimensional spaces may be enlarged
by a zero summand to implement the stated padding convention.

(v) The marginal identities for Bochner integrals give
$\mathcal K_\sigma-\mathcal K_{\sigma'}=\int(\mathcal P_x-\mathcal P_y)\dd\pi$.
Part (i), $1-e^{-2a}\le2a$, and Cauchy--Schwarz imply
\[
 \|\mathcal K_\sigma-\mathcal K_{\sigma'}\|_1
 \le2\int\sqrt{1-e^{-2|x-y|}}\dd\pi
 \le2\sqrt{2M_\pi\int|x-y|\dd\pi}.\qedhere
\]
\end{proof}

\begin{proof}[Proof of \Cref{thm:packing}]\label{proof:thm:packing}
Split $J$ at its zero off-diagonal entries into irreducible diagonal
blocks $J^{(1)},\dots,J^{(q)}$; each is real symmetric tridiagonal with
positive off-diagonal entries and $\|J^{(p)}\|\le\|J\|<1$, and
$G=(I-J)^{-1}$ is block diagonal with blocks $(I-J^{(p)})^{-1}$. For a
block of size at least two, \Cref{cor:centering-packing} gives points
$y_i$ and closed intervals $I_i\ni y_i$ of lengths $1/d_i$ with disjoint
interiors such that $G_{ij}=\sqrt{d_id_j}e^{-|y_i-y_j|}$ inside the block.
For a block of size one with entry $b$, $d_i=(1-b)^{-1}>0$; put $y_i=0$
and $I_i=[-1/d_i,0]$. Translate the blocks, in their order, so that the
union of the intervals of each block lies at distance at least $2\log N$
from the union of the intervals of every other block; translations
preserve the factorization within a block. Define $m_J$ as in the
statement. Since the $I_i$ have disjoint interiors and
$|I_i|=1/d_i$,
\[
 \int m_J=\sum_i\frac{d_i^2}{N^2}\cdot\frac1{d_i}=\frac{\tr G}{N^2}=1,\qquad
 \int\sqrt{m_J}=\sum_i\frac{d_i}N\cdot\frac1{d_i}=1,\qquad
 E(m_J)=\sum_i\frac{d_i^3}{N^3}\cdot\frac1{d_i}=\frac{\E_N(J)}{N^3},
\]
by \Cref{lem:trace}; $m_J$ is bounded with compact support, so
$m_J\in\Mc$.

Put $w_i=d_i/N^2$, so that $\sum_iw_i=1$, and
$\sigma_J=\sum_iw_i\delta_{y_i}$. The measure
$\pi=\sum_i\frac{d_i^2}{N^2}\ind_{I_i}(x)\dd x\otimes\delta_{y_i}$ has
marginals $m_J\dd x$ and $\sigma_J$ and mass one. Since
$\int_{I}|x-y|\dd x\le|I|^2/2$ for $y\in I$,
\[
 \int|x-y|\dd\pi\le\sum_i\frac{d_i^2}{N^2}\cdot\frac1{2d_i^2}=\frac1{2N},
\]
and \eqref{eq:coupling-spectrum} gives
$\|\mathcal K_{m_J}-\mathcal K_{\sigma_J}\|_1\le2/\sqrt N$. By
\Cref{lem:features}(iii) and (iv),
$\sum_j|\lambda_j(K_{m_J})-\lambda_j(\widetilde G)|\le2/\sqrt N$, where
$\widetilde G_{ij}=\sqrt{w_iw_j}\,e^{-|y_i-y_j|}$. For $i,j$ in the same
block $\widetilde G_{ij}=G_{ij}/N^2$; for $i,j$ in different blocks
$G_{ij}=0$ and $|y_i-y_j|\ge2\log N$, so
$0<\widetilde G_{ij}\le\sqrt{w_iw_j}/N^2$. Each matrix unit has trace norm
one, so
\[
 \Bigl\|\widetilde G-\frac G{N^2}\Bigr\|_1\le\sum_{i,j}\frac{\sqrt{w_iw_j}}{N^2}
 =\frac1{N^2}\Bigl(\sum_i\sqrt{w_i}\Bigr)^2\le\frac{N\sum_iw_i}{N^2}=\frac1N .
\]
Both matrices are positive semidefinite, and \eqref{eq:trace-lipschitz}
together with the triangle inequality gives
\eqref{eq:block-spectral-error}.
\end{proof}

\begin{proof}[Proof of \Cref{thm:transfer}]\label{proof:thm:transfer}
By \Cref{lem:trace}, $\lambda_j(G/N^2)=r_{N,j}$ for $j<N$ and zero for
$j\ge N$. For $0\le j<N$ the argument $a=(2j+1)\pi/(4N)$ lies in
$(0,\pi/2)$, where $\frac2\pi a\le\sin a\le a$; hence
\[
 \Lambda_j=\frac{8}{\pi^2(2j+1)^2}=\frac1{2N^2a^2}\le r_{N,j}
 \le\frac{\pi^2}{8N^2a^2}=\frac2{(2j+1)^2}.
\]
For fixed $j$, $r_{N,j}\to\Lambda_j$ because $\sin a/a\to1$. Since
$\sum_jr_{N,j}=\tr G/N^2=1$ and $r_{N,j}\le2/(2j+1)^2$, dominated
convergence gives $\sum_j\Lambda_j=1$. Similarly
$\sum_jr_{N,j}^2=\tr(G^2)/N^4=\frac23+\frac1{3N^2}$ by \Cref{lem:trace},
and $r_{N,j}^2\le4/(2j+1)^4$, so $\sum_j\Lambda_j^2=\frac23$. Since
$r_{N,j}\ge\Lambda_j$ for $j<N$,
\[
 \sum_j|r_{N,j}-\Lambda_j|
 =\Bigl(1-\sum_{j<N}\Lambda_j\Bigr)+\sum_{j\ge N}\Lambda_j
 =2\sum_{j\ge N}\Lambda_j,
\]
and $\frac1{(2j+1)^2}\le\frac1{(2j-1)(2j+1)}=\frac12\bigl(\frac1{2j-1}-\frac1{2j+1}\bigr)$
gives, by telescoping, $\sum_{j\ge N}\Lambda_j\le\frac8{\pi^2}\cdot\frac1{2(2N-1)}$.
Hence $\sum_j|r_{N,j}-\Lambda_j|\le8/(\pi^2(2N-1))$, and with
\eqref{eq:block-spectral-error} and the triangle inequality this proves
\eqref{eq:deltaN}.

For the determinant, $P(z)=\prod_j(z-\xi_j)=2^{1-N}T_N(z)$ and
$T_N(1)=1$ give
\[
 \prod_{j<N}\Bigl(1+\frac s2r_{N,j}\Bigr)
 =\prod_{j=1}^N\frac{1+s/(2N^2)-\xi_j}{1-\xi_j}
 =\frac{T_N(1+s/(2N^2))}{T_N(1)}=T_N\Bigl(1+\frac s{2N^2}\Bigr).
\]
Use $T_N(\cosh t)=\cosh(Nt)$ with
$t_N=2\operatorname{arsinh}(\sqrt s/(2N))$. Then
$\cosh t_N=1+s/(2N^2)$ and $Nt_N\to\sqrt s$, giving
$\sum_j\log(1+sr_{N,j}/2)\to\log\cosh\sqrt s=H_*(s)$, also for $s=0$.
The map
$x\mapsto\log(1+sx/2)$ is $(s/2)$-Lipschitz on $[0,\infty)$, so
\[
 \Bigl|\sum_j\log(1+sr_{N,j}/2)-\sum_j\log(1+s\Lambda_j/2)\Bigr|
 \le\frac s2\sum_j|r_{N,j}-\Lambda_j|\longrightarrow0,
\]
which proves $\sum_j\log(1+s\Lambda_j/2)=H_*(s)$; the series converges
because its terms are at most $s\Lambda_j/2$. The same Lipschitz bound
and \eqref{eq:deltaN} give $|H_{m_J}(s)-H_*(s)|\le\frac s2\delta_N$. Since
$\sum_j\lambda_j(K_{m_J})=\int m_J=1$ and $\sum_j\Lambda_j=1$, all
$\lambda_j(K_{m_J})$ and $\Lambda_j$ lie in $[0,1]$, so
\[
 \Bigl|\tr K_{m_J}^2-\frac23\Bigr|
 =\Bigl|\sum_j(\lambda_j(K_{m_J})-\Lambda_j)(\lambda_j(K_{m_J})+\Lambda_j)\Bigr|
 \le2\delta_N .
\]
Finally, $|\lambda_0(K_{m_J})-\Lambda_0|\le\delta_N<\Lambda_0/2$ gives
$\lambda_0(K_{m_J})>\Lambda_0/2$ and
\[
 |\kappa(m_J)-k|=\frac{2|\lambda_0(K_{m_J})-\Lambda_0|}{\lambda_0(K_{m_J})\Lambda_0}
 \le\frac{4\delta_N}{\Lambda_0^2},
\]
where $k=2/\Lambda_0=\pi^2/4$.
\end{proof}

\section{The sine reference}\label{app:sine}

The sine reference requires both a finite-dimensional extremum and its
continuum limit. Orthogonal polynomials on the Chebyshev nodes diagonalize
$J_N^\star$ and place the gain maximum at $x=y=1$. Their limit gives the
Legendre Green kernel $G_*$; its diagonal is $\phi_*/2$. A uniform bound
on the discrete diagonals upgrades convergence to $L^2$, while Jacobi's
theta identities give an explicit density and a finite-sum enclosure
for its energy.

Throughout this appendix $h=\pi/N$, $\theta_i=(i-\frac12)h$,
$x_i=\cos\theta_i$ and $p_i=\sin\theta_i$ for $1\le i\le N$.

\begin{proof}[Proof of the finite part of \Cref{thm:sine-finite}]
\label{proof:thm:sine-finite}
{Spectrum.} Put $w_i=p_i\sin(h/2)$. Since
$2\sin(h/2)\sin\theta_i=\cos((i-1)h)-\cos(ih)$, the sum $\sum_iw_i$
telescopes to $\frac12(1-\cos(Nh))=1$. For functions $f$ on the nodes
$\{x_1,\dots,x_N\}$, written $f_i=f(x_i)$, define
\[
 (\mathcal L_Nf)_i=\frac{\sin(ih)f_{i+1}+\sin((i-1)h)f_{i-1}}{2p_i}
 \qquad(1\le i\le N),
\]
where the coefficient of $f_0$ at $i=1$ and that of $f_{N+1}$ at $i=N$
vanish. Since
$w_i(\mathcal L_N)_{i,i+1}=\frac12\sin(h/2)\sin(ih)=w_{i+1}(\mathcal L_N)_{i+1,i}$,
$\mathcal L_N$ is symmetric in $\ell^2(w)$. The matrix
$D^{1/2}\mathcal L_ND^{-1/2}$, $D=\diag(w_i)$, is symmetric tridiagonal with
zero diagonal and entries
$\sqrt{w_i/w_{i+1}}\,\sin(ih)/(2p_i)=\sin(ih)/(2\sqrt{p_ip_{i+1}})$, that is,
it equals $J_N^\star$; these entries are positive for $1\le i<N$.

Let $f=T_j$, so that $f_i=\cos(j\theta_i)$, and use $\theta_{i\pm1}=\theta_i\pm h$,
$ih=\theta_i+h/2$, $(i-1)h=\theta_i-h/2$. The terms with the vanishing
coefficients may be evaluated by the same formula, so for every $i$
\begin{align*}
 2p_i(\mathcal L_NT_j)_i
 &=\sin(\theta_i+\tfrac h2)\cos(j\theta_i+jh)
  +\sin(\theta_i-\tfrac h2)\cos(j\theta_i-jh)\\
 &=\cos(j\theta_i)\cos(jh)\cdot2\sin\theta_i\cos\tfrac h2
  -\sin(j\theta_i)\sin(jh)\cdot2\cos\theta_i\sin\tfrac h2 .
\end{align*}
With $U_{j-1}(\cos\theta)=\sin(j\theta)/\sin\theta$ this reads
\[
 \mathcal L_NT_j=\cos\tfrac h2\cos(jh)\,T_j-\sin\tfrac h2\sin(jh)\,xU_{j-1}
 \quad\text{on the nodes}.
\]
For $j\ge1$ both $T_j$ and $xU_{j-1}$ have degree $j$ and leading
coefficient $2^{j-1}$, and $\mathcal L_NT_0=\cos\frac h2$. Hence, for
$0\le j<N$, $\mathcal L_N$ maps the restrictions to the nodes of the
polynomials of degree at most $j$ into themselves, and multiplies the
coefficient of $x^j$ by $\cos\frac h2\cos(jh)-\sin\frac h2\sin(jh)=\cos((j+\frac12)h)$.
The $N$ nodes are distinct, so the restrictions of the polynomials of
degree less than $N$ form all of $\R^N$, and $\ell^2(w)$ is an inner
product on them. Let $\phi_{N,j}$ be the orthonormal polynomials of
$\ell^2(w)$ with positive leading coefficients, $0\le j<N$. Then
$\mathcal L_N\phi_{N,j}$ has degree at most $j$, and for $l<j$,
$\ip{\mathcal L_N\phi_{N,j}}{\phi_{N,l}}_w=\ip{\phi_{N,j}}{\mathcal L_N\phi_{N,l}}_w=0$
because $\mathcal L_N\phi_{N,l}$ has degree at most $l$. Hence
$\mathcal L_N\phi_{N,j}=\cos((j+\frac12)h)\phi_{N,j}$. The vectors
$U_{\cdot j}=(\sqrt{w_i}\phi_{N,j}(x_i))_i$ are orthonormal eigenvectors of
$J_N^\star$ with eigenvalues $\cos((j+\frac12)h)$, $0\le j<N$; these are
the $N$ distinct roots of $T_N$, so the monic characteristic polynomial
of $J_N^\star$ is $2^{1-N}T_N$ and $J_N^\star\in\X_N$. By
\Cref{thm:jacobi}, $\mathcal R_N^\star$ is centered.

{Reversal.} Since $\sin((N-i)h)=\sin(ih)$,
$\sin\theta_{N-i}=\sin\theta_{i+1}$ and $\sin\theta_{N+1-i}=\sin\theta_i$,
the entries satisfy $J^\star_{N-i,N+1-i}=J^\star_{i,i+1}$. Hence
$RJ_N^\star R=J_N^\star$ for the reversal permutation $R$, which maps
$J^\star_{[s+2:N]}$ onto $J^\star_{[1:N-s-1]}$, and
$\widehat P_s=P_{N-s-1}$ for $0\le s<N$.

{Leading principal polynomials.} Put $A_0=1$ and
$A_s=\prod_{r=0}^{s-1}\frac{\sin((r+\frac12)h)}{\sin((r+1)h)}>0$ for
$1\le s\le N-1$, and
\[
 C_s(\cos\theta)=\sum_{j=0}^sA_jA_{s-j}\cos((s-2j)\theta)\qquad(0\le s\le N-1).
\]
We show that for $1\le s\le N-2$
\begin{equation}\label{eq:rogers}
 \sin((s+1)h)\,C_{s+1}=2\sin((s+\tfrac12)h)\,xC_s-\sin(sh)\,C_{s-1}.
\end{equation}
Using $2\cos\theta\cos(n\theta)=\cos((n+1)\theta)+\cos((n-1)\theta)$, the
three functions in \eqref{eq:rogers} are combinations of
$\cos((s+1-2k)\theta)$, $0\le k\le s+1$, with coefficients
\[
 [C_{s+1}]_k=A_kA_{s+1-k},\qquad
 [2xC_s]_k=A_kA_{s-k}+A_{k-1}A_{s+1-k},\qquad
 [C_{s-1}]_k=A_{k-1}A_{s-k},
\]
where $A_{-1}=0$ and $A_n=0$ for $n<0$. It suffices to compare these
coefficients. For $k=0$ the comparison reads
$\sin((s+1)h)A_{s+1}=\sin((s+\frac12)h)A_s$, which is the definition of
$A_{s+1}$; the case $k=s+1$ is the same. For $1\le k\le s$ put
$l=s+1-k\ge1$; dividing by $A_{k-1}A_{l-1}$ and using
$A_n/A_{n-1}=\sin((n-\frac12)h)/\sin(nh)$, then multiplying by
$\sin(kh)\sin(lh)$, the comparison becomes, with $\alpha=kh$,
$\beta=lh$, $\delta=h/2$,
\[
 \sin(\alpha+\beta)\sin(\alpha-\delta)\sin(\beta-\delta)
 +\sin(\alpha+\beta-2\delta)\sin\alpha\sin\beta
 =\sin(\alpha+\beta-\delta)\bigl[\sin(\alpha-\delta)\sin\beta
 +\sin(\beta-\delta)\sin\alpha\bigr].
\]
Both sides equal
$\cos(\alpha-\beta)\cos\delta\sin\sigma-\sin\sigma\cos\sigma$,
where $\sigma=\alpha+\beta-\delta$, by the product-to-sum identities.
This proves \eqref{eq:rogers}.

The coefficient of $x^s$ in $C_s$ is $2^sA_s$, since $\cos(n\theta)=T_n(x)$
has leading coefficient $2^{n-1}$ for $n\ge1$ and the degree $s$ occurs
for $j=0$ and $j=s$. Put $\widetilde P_s=C_s/(2^sA_s)$. Then
$\widetilde P_0=1$, $\widetilde P_1=x$, and dividing \eqref{eq:rogers} by
$2^{s+1}\sin((s+1)h)A_{s+1}=2^{s+1}\sin((s+\frac12)h)A_s$ gives
\[
 \widetilde P_{s+1}=x\widetilde P_s
 -\frac{\sin(sh)A_{s-1}}{4\sin((s+\frac12)h)A_s}\widetilde P_{s-1}
 =x\widetilde P_s-\frac{\sin^2(sh)}{4\sin\theta_s\sin\theta_{s+1}}\widetilde P_{s-1}.
\]
The last coefficient is $(J^\star_{s,s+1})^2$, and $J_N^\star$ has zero
diagonal, so $\widetilde P_s$ and $P_s$ satisfy the same recurrence
\eqref{eq:continuant} with the same initial values: $P_s=C_s/(2^sA_s)$
for $0\le s\le N-1$. All coefficients $A_jA_{s-j}$ are positive and
$|\cos((s-2j)\theta)|\le1$, so $|P_s(x)|\le P_s(1)$ for $x\in[-1,1]$.

{The gain.} By reversal, $|\widehat P_s(x)|=|P_{N-s-1}(x)|\le
P_{N-s-1}(1)=\widehat P_s(1)$ on $[-1,1]$, so
$\sum_s\widehat P_s(x)^2P_s(y)^2\le\sum_s\widehat P_s(1)^2P_s(1)^2$ for
$x,y\in[-1,1]$. The maximum in \eqref{eq:full-gain} is attained at
$(1,1)$, and the endpoint evaluation in the proof of \Cref{thm:gain}
gives \eqref{eq:sine-gain}.
\end{proof}

Let $\mathcal L$ be the nonnegative self-adjoint Legendre operator on
$L^2(-1,1)$, defined by the orthonormal basis
$p_j=\sqrt{j+\tfrac12}\,\mathsf P_j$ and $\mathcal Lp_j=j(j+1)p_j$, with
\[
 \mathcal D(\mathcal L)=\Bigl\{\sum_ja_jp_j:
 \sum_j j^2(j+1)^2|a_j|^2<\infty\Bigr\}.
\]
The unitary map $(Vf)(t)=\sqrt{\pi\sin\pi t}\,f(\cos\pi t)$ defines
$\mathscr H=\pi^2V(\mathcal L+\tfrac14)V^{-1}$. Its differential expression
is $-\partial_t^2-\pi^2/(4\sin^2\pi t)$, as the calculation below shows.
This specifies the no-logarithm endpoint realization. Requiring only
that functions vanish at both endpoints would not distinguish it from
the logarithmic branch, which also vanishes.

\begin{lemma}[The continuum sine kernel]\label{lem:legendre}
Let $\mathsf P_j$ be the Legendre polynomials, and let
\[
 \upsilon_j(t)=\sqrt{\pi(j+\tfrac12)\sin\pi t}\;\mathsf P_j(\cos\pi t),\qquad
 \mu_j=\pi^2(j+\tfrac12)^2=2/\Lambda_j .
\]
Then $(\upsilon_j)$ is an orthonormal
basis of $L^2(0,1)$, and each $\upsilon_j$ solves
$-\upsilon''-\frac{\pi^2}{4\sin^2\pi t}\upsilon=\mu_j\upsilon$ on $(0,1)$. Put
\[
 l(t)=\frac2\pi\sqrt{\sin\pi t}\,\mathbf K\bigl(\sin^2\tfrac{\pi t}2\bigr),
 \qquad \bar l(t)=l(1-t).
\]
Then $l,\bar l$ extend continuously to $[0,1]$ with value zero at both
ends, $l'\bar l-l\bar l'=2$ on $(0,1)$, and the operator with kernel
$G_*(t,s)=\frac12l(\min(t,s))\,\bar l(\max(t,s))$ is
$\sum_j\mu_j^{-1}\ip\cdot{\upsilon_j}\upsilon_j$. The kernel is continuous on
$[0,1]^2$, $G_*(t,t)=\sum_j\upsilon_j(t)^2/\mu_j$ for every $t$, and
$\int_0^1G_*(t,t)\dd t=\frac12$.
\end{lemma}

\begin{proof}\label{proof:lem:legendre}
{The basis.} Legendre orthogonality and density of polynomials in
$L^2(-1,1)$ show that $(p_j)$ is an orthonormal basis. The substitution
$z=\cos\pi t$ proves that $V$ is unitary; hence $(\upsilon_j)=(Vp_j)$
is an orthonormal basis of $L^2(0,1)$.

{The equation.} Let $s=\sin^2(\pi t/2)$, so $\cos\pi t=1-2s$ and
$\sin\pi t=2\sqrt{s(1-s)}$. For $F\in C^2(0,1)$, the chain rule gives
\begin{equation}\label{eq:legendre-transform}
 \Bigl(-\partial_t^2-\frac{\pi^2}{4\sin^2\pi t}\Bigr)\bigl\{\sqrt{\sin\pi t}\,F(s)\bigr\}
 =\pi^2\sqrt{\sin\pi t}\,\bigl\{\tfrac14F-(1-2s)F'-s(1-s)F''\bigr\}.
\end{equation}
For $F(s)=\mathsf P_j(1-2s)$, $\frac{\dd}{\dd x}=-\frac12\frac{\dd}{\dd s}$ and
$1-x^2=4s(1-s)$ turn the Legendre equation into
$s(1-s)F''+(1-2s)F'+j(j+1)F=0$, so the brace in
\eqref{eq:legendre-transform} equals $(\frac14+j(j+1))F=(j+\frac12)^2F$, and
$\upsilon_j$ solves the stated equation.

{Zero-energy solutions.} For $0\le s<1$, the binomial series
$(1-u)^{-1/2}=\sum_nc_nu^n$, $c_n=\binom{2n}n4^{-n}$, and
$\int_0^{\pi/2}\sin^{2n}\theta\dd\theta=\frac\pi2c_n$ give
$\mathbf K(s)=\frac\pi2\sum_nc_n^2s^n$. Since $c_{n+1}/c_n=(n+\frac12)/(n+1)$,
the coefficients $a_n=\frac\pi2c_n^2$ satisfy $(n+1)^2a_{n+1}=(n+\frac12)^2a_n$,
which is the coefficient form of
\[
 s(1-s)\mathbf K''+(1-2s)\mathbf K'-\tfrac14\mathbf K=0 .
\]
This equation is invariant under $s\mapsto1-s$, so $\mathbf K(1-s)$ also
solves it. By \eqref{eq:legendre-transform}, $l$ and $\bar l$, which equal
$\frac2\pi\sqrt{\sin\pi t}$ times $\mathbf K(s)$ and $\mathbf K(1-s)$, solve
$-y''-\frac{\pi^2}{4\sin^2\pi t}y=0$ on $(0,1)$, and both are positive.

{Endpoint behavior.} Our elliptic integral uses the parameter:
$\mathbf K(s)=K_{\mathrm{DLMF}}(\sqrt s)$. The convergent endpoint
expansion \cite[Eq.~19.12.1]{DLMF2026}, with complementary modulus
$\sqrt s$, gives
\[
 \mathbf K(s)=\frac\pi2+O(s),\qquad
 \mathbf K(1-s)=\log4-\frac12\log s+O(s\log(1/s))
 \quad(s\downarrow0).
\]
As $t\downarrow0$, $\sin\pi t=\pi t(1+O(t^2))$ and $s=(\pi t/2)^2(1+O(t^2))$,
so $\frac12\log\frac1s=\log\frac1t+O(1)$ and
\[
 l(t)=\sqrt{\pi t}\,(1+O(t^2)),\qquad
 \bar l(t)=\frac2{\sqrt\pi}\sqrt t\log\frac1t+O(\sqrt t).
\]
In particular $l(t)\to0$ and $\bar l(t)\to0$ as $t\downarrow0$, and by
$\bar l(t)=l(1-t)$ also as $t\uparrow1$.

{Wronskian.} $W=l'\bar l-l\bar l'$ satisfies $W'=l''\bar l-l\bar l''=0$,
since both functions solve the same equation, so $W$ is constant. Since
$(\bar l/l)'=-W/l^2$,
\[
 \frac{\bar l}l(t)=\frac{\bar l}l(\tfrac12)+W\int_t^{1/2}\frac{\dd u}{l(u)^2}
 =\frac W\pi\log\frac1t+O(1)\qquad(t\downarrow0),
\]
by $l(u)^2=\pi u(1+O(u^2))$. On the other hand
$\bar l/l=\mathbf K(1-s)/\mathbf K(s)=\frac2\pi\log\frac1t+O(1)$. Hence $W=2$.

{The inverse.} Let $\mathcal G$ be the operator with kernel $G_*$,
which is continuous on $[0,1]^2$ as a product of continuous functions of
$\min(t,s)$ and $\max(t,s)$. For bounded continuous $g$ on $(0,1)$, the
function
$u=\mathcal Gg=\frac12\bar l(t)\int_0^tlg+\frac12l(t)\int_t^1\bar lg$ is $C^2$ on
$(0,1)$, with
\[
 u'=\tfrac12\bar l'\!\int_0^tlg+\tfrac12l'\!\int_t^1\bar lg,\qquad
 u''=-\frac{\pi^2}{4\sin^2\pi t}u-\tfrac12Wg=-\frac{\pi^2}{4\sin^2\pi t}u-g ,
\]
since the terms $\pm\frac12l\bar lg$ cancel in $u'$. Apply this to $g=\upsilon_j$ and set
$f=\mathcal G\upsilon_j-\upsilon_j/\mu_j$. Then $f=al+b\bar l$, since
it solves the homogeneous equation and $W=2$. At $0$, both
$\mathcal G\upsilon_j$ and $\upsilon_j/\mu_j$ are $O(\sqrt t)$:
the first integral term is $O(t^{5/2}\log(1/t))$ and the second is
$O(\sqrt t)$. Thus $b=0$, because $\bar l(t)/\sqrt t\to\infty$.
Reflection gives $f(t)=O(\sqrt{1-t})$ at $1$, and hence $a=0$.
Consequently $\mathcal G\upsilon_j=\mu_j^{-1}\upsilon_j$ for every $j$.
The bounded kernel and completeness of the basis give
$\mathcal G=\mathscr H^{-1}=\sum_j\mu_j^{-1}\ip\cdot{\upsilon_j}\upsilon_j$.

{The diagonal.} By Mercer's theorem \cite[\S98]{RieszNagy} for the
continuous symmetric kernel $G_*$ of the positive operator $\mathcal G$,
$G_*(t,s)=\sum_j\mu_j^{-1}\upsilon_j(t)\upsilon_j(s)$ with absolute and uniform
convergence on $[0,1]^2$. On the diagonal this gives
$G_*(t,t)=\sum_j\upsilon_j(t)^2/\mu_j$, and integration gives
$\int_0^1G_*(t,t)\dd t=\sum_j\mu_j^{-1}=\frac12\sum_j\Lambda_j=\frac12$ by
\Cref{thm:transfer}.
\end{proof}

\begin{proof}[Proof of \Cref{prop:target}]\label{proof:prop:target}
{The map $X_*$.} With $s=\sin^2(\pi t/2)$,
$l\bar l=\frac4{\pi^2}\sin(\pi t)\mathbf K(s)\mathbf K(1-s)=\phi_*$, and
$(\log(l/\bar l))'=W/(l\bar l)=2/\phi_*$. Since $l(\frac12)=\bar l(\frac12)$,
\[
 X_*(t)=\tfrac12\log\frac{l(t)}{\bar l(t)}=-\tfrac12\log\varpi(t),\qquad
 \varpi(t)=\frac{\mathbf K(\cos^2(\pi t/2))}{\mathbf K(\sin^2(\pi t/2))}.
\]
Thus $X_*$ is continuous and strictly increasing, $X_*(1-t)=-X_*(t)$, and
$\varpi(t)\to\infty$ as $t\downarrow0$ because $\mathbf K(1-s)\to\infty$ and
$\mathbf K(s)\to\frac\pi2$. Hence $X_*(t)\to-\infty$ as $t\downarrow0$ and
$X_*(t)\to+\infty$ as $t\uparrow1$, and $X_*$ is a bijection onto $\R$.

{Normalizations and spectrum.} The substitution $x=X_*(t)$,
$\dd x=\dd t/\phi_*(t)$, and $m_*(X_*(t))=\phi_*(t)^2$ give
\[
 \int m_*=\int_0^1\phi_*,\qquad \int\sqrt{m_*}=\int_0^1\dd t=1,\qquad
 E(m_*)=\int_0^1\phi_*^2=\cs ,
\]
and $\int_0^1\phi_*=\int_0^1l\bar l=2\int_0^1G_*(t,t)\dd t=1$ by
\Cref{lem:legendre}. For $s<t$,
\[
 \sqrt{\phi_*(s)\phi_*(t)}\,e^{-(X_*(t)-X_*(s))}
 =\sqrt{l(s)\bar l(s)l(t)\bar l(t)}\sqrt{\frac{l(s)\bar l(t)}{\bar l(s)l(t)}}
 =l(s)\bar l(t)=2G_*(s,t).
\]
The map
\[
 U:L^2(\R)\to L^2(0,1),\qquad (Uf)(t)=f(X_*(t))/\sqrt{\phi_*(t)},
\]
is unitary, since $\int_0^1|f(X_*(t))|^2\dd t/\phi_*(t)=\int|f|^2$ and
$g\mapsto(\sqrt{\phi_*}\,g)\circ X_*^{-1}$ is its inverse. For
$g\in L^2(0,1)$ and $f=U^{-1}g$, the substitution $y=X_*(s)$ gives
\[
 (UK_{m_*}f)(t)=\frac1{\sqrt{\phi_*(t)}}\int_0^1\phi_*(t)e^{-|X_*(t)-X_*(s)|}
 \phi_*(s)\sqrt{\phi_*(s)}g(s)\frac{\dd s}{\phi_*(s)}
 =\int_0^12G_*(t,s)g(s)\dd s ,
\]
so $UK_{m_*}U^{-1}=2\mathcal G$, and by \Cref{lem:legendre}
$\lambda_j(K_{m_*})=2/\mu_j=\Lambda_j$.

{Theta form.} Fix $t\le\frac12$, so that $\varpi=\varpi(t)\ge1$, and put
$z=\sin^2(\pi t/2)$ and $a=e^{-\pi\varpi}$. For $0<q<1$ let
$\vartheta_3(q)=\sum_{j\in\mathbb Z}q^{j^2}$,
$\vartheta_4(q)=\sum_{j\in\mathbb Z}(-1)^jq^{j^2}$,
$\vartheta_2(q)=\sum_{j\in\mathbb Z}q^{(j+1/2)^2}$. Jacobi's
theta--elliptic identities \cite[\S20.9(i)]{DLMF2026}, in our parameter
convention, read
\[
 z=\frac{\vartheta_2(a)^4}{\vartheta_3(a)^4},\qquad
 1-z=\frac{\vartheta_4(a)^4}{\vartheta_3(a)^4},\qquad
 \mathbf K(z)=\frac\pi2\vartheta_3(a)^2,
 \quad a=\exp\!\left(-\pi\frac{\mathbf K(1-z)}{\mathbf K(z)}\right).
\]
These are also given in \cite[\S4.6]{FGS}; the equivalent modular
transformation follows from Gaussian Poisson summation
\cite[Chapter~5]{SteinShakarchi}.
Using $\sin\pi t=2\sqrt{z(1-z)}$ and
$\mathbf K(1-z)=\varpi\mathbf K(z)$,
\[
 \phi_*(t)=\frac8{\pi^2}\varpi\sqrt{z(1-z)}\,\mathbf K(z)^2
 =\frac8{\pi^2}\varpi\,\frac{\vartheta_2(a)^2\vartheta_4(a)^2}{\vartheta_3(a)^4}
 \cdot\frac{\pi^2}4\vartheta_3(a)^4=2\varpi\,\vartheta_2(a)^2\vartheta_4(a)^2 .
\]
Since $(j+\frac12)^2=\frac14+j(j+1)$ and the terms with $j$ and $-1-j$ agree,
$\vartheta_2(a)=2a^{1/4}\Sigma_1(a)$; and $\vartheta_4(a)=\Sigma_2(a)$. With
$x=X_*(t)=-\frac12\log\varpi\le0$, that is $\varpi=e^{2|x|}$, and
$n_*(x)=\phi_*(t)$, this is \eqref{eq:sine-theta} for $x\le0$; since
$\phi_*(1-t)=\phi_*(t)$ and $X_*(1-t)=-X_*(t)$, $n_*$ is even, which gives
$x\ge0$.

{Asymptotics.} For $\varpi\ge1$, $j(j+1)\ge2j$ and
$j^2\ge3j-2$ give
\[
 0\le\Sigma_1(a)-1\le\frac{a^2}{1-a^2}\le\frac{e^{-2\pi\varpi}}{1-e^{-2\pi}},\qquad
 |\Sigma_2(a)-1|\le\frac{2a}{1-a^3}\le\frac{2e^{-\pi\varpi}}{1-e^{-3\pi}} .
\]
Hence, for $\varpi\ge1$, $1\le\Sigma_1(a)\le M_1$ and
$0<1-2e^{-\pi}/(1-e^{-3\pi})\le\Sigma_2(a)\le M_2$, with $M_1,M_2$ of
\Cref{lem:sine-sum}; so $\log\Sigma_1$ and $\log\Sigma_2$ are bounded, and
$\log n_*(x)=-\frac\pi2e^{2|x|}+2|x|+O(1)$.
\end{proof}

\begin{lemma}[Uniform bound for the sine diagonal]\label{lem:uniform-diagonal}
For every $N\ge2$ and $i$, $0<[(I-J_N^\star)^{-1}]_{ii}\le64N$.
\end{lemma}

\begin{proof}\label{proof:lem:uniform-diagonal}
Put $c=\cos\frac\pi{2N}$. Since $\sin y\ge\frac{2y}\pi$ on
$[0,\frac\pi2]$, $1-c=2\sin^2\frac\pi{4N}\ge\frac1{2N^2}$. For $N=2,3$, the
diagonal entries of the positive definite matrix
$(I-J_N^\star)^{-1}$ are at most its norm $(1-c)^{-1}\le2N^2\le6N$.

Let $N\ge4$. The vector $(\sqrt{p_i})_i$ satisfies
$(J_N^\star\sqrt p)_i=\frac{\sin(ih)+\sin((i-1)h)}{2\sqrt{p_i}}=c\sqrt{p_i}$, by
$\sin(ih)+\sin((i-1)h)=2\sin\theta_i\cos\frac h2$. Write $x\in\R^N$ as
$x_i=\sqrt{p_i}f_i$. Since
$2J^\star_{i,i+1}\sqrt{p_ip_{i+1}}=\sin(ih)$ and
$-f_if_{i+1}=\frac12(f_{i+1}-f_i)^2-\frac12f_i^2-\frac12f_{i+1}^2$,
\[
 \mathfrak q(f):=x^T(I-J_N^\star)x
 =\sum_ip_if_i^2+\frac12\sum_{i<N}\sin(ih)(f_{i+1}-f_i)^2
 -\frac12\sum_{i<N}\sin(ih)(f_i^2+f_{i+1}^2),
\]
and the coefficient of $f_i^2$ in the last sum is
$\frac12(\sin(ih)+\sin((i-1)h))=cp_i$. Hence
\[
 \mathfrak q(f)=(1-c)\sum_ip_if_i^2+\frac12\sum_{i<N}\sin(ih)(f_{i+1}-f_i)^2 .
\]
Let $S$ be the set of indices with $\theta_i\in[\frac\pi4,\frac{3\pi}4]$,
that is, $i\in[\frac N4+\frac12,\frac{3N}4+\frac12]$. This interval has
length $N/2$ and contains at least $\lfloor N/2\rfloor\ge N/4$ integers,
and $p_i\ge\frac1{\sqrt2}\ge\frac12$ on $S$. Hence
$\mathfrak q(f)\ge\frac1{2N^2}\cdot\frac12\cdot\frac N4\min_Sf_i^2$, and some
$j\in S$ has $f_j^2\le16N\mathfrak q(f)$.

Fix $i$ and let $E$ be the set of edges $k$, joining $k$ and $k+1$, between
$i$ and $j$. The Cauchy--Schwarz inequality gives
$(f_i-f_j)^2\le\sum_{k\in E}\frac1{\sin(kh)}\sum_{k\in E}\sin(kh)(f_{k+1}-f_k)^2
\le2\mathfrak q(f)\sum_{k\in E}\frac1{\sin(kh)}$, hence
\[
 f_i^2\le2f_j^2+2(f_i-f_j)^2\le2f_j^2+4\mathfrak q(f)\sum_{k\in E}\frac1{\sin(kh)} .
\]
Since $\sin(kh)=\sin(\min(k,N-k)h)\ge2\min(k,N-k)/N$, each term is at
most $N/(2\min(k,N-k))$. Put $m'=\lfloor N/4\rfloor+1$, so that $4m'\ge N+1$.
The edges with $\min(k,N-k)>N/4$ are those with $m'\le k\le N-m'$; there are
at most $N-2m'+1\le2m'$ of them, and each contributes at most $N/(2m')$, so
together they contribute at most $N$. Let $k\in E$ with $\min(k,N-k)\le N/4$. If
$k\le N/4$, then $k<j$, so $i\le k$; if $k\ge3N/4$, then $k\ge j$, since
$j\le\frac{3N}4+\frac12$, so $k\le i-1$. Only one of the two cases occurs,
and in either case $n=\min(k,N-k)$ runs through distinct integers in
$[d,N/4]$, where $d=\min(i,N+1-i)$. These edges contribute at most
$\frac N2\sum_{n=d}^{\lfloor N/4\rfloor}\frac1n\le\frac N2(1+\log\frac Nd)$.
Thus $\sum_{k\in E}1/\sin(kh)\le\frac N2(1+\log\frac Nd)+N$. Moreover
$p_i=\sin\theta_i=\sin(\pi-\theta_i)\le\pi d/N$. With $u=d/N\in(0,1]$,
\[
 p_if_i^2\le32N\mathfrak q(f)+2\pi N\mathfrak q(f)\,u\Bigl(3+\log\frac1u\Bigr)
 \le(32+6\pi)N\mathfrak q(f)<64N\mathfrak q(f),
\]
where $u(3+\log\frac1u)$ is increasing on $(0,1]$, with derivative
$2+\log\frac1u>0$, and equals $3$ at $u=1$. For a positive definite matrix
$M$, the Cauchy--Schwarz inequality in the inner product $\ip{x}{y}_M=x^TMy$
gives $x_i^2=\ip{M^{-1}e_i}x_M^2\le(M^{-1})_{ii}\,x^TMx$ with equality at
$x=M^{-1}e_i$, so $(M^{-1})_{ii}=\sup_{x\ne0}x_i^2/x^TMx$. With
$M=I-J_N^\star$ and $x_i^2=p_if_i^2$ this proves the claim.
\end{proof}

\begin{proof}[Proof of the limit in \Cref{thm:sine-finite}]
\label{proof:thm:reference-limit}
By the finite part, $J_N^\star$ has the orthonormal eigenvectors
$U_{\cdot j}$, $U_{ij}=\sqrt{w_i}\phi_{N,j}(x_i)$, with eigenvalues
$\cos((j+\frac12)h)$, and
$(1-\cos((j+\frac12)h))^{-1}=N^2r_{N,j}$ by \Cref{lem:trace}. Hence
$d_i=\sum_{j<N}N^2r_{N,j}U_{ij}^2$ and
\[
 D_N(t)=\sum_{j<N}r_{N,j}\,NU_{ij}^2\qquad(t\in[\tfrac{i-1}N,\tfrac iN)).
\]

{Convergence of the first $K$ terms.} The intervals
$[\cos(ih),\cos((i-1)h)]$, $1\le i\le N$, partition $[-1,1]$, have lengths
$2w_i$ and contain $x_i$; hence $\sum_iw_iF(x_i)$ is a Riemann sum of
$\frac12\int_{-1}^1F$ with mesh at most $h$, and converges to it for every
$F\in C[-1,1]$. For fixed $j$ and $N>j$, the moment matrix of $1,x,\dots,x^j$ is
positive definite: a nonzero polynomial of degree at most $j$ cannot
vanish at all $N$ distinct nodes. These matrices converge to the positive
definite moment matrix of $\tfrac12\dd x$ on $[-1,1]$. Cholesky
factorization, and hence Gram--Schmidt with positive leading coefficients,
depends continuously on a positive definite moment matrix. Therefore
$\phi_{N,j}\to\sqrt{2j+1}\,\mathsf P_j$ coefficientwise and uniformly on
$[-1,1]$.
For
$t\in[\frac{i-1}N,\frac iN)$, $|\theta_i-\pi t|\le h$, so
$|x_i-\cos\pi t|\le h$, $|p_i-\sin\pi t|\le h$, and
$Nw_i=p_i\cdot N\sin\frac\pi{2N}$ with $N\sin\frac\pi{2N}\to\frac\pi2$. By the
uniform continuity of $\mathsf P_j$ on $[-1,1]$, for fixed $K$
\[
 D_{N,K}(t):=\sum_{j<K}r_{N,j}NU_{ij}^2\ \longrightarrow\
 D_K(t):=\sum_{j<K}\Lambda_j\upsilon_j(t)^2\quad\text{uniformly on }[0,1),
\]
since $NU_{ij}^2=Nw_i\phi_{N,j}(x_i)^2\to\frac\pi2\sin(\pi t)(2j+1)\mathsf P_j(\cos\pi t)^2=\upsilon_j(t)^2$
uniformly and $r_{N,j}\to\Lambda_j$.

{The tails.} By \Cref{lem:uniform-diagonal}, $0\le D_N\le64$, and
$0\le D_N-D_{N,K}\le D_N$. Since the columns of $U$ are unit vectors,
\[
 \|D_N-D_{N,K}\|_{L^1(0,1)}=\sum_i\frac1N\sum_{j=K}^{N-1}r_{N,j}NU_{ij}^2
 =\sum_{j=K}^{N-1}r_{N,j}\le\sum_{j\ge K}\frac2{(2j+1)^2}\le\frac1{2K-1},
\]
by the bound of \Cref{lem:trace} used in the proof of
\Cref{thm:transfer} and telescoping; hence
$\|D_N-D_{N,K}\|_2^2\le64/(2K-1)$. By \Cref{lem:legendre},
$\sum_j\Lambda_j\upsilon_j^2=2\sum_j\upsilon_j^2/\mu_j=2G_*(t,t)=l\bar l=\phi_*$, so
$0\le\phi_*-D_K\le\phi_*$ and
$\|\phi_*-D_K\|_2^2\le\|\phi_*\|_\infty\int_0^1(\phi_*-D_K)
=\|\phi_*\|_\infty\sum_{j\ge K}\Lambda_j$, where $\phi_*$ is bounded as a
continuous function on $[0,1]$. Therefore
\[
 \limsup_{N\to\infty}\|D_N-\phi_*\|_2
 \le\frac8{\sqrt{2K-1}}+\Bigl(\|\phi_*\|_\infty\sum_{j\ge K}\Lambda_j\Bigr)^{1/2}
\]
for every $K$, and the right-hand side tends to zero as $K\to\infty$.

{Conclusion.} The definition of $D_N$ gives
\[
 N^{-3}\E_N(J_N^\star)=\sum_iN^{-1}(d_i/N)^2=\|D_N\|_2^2\to\|\phi_*\|_2^2=\cs .
\]
By \eqref{eq:sine-gain},
$\Delta A_N(\mathcal R_N^\star)/(N^{3/2}\epsN)=2N^{-3/2}\E_N(J_N^\star)^{1/2}
=2\|D_N\|_2\to2\sqrt{\cs}$, and since $\mathcal R_N^\star\in\mathfrak E_N(\mu,L)$,
$\mathscr G(\mu,L)\le2\sqrt{\cs}$.
\end{proof}

\begin{lemma}[A finite sum for the sine energy]\label{lem:sine-sum}
For an integer $K\ge0$ let
\[
 \Sigma_{1,K}(a)=\sum_{j=0}^Ka^{j(j+1)},\qquad
 \Sigma_{2,K}(a)=1+2\sum_{j=1}^K(-1)^ja^{j^2},\qquad
 \Sigma_{1,K}^6\Sigma_{2,K}^6=\sum_{n=0}^{12K^2+6K}C_na^n,
\]
\[
 I(\lambda)=\int_1^\infty\varpi^2e^{-\lambda\pi\varpi}\dd\varpi
 =e^{-\lambda\pi}\Bigl(\frac1{\lambda\pi}+\frac2{(\lambda\pi)^2}
 +\frac2{(\lambda\pi)^3}\Bigr),\qquad
 M_1=\frac1{1-e^{-2\pi}},\qquad M_2=1+\frac{2e^{-\pi}}{1-e^{-3\pi}},
\]
\[
 E_K=512\sum_nC_nI\bigl(n+\tfrac32\bigr),
\]
\[
 R_K=512\Bigl[\frac{6M_1^5M_2^6\,I(\frac32+(K+1)(K+2))}{1-e^{-2(K+2)\pi}}
 +\frac{12M_1^6M_2^5\,I(\frac32+(K+1)^2)}{1-e^{-(2K+3)\pi}}\Bigr].
\]
Then $|\cs-E_K|\le R_K$, and $R_K\to0$.
\end{lemma}

\begin{proof}\label{proof:lem:sine-sum}
The formula for $I(\lambda)$ follows from the antiderivative
$-e^{-cv}(v^2/c+2v/c^2+2/c^3)$ of $v^2e^{-cv}$, $c=\lambda\pi$. By
\Cref{prop:target}, $\cs=\int n_*^3$, $n_*$ is even, and for $x\le0$ the
variable $\varpi=e^{-2x}$ ranges over $[1,\infty)$ with
$\dd x=-\dd\varpi/(2\varpi)$; by \eqref{eq:sine-theta},
\[
 \cs=2\int_{-\infty}^0n_*^3\dd x
 =512\int_1^\infty\varpi^2e^{-3\pi\varpi/2}\,\Sigma_1(e^{-\pi\varpi})^6
 \Sigma_2(e^{-\pi\varpi})^6\dd\varpi .
\]
The polynomial $\Sigma_{1,K}^6\Sigma_{2,K}^6$ has degree
$6K(K+1)+6K^2$, and integrating it against $512\varpi^2e^{-3\pi\varpi/2}$ with
$a=e^{-\pi\varpi}$ gives $E_K$, since $a^ne^{-3\pi\varpi/2}=e^{-(n+3/2)\pi\varpi}$.

Let $\varpi\ge1$ and $a=e^{-\pi\varpi}\le e^{-\pi}$. Since $j(j+1)\ge2j$,
$0\le\Sigma_{1,K}\le\Sigma_1\le\sum_ja^{2j}\le M_1$; since $j^2\ge3j-2$ for
$j\ge1$, $|\Sigma_2|,|\Sigma_{2,K}|\le1+2\sum_{j\ge1}a^{3j-2}\le M_2$. The
omitted exponents of $\Sigma_1$ are $j(j+1)$, $j\ge K+1$, with consecutive
differences $2(j+1)\ge2(K+2)$; those of $\Sigma_2$ are $j^2$, $j\ge K+1$,
with differences $2j+1\ge2K+3$. Comparison with geometric series gives
\[
 |\Sigma_1-\Sigma_{1,K}|\le\frac{a^{(K+1)(K+2)}}{1-e^{-2(K+2)\pi}},\qquad
 |\Sigma_2-\Sigma_{2,K}|\le\frac{2a^{(K+1)^2}}{1-e^{-(2K+3)\pi}} .
\]
For $|X|,|x|\le M_1$ and $|Y|,|y|\le M_2$,
$X^6Y^6-x^6y^6=(X^6-x^6)Y^6+x^6(Y^6-y^6)$ and
$|X^6-x^6|=|X-x|\,|\sum_{i<6}X^ix^{5-i}|\le6M_1^5|X-x|$ give
\[
 |\Sigma_1^6\Sigma_2^6-\Sigma_{1,K}^6\Sigma_{2,K}^6|
 \le6M_1^5M_2^6|\Sigma_1-\Sigma_{1,K}|+6M_1^6M_2^5|\Sigma_2-\Sigma_{2,K}| .
\]
Integrating this bound against $512\varpi^2e^{-3\pi\varpi/2}$ gives
$|\cs-E_K|\le R_K$. Finally, $I(\lambda)\le e^{-\lambda\pi}$ for $\lambda\ge1$,
since $\frac1\pi+\frac2{\pi^2}+\frac2{\pi^3}<1$, so $R_K\to0$.
\end{proof}
\section{The ground coordinate, the filters and the state region}\label{app:states}

The ground-state square identity supplies the coordinate $q$ and its
integral constraints. The Fredholm--Jost identity then expresses each
determinant as the integral of a Riccati filter. To describe their joint
range, we use the spectral measure of a Dirichlet-truncated
Birman--Schwinger operator. For the three retained filters, the resulting
moment problem has explicit two-point extremizers.

\subsection{Ground state, filters and the common measure}\label{sec:ground}

\begin{proof}[Proof of \Cref{thm:ground}]\label{proof:thm:ground}
{Ground state and square identity.}
The positive compact operator $K_m$ has a positive largest eigenvalue
$\lambda_0$. Its nonnegative kernel allows a nonnegative normalized
maximizer $\varphi$ of the Rayleigh quotient, with
$K_m\varphi=\lambda_0\varphi$. Put $j=\sqrt m\,\varphi\in L^1$ and
\[
 u_-(x)=e^{-x}\int_{-\infty}^xe^yj(y)\dd y,\qquad
 u_+(x)=e^x\int_x^\infty e^{-y}j(y)\dd y,\qquad u=u_-+u_+.
\]
Then $u_-'=-u_-+j$, $u_+'=u_+-j$, and $j$ is nonnegative and nonzero.
Thus $u>0$, $u'=u_+-u_-$ is continuous and locally absolutely continuous,
$|u'|\le u$, and $u''=u-2j$ almost everywhere. The functions $u_\pm$
are in $L^1\cap L^\infty$, so $u\in H^1$. Moreover,
$\sqrt m\,u=\lambda_0\varphi$, whence $2j=\kappa mu$.

Set $a=u'/u$. Then $|a|\le1$ and $a'=1-\kappa m-a^2$. For $f\in H^1$,
\[
 (f'-af)^2=f'^2+f^2-\kappa mf^2-(af^2)' .
\]
All nondifferentiated terms are integrable, and $af^2\to0$ at both
infinities by \Cref{lem:features}(i). Integration proves
\eqref{eq:ground-square}. Its right-hand side vanishes precisely when
$(f/u)'=0$, or $f=cu$.

{Simplicity and uniqueness.}
Any $H^1\cap C^1$ solution $v$ of $-v''+v=\kappa mv$ with locally
absolutely continuous derivative satisfies
$\|v\|_\HH^2=\kappa\int mv^2$: integrate $(vv')'=v'^2+v^2-\kappa mv^2$
between endpoints tending to $\pm\infty$ along which $vv'\to0$.
Such endpoints exist because $vv'\in L^1$; the integrands are in $L^1$
because $v$ is bounded. The square identity therefore gives $v=cu$.
For an eigenvector $\psi$ of $K_m$ at $\lambda_0$, the convolution
$v=e^{-|\cdot|}*(\sqrt m\,\psi)$ has this regularity and satisfies
$\sqrt m\,v=\lambda_0\psi$. Thus $\psi$ is a multiple of $\varphi$,
proving simplicity as well as uniqueness of the positive solution.
The square identity proves the variational characterization of $\kappa$;
testing with $u$ excludes every larger coupling.

{Ground coordinate.}
The function $q=1-a$ lies in $[0,2]$ and satisfies
$q'=\kappa m-Q(q)$. For $R>0$,
\[
 \int_{-R}^RQ(q)=\kappa\int_{-R}^Rm-q(R)+q(-R)\le\kappa\int m+2.
\]
Hence $Q(q),q'\in L^1$, $q$ has limits at both infinities, and these
limits belong to $\{0,2\}$. A limit $q(+\infty)=0$ would give
$u'/u\ge1/2$ eventually and exponential growth of $u$; a limit
$q(-\infty)=2$ gives the same contradiction at the other end.
Thus the limits are $0$ and $2$, and integration gives
$\int Q(q)=\kappa\int m-2$. Outside a compact support, $u''=u$ and
$u\in H^1$ select $u=c_-e^x$ on the left and $u=c_+e^{-x}$ on the
right, giving the stated constant values of $q$. If $m$ is bounded,
$|q'|\le\kappa\|m\|_\infty+1$, so $q$ is Lipschitz.
\end{proof}

\begin{proof}[Proof of \Cref{lem:filter}]\label{proof:lem:filter}
Put $M=\int m$ and $M(x)=\int_{-\infty}^xm$.

{The Volterra equation.} Let $V_0=1$ and
$V_{j+1}(x)=\frac s2\int_{-\infty}^x(1-e^{-2(x-y)})m(y)V_j(y)\dd y$. By
induction on $j$, using $0\le1-e^{-2(x-y)}\le1$ and
$\frac{\dd}{\dd y}M(y)^{j+1}=(j+1)M(y)^jm(y)$,
\[
 0\le V_j(x)\le\frac{(sM(x)/2)^j}{j!}\le\frac{(sM/2)^j}{j!} ,
\]
and, also by induction,
\begin{equation}\label{eq:volterra-terms}
 V_j(x)=\Bigl(\frac s2\Bigr)^j\int_{y_1<\dots<y_j<x}
 \prod_{i=1}^j\bigl(1-e^{-2(y_{i+1}-y_i)}\bigr)m(y_i)\dd y_i
 \qquad(y_{j+1}=x).
\end{equation}
The series $V_s=\sum_jV_j$ converges uniformly, $1\le V_s\le e^{sM/2}$,
$V_s(x)\to1$ as $x\to-\infty$ because $M(x)\to0$, and summation of the
recursive definitions gives
$V_s(x)=1+\frac s2\int_{-\infty}^x(1-e^{-2(x-y)})m(y)V_s(y)\dd y$. If $V$ and
$\widetilde V$ are bounded solutions, $D=V-\widetilde V$ satisfies the
homogeneous equation, and induction gives
$|D(x)|\le\|D\|_\infty(sM/2)^j/j!$ for every $j$; hence $D=0$.

{The filter.} Writing the equation as
$V_s(x)=1+\frac s2\int_{-\infty}^xmV_s-\frac s2e^{-2x}\int_{-\infty}^xe^{2y}mV_s$,
both terms are absolutely continuous, and
\[
 V_s'(x)=s\int_{-\infty}^xe^{-2(x-y)}m(y)V_s(y)\dd y\ge0,\qquad
 V_s''+2V_s'=smV_s\ \text{a.e.} ;
\]
the right-hand side of the first identity is continuous, so $V_s\in C^1$.
The function $r_s=V_s'/V_s$ is locally absolutely continuous and
$r_s'=V_s''/V_s-r_s^2=sm-2r_s-r_s^2$; since $(\log V_s)'=r_s$ and
$V_s(-\infty)=1$, $V_s=\exp\int_{-\infty}^xr_s$. Since $V_s$ is nondecreasing,
$V_s(y)/V_s(x)\le1$ for $y\le x$, so
\begin{equation}\label{eq:filter-linear-bound}
 0\le r_s(x)\le s\int_{-\infty}^xe^{-2(x-y)}m(y)\dd y=\frac s\kappa r_\circ(x).
\end{equation}
The linear filter satisfies $r_\circ(x)\le\kappa M(x)\to0$ as $x\to-\infty$
and, for every $R$,
$r_\circ(x)\le\kappa e^{-2(x-R)}M+\kappa\int_R^\infty m$, so
$r_\circ(x)\to0$ as $x\to+\infty$. Hence $r_s(\pm\infty)=0$, and $r_s$ is
bounded.

{Uniqueness.} Let $\rho$ be a bounded locally absolutely continuous
solution of $\rho'=sm-2\rho-\rho^2$ with $\rho(x)\to0$ as $x\to-\infty$. The
difference $\delta=\rho-r_s$ satisfies $\delta'=-(2+\rho+r_s)\delta$, and there
is $x_0$ with $2+\rho+r_s\ge1$ on $(-\infty,x_0]$. For $a<x\le x_0$,
$|\delta(x)|=|\delta(a)|\exp(-\int_a^x(2+\rho+r_s))\le\|\delta\|_\infty e^{-(x-a)}$;
letting $a\to-\infty$ gives $\delta=0$ on $(-\infty,x_0]$, and then
$\delta(x)=\delta(x_0)\exp(-\int_{x_0}^x(2+\rho+r_s))=0$ for all $x$.

{The determinant.}
For $s>0$, dominated convergence in the Volterra equation gives
\[
 V_s(+\infty)=1+\frac s2\int_\R m(y)V_s(y)\dd y.
\]
The function $f_-(x)=e^xV_s(x)$ is the left Jost solution for
$-\partial_x^2+sm$ at spectral parameter $z=-1$. Apply the
Fredholm--Jost identity \cite[Theorem~4.6]{GesztesyMakarov2004} with
$sm\in L^1(\R)$ and $\sqrt z=i$. In its notation the Birman--Schwinger
operator is $-sK_m/2$, since $(-\partial_x^2+1)^{-1}$ has kernel
$e^{-|x-y|}/2$, and the inverse transmission coefficient is the
Volterra limit above. Therefore
\[
 \det\Bigl(I+\frac s2K_m\Bigr)=V_s(+\infty),\qquad
 H_m(s)=\int_\R r_s.
\]
The case $s=0$ is immediate. The Riccati equation then gives
$\|r_s'\|_1\le sM+(2+\|r_s\|_\infty)H_m(s)$.

{The linear filter.} The function
$r_\circ(x)=\kappa e^{-2x}\int_{-\infty}^xe^{2y}m(y)\dd y$ is locally absolutely
continuous by the product rule, with $r_\circ'=\kappa m-2r_\circ$; it is
nonnegative, and $r_\circ(\pm\infty)=0$ as shown above. By Tonelli's theorem $\int r_\circ=\frac\kappa2M$, so
$\|r_\circ'\|_1\le2\kappa M$. The function $h=q-r_\circ$ satisfies
$h'=\kappa m-Q(q)-\kappa m+2r_\circ=q^2-2h$, so as before
$h(x)=\int_{-\infty}^xe^{-2(x-y)}q(y)^2\dd y\ge0$, that is, $r_\circ\le q$; and
$0\le h\le q$, $h'=q'-r_\circ'\in L^1$. Together with
\eqref{eq:filter-linear-bound}, this gives $r_s\le(s/\kappa)q$. Finally, for $c>0$, $0\le\tilde m\in L^1$ and
$\varrho(x)=c\int_{-\infty}^xe^{-2(x-y)}\tilde m(y)\dd y$, Tonelli's theorem and
$\int_{\max(y,z)}^\infty e^{-2(2x-y-z)}\dd x=\frac14e^{-2|y-z|}$ give
\[
 \int_\R\varrho^2
 =c^2\iint \tilde m(y)\tilde m(z)\int_{\max(y,z)}^\infty e^{-2(2x-y-z)}\dd x\dd y\dd z
 =\frac{c^2}4\iint e^{-2|y-z|}\tilde m(y)\tilde m(z)\dd y\dd z .
\]
For $c=\kappa$ and $\tilde m=m$, $\varrho=r_\circ$, and the last double integral is
$\|K_m\|_{\HS}^2=\tr K_m^2$.
\end{proof}

\begin{proof}[Proof of \Cref{thm:common-measure}]\label{proof:thm:common-measure}
Fix $x$ and put $I_x=(-\infty,x)$ and
\[
 G_x(t,y)=\tfrac12\bigl(e^{-|t-y|}-e^{t+y-2x}\bigr)\qquad(t,y<x).
\]
Since $(x-t)+(x-y)\ge|t-y|$, $0\le G_x\le\frac12e^{-|t-y|}$, and
$G_x(x,y)=0$. Let $\mathcal K_x$ be the operator on $L^2(I_x)$ with kernel
$\sqrt{m(t)}G_x(t,y)\sqrt{m(y)}$ and $v(t)=\sqrt{m(t)}e^{t-x}$. The kernel is
symmetric with
$\iint m(t)G_x(t,y)^2m(y)\le\frac14M^2$, so $\mathcal K_x$ is self-adjoint and
Hilbert--Schmidt, in particular compact; and $\|v\|_2^2=\int_{I_x}e^{2(t-x)}m(t)\dd t$.

{The Green representation.} For $\phi\in L^1(I_x)$ put
$g=G_x\phi=\int_{I_x}G_x(\cdot,y)\phi(y)\dd y$. Then
$g(t)=\frac12\int e^{-|t-y|}\phi(y)\dd y-\frac12e^{t-2x}\int e^y\phi(y)\dd y$. As in
the proof of \Cref{thm:ground}, $g\in C^1(I_x\cup\{x\})$, $g'$ is locally
absolutely continuous, $-g''+g=\phi$ almost everywhere, $g(x)=0$, $g,g'$
are bounded and $g(t),g'(t)\to0$ as $t\to-\infty$, and
\[
 g'(x)=-\int_{I_x}e^{y-x}\phi(y)\dd y,
\]
since $\partial_tG_x(t,y)\to-e^{y-x}$ as $t\uparrow x$ for every $y<x$ and
$|\partial_tG_x|\le1$. Integration by parts on $[a,x]$ and $a\to-\infty$ give
$\int_{I_x}\phi g=\int_{I_x}(g'^2+g^2)$, and $g\in H^1(I_x)$.

{Positivity and the norm.} For $f\in L^2(I_x)$, $\phi=\sqrt m\,f\in L^1$
and $\ip f{\mathcal K_xf}=\int\phi G_x\phi=\int_{I_x}(g'^2+g^2)\ge0$. Let
$\mathcal K_xf=\lambda f$ with $\lambda>0$ and $f\ne0$, and $g=G_x(\sqrt mf)$.
Then $\sqrt m\,g=\lambda f$, so $\int_{I_x}(g'^2+g^2)=\lambda\|f\|^2$ and
$\int_{I_x}mg^2=\lambda^2\|f\|^2$. The extension $\hat g$ of $g$ by zero on
$[x,\infty)$ lies in $H^1(\R)$ because $g(x)=0$, and \eqref{eq:ground-square}
gives
\[
 \lambda(1-\kappa\lambda)\|f\|^2=\int(\hat g'^2+\hat g^2-\kappa m\hat g^2)
 =\int\Bigl(\hat g'-\frac{u'}u\hat g\Bigr)^2\ge0 .
\]
If the right-hand side were zero, $\hat g$ would be a multiple of $u>0$
vanishing on $(x,\infty)$, so $\hat g=0$ and $f=\lambda^{-1}\sqrt m\,g=0$.
Hence $\lambda<1/\kappa$ for every positive eigenvalue. Since $\mathcal K_x$
is compact and positive, $\|\mathcal K_x\|$ is its largest eigenvalue or
zero, so $\|\mathcal K_x\|<1/\kappa$, and $I+s\mathcal K_x$ ($s\ge0$) and
$I-\kappa\mathcal K_x$ are invertible.

{Boundary values.} For $s\ge0$ put $y_s(t)=e^{t-x}V_s(t)/V_s(x)$, with
$V_s$ of \Cref{lem:filter}(ii), and put $y_{-\kappa}=u/u(x)$.
From $V_s''+2V_s'=smV_s$ and $-u''+u=\kappa mu$,
$-y_s''+(1+sm)y_s=0$ for $s\in\{-\kappa\}\cup[0,\infty)$; moreover $y_s(x)=1$,
$y_s$ is bounded on $I_x$, $y_s(t)\to0$ as $t\to-\infty$, and
\[
 y_s'(x)=1+r_s(x)\quad(s\ge0),\qquad y_{-\kappa}'(x)=\frac{u'(x)}{u(x)}=1-q(x).
\]
Let $z=e^{t-x}-sG_x(my_s)$, which is defined since $my_s\in L^1(I_x)$. Then
$-z''+z=-smy_s$, so $w=y_s-z$ satisfies $-w''+w=0$ on $I_x$, is bounded,
and $w(x)=0$; hence $w=\alpha e^t+\beta e^{-t}$ with $\beta=0$ by boundedness
and $\alpha=0$ by $w(x)=0$. Thus $y_s=e^{t-x}-sG_x(my_s)$. Multiplication
by $\sqrt m$ gives $(I+s\mathcal K_x)(\sqrt m\,y_s)=v$, so
$\sqrt m\,y_s=(I+s\mathcal K_x)^{-1}v$, and differentiation at $t=x$ gives
\[
 y_s'(x)=1+s\int_{I_x}e^{y-x}m(y)y_s(y)\dd y
 =1+s\ip v{(I+s\mathcal K_x)^{-1}v} .
\]
Hence $r_s(x)=s\ip v{(I+s\mathcal K_x)^{-1}v}$ for $s\ge0$ and
$q(x)=\kappa\ip v{(I-\kappa\mathcal K_x)^{-1}v}$; also
$r_\circ(x)=\kappa\int_{I_x}e^{-2(x-y)}m(y)\dd y=\kappa\|v\|^2$.

{The measure.} Let $\sigma_x$ be the spectral measure of the positive
compact operator $\mathcal K_x$ at $v$. It has mass $\|v\|^2$, is supported
in $[0,\|\mathcal K_x\|]\subset[0,1/\kappa)$, and includes any component of
$v$ in the kernel as an atom at $0$. The spectral theorem gives
$\ip v{\Phi(\mathcal K_x)v}=\int\Phi\dd\sigma_x$ for bounded Borel $\Phi$.
Therefore
\[
 r_s(x)=s\int\frac{\dd\sigma_x(\lambda)}{1+s\lambda},\qquad
 q(x)=\kappa\int\frac{\dd\sigma_x(\lambda)}{1-\kappa\lambda},\qquad
 r_\circ(x)=\kappa\int\dd\sigma_x .
\]
If $q(x)>0$, the measure $\kappa(1-\kappa\lambda)^{-1}q(x)^{-1}\dd\sigma_x$ is a
probability measure on $[0,1/\kappa)$; its image $\mu_x$ under
$\lambda\mapsto z=\kappa\lambda$ satisfies \eqref{eq:common}, because
$\kappa r_s/(sq)=\int(1-z)(1+(s/\kappa)z)^{-1}\dd\mu_x$ and
$r_\circ/q=\int(1-z)\dd\mu_x$ by the three displayed identities. If
$q(x)=0$, the positive integrand $(1-\kappa\lambda)^{-1}$ forces $\sigma_x=0$,
and all filters vanish at $x$.
\end{proof}

\subsection{The range of three filters}\label{app:range}

By \Cref{thm:common-measure}, the normalized values of the three filters
at a point are three integrals of one probability measure on $[0,1]$. Their
joint range is computed by a change of variables that makes the first
integrand linear, Jensen's inequality for the second, and Hermite
interpolation for the third.

For two scales $0<s_0<s_1$ and $\kappa>0$ put
\[
 \kappa_0=\kappa+s_0,\qquad \kappa_1=\kappa+s_1,\qquad \delta=s_1-s_0,
 \qquad f(U)=\frac{\kappa_0U}{\kappa_1-\delta U}.
\]

\begin{proposition}[The range of three filters]\label{prop:three-filter}
At a point with $q>0$, let $U_i=\kappa r_{s_i}/(s_iq)$ for $i=0,1$ and
$\vartheta=1-r_\circ/q$. Then $0\le U_0\le1$ and $f(U_0)\le U_1\le U_0$.
If $0<U_0<1$ and $f(U_0)<U_1<U_0$, put
\[
 y_-=\frac{\kappa_1U_1-\kappa_0U_0}{\delta U_1},\qquad
 y_+=\frac{\kappa_1(U_0-U_1)}{\delta(1-U_1)};
\]
then $0<y_+<U_0<y_-<1$ and
\begin{equation}\label{eq:third-interval}
 \vartheta_-:=\frac{\kappa(1-U_0)}{\kappa+s_0y_+}
 \ \le\ \vartheta\ \le\
 \vartheta_+:=1-\frac{\kappa_0U_0}{\kappa+s_0y_-}.
\end{equation}
On the boundary of the region, $\vartheta=1$ if $U_0=0$; $\vartheta=0$ if
$U_0=1$; $\vartheta=1-U_0$ if $U_1=U_0$; and
$\vartheta=\kappa(1-U_0)/(\kappa+s_0U_0)$ if $U_1=f(U_0)$. Every point of the
region and of the interval \eqref{eq:third-interval} arises from some
probability measure in \eqref{eq:common}.
\end{proposition}

\begin{proof}[Proof of \Cref{prop:three-filter}]\label{proof:prop:three-filter}
{Moment coordinates.}
The decreasing bijection $w=\kappa(1-z)/(\kappa+s_0z)$ has inverse
$z=\kappa(1-w)/(\kappa+s_0w)$. For the image measure $\tilde\omega$,
\[
 U_0=\int w\dd\tilde\omega,\qquad U_1=\int f(w)\dd\tilde\omega,
 \qquad 1-\vartheta=\int g_0(w)\dd\tilde\omega,
 \quad g_0(w)=\frac{\kappa_0w}{\kappa+s_0w}.
\]
Every probability measure on $[0,1]$ can occur in these coordinates.
Since $f''(w)=2\kappa_0\kappa_1\delta/(\kappa_1-\delta w)^3>0$,
Jensen's inequality and $f(w)\le w$ give $f(U_0)\le U_1\le U_0$.
Equality on the lower boundary forces $\tilde\omega=\delta_{U_0}$;
on the upper boundary it forces
$\tilde\omega=(1-U_0)\delta_0+U_0\delta_1$.
These measures, including $U_0=0,1$, give all four boundary values.

{Extremizing measures.}
For an interior pair, the increasing ratios
\[
 \frac{f(y)}y=\frac{\kappa_0}{\kappa_1-\delta y},\qquad
 \frac{1-f(y)}{1-y}=\frac{\kappa_1}{\kappa_1-\delta y}
\]
take the values $U_1/U_0$ and $(1-U_1)/(1-U_0)$ at
$y_-$ and $y_+$, respectively. The strict moment inequalities give
$0<y_+<U_0<y_-<1$. Thus
\[
 \tilde\omega_-=\Bigl(1-\frac{U_0}{y_-}\Bigr)\delta_0
       +\frac{U_0}{y_-}\delta_{y_-},\qquad
 \tilde\omega_+=\frac{1-U_0}{1-y_+}\delta_{y_+}
       +\frac{U_0-y_+}{1-y_+}\delta_1
\]
are probability measures with moments $(U_0,U_1)$.

{Sharp bounds by interpolation.}
For $0<y<1$, let $P_\pm\in\Span\{1,w,f(w)\}$ agree with $g_0$ and
its derivative at $y$, and agree with $g_0$ at $0$ for $P_-$ or at $1$
for $P_+$. Multiplication by $\kappa_1-\delta w$ maps this span
bijectively onto the quadratic polynomials; Hermite interpolation
therefore gives unique $P_\pm$. The polynomials
\[
 E_\pm=(\kappa_1-\delta w)(\kappa+s_0w)(g_0-P_\pm)
\]
have degree at most three, with zeros $0,y,y$ or $y,y,1$.
Their common value at $w=-\kappa/s_0$ is
$-\kappa_0^2\kappa s_1/s_0^2$, fixing the remaining factor. Hence
\[
 g_0-P_-=
 \frac{\kappa_0^2s_0s_1w(w-y)^2}
 {(\kappa+s_0y)^2(\kappa+s_0w)(\kappa_1-\delta w)}\ge0,
\]
\[
 g_0-P_+=
 \frac{\kappa_0\kappa s_0s_1(w-1)(w-y)^2}
 {(\kappa+s_0y)^2(\kappa+s_0w)(\kappa_1-\delta w)}\le0.
\]
The integral of $P_\pm$ depends only on $(1,U_0,U_1)$. Take $y=y_-$
in $P_-$ and $y=y_+$ in $P_+$; equality holds on the corresponding
two-point supports. Therefore
\[
 \frac{\kappa_0U_0}{\kappa+s_0y_-}
 =\int g_0\dd\tilde\omega_-
 \le\int g_0\dd\tilde\omega
 \le\int g_0\dd\tilde\omega_+
 =1-\frac{\kappa(1-U_0)}{\kappa+s_0y_+}.
\]
This is \eqref{eq:third-interval}. Mixtures of $\tilde\omega_-$ and
$\tilde\omega_+$ preserve the two moments and attain every intermediate
value of $\vartheta$; the boundary measures already identified complete
the attainment assertion.
\end{proof}

Let $\mathcal R_\kappa\subset\R^3$ be the set of triples $(U_0,U_1,\vartheta)$
such that either $0<U_0<1$, $f(U_0)<U_1<U_0$ and
$\vartheta_-\le\vartheta\le\vartheta_+$, or $(U_0,U_1)$ lies on the boundary of
$\{0\le U_0\le1,\ f(U_0)\le U_1\le U_0\}$ and $\vartheta$ is the boundary value
listed in \Cref{prop:three-filter}.

\begin{lemma}[Parametrization of the state region]\label{lem:range}
For $(U,T)\in[0,1]^2$ put
\[
 V=\frac{U[\kappa_0+\delta T(1-U)]}{\kappa_1-\delta U},\qquad
 y_+=\frac{\kappa_1U(1-T)}{\kappa_1-\delta UT},\qquad
 y_-=\frac{\kappa_0U+\kappa_1T(1-U)}{\kappa_0+\delta T(1-U)},
\]
\[
 \vartheta_-=\frac{\kappa(1-U)}{\kappa+s_0y_+},\qquad
 \vartheta_+=1-\frac{\kappa_0U}{\kappa+s_0y_-}.
\]
All denominators are at least $\kappa$, $\kappa_1-\delta U\ge\kappa_0$ and
$\kappa_1-\delta UT\ge\kappa_0$, these functions are continuous
on $[0,1]^2$, and $V,y_\pm,\vartheta_\pm$ take values in $[0,1]$. Then
$\mathcal R_\kappa=\{(U,V,\vartheta):(U,T)\in[0,1]^2,\ \vartheta_-\le\vartheta\le\vartheta_+\}$,
and on the interior of $[0,1]^2$ the functions $y_\pm,\vartheta_\pm$ coincide
with those of \Cref{prop:three-filter} at $(U_0,U_1)=(U,V)$. The sets
$\mathcal R_\kappa$ and $\Omega_\kappa$ are compact, and the state of every $m$ as
in \Cref{lem:filter} with $\kappa(m)=\kappa$ lies in $\Omega_\kappa$ at every point.
\end{lemma}

\begin{proof}[Proof of \Cref{lem:range}]\label{proof:lem:range}
The identity $V=(1-T)f(U)+TU$ maps the square onto the two-filter region;
on its interior, $T\mapsto V$ is an increasing affine bijection from
$(0,1)$ to $(f(U),U)$. Substitution in the formulas of
\Cref{prop:three-filter} gives
\[
 y_+=\frac{\kappa_1(U-V)}{\delta(1-V)},\qquad
 y_-=\frac{\kappa_1V-\kappa_0U}{\delta V}.
\]
These are the displayed formulas after cancellation of common factors.
Their denominators satisfy
$\kappa_1-\delta U,\kappa_1-\delta UT,\kappa_0+\delta T(1-U)\ge\kappa_0$.
The inequalities $0\le y_+\le U\le y_-\le1$ follow by clearing these
positive denominators. Hence $\kappa+s_0y_\pm\ge\kappa$, all functions
are continuous on the square, and $V,y_\pm,\vartheta_\pm\in[0,1]$.

At $T=0$, $y_+=y_-=U$ and
$\vartheta_\pm=\kappa(1-U)/(\kappa+s_0U)$; at $T=1$,
$y_+=0$, $y_-=1$, and $\vartheta_\pm=1-U$. At $U=0,1$ the common
endpoint values are $1,0$, respectively. These agree with all the
boundary cases of \Cref{prop:three-filter}, proving the range assertion
and $\vartheta_-\le\vartheta_+$ everywhere.

The continuous map
$(U,T,\lambda)\mapsto(U,V,(1-\lambda)\vartheta_-+\lambda\vartheta_+)$
maps $[0,1]^3$ onto $\mathcal R_\kappa$, so this set is compact.
Likewise, $\Omega_\kappa$ is the image of $[0,2]\times\mathcal R_\kappa$
under $(q,U_0,U_1,\vartheta)\mapsto(q,qU_0/\kappa,qU_1/\kappa,q\vartheta)$.
For $q>0$, state inclusion follows from \Cref{thm:common-measure,prop:three-filter};
for $q=0$, all filters vanish and the state is $Y_-$.
\end{proof}

\begin{proof}[Proof of \Cref{prop:state-region}]\label{proof:prop:state-region}
\Cref{lem:range} gives continuous $V,\vartheta_\pm$, their bounds,
compactness, and state inclusion. Every normalized triple in
$\mathcal R_\kappa$ is attained by a probability measure by
\Cref{prop:three-filter}; multiplying its coordinates by $q$ gives
the last assertion.
\end{proof}
\section{The calibration and the contact potential}\label{app:calibration}

The calibration begins with an exact defect identity. Its integral is
the one-determinant lower bound; the remaining task is to choose a
potential that makes the calibration defect vanish along a reference
density. The contact theorem does this by proving monotonicity of the
contact coordinate and locating the unique minimum of $F_\varsigma$.
The last two results evaluate the calibration constant and characterize
exactness at a self-consistent reference.

\begin{proof}[Proof of \Cref{lem:base-identity}]\label{proof:lem:base-identity}
{Calibration and root.}
For $0\le h\le q$, write $N=2(q-h)^2-(q^2-2h)g(h)$. If $q^2\le2h$,
then $N\ge0$; otherwise $g(h)\le2-h$ gives
$N\ge2(q-h)^2-(q^2-2h)(2-h)=h(2-q)^2\ge0$. Furthermore,
\[
 F_\varsigma(q,0)=\frac{(2-g(0))q^2}{2+\varsigma g(0)}\le q^2,
 \qquad F_\varsigma(q,q)=\frac{Q(q)g(q)}{2+\varsigma g(q)}\le(2-q)^2.
\]
Thus $0\le\Gamma_\varsigma\le\min\{q^2,(2-q)^2\}\le Q^2$.
Continuity follows by taking the minimum of the continuous function
$F_\varsigma(q,qu)$ over the fixed compact interval $0\le u\le1$.

For $0<q<2$, put $c=2bQ+2\kappa\omega\Gamma_\varsigma>0$.
The function $\kappa\nu^2-3Q-c/\nu$ increases strictly from $-\infty$
to $+\infty$, so the positive root is unique and
$d_\kappa=2Q+c/\nu_\kappa>2Q$. Its strict monotonicity gives joint
continuity for $0<q<2$. Since $Q\le1$ and $\Gamma_\varsigma\le Q^2$,
\[
 \nu_\kappa\le\max\{1,((3+2b+2\kappa\omega)/\kappa)^{1/2}\},
 \qquad \nu_\kappa^3\le C Q
\]
with $C$ uniform on compact $\kappa$-sets. This proves continuity at
$q=0,2$ as well.

{Pointwise identities.}
By \Cref{lem:filter}, $0\le f_\varsigma\le q$, $h_\varsigma=q-f_\varsigma$,
and $h_\varsigma'=q^2+\varsigma f_\varsigma^2-2h_\varsigma$.
Substituting $\mathcal J_\varsigma'=-g/(2+\varsigma g)$ gives
\eqref{eq:calibration-pointwise} and $\mathcal T\ge0$.
Put $\mathcal V(q)=-\kappa^{-1}\int_0^q(b+\tfrac32\nu_\kappa)$.
The cubic equation and \eqref{eq:cubic-loss} yield
\[
 n^3+bn^2+(\mathcal V\circ q)'
 =\rho(n,\nu_\kappa)-\tfrac12\nu_\kappa^3
   +\frac{(3\nu_\kappa+2b)Q}{2\kappa}
 =\rho(n,\nu_\kappa)-\omega\Gamma_\varsigma(q).
\]
Adding $\omega$ times the calibration identity gives
\begin{equation}\label{eq:defect-identity}
 n^3+bn^2+\omega f_\varsigma^2+
 \bigl\{\mathcal V(q)+\omega\mathcal J_\varsigma(h_\varsigma)\bigr\}'
 =\rho(n,\nu_\kappa)+\omega\mathcal T.
\end{equation}

{Integration.}
The derivatives $q'$ and $h_\varsigma'$ are integrable by
\Cref{thm:ground,lem:filter}; $\mathcal V$ and $\mathcal J_\varsigma$
are $C^1$ with bounded derivatives. Their compositions therefore have
integrable derivatives, with endpoint changes
$-\kappa^{-1}(2b+\tfrac32\int_0^2\nu_\kappa)$ and
$\mathcal J_\varsigma(2)$, respectively. Also $n^2,n^3,f_\varsigma^2\in L^1$.
The two nonnegative terms on the right of \eqref{eq:defect-identity}
are consequently integrable. Finally, integration of
$f_\varsigma'=\kappa m-2f_\varsigma-\varsigma f_\varsigma^2$ gives
\eqref{eq:base-observable}. Integrating \eqref{eq:defect-identity}, with
$\int m=1$, proves \eqref{eq:base-deficit}.
\end{proof}

Throughout the proof of \Cref{thm:contact},
$X_2=X_\circ(2)$, $q_\circ$ denotes the ground coordinate of $m_\circ$ (equal
to $X_\circ^{-1}$ on $[0,X_2]$, $0$ to the left and $2$ to the right), and
$\varsigma>0$ is fixed.

\begin{proof}[Proof of \Cref{thm:contact}: the reference density]
The function $b_\circ$ is positive and nondecreasing, so
$k\nu_\circ^3=3Q\nu_\circ+2b_\circ(Q)Q\ge2b_\circ(0)Q$ and, as in the proof of
\Cref{lem:base-identity}, $k\nu_\circ^2>3Q$ and
$d_\circ>\frac23k\nu_\circ^2\ge\frac23k(2b_\circ(0)Q/k)^{2/3}$. Hence
$1/d_\circ\le CQ^{-2/3}$ is integrable on $(0,2)$ and $X_2<\infty$. Since $Q\le1$
and $b_\circ(Q)\le b_\circ(1)$, $k\nu_\circ^3\le3Q\nu_\circ+2b_\circ(1)Q$; as in the proof of
\Cref{lem:base-identity}, this gives
$\nu_\circ\le\nu_{\max}:=\max\{1,((3+2b_\circ(1))/k)^{1/2}\}$ and
$\nu_\circ^3\le(3\nu_{\max}+2b_\circ(1))Q/k$, so $\nu_\circ(q)\to0$ as $q\to0$ or $q\to2$,
and $d_\circ\le k\nu_\circ^2$ is bounded. The function
$(\nu,q)\mapsto k\nu^3-3Q(q)\nu-2b_\circ(Q(q))Q(q)$ is $C^\infty$ on
$(0,\infty)\times(0,2)$, and its $\nu$-derivative at the root is
$3k\nu_\circ^2-3Q=3d_\circ>0$; by the implicit function theorem
$\nu_\circ\in C^\infty(0,2)$, and with the limits at the ends,
$\nu_\circ\in C[0,2]$. \Cref{thm:primal-map}, applied with
$\kappa=k$ and $D=d_\circ$ after the translation $x\mapsto x-X_\circ(1)$, shows
that $m_\circ$ is bounded with support $[0,X_2]$, has ground parameter $k$,
and has ground coordinate $X_\circ^{-1}$ on $[0,X_2]$. Since $\nu_\circ$ depends
on $q$ only through $Q(q)$, $d_\circ(2-q)=d_\circ(q)$, so
$X_\circ(2-q)=X_2-X_\circ(q)$, $q_\circ(X_2-x)=2-q_\circ(x)$ and
$m_\circ(X_2-x)=m_\circ(x)$. Differentiation of $k\nu^3-3Q\nu-2b_\circ(Q)Q=0$
in $Q$ gives
\[
 \frac{\dd\nu_\circ}{\dd Q}=\frac{3\nu_\circ+2b_\circ(Q)+2Qb_\circ'(Q)}{3k\nu_\circ^2-3Q}
 =\frac{3\nu_\circ+2b_\circ+2Qb_\circ'}{3d_\circ}>0 ,
\]
and $Q(q_\circ(x))$ increases strictly on $(0,X_2/2)$, where $q_\circ<1$. Hence
$m_\circ=\nu_\circ(q_\circ)^2$ is continuous on $\R$, zero outside $[0,X_2]$,
strictly increasing on $[0,X_2/2]$ and symmetric about $X_2/2$.

The filter $f_\circ=r_{\varsigma k}/\varsigma$ of $m_\circ$ satisfies
$f_\circ'=km_\circ-2f_\circ-\varsigma f_\circ^2$, $f_\circ=0$ on $(-\infty,0]$, and
$0\le f_\circ\le q_\circ$ by \Cref{lem:filter} with $s=\varsigma k$; it is $C^1$
because $m_\circ$ is continuous, and $f_\circ>0$ on $(0,\infty)$, since
$f_\circ(x)=\int_0^xe^{-\int_t^x(2+\varsigma f_\circ)}km_\circ(t)\dd t$. The reflected
filter $\check f_\circ(x)=f_\circ(X_2-x)$ satisfies
$\check f_\circ'=-km_\circ+2\check f_\circ+\varsigma\check f_\circ^2$ and
$0\le\check f_\circ\le q_\circ(X_2-x)=2-q_\circ$.
\end{proof}

\begin{proof}[Proof of \Cref{thm:contact}: the contact point and $g_\circ$]
{Monotonicity of $h_c$.} On $[0,X_2]$, $q_\circ'=km_\circ-Q(q_\circ)$, so
$h_c=q_\circ-f_\circ$ has the derivative
\[
 h_c'=2f_\circ+\varsigma f_\circ^2-Q(q_\circ),\qquad h_c'(0)=0 .
\]
On $(0,X_2)$ the functions $q_\circ$ and $m_\circ$ are locally Lipschitz, so
$f_\circ'$ and $h_c'$ are locally absolutely continuous, and
\[
 h_c''+(2+2\varsigma f_\circ)h_c'
 =(2+2\varsigma f_\circ)(km_\circ-Q(q_\circ))-Q'(q_\circ)q_\circ'
 =(2q_\circ+2\varsigma f_\circ)q_\circ' ,
\]
using $Q'(q)=2-2q$ and $km_\circ-Q(q_\circ)=q_\circ'$. The function $h_c'$ is
continuous on $[0,X_2]$, and by this equation
$|h_c''|\le(2+2\varsigma\|f_\circ\|_\infty)|h_c'|+(4+2\varsigma\|f_\circ\|_\infty)\|q_\circ'\|_\infty$
almost everywhere on $(0,X_2)$; so $h_c'$ is Lipschitz on $[0,X_2]$, and
variation of constants on $[0,x]$ gives
\[
 h_c'(x)=\int_0^x\exp\Bigl(-\int_t^x(2+2\varsigma f_\circ)\Bigr)
 \bigl(2q_\circ(t)+2\varsigma f_\circ(t)\bigr)q_\circ'(t)\dd t>0
 \qquad(0<x\le X_2),
\]
because $q_\circ>0$ and $q_\circ'=d_\circ(q_\circ)>0$ on $(0,X_2)$. For $x\ge X_2$,
$q_\circ=2$, $m_\circ=0$, $h_c=2-f_\circ$ and $h_c'=2f_\circ+\varsigma f_\circ^2>0$, and
$f_\circ(+\infty)=0$ by \Cref{lem:filter}. Hence $h_c$ is continuous and
strictly increasing on $[0,\infty)$, $h_c(0)=0$, $h_c(x)\to2$ as $x\to\infty$.

{The potential.} Define $g_\circ$ on $[0,2)$ by
$g_\circ(h_c(x))=f_\circ(x)+\check f_\circ(x)$ and put $g_\circ(2)=0$. Since
$f_\circ(x)\to0$ and $\check f_\circ(x)=f_\circ(X_2-x)\to0$ as $x\to\infty$,
$g_\circ$ is continuous on $[0,2]$. The bounds
$\check f_\circ\le2-q_\circ$ and $f_\circ\ge0$ give
$0\le g_\circ(h_c)\le f_\circ+2-q_\circ=2-h_c$. For $x\ge X_2$,
$\check f_\circ(x)=f_\circ(X_2-x)=0$, so $g_\circ(h)=2-h$ for
$h\ge h_c(X_2)=2-f_\circ(X_2)$. On $(0,X_2)$, $g_\circ$ is $C^1$ as a function of
$h$, since $h_c'>0$ there, and
\[
 1+g_\circ'(h_c)=\frac{h_c'+f_\circ'+\check f_\circ'}{h_c'}
 =\frac{2\check f_\circ+\varsigma\check f_\circ^2-Q(q_\circ)}{h_c'},
\]
because $f_\circ'+\check f_\circ'=(\check f_\circ-f_\circ)(2+\varsigma(f_\circ+\check f_\circ))$.
The numerator $\check P:=2\check f_\circ+\varsigma\check f_\circ^2-Q(q_\circ)$ at $x$ equals
$h_c'(X_2-x)>0$, by $\check f_\circ(x)=f_\circ(X_2-x)$ and
$Q(q_\circ(x))=Q(q_\circ(X_2-x))$. Hence $h+g_\circ(h)$ is nondecreasing on $[0,2]$:
strictly increasing on $[0,h_c(X_2)]$ and constant $2$ on $[h_c(X_2),2]$.

{Comparison of the two filters.} Let $m_1,m_2\ge0$ be bounded on
$(-\infty,x]$, vanish near $-\infty$, and let $f_1,f_2$ solve
$f_i'=km_i-2f_i-\varsigma f_i^2$ with $f_i(-\infty)=0$. Then
$\delta=f_2-f_1$ satisfies $\delta'=k(m_2-m_1)-c\delta$ with
$c=2+\varsigma(f_1+f_2)\ge2$, and, since $\delta$ is bounded,
\[
 \delta(x)=k\int_{-\infty}^x\exp\Bigl(-\int_t^xc\Bigr)\bigl(m_2(t)-m_1(t)\bigr)\dd t .
\]
Fix $0<x<X_2/2$, and take $m_1=m_\circ$ and $m_2(\tau)=m_\circ(2x-\tau)$. By
translation invariance and the symmetry of $m_\circ$, the filter of $m_2$
at $x$ is $f_\circ$ at $X_2-x$, that is, $f_2(x)=\check f_\circ(x)$, while
$f_1(x)=f_\circ(x)$. For $\tau<x$ write $\tau=x-t$ with $t>0$; then
$m_2(\tau)-m_1(\tau)=m_\circ(x+t)-m_\circ(x-t)$. If $x+t\le X_2/2$, this is
nonnegative by monotonicity; if $x+t>X_2/2$, then
$m_\circ(x+t)=m_\circ(X_2-x-t)$ with $x-t<X_2-x-t<X_2/2$, so it is again
nonnegative. For $0<t<\min\{x,X_2/2-x\}$ it is positive. Hence
$\check f_\circ(x)>f_\circ(x)$ for $0<x<X_2/2$, and by reflection
$\check f_\circ(x)<f_\circ(x)$ for $X_2/2<x<X_2$.
\end{proof}

\begin{proof}[Proof of \Cref{thm:contact}: the minimum]
\label{proof:thm:contact}
{The derivative.} Put $A=q-h$ and $B=q^2-2h$, so that
$F_\varsigma=(2A^2-Bg_\circ(h))/(2+\varsigma g_\circ(h))$. Where $g_\circ$ is
differentiable,
\[
 \partial_hF_\varsigma
 =\frac{(-4A+2g_\circ-Bg_\circ')(2+\varsigma g_\circ)-\varsigma g_\circ'(2A^2-Bg_\circ)}{(2+\varsigma g_\circ)^2}
 =\frac{2\bigl[(g_\circ-2A)(2+\varsigma g_\circ)-g_\circ'(B+\varsigma A^2)\bigr]}{(2+\varsigma g_\circ)^2}.
\]
Let $h=h_c(x)$, $\xi=q_\circ(x)$, $f=f_\circ(x)$ and
$\check f=\check f_\circ(x)$ for $0<x<X_2$. Write
$\Delta_g=2+\varsigma(f+\check f)$ and
$\Xi=(2-\xi)f+\xi\check f+\varsigma f\check f-Q(\xi)$.
Use $h=\xi-f$, $g_\circ=f+\check f$,
$h_c'g_\circ'=(\check f-f)\Delta_g$, and
\[
 B+\varsigma A^2=h_c'+2(q-\xi)(\xi+\varsigma f)
 +(1+\varsigma)(q-\xi)^2.
\]
Substitution and collection of the factor $q-\xi$ give
\begin{equation}\label{eq:contact-sign}
 \partial_hF_\varsigma(q,h_c(x))
 =\frac{2(q-\xi)}{\Delta_gh_c'}
 \bigl[(1+\varsigma)(f-\check f)(q-\xi)-2\Xi\bigr].
\end{equation}

{Positivity of $\Xi$.} Let $\epsilon_\xi=Q(\xi)/(1+\sqrt{1+\varsigma Q(\xi)})$,
the positive solution of $2\epsilon+\varsigma\epsilon^2=Q(\xi)$. The function
$y\mapsto2y+\varsigma y^2$ is increasing on $[0,\infty)$, and
$2f+\varsigma f^2-Q(\xi)=h_c'>0$, $2\check f+\varsigma\check f^2-Q(\xi)=\check P>0$;
hence $f=\epsilon_\xi+\alpha_1$, $\check f=\epsilon_\xi+\alpha_2$ with
$\alpha_1,\alpha_2>0$. Substituting and using $2\epsilon_\xi+\varsigma\epsilon_\xi^2=Q(\xi)$,
\[
 \Xi=(2-\xi)\alpha_1+\xi\alpha_2+\varsigma\epsilon_\xi(\alpha_1+\alpha_2)+\varsigma\alpha_1\alpha_2>0 .
\]

{The case $0<q\le1$ on $[0,q]$.} Let $h=h_c(x)\in(0,q]$ with
$0<x<X_2$. If $\xi=q$, \eqref{eq:contact-sign} gives $\partial_hF_\varsigma=0$. If $\xi<q$, then $\xi<1$, so $x<X_2/2$ and $f<\check f$; the
bracket in \eqref{eq:contact-sign} is negative and $q-\xi>0$, so
$\partial_hF_\varsigma<0$. If $\xi\ge1$, then $f\ge\check f$ and $q-\xi\le0$,
so the bracket is at most $-2\Xi<0$ and $\partial_hF_\varsigma\ge0$, with
equality only if $\xi=q$. If $q<\xi<1$, then $\xi-q\le\xi-h=f$ because
$h\le q$, and $\check f>f$; the bracket equals
$(1+\varsigma)(\xi-q)(\check f-f)-2\Xi\le(1+\varsigma)f(\check f-f)-2\Xi$, and
\[
 2\Xi-(1+\varsigma)f(\check f-f)
 =\alpha_1\bigl[4-2\xi+(3\varsigma+1)\epsilon_\xi\bigr]
 +\alpha_2\bigl[2\xi+(\varsigma-1)f\bigr]+(\varsigma+1)\alpha_1^2>0 .
\]
Here $2\xi+(\varsigma-1)f\ge\xi+\varsigma f>0$ by $f\le\xi$; so $\partial_hF_\varsigma>0$.
On $[h_c(X_2),2]$, where $g_\circ(h)=2-h$, the numerator of $F_\varsigma$ is
$h(2-q)^2$ and $\partial_hF_\varsigma=(2+2\varsigma)(2-q)^2/(2+\varsigma(2-h))^2>0$. The
function $F_\varsigma(q,\cdot)$ is continuous on $[0,q]$ and continuously
differentiable away from the two points $0$ and $h_c(X_2)$; it is
therefore strictly decreasing on $[0,h^*]$ and strictly increasing on
$[h^*,q]$, where $h^*=h_c(x^*)$, $q_\circ(x^*)=q$. Here $h^*=q-f_\circ(x^*)<q$.

{Points $h>q$.} Let $0<q\le1$, $0<a\le2-q$, $G=g_\circ(q)$ and
$g=g_\circ(q+a)$. Since $h+g_\circ(h)$ is nondecreasing, $g\ge G-a$, and
$g\ge0$. The identity
$2g_\circ(h)+\varsigma g_\circ(h)^2-Q(h)=\check P+2(2-\xi)f+(1+\varsigma)f^2+2\varsigma f\check f$
at $h=h_c(x)$, $0\le x\le X_2$,
and the identity
$2(2-h)+\varsigma(2-h)^2-Q(h)=(1+\varsigma)(2-h)^2$ on the tail give
$2G+\varsigma G^2\ge Q(q)$. Now $q-(q+a)=-a$ and $q^2-2(q+a)=-Q-2a$ with
$Q=Q(q)$, so
\[
 F_\varsigma(q,q+a)=\varphi(g),\qquad \varphi(\gamma)=\frac{2a^2+(Q+2a)\gamma}{2+\varsigma\gamma},
 \qquad F_\varsigma(q,q)=\frac{QG}{2+\varsigma G}.
\]
If $g\ge G$, then
$(2+\varsigma G)(2+\varsigma g)[\varphi(g)-F_\varsigma(q,q)]=2a^2(2+\varsigma G)+2ag(2+\varsigma G)+2Q(g-G)>0$.
If $g<G$, note $\varphi'(\gamma)=2(Q+2a-\varsigma a^2)/(2+\varsigma\gamma)^2$. When
$Q+2a-\varsigma a^2<0$, $\varphi(g)>\varphi(G)>F_\varsigma(q,q)$ by the previous case
with $g$ replaced by $G$. When $Q+2a-\varsigma a^2\ge0$, $\varphi$ is
nondecreasing; for $a\le G$,
\[
 \varphi(g)-F_\varsigma(q,q)\ge\varphi(G-a)-\frac{QG}{2+\varsigma G}
 =\frac{2a(2G+\varsigma G^2-Q)}{(2+\varsigma G)(2+\varsigma(G-a))}\ge0,
\]
while for $a>G$, $\varphi(g)\ge\varphi(0)=a^2>G^2\ge QG/(2+\varsigma G)$, the last
inequality by $2G+\varsigma G^2\ge Q$. In all cases
$F_\varsigma(q,q+a)\ge F_\varsigma(q,q)>F_\varsigma(q,h^*)$. Hence $h^*$ is the unique
minimizer of $F_\varsigma(q,\cdot)$ on $[0,2]$ for $0<q\le1$.

{The case $1<q<2$.} On $[0,h_c(X_2)]$ put
$\mathcal U(h)=2-h-g_\circ(h)$. For $0\le x\le X_2$,
\[
 \mathcal U(h_c(x))=2-(q_\circ-f_\circ)-(f_\circ+\check f_\circ)
 =2-q_\circ(x)-f_\circ(X_2-x)=h_c(X_2-x),
\]
so $\mathcal U$ is an involution of $[0,h_c(X_2)]$, and
$g_\circ(\mathcal U(h_c(x)))=f_\circ(X_2-x)+\check f_\circ(X_2-x)=g_\circ(h_c(x))$. For $h\in[0,h_c(X_2)]$, substitution of
$\mathcal U(h)=2-h-g_\circ(h)$ and $g_\circ(\mathcal U(h))=g_\circ(h)$ gives
$F_\varsigma(q,h)=F_\varsigma(2-q,\mathcal U(h))$.
Since $F_\varsigma(q,\cdot)$ increases strictly on $[h_c(X_2),2]$ for every $q$,
every minimizer of $F_\varsigma(q,\cdot)$ on $[0,2]$ lies in $[0,h_c(X_2)]$, and
\[
 \min_{[0,2]}F_\varsigma(q,\cdot)=\min_{[0,h_c(X_2)]}F_\varsigma(2-q,\mathcal U(\cdot))
 =\min_{[0,2]}F_\varsigma(2-q,\cdot).
\]
For $1<q<2$, $0<2-q<1$, so the unique minimizer of $F_\varsigma(2-q,\cdot)$ is
$h_c(X_\circ(2-q))=h_c(X_2-X_\circ(q))$, and the unique minimizer of
$F_\varsigma(q,\cdot)$ is $\mathcal U(h_c(X_2-X_\circ(q)))=h_c(X_\circ(q))$.

{The value.} At the minimizer $h=q-f$, $f=f_\circ(x)$, $\check f=\check f_\circ(x)$,
$q_\circ(x)=q$: $A=f$, $B=2f-Q(q)$, and
$2f^2-(2f-Q)(f+\check f)=Q(f+\check f)-2f\check f$. Since $h=q-f\in[0,q]$ minimizes
$F_\varsigma(q,\cdot)$ on $[0,2]\supset[0,q]$,
$\Gamma_\varsigma(q)=\min_{[0,2]}F_\varsigma(q,\cdot)=F_\varsigma(q,q-f)$, which gives
\eqref{eq:contact}; with the equality of the minima over $[0,2]$ for $q$ and
$2-q$ proved above, this gives $\Gamma_\varsigma(2-q)=\Gamma_\varsigma(q)$.
\end{proof}

\begin{lemma}[The calibration constant]\label{lem:J-value}
In the setting of \Cref{thm:contact}, with $g=g_\circ$, $X_2=X_\circ(2)$ and
$f_{\rm end}=f_\circ(X_2)$,
\[
 \mathcal J_\varsigma(2)=\int_0^{X_2/2}\Bigl(2\Gamma_\varsigma(q_\circ)-f_\circ^2-\check f_\circ^2\Bigr)\dd x
 -\frac{\varsigma f_{\rm end}-2\log(1+\varsigma f_{\rm end}/2)}{\varsigma^2} .
\]
\end{lemma}

\begin{proof}\label{proof:lem:J-value}
For $0<x<X_2$, $h_c(x)$ is the minimizer of $F_\varsigma(q_\circ(x),\cdot)$, so
\eqref{eq:calibration-pointwise}, which is an algebraic identity in
$(q,h_\varsigma,f_\varsigma)$ given the two state equations and holds along $m_\circ$
with $(q,f_\varsigma,h_\varsigma)=(q_\circ,f_\circ,h_c)$, gives
$(\mathcal J_\varsigma(h_c))'=\Gamma_\varsigma(q_\circ)-f_\circ^2$. Integration over $[0,X_2]$, with
$h_c(0)=0$ and $\mathcal J_\varsigma(0)=0$, gives
$\mathcal J_\varsigma(h_c(X_2))=\int_0^{X_2}(\Gamma_\varsigma(q_\circ)-f_\circ^2)\dd x$. By the symmetry
$q_\circ(X_2-x)=2-q_\circ(x)$, $\Gamma_\varsigma(q_\circ(X_2-x))=\Gamma_\varsigma(q_\circ(x))$, and
$\int_{X_2/2}^{X_2}f_\circ^2=\int_0^{X_2/2}\check f_\circ^2$, so the integral equals
the first term of the claim. On $[h_c(X_2),2]=[2-f_{\rm end},2]$,
$g_\circ(h)=2-h$, and the substitution $w=2-h$ gives
\[
 \mathcal J_\varsigma(2)-\mathcal J_\varsigma(h_c(X_2))=-\int_0^{f_{\rm end}}\frac{w\dd w}{2+\varsigma w}
 =-\Bigl[\frac w\varsigma-\frac2{\varsigma^2}\log\Bigl(1+\frac{\varsigma w}2\Bigr)\Bigr]_0^{f_{\rm end}} .
\qedhere
\]
\end{proof}

\begin{proof}[Proof of \Cref{prop:exact-reference}]\label{proof:prop:exact-reference}
The condition \eqref{eq:self-consistent} makes the defining cubics for
$\nu_\circ$ and $\nu_k$ identical. Uniqueness of the positive root and
continuity at $0,2$ give $\nu_\circ=\nu_k$, hence $d_\circ=d_k$ and
$X_\circ(q)=X_k(q)+X_\circ(1)$. This proves the stated translation.

Suppose $\int m_\circ=1$. Then $m_\circ\in\mathcal M^1_c$ and
$\kappa(m_\circ)=k$ by \Cref{thm:contact}. Along its support,
$n=\nu_k(q_\circ)$ and $h_\varsigma=h_c$ minimizes $F_\varsigma(q_\circ,\cdot)$;
therefore both defects in \eqref{eq:base-deficit} vanish. They also
vanish outside: on the left $(q,f,h)=(0,0,0)$; on the right $q=2$,
$n=\nu_k(2)=0$ and $g_\circ(h)=2-h$, which gives
$F_\varsigma(2,h)=0=\Gamma_\varsigma(2)$. Thus
$E(m_\circ)=\mathcal B_k(m_\circ)$.

For every density in the stated class, \eqref{eq:base-deficit} gives
$E(m)\ge\mathcal B_k(m)$. By \eqref{eq:base-value} and \eqref{eq:base-observable},
the prescribed ground parameter and determinant force
$\mathcal B_k(m)=\mathcal B_k(m_\circ)=E(m_\circ)$. The reference itself
belongs to the class and attains this value.
\end{proof}
\section{Comparisons, integration and transfer}\label{app:integration}

The comparison identities are obtained by differentiating explicit
potentials of $(q,r)$. We prove them for a class of comparison profiles
that also includes the primal family. After establishing continuity of
the comparison filters, we integrate the certificate residual, eliminate
the amplitude, and transfer strict admissibility from the target ground
parameter to the packed densities.

\subsection{The comparison density and the comparison identities}\label{app:comparison}

\begin{lemma}[Regularity of the comparison density]\label{lem:local-profile}
The comparison density $\bar m_\kappa$ is bounded, compactly supported,
has ground parameter $\kappa$ and ground coordinate $q$ at $X_\kappa(q)$.
On compact $\kappa$-sets, uniformly in $0<q<2$,
\[
 d_\kappa^{-1}\le CQ^{-2/3},\qquad
 \frac{\nu_\kappa^2}{d_\kappa}\le\frac3{2\kappa},\qquad
 \frac Q{d_\kappa\nu_\kappa}\le\frac1{2b},\qquad
 \nu_\kappa^3\le CQ ,
\]
$\nu_\kappa(q)$ is jointly continuous in $(\kappa,q)\in(0,\infty)\times[0,2]$,
and $\kappa\mapsto d_\kappa^{-1}$ is continuous into $L^1(0,2)$. If
$\Gamma_\varsigma(2-q)=\Gamma_\varsigma(q)$, then $\bar m_\kappa$ is symmetric about $0$.
\end{lemma}

\begin{proof}\label{proof:lem:local-profile}
Write $\nu=\nu_\kappa(q)$, $d=d_\kappa(q)$. From \eqref{eq:base-root},
$\kappa\nu^3\ge2bQ$, $\kappa\nu^2>3Q$, and
$d\nu=\kappa\nu^3-Q\nu=2Q\nu+2bQ+2\kappa\omega\Gamma_\varsigma\ge2bQ$. Hence
$d>\kappa\nu^2-\frac13\kappa\nu^2=\frac23\kappa\nu^2$, which gives $\nu^2/d<3/(2\kappa)$
and $d^{-1}<\frac3{2\kappa}\nu^{-2}\le\frac3{2\kappa}(\kappa/(2bQ))^{2/3}$, and
$Q/(d\nu)\le1/(2b)$. The bound $\nu^3\le CQ$ and the joint continuity were
proved with \Cref{lem:base-identity}. Since $Q^{-2/3}$ is integrable on
$(0,2)$, $1/d_\kappa\in L^1$, $d_\kappa\le\kappa\nu_\kappa^2$ is bounded, and
\Cref{thm:primal-map} with $D=d_\kappa$, for which $(D+Q)/\kappa=\nu_\kappa^2$,
gives the first assertion. For $\kappa'\to\kappa$, $1/d_{\kappa'}\to1/d_\kappa$
pointwise on $(0,2)$, with the common majorant $CQ^{-2/3}$ on a compact
neighborhood of $\kappa$; dominated convergence gives the $L^1$ continuity.
If $\Gamma_\varsigma$ is symmetric, $\nu_\kappa(2-q)=\nu_\kappa(q)$ and $d_\kappa(2-q)=d_\kappa(q)$,
so $X_\kappa(2-q)=-X_\kappa(q)$ and $\bar m_\kappa(-x)=\bar m_\kappa(x)$.
\end{proof}

\begin{definition}[Comparison profiles]\label{def:comparison-profile}
A {comparison profile} at $\kappa>0$ is a continuous
$\nu:[0,2]\to[0,\infty)$, positive on $(0,2)$, such that
$d=\kappa\nu^2-Q$ is positive on $(0,2)$, $1/d\in L^1(0,2)$, and, for a
constant $C_\nu$,
\[
 \frac{\nu^2}d+\frac Q{d\nu}+\frac{\nu^3}Q\le C_\nu\qquad\text{on }(0,2).
\]
Its density $\bar m$ is the density of \Cref{thm:primal-map} with $D=d$,
that is, $\bar m(X(q))=\nu(q)^2$, $X(q)=\int_1^qd^{-1}$. Its left filters are
those of \Cref{lem:filter}, and its {right filters}, as functions of
$q$, are $R_s(q)=\check r_s(-X(q))$ and $R_\circ(q)=\check r_\circ(-X(q))$, where
$\check r_s,\check r_\circ$ are the left filters of $\check m(x)=\bar m(-x)$.
\end{definition}

By \Cref{lem:local-profile}, $\nu_\kappa$ is a comparison profile at $\kappa$,
and its right filters are the functions $R_s,R_\circ$ of
\Cref{sec:comparisons}, by the next lemma.

\begin{lemma}[Right filters]\label{lem:right-filters}
Let $\nu$ be a comparison profile at $\kappa$ with density $\bar m$. Then
$\bar m$ has ground parameter $\kappa$, and $R_s,R_\circ$ are absolutely
continuous on $[0,2]$ and satisfy
\[
 d\frac{\dd R_s}{\dd q}=-s\nu^2+2R_s+R_s^2,\qquad
 d\frac{\dd R_\circ}{\dd q}=-\kappa\nu^2+2R_\circ,\qquad R_s(2)=R_\circ(2)=0,
\]
\[
 0\le R_s\le\frac s\kappa(2-q),\qquad 0\le R_\circ\le2-q,\qquad
 R_\circ(q)=\kappa\int_{X(q)}^\infty e^{-2(y-X(q))}\bar m(y)\dd y .
\]
Let $\bar v_s$ solve $-\bar v''+(1+s\bar m)\bar v=0$ with $\bar v_s=e^{-x}$ for
$x\ge X(2)$. Then $R_s(q)=-\bar v_s'/\bar v_s-1$ at $x=X(q)$, and for
$x\le X(0)$, $\bar v_s(x)=c_s^-e^{-x}+c_s^+e^x$ with $c_s^-=\exp H_{\bar m}(s)$.
If $\bar m$ is symmetric about $0$, then, with $\bar r_s,\bar r_\circ$ the left filters
of $\bar m$ (functions of $x$), $R_s(q)=\bar r_s(X(2-q))$ and
$R_\circ(q)=\bar r_\circ(X(2-q))$.
\end{lemma}

\begin{proof}\label{proof:lem:right-filters}
The function $\nu$ is bounded and $d\le\kappa\nu^2$, so \Cref{thm:primal-map}
applies: $\bar m$ is bounded, supported in $[X(0),X(2)]$, has ground
parameter $\kappa$ and ground coordinate $q(x)=X^{-1}(x)$ there. The reflection
$(Rf)(x)=f(-x)$ is unitary on $L^2(\R)$ and $K_{\check m}=RK_{\bar m}R$, so
$\kappa(\check m)=\kappa$. With $u$ the ground state of $\bar m$, the function
$\check u(x)=u(-x)$ is positive, lies in $H^1\cap C^1$, has a locally absolutely
continuous derivative and solves $-\check u''+\check u=\kappa\check m\check u$ almost
everywhere; by the uniqueness statement of \Cref{thm:ground}, the ground
coordinate of $\check m$ is $\check q(x)=1-\check u'(x)/\check u(x)=1+u'(-x)/u(-x)=2-q(-x)$.
\Cref{lem:filter}, applied to $\check m$, gives $0\le\check r_s\le\frac s\kappa\check q$
and $0\le\check r_\circ\le\check q$, that is, the two bounds. The filter
equations $\check r_s'=s\check m-2\check r_s-\check r_s^2$ evaluated at
$-x$, with $\check m(-x)=\bar m(x)=\nu(q(x))^2$ and $\frac{\dd}{\dd x}=d(q)\frac{\dd}{\dd q}$
on $(X(0),X(2))$, give the differential equations; the composition of the
Lipschitz function $\check r_s$ (its derivative is bounded because $\check m$
is) with the absolutely continuous function $-X$ is absolutely continuous
on $[0,2]$. The filters of $\check m$ vanish to the
left of its support $[-X(2),-X(0)]$, which gives $R_s(2)=R_\circ(2)=0$; the
convolution formula for $\check r_\circ$ gives the formula for $R_\circ$.

The function $Y=e^x\check V_s$, with $\check V_s$ the function $V_s$ of \Cref{lem:filter}(ii) for
$\check m$, solves $-Y''+(1+s\check m)Y=0$ and equals $e^x$ to the left of
the support of $\check m$; so $\bar v_s(x)=Y(-x)$, and
$-\bar v_s'(x)/\bar v_s(x)=Y'(-x)/Y(-x)=1+\check r_s(-x)$. To the right of
the support of $\check m$, $\check V_s''+2\check V_s'=0$, so
$\check V_s=\alpha+\beta e^{-2x}$ with $\alpha=\check V_s(+\infty)=\exp H_{\check m}(s)$,
and $Y=\alpha e^x+\beta e^{-x}$; since $K_{\check m}$ is unitarily equivalent
to $K_{\bar m}$ by the reflection $f(x)\mapsto f(-x)$, $\alpha=\exp H_{\bar m}(s)$,
and $\bar v_s(x)=\alpha e^{-x}+\beta e^x$ for $x\le X(0)$. If $\bar m$ is
symmetric, $\check m=\bar m$ and $-X(q)=X(2-q)$.
\end{proof}

\begin{lemma}[Comparison with a comparison profile]\label{lem:comparison-general}
Let $\nu$ be a comparison profile at $\kappa$ with density $\bar m$ and right
filters $R_s,R_\circ$, and let $m\ge0$ be integrable with $E(m)<\infty$,
$\kappa(m)=\kappa$, ground coordinate $q$ and left filters $r_s,r_\circ$. Define
$a_s,a_\circ$ by the formulas of \Cref{thm:comparisons} with $\nu,d,R_s,R_\circ$
in place of $\nu_\kappa,d_\kappa,R_s,R_\circ$. Then
$-sQ/(2d)\le a_s\le0$, $-\kappa Q/d\le a_\circ\le0$, and
\[
 H_m(s)-H_{\bar m}(s)=\int a_s(q,r_s)\bigl(n^2-\nu(q)^2\bigr)\dd x,
 \qquad
 S(m)-S(\bar m)=\int a_\circ(q,r_\circ)\bigl(n^2-\nu(q)^2\bigr)\dd x .
\]
\end{lemma}

\begin{proof}\label{proof:lem:comparison-general}
{Bounds and integrability.}
The inequalities $r_s\le sq/\kappa$ and
$R_s\le s(2-q)/\kappa$ imply
$2\kappa r_sR_s\le sQ(r_s+R_s)$, including when either factor vanishes.
Thus $-sQ/(2d)\le a_s\le0$. Similarly,
$2r_\circ R_\circ\le Q(r_\circ+R_\circ)$ and
$r_\circ+R_\circ\le2$ give $-\kappa Q/d\le a_\circ\le0$.
The profile bounds imply $|a_s|+|a_\circ|\le C\nu$ and
$\nu^3\le C Q(q)\in L^1(\R)$. Since
$\nu n^2\le\nu^3/3+2n^3/3$, both claimed integrands are integrable.

{Determinant potential.}
Let $L_s(q)=\log(e^{X(q)}\bar v_s(X(q)))$. By \Cref{lem:right-filters},
$dL_s'=-R_s$ and $L_s(2)=0$. On the triangle
$\Delta_s=\{0\le q\le2,\ 0\le r\le sq/\kappa\}$ define
\[
 W_s(q,r)=L_s(q)+\log\frac{2+r+R_s(q)}2.
\]
On $0<q<2$, the right-filter equation gives
\[
 \partial_qW_s=-\frac{rR_s+s\nu^2}{d(2+r+R_s)},\qquad
 \partial_rW_s=\frac1{2+r+R_s}.
\]
These derivatives are bounded on $\Delta_s$ because
$rR_s\le s^2Q/\kappa^2$, $Q/d\le C_\nu\nu$, and
$\nu^2/d\le C_\nu$. The potential is continuous on the closed triangle,
so integration along segments makes it Lipschitz there. Its endpoint
values are
\[
 W_s(0,0)=\log c_s^-=H_{\bar m}(s),\qquad
 W_s(2,r)=\log(1+r/2),
\]
where the first identity follows from
$\bar v_s=c_s^-e^{-x}+c_s^+e^x$ on the left of the support.
Along $m$, substitution of $q'=\kappa n^2-Q$ and
$r_s'=sn^2-2r_s-r_s^2$ gives, for $0<q<2$,
\[
 \frac{\dd}{\dd x}W_s(q,r_s)=-r_s+a_s(q,r_s)(n^2-\nu^2).
\]
Indeed, the coefficient of $n^2$ is
$-(sQ+\kappa r_sR_s)/(d(2+r_s+R_s))=a_s$; the remaining term is
$-r_s-a_s\nu^2$, using $d=\kappa\nu^2-Q$.

The composition is locally absolutely continuous. On the level sets
$q=0,2$, one has $q'=0$, and hence $n=0$, almost everywhere. On $q=0$
also $r_s=0$; on $q=2$, differentiating $\log(1+r_s/2)$ gives $-r_s$.
Locality of derivatives of absolutely continuous functions justifies
these identities on the level sets. The derivative formula therefore
holds almost everywhere, with the endpoint convention $a_s=0$.
Its right-hand side is integrable, and $(q,r_s)$ tends to $(0,0)$ and
$(2,0)$ at the two infinities. Integration proves the determinant identity.

{Second-trace potential.}
Put
\[
 G_\circ(q)=\frac\kappa2\int_q^2\frac{\nu^2R_\circ}{d},\qquad
 W_\circ(q,r)=\frac{r^2}4+\frac{rR_\circ(q)}2+G_\circ(q).
\]
The integral is finite, and on $\Delta_\circ=\{0\le q\le2,\ 0\le r\le q\}$,
\[
 \partial_rW_\circ=\frac{r+R_\circ}2,\qquad
 \partial_qW_\circ=
 \frac{2rR_\circ-\kappa\nu^2(r+R_\circ)}{2d}.
\]
The same profile bounds, with $rR_\circ\le Q$ and $r+R_\circ\le2$,
make these derivatives bounded; thus $W_\circ$ is Lipschitz on
$\Delta_\circ$. The state equations yield
\[
 \frac{\dd}{\dd x}W_\circ(q,r_\circ)
 =-r_\circ^2+a_\circ(q,r_\circ)(n^2-\nu^2).
\]
The coefficient of $n^2$ is
$\kappa[2r_\circ R_\circ-Q(r_\circ+R_\circ)]/(2d)=a_\circ$;
the constant term follows from $d=\kappa\nu^2-Q$.
At $q=0$, $r_\circ=0$; at $q=2$, $W_\circ=r_\circ^2/4$ and its derivative
is $-r_\circ^2$. The same level-set argument extends the formula almost
everywhere. Its endpoint values are $G_\circ(0)$ and $0$.

Finally, integration of $(\bar r_\circ^2)'$ and Fubini's theorem give
\[
 S(\bar m)=\frac\kappa2\int\bar m\bar r_\circ
 =\frac\kappa2\int\bar mR_\circ
 =\frac\kappa2\int_0^2\frac{\nu^2R_\circ}{d}=G_\circ(0).
\]
Integrating the derivative of $W_\circ$ proves the second identity.
\end{proof}

\begin{proof}[Proof of \Cref{thm:comparisons}]\label{proof:thm:comparisons}
By \Cref{lem:local-profile}, $\nu_\kappa$ is a comparison profile at $\kappa$
with density $\bar m_\kappa$, and \Cref{lem:comparison-general} applies.
\end{proof}

\begin{lemma}[Continuity of the comparison filters]\label{lem:filter-continuity}
For $s>0$ let $\bar r_s(q),\bar r_\circ(q)$ denote the left filters of $\bar m_\kappa$ at
the point $X_\kappa(q)$, $q\in[0,2]$, and put
\[
 w_s=q-\frac\kappa s\bar r_s,\qquad w_\circ=q-\bar r_\circ .
\]
Then $\bar r_s,\bar r_\circ$ are jointly continuous in
$(q,\kappa)\in[0,2]\times(0,\infty)$, and on every compact $\kappa$-set
$0\le w_s,w_\circ\le Cq^{7/3}$. If $\Gamma_\varsigma$ is symmetric, the right
filters are $R_s(q)=\bar r_s(2-q)$ and $R_\circ(q)=\bar r_\circ(2-q)$.
\end{lemma}

\begin{proof}\label{proof:lem:filter-continuity}
Fix a compact set $\mathsf K\subset(0,\infty)$ and $\kappa\in\mathsf K$. Along
$\bar m_\kappa$, $x=X_\kappa(q)$ and $\frac{\dd q}{\dd x}=d_\kappa(q)$ on $(0,2)$;
$\bar r_s$ is Lipschitz in $x$, and $X_\kappa$ is absolutely continuous, so
$\bar r_s(q)$ is absolutely continuous on $[0,2]$ with $\bar r_s(0)=0$ and
\[
 \frac{\dd\bar r_s}{\dd q}=G_\kappa(q,\bar r_s),\qquad
 G_\kappa(q,r)=\frac{s\nu_\kappa(q)^2-2r-r^2}{d_\kappa(q)} .
\]
Its values lie in $[0,R_{\max}]$, $R_{\max}=2s/\min\mathsf K$. For
$r,r'\in[0,R_{\max}]$,
$|G_\kappa(q,r)-G_\kappa(q,r')|\le\ell_\kappa(q)|r-r'|$ with
$\ell_\kappa=(2+2R_{\max})/d_\kappa$, and
$|G_\kappa(q,r)-G_{\kappa'}(q,r)|\le\epsilon_{\kappa,\kappa'}(q)$ with
$\epsilon_{\kappa,\kappa'}=s|\nu_\kappa^2/d_\kappa-\nu_{\kappa'}^2/d_{\kappa'}|+(2R_{\max}+R_{\max}^2)|1/d_\kappa-1/d_{\kappa'}|$.
The function $e=|\bar r_s^{(\kappa)}-\bar r_s^{(\kappa')}|$ satisfies
$e(q)\le\int_0^q\ell_\kappa e+\int_0^2\epsilon_{\kappa,\kappa'}$, so Gronwall's
inequality gives
$\sup_qe\le\exp(\int_0^2\ell_\kappa)\int_0^2\epsilon_{\kappa,\kappa'}$. As
$\kappa'\to\kappa$, $\epsilon_{\kappa,\kappa'}\to0$ pointwise on $(0,2)$ by
\Cref{lem:local-profile}, with the integrable majorant
$3s/\min\mathsf K+CQ^{-2/3}$; hence $\bar r_s$ is continuous in $\kappa$
uniformly in $q$, and being continuous in $q$ for each $\kappa$, it is
jointly continuous. The same argument, with
$G_\kappa(q,r)=(\kappa\nu_\kappa^2-2r)/d_\kappa$, applies to $\bar r_\circ$.

By \Cref{lem:filter}, $0\le w_s\le q$, and along $\bar m_\kappa$
\[
 w_s'=q'-\frac\kappa s\bar r_s'=\Bigl(1+\frac s\kappa\Bigr)q^2-2\Bigl(1+\frac{sq}\kappa\Bigr)w_s
 +\frac s\kappa w_s^2 ,
\]
since
$q'-\frac\kappa s(s\bar m-2\bar r_s-\bar r_s^2)=-Q(q)+\frac{2\kappa}s\bar r_s+\frac\kappa s\bar r_s^2$,
which for $\bar r_s=\frac s\kappa(q-w_s)$ equals
$-Q(q)+2(q-w_s)+\frac s\kappa(q-w_s)^2$. For $0\le w\le q$,
$-2(1+\frac{sq}\kappa)w+\frac s\kappa w^2\le w(\frac s\kappa(w-2q)-2)\le0$, so
$\frac{\dd w_s}{\dd q}\le(1+\frac s\kappa)q^2/d_\kappa\le Cq^2Q^{-2/3}$, and
$w_s(0)=0$ gives $w_s\le Cq^{7/3}$ for $q\le1$, hence for $q\le2$ after
enlarging $C$. Likewise $w_\circ'=q^2-2w_\circ\le q^2$ gives $w_\circ\le Cq^{7/3}$.
The last assertion is \Cref{lem:right-filters} with \Cref{lem:local-profile}.
\end{proof}

\begin{lemma}[Spectral values of a symmetric density]\label{lem:symmetric-values}
Let $m\ge0$ be integrable and not zero, with support in $[x_-,2x_c-x_-]$
and symmetric about $x_c$, let $\kappa=\kappa(m)$, and let $r_s,r_\circ$ be its left filters.
Then
\[
 H_m(s)=2\int_{x_-}^{x_c}r_s\dd x+\log(1+r_s(x_c)),\qquad
 S(m)=2\int_{x_-}^{x_c}r_\circ^2\dd x+r_\circ(x_c)^2 .
\]
\end{lemma}

\begin{proof}\label{proof:lem:symmetric-values}
Translate so that $x_c=0$. Let $u_L=e^xV_s$, with $V_s$ of
\Cref{lem:filter}(ii); it solves $-u''+(1+sm)u=0$, equals $e^x$
to the left of the support, and $u_L'/u_L=1+r_s$, $u_L(x)=e^x\exp\int_{-\infty}^xr_s$.
Since $m$ is even, $u_R(x)=u_L(-x)$ solves the same equation and equals
$e^{-x}$ to the right of the support. The Wronskian
$W=u_L'u_R-u_Lu_R'$ is constant. At $x=0$, $u_R(0)=u_L(0)$ and
$u_R'(0)=-u_L'(0)$, so $W=2u_L'(0)u_L(0)=2(1+r_s(0))\exp(2\int_{x_-}^0r_s)$,
using $r_s=0$ on $(-\infty,x_-]$. To the right of the support,
$V_s=\alpha+\beta e^{-2x}$ with $\alpha=V_s(+\infty)=\exp H_m(s)$
(\Cref{lem:filter}), so $u_L=\alpha e^x+\beta e^{-x}$, $u_R=e^{-x}$ and
$W=2\alpha$. Comparing the two expressions gives the first identity.

For the second, write $S(m)=\frac{\kappa^2}4\iint m(x)m(y)e^{-2|x-y|}$ and
split $\R^2$ into the four quadrants of the signs of $x,y$. For the
quadrant $x,y<0$, the identity of \Cref{lem:filter}(i) with $c=\kappa$ and
$\tilde m=m\ind_{(-\infty,0)}$, for which
$\kappa\int_{-\infty}^xe^{-2(x-y)}\tilde m(y)\dd y$ is $r_\circ(x)$ on $(-\infty,0]$ and
$r_\circ(0)e^{-2x}$ on $[0,\infty)$, gives
$\int_{-\infty}^0r_\circ^2+\int_0^\infty r_\circ(0)^2e^{-4x}\dd x=\int_{-\infty}^0r_\circ^2+r_\circ(0)^2/4$. For the quadrant
$x<0<y$, the integral is
$\frac{\kappa^2}4\int_{x<0}e^{2x}m\int_{y>0}e^{-2y}m=\frac{\kappa^2}4(r_\circ(0)/\kappa)^2$,
because $\int_{y>0}e^{-2y}m(y)\dd y=\int_{x<0}e^{2x}m(x)\dd x=r_\circ(0)/\kappa$ by
symmetry. By symmetry the quadrant $x,y>0$ equals the first and the
quadrant $y<0<x$ equals the second. Summation gives the claim.
\end{proof}

\subsection{Integration of the residual}\label{app:integration-proof}

\begin{proof}[Proof of \Cref{prop:certificate-mechanism}]
\label{proof:prop:certificate-mechanism}
Let $m\in\mathcal M^1_c$ with ground parameter $\kappa$ and state $Y$.
Then $E(m)<\infty$, and $Y(x)\in\Omega_\kappa$ by \Cref{lem:range}.
On the level sets $q=0,2$, absolute continuity gives $q'=0$ almost
everywhere, and the ground equation then gives $n=0$. The hypotheses
therefore imply $\mathscr R_\kappa(Y,n)\ge0$ almost everywhere.

The terms $\rho(n,\nu_\kappa(q))$ and $a(n^2-\nu_\kappa^2)$ are integrable
by \Cref{lem:base-identity,thm:comparisons}. Since $Y'$ is integrable and
$Y$ lies in the compact set $\Omega_\kappa$, the chain rule gives
\[
 (\Psi\circ Y)'=\mathsf b_\Psi n^2-\mathsf d_\Psi\in L^1,\qquad
 \int_\R(\Psi\circ Y)'=\Psi(Y_+)-\Psi(Y_-)=0.
\]
For the price, let $\Phi(y)=\int_0^y\eta$. This function is Lipschitz,
with $\Phi'=\eta$ outside the finite discontinuity set of $\eta$.
The chain rule applies off its preimage under $q$; on that preimage,
both $q$ and $\Phi\circ q$ have derivative zero almost everywhere on
each level set. Hence
\[
 (\Phi\circ q)'=\eta(q)q'\in L^1,\qquad
 \int_\R\eta(q)q'=\Phi(2)-\Phi(0)=\int_0^2\eta.
\]
Also $\int Q(q)=\kappa-2$. Rewrite \eqref{eq:residual} as
\[
 \rho=\mathscr R_\kappa+a(n^2-\nu_\kappa^2)
 -\frac{\eta(q)q'}\kappa-\frac\tau\kappa Q(q)-(\Psi\circ Y)' .
\]
Every term is integrable. Integrating and discarding
$\mathscr R_\kappa\ge0$, then using \Cref{thm:comparisons}, gives
\[
 \int\rho\ge\sum_i\xi_i[H_m(s_i)-H_{\bar m_\kappa}(s_i)]
 +\chi[S(m)-S(\bar m_\kappa)]-\frac1\kappa\int_0^2\eta
 -\frac{\tau(\kappa-2)}\kappa.
\]
Combine this with $E(m)\ge\mathcal B_\kappa(m)+\int\rho$ from
\eqref{eq:base-deficit}.
\end{proof}

\begin{lemma}[Cubic conjugate]\label{lem:conjugate}
For $\nu\ge0$ and $y\in\R$,
\[
 \mathcal C(\nu,y):=\max_{n\ge0}\bigl\{y(n^2-\nu^2)-\rho(n,\nu)\bigr\}
 =-y\nu^2-\frac{\nu^3}2+\frac4{27}\Bigl(y+\frac{3\nu}2\Bigr)_+^3 .
\]
For $\nu>0$, $\mathcal C(\nu,y)=\nu^3\psi(y/\nu)$ with
$\psi(x)=-x-\frac12+\frac4{27}(x+\frac32)_+^3$, which equals
$\frac23x^2+\frac4{27}x^3$ for $x\ge-\frac32$. The function $\psi$ is
twice continuously differentiable and convex.
\end{lemma}

\begin{proof}[Proof of \Cref{lem:conjugate}]\label{proof:lem:conjugate}
Expanding \eqref{eq:cubic-loss} gives
$y(n^2-\nu^2)-\rho=-n^3+cn^2-y\nu^2-\nu^3/2$, where
$c=y+3\nu/2$. Its derivative is $n(2c-3n)$, so the maximum on
$[0,\infty)$ occurs at $n=2c_+/3$. Substitution gives the stated formula.
For $\nu>0$, divide by $\nu^3$ and set $x=y/\nu$. Expansion on
$x\ge-3/2$ gives the polynomial form, while
$\psi''(x)=\tfrac89(x+\tfrac32)_+$ is continuous and nonnegative.
\end{proof}

\begin{proof}[Proof of \Cref{thm:master}]\label{proof:thm:master}
Since $d_\kappa=\kappa\nu_\kappa^2-Q$,
$\frac{\eta}\kappa(\kappa n^2-Q)=\eta(n^2-\nu_\kappa^2)+\frac{\eta d_\kappa}\kappa$, and
$\mathsf b_\Psi n^2=\mathsf b_\Psi(n^2-\nu_\kappa^2)+\mathsf b_\Psi\nu_\kappa^2$. Hence
\[
 \mathscr R_\kappa(Y,n)=-\bigl[y(n^2-\nu_\kappa^2)-\rho(n,\nu_\kappa)\bigr]
 +\frac{\eta d_\kappa+\tau Q}\kappa+\mathsf b_\Psi\nu_\kappa^2-\mathsf d_\Psi,\qquad
 y=a-\eta-\mathsf b_\Psi ,
\]
and $\inf_{n\ge0}\mathscr R_\kappa=-\mathcal C(\nu_\kappa,y)+(\eta d_\kappa+\tau Q)/\kappa
+\mathsf b_\Psi\nu_\kappa^2-\mathsf d_\Psi$. The replacement changes $\mathsf b_\Psi$ to
$\mathsf b_\Psi-\kappa F'(q)$ and $\mathsf d_\Psi$ to $\mathsf d_\Psi-Q F'(q)$, because the
first components of $\mathsf b$ and $\mathsf d$ are $\kappa$ and $Q$. Thus $y$ is
unchanged, the right side changes by
$F'(q)(d_\kappa-\kappa\nu_\kappa^2+Q)=0$, and $\mathcal P_\kappa$ changes by
$F(2)-F(0)=0$. For the last claim,
$(\eta d_\kappa+\tau Q)/\kappa=\nu_\kappa ed_\kappa/\kappa$, and
\[
 \int\bar m_\kappa=\int_0^2\frac{\nu_\kappa^2}{d_\kappa}\dd q
 =\frac1\kappa\int_0^2\frac{d_\kappa+Q}{d_\kappa}\dd q=\frac1\kappa\Bigl(2+\int_0^2\frac Q{d_\kappa}\dd q\Bigr)
\]
gives $\frac1\kappa\int_0^2\eta+\frac{\tau(\kappa-2)}\kappa
=\frac1\kappa\int_0^2\nu_\kappa e-\frac\tau\kappa\int_0^2\frac Q{d_\kappa}+\frac{\tau(\kappa-2)}\kappa
=\frac1\kappa\int_0^2\nu_\kappa e+\tau\bigl(1-\int\bar m_\kappa\bigr)$; here $Q/d_\kappa\le\nu_\kappa/(2b)$
is integrable (\Cref{lem:local-profile}).
\end{proof}

\subsection{Uniform transfer}\label{sec:transport}

\begin{theorem}[Transfer of an admissible certificate]
\label{thm:certificate-transfer}
Let $\Theta$ be a dual certificate admissible at $k$. There is
$\delta_*>0$ such that $\Theta$ is admissible at every $\kappa$ with
$|\kappa-k|<\delta_*$, and
\[
 \sup_{J\in\overline{\mathcal Y}_N}\bigl|\ell_{\kappa(m_J)}(\Theta;m_J)-\ell_*(\Theta)\bigr|
 \longrightarrow0\qquad(N\to\infty).
\]
\end{theorem}

\begin{proof}\label{proof:thm:certificate-transfer}
{Admissibility near $k$.} By \Cref{lem:regular-criterion}(ii), and
since $e_{\mathsf C}$ and $\psi$ are continuous, for each side, cell $\mathsf C$ and end
$\sigma$, the bracket in \eqref{eq:witness} is a continuous function on the
compact set $\mathsf C\times[0,1]^2\times[k/2,2k]$, hence uniformly continuous.
Its maximum over $\mathsf C\times[0,1]^2$ at $\kappa=k$ is negative, so there is
$\delta_{\mathsf C,\sigma}>0$ such that the maximum is negative for
$|\kappa-k|<\delta_{\mathsf C,\sigma}$. There are finitely many pairs
$(\mathsf C,\sigma)$; $\delta_*$ is the smallest of these numbers and $k/2$.

{The value.} Write $\kappa_J=\kappa(m_J)$. By \Cref{thm:transfer},
$\sup_J|\kappa_J-k|\le4\delta_N/\Lambda_0^2\to0$, and every $m_J$ has
$\int m_J=1$. Compare the terms of \eqref{eq:dual-value} at
$(\kappa_J,m_J)$ with those of \eqref{eq:dual-target-value}.
\begin{enumerate}
\item By \Cref{thm:transfer}, $|H_{m_J}(s_i)-H_*(s_i)|\le\frac{s_i}2\delta_N$ and
$|S(m_J)-\frac{k^2}6|=|\frac{\kappa_J^2}4\tr K_{m_J}^2-\frac{k^2}4\cdot\frac23|
\le\frac{|\kappa_J^2-k^2|}4+\frac{k^2}2\delta_N$, using $\tr K_{m_J}^2\le1$.
\item In $\mathcal B_{\kappa_J}(m_J)$: $\kappa\mapsto\int_0^2\nu_\kappa\dd q$ is continuous
by joint continuity and boundedness of $\nu_\kappa$ (\Cref{lem:local-profile});
$\mathcal J_\varsigma(2)$ does not depend on $\kappa$; and by
\eqref{eq:base-observable}, $S_\varsigma(m_J)=\kappa_J/\varsigma-2H_{m_J}(\varsigma\kappa_J)/\varsigma^2$,
where
$|H_{m_J}(\varsigma\kappa_J)-H_*(\varsigma k)|\le\frac\varsigma2|\kappa_J-k|+\frac{\varsigma k}2\delta_N$,
because $0\le\frac{\dd}{\dd s}H_m(s)=\sum_j\lambda_j/(2+s\lambda_j)\le\frac12\int m$.
\item The price \eqref{eq:closed-price-main} is continuous in $\kappa$: by
\Cref{lem:regular-criterion}(ii), $\hat\nu_\kappa(t)$ is jointly continuous on
$[0,1]\times[k/2,2k]$, hence uniformly continuous; and
$\int\bar m_\kappa=\int_0^2\nu_\kappa^2/d_\kappa\dd q$ is continuous by dominated
convergence with the majorant $3/k$ for $\kappa\ge k/2$.
\item $H_{\bar m_\kappa}(s)$ and $S(\bar m_\kappa)$ are continuous in $\kappa$. The left
filters of $\bar m_\kappa$ vanish to the left of its support; on the support
$\dd x=\dd q/d_\kappa$; and to the right of it they solve $r'=-2r-r^2$, respectively
$r'=-2r$, whose solutions with initial value $r_0$ have
$\int_0^\infty r=\log(1+r_0/2)$, respectively $\int_0^\infty r^2=r_0^2/4$. By
\Cref{lem:filter}, with $\kappa(\bar m_\kappa)=\kappa$,
\[
 H_{\bar m_\kappa}(s)=\int_0^2\frac{\bar r_s}{d_\kappa}\dd q+\log\Bigl(1+\frac{\bar r_s(2)}2\Bigr),
 \qquad
 S(\bar m_\kappa)=\int_0^2\frac{\bar r_\circ^2}{d_\kappa}\dd q+\frac{\bar r_\circ(2)^2}4 .
\]
By \Cref{lem:filter-continuity} and its proof, $\bar r_s$ and $\bar r_\circ$ are
bounded on $[0,2]\times[k/2,2k]$ and continuous in $\kappa$ uniformly in $q$, and
by \Cref{lem:local-profile} $d_\kappa^{-1}$ is continuous into $L^1(0,2)$; hence
both quantities are continuous in $\kappa$.
\end{enumerate}
The value is a continuous function of $\kappa$ and of these quantities, with
coefficients fixed by $\Theta$; since $\sup_J|\kappa_J-k|\to0$ and $\delta_N\to0$, the
differences tend to zero uniformly in $J$.
\end{proof}

\begin{proof}[Proof of \Cref{thm:dual}]\label{proof:thm:dual}
We first show, for every $\kappa>0$ at which $\Theta$ is admissible, that
$E(m)\ge\ell_\kappa(\Theta;m)$ for all $m\in\mathcal M^1_c$ with $\kappa(m)=\kappa$. Let
$Y\in\Omega_\kappa$ with $0<q<2$. There
are a side and $t\in(0,1]$ with $q=q_{\mathrm s}(t)$, a cell $\mathsf C\ni t$, and,
by \Cref{lem:regular-criterion}(iii), $(U,T)\in[0,1]^2$ and
$\vartheta\in[\vartheta_-,\vartheta_+]$ such that $Y$ has the normalized coordinates
$(U,V,\vartheta)$. By \eqref{eq:witness},
$\psi(\mathfrak b_\sigma-e)-\mathfrak c_\sigma-e(1-\mathfrak l)<0$ at $(t,U,T)$ for both
$\sigma$; by convexity in $\vartheta$ (\Cref{lem:regular-criterion}(iii)),
\eqref{eq:closed-criterion} holds at $\vartheta$, hence \eqref{eq:master} holds at
$Y$, and by \Cref{thm:master} $\mathscr R_\kappa(Y,n)\ge0$ for all $n\ge0$. At
$q\in\{0,2\}$, $\mathscr R_\kappa(Y,0)=0$ by \Cref{lem:regular-criterion}(iv), and
$\Psi\in C^1(\R^4)$ with $\Psi(Y_-)=\Psi(Y_+)$. By
\Cref{prop:certificate-mechanism} and \eqref{eq:centered-price}, every
$m\in\mathcal M^{1}_c$ with $\kappa(m)=\kappa$ satisfies $E(m)\ge\ell_\kappa(\Theta;m)$.
For $m\in\mathcal A_{\mathrm{obs}}(\varsigma)$ one has $\kappa(m)=k$ and
$\mathbf o_\varsigma(m)=\mathbf o_\varsigma^*$, so
$S_\varsigma(m)=k/\varsigma-2H_*(\varsigma k)/\varsigma^2$,
$\mathcal B_k(m)=\mathcal B_*$, $S(m)=k^2/6$, and $\ell_k(\Theta;m)=\ell_*(\Theta)$.
This proves (i).

By \Cref{thm:certificate-transfer}, there is $\delta_*>0$ such that
$\Theta$ is admissible at every $\kappa$ with $|\kappa-k|<\delta_*$, and
$\ell_{\kappa(m_J)}(\Theta;m_J)\to\ell_*(\Theta)$ uniformly over
$J\in\overline{\mathcal Y}_N$ as $N\to\infty$, where $m_J$ is the packed density of
\Cref{thm:packing}. By \Cref{thm:transfer}, $|\kappa(m_J)-k|<\delta_*$
for all large $N$, and then
$N^{-3}\E_N(J)=E(m_J)\ge\ell_{\kappa(m_J)}(\Theta;m_J)$. This proves the
first claim of (ii); the second follows from the definition of
$\underline{\mathsf e}$, and the third from \Cref{thm:gain}.
\end{proof}
\section{The normalized criterion and the verification}\label{app:verification}

The certificate inequality must remain regular at the endpoints before
interval arithmetic can verify it on closed boxes. We first derive that
regular form and prove its continuity, including the cancellations at
$t=0$. The enclosure lemmas then specify the error bounds for roots,
flows, quadrature, and box maxima. Applying those bounds to the fixed
certificate data proves the computed propositions and the main theorems.

\subsection{The normalized criterion}\label{sec:normalized-criterion}

The criterion \eqref{eq:master} is divided by $\nu_\kappa^3$ and written on the two
sides in the variable $t$. All functions of $t$ that occur are given by
cubic roots and by ordinary differential equations (items (R), (G),
(C), (F) below), and \Cref{lem:regular-criterion} identifies them with the
objects of \Cref{sec:mechanism}, proves their regularity on closed
intervals, and gives the normalized criterion.

On each side $\mathrm s\in\{\mathrm L,\mathrm R\}$ write $q=q_{\mathrm s}(t)$
and $\mathscr D P=P/q$ for $\mathrm s=\mathrm L$, $\mathscr DP=P/(2-q)$ for
$\mathrm s=\mathrm R$, for polynomials $P$ divisible by the indicated
factor. Thus $\mathscr DP=P/t^3$ for $P$ divisible by $t^3$ as a function of
$t$. Put
\[
 \nu_\kappa=t\hat\nu,\qquad d_\kappa=t^2\mathfrak d,\qquad
 \mathfrak l=\frac{t(2-t^3)}{\kappa\hat\nu^2},\qquad
 R_{s_i}=(2-q)\widehat R_i,\qquad R_\circ=(2-q)\widehat R_\circ ,
\]
where $R_{s_i},R_\circ$ are the right filters of $\bar m_\kappa$, and let
$\bar U_i=\kappa\bar u_i/q$ and $\bar\vartheta=\bar h/q$ be the normalized
states of the comparison trajectory $\bar Y(q)$. For $(U,T)\in[0,1]^2$ let
$V,y_\pm,\vartheta_\pm$ be the functions of \Cref{lem:range}, and evaluate the
state at $(U_0,U_1)=(U,V)$, $u_i=qU_i/\kappa$ and $h=q\vartheta$ with
$\vartheta=\vartheta_\sigma$, $\sigma\in\{-,+\}$.

Write $\Psi_{\rm raw}$ as a sum of terms $P(q)h$ and $P(q)u_0^iu_1^j$
($i+j\ge1$), where $(2-q)^2$ divides every coefficient $P$, and $q$
divides $P$ when $i+j=1$; in \eqref{eq:transport-potential} the
coefficients of $h,u_0,u_1,u_0^2,u_0u_1,u_1^2$ are $(2-q)^2P_h$,
$q(2-q)^2P_0$, $q(2-q)^2P_1$, $(2-q)^2P_{00}$, $(2-q)^2P_{01}$, $(2-q)^2P_{11}$.
For a term $P(q)h$ put
\begin{equation}\label{eq:regular-h}
 b_P=-\frac{\kappa t^2}{\hat\nu}\mathscr D(P'q)\vartheta,\qquad
 c_P=\frac{t^2\mathscr D(P'q)\mathfrak d\vartheta+\mathscr D(Pq)(q-2\vartheta)}{\hat\nu^3}.
\end{equation}
For a term $P(q)u_0^iu_1^j$ put $l=i+j$, $\upsilon_0=U_0/\kappa$,
$\upsilon_1=U_1/\kappa$,
$\mathsf m=\upsilon_0^i\upsilon_1^j$,
$\mathsf m_1=\partial_{\upsilon_0}\mathsf m+\partial_{\upsilon_1}\mathsf m$ and
$\mathsf m_2=s_0\upsilon_0^2\partial_{\upsilon_0}\mathsf m+s_1\upsilon_1^2\partial_{\upsilon_1}\mathsf m$, and
\begin{equation}\label{eq:regular-u}
\begin{aligned}
 b_P&=-\frac{t^2}{\hat\nu}\bigl[\kappa\mathscr D(P'q^l)\mathsf m
 +\mathscr D(Pq^{l-1})\mathsf m_1\bigr],\\
 c_P&=\frac1{\hat\nu^3}\Bigl\{t^2\bigl[\mathscr D(P'q^l)\mathfrak d\mathsf m
 +\mathscr D(Pq^{l-1})\mathsf m_1\hat\nu^2\bigr]
 -2l\,\mathscr D(Pq^l)\mathsf m-\mathscr D(Pq^{l+1})\mathsf m_2\Bigr\}.
\end{aligned}
\end{equation}
Let $\bar c_P$ be $c_P$ evaluated at the comparison state
$(\bar U_0,\bar U_1,\bar\vartheta)$, and
$b_{P,c}=b_P+\kappa\hat\nu^2\bar c_P/\mathfrak d$,
$c_{P,c}=c_P-\bar c_P$. For multipliers $\xi_0,\xi_1,\chi,\tau$ define
\[
\begin{aligned}
 \mathfrak b_\sigma={}&\frac{\tau(2-t^3)}{\mathfrak d\hat\nu}
 -\sum_{i=0}^1\frac{\xi_is_i(2-t^3)(1+U_i\widehat R_i)}
 {\mathfrak d\hat\nu\,[2+s_iqU_i/\kappa+(2-q)\widehat R_i]}\\
 &+\frac{\chi\kappa(2-t^3)}{2\mathfrak d\hat\nu}
 \bigl[2(1-\vartheta)\widehat R_\circ-q(1-\vartheta)-(2-q)\widehat R_\circ\bigr]
 +\sum_Pb_{P,c},\qquad
 \mathfrak c_\sigma=\sum_Pc_{P,c},
\end{aligned}
\]
the sums running over the terms of $\Psi_{\rm raw}$, at
$\vartheta=\vartheta_\sigma$.

{Regrouped form.} For a term $P(q)u_0^iu_1^j$ with $l=i+j$ put
\[
\begin{gathered}
 b^W_P=-\frac{\kappa t^2\mathscr D(P'q^l)}{\hat\nu},\qquad
 b^1_P=-\frac{t^2\mathscr D(Pq^{l-1})}{\hat\nu},\qquad
 c^W_P=\frac{t^2\mathscr D(P'q^l)\mathfrak d-2l\mathscr D(Pq^l)}{\hat\nu^3},\\
 c^1_P=\frac{t^2\mathscr D(Pq^{l-1})}{\hat\nu},\qquad
 c^2_P=-\frac{\mathscr D(Pq^{l+1})}{\hat\nu^3};
\end{gathered}
\]
then $b_P=b^W_P\mathsf m+b^1_P\mathsf m_1$ and
$c_P=c^W_P\mathsf m+c^1_P\mathsf m_1+c^2_P\mathsf m_2$ by \eqref{eq:regular-u}. For a term
$P(q)h$ put $b^W_P=-\kappa t^2\mathscr D(P'q)/\hat\nu$,
$c^W_P=(t^2\mathscr D(P'q)\mathfrak d-2\mathscr D(Pq))/\hat\nu^3$ and $c^0_P=q\mathscr D(Pq)/\hat\nu^3$;
then $b_P=b^W_P\vartheta$ and $c_P=c^W_P\vartheta+c^0_P$ by \eqref{eq:regular-h}. With
$\bar{\mathsf m},\bar{\mathsf m}_1,\bar{\mathsf m}_2$ the values of $\mathsf m,\mathsf m_1,\mathsf m_2$
at $(\bar U_0,\bar U_1)$,
\[
 c_{P,c}=c^W_P(\mathsf m-\bar{\mathsf m})+c^1_P(\mathsf m_1-\bar{\mathsf m}_1)+c^2_P(\mathsf m_2-\bar{\mathsf m}_2),
 \qquad\text{respectively}\qquad c_{P,c}=c^W_P(\vartheta-\bar\vartheta).
\]
We call the functions of $t$ alone that occur here and in
$\mathfrak b_\sigma$, $\mathfrak c_\sigma$, $\mathfrak l$ the {$t$-coefficients}: $q$, $\hat\nu$, $\mathfrak d$,
$\mathfrak l$, $\bar U_i$, $\bar\vartheta$, $\widehat R_i$, $\widehat R_\circ$, $(2-t^3)/(\mathfrak d\hat\nu)$, and,
for each term, those of $b^W_P$, $b^1_P$, $c^W_P$, $c^1_P$, $c^2_P$, $\bar{\mathsf m}$,
$\bar{\mathsf m}_1$, $\bar{\mathsf m}_2$ defined for it, and $\kappa\hat\nu^2\bar c_P/\mathfrak d$.

The following functions of $t\in[0,1]$ describe the reference and
comparison densities on the left side; by the symmetry of both densities,
the right side is described by the same functions. Let $h_t=2-t^3$ and
$Q_t=t^3h_t$.
\begin{enumerate}
\item[(R)] $z_\circ(t)$ is the positive root of
$z^3-\frac{3th_t}kz-\frac{2b_\circ(Q_t)h_t}k=0$, and
$\mathfrak p_\circ=3/(2th_t+2b_\circ(Q_t)h_t/z_\circ)$. On $v\in[0,2]$, with $t=v$ for $v\le1$ and
$t=2-v$ for $v\ge1$, $\varphi_\circ$ solves
$\frac{\dd\varphi_\circ}{\dd v}=\mathfrak p_\circ(kt^2z_\circ^2-2\varphi_\circ-\varsigma\varphi_\circ^2)$, $\varphi_\circ(0)=0$; and on
$[0,1]$, $\epsilon_\circ$ solves
$\frac{\dd\epsilon_\circ}{\dd t}=\mathfrak p_\circ[(1+\varsigma)t^6-2(1+\varsigma t^3)\epsilon_\circ+\varsigma\epsilon_\circ^2]$,
$\epsilon_\circ(0)=0$.
\item[(G)] $\displaystyle\gamma(t)=\frac{t^3(h_t-\varphi_\circ(2-t))+(\epsilon_\circ(t)/t^3)\bigl(2\varphi_\circ(2-t)-Q_t\bigr)}
{h_t\bigl[2+\varsigma(\varphi_\circ(t)+\varphi_\circ(2-t))\bigr]}$.
\item[(C)] $\hat\nu(t)$ is the positive root of
$z^3-\frac{3th_t}\kappa z-\frac{2(b+\kappa\omega\gamma(t))h_t}\kappa=0$,
$\mathfrak d=2th_t+2(b+\kappa\omega\gamma)h_t/\hat\nu$, and $\mathfrak p=3/\mathfrak d$.
\item[(F)] For $s>0$, on $v\in[0,2]$, $\varrho_s$ solves
$\frac{\dd\varrho_s}{\dd v}=\mathfrak p(st^2\hat\nu^2-2\varrho_s-\varrho_s^2)$ and $\varrho_\circ$ solves
$\frac{\dd\varrho_\circ}{\dd v}=\mathfrak p(\kappa t^2\hat\nu^2-2\varrho_\circ)$, with zero initial values; on
$[0,1]$, $w_s$ solves
$\frac{\dd w_s}{\dd t}=\mathfrak p[(1+\frac s\kappa)t^6-2(1+\frac{st^3}\kappa)w_s+\frac s\kappa w_s^2]$ and
$w_\circ$ solves $\frac{\dd w_\circ}{\dd t}=\mathfrak p(t^6-2w_\circ)$, with zero initial values.
\end{enumerate}

\begin{lemma}[The normalized criterion]\label{lem:regular-criterion}
Let $\Theta$ be a dual certificate and $\kappa>0$.
\begin{enumerate}
\item With $v=t$ on the left side, $f_\circ=\varphi_\circ$, $q_\circ-f_\circ=h_c=\epsilon_\circ$
and $\check f_\circ(t)=\varphi_\circ(2-t)$ at $q_\circ=t^3$; $\Gamma_\varsigma(q)/Q(q)=\gamma(t)$ at
$q=t^3$; $\nu_\kappa(t^3)=t\hat\nu(t)$, $d_\kappa(t^3)=t^2\mathfrak d(t)$, and
$\frac{\dd x}{\dd t}=\mathfrak p$ along $\bar m_\kappa$; the left filters of $\bar m_\kappa$ at
$q=t^3$ are $\bar r_s=\varrho_s(t)=\frac s\kappa(t^3-w_s(t))$ and
$\bar r_\circ=\varrho_\circ(t)=t^3-w_\circ(t)$, and at $q=2-t^3$ they are $\varrho_s(2-t)$ and
$\varrho_\circ(2-t)$. Consequently, on the left side
\[
 \bar U_i=1-\frac{w_{s_i}}{t^3},\quad \bar\vartheta=\frac{w_\circ}{t^3},\quad
 \widehat R_i=\frac{\varrho_{s_i}(2-t)}{h_t},\quad \widehat R_\circ=\frac{\varrho_\circ(2-t)}{h_t},
\]
and on the right side
\[
 \bar U_i=\frac{\kappa\varrho_{s_i}(2-t)}{s_ih_t},\quad \bar\vartheta=1-\frac{\varrho_\circ(2-t)}{h_t},\quad
 \widehat R_i=\frac{s_i}\kappa\Bigl(1-\frac{w_{s_i}}{t^3}\Bigr),\quad \widehat R_\circ=1-\frac{w_\circ}{t^3}.
\]
\item The functions $z_\circ,\mathfrak p_\circ,\gamma,\hat\nu,\mathfrak d,\mathfrak p$ are $C^\infty$ on
$[0,1]$, and the restrictions of $\varphi_\circ,\varrho_s,\varrho_\circ$ to $[0,1]$ and to
$[1,2]$ are $C^\infty$ on these closed intervals; $\epsilon_\circ,w_s,w_\circ$ vanish to order seven at
$t=0$, and $\epsilon_\circ/t^3$, $w_s/t^3$, $w_\circ/t^3$ are $C^\infty$ on $[0,1]$. The
functions $\hat\nu,\mathfrak d,\widehat R_i,\widehat R_\circ,\bar U_i,\bar\vartheta$ are jointly
continuous in $(t,\kappa)\in[0,1]\times(0,\infty)$, $\hat\nu>0$, $\mathfrak d>0$,
$\bar U_i,\bar\vartheta,\widehat R_\circ\in[0,1]$, $\widehat R_i\in[0,s_i/\kappa]$, and
$\varphi_\circ,\epsilon_\circ,\gamma\ge0$. All quotients
$\mathscr D(\cdot)$ above are polynomials. Consequently $\mathfrak b_\sigma$, $\mathfrak c_\sigma$ and
$\mathfrak l$ are continuous in $(t,U,T,\kappa)\in[0,1]^3\times(0,\infty)$.
\item For $0<t\le1$, the normalized states in $\Omega_\kappa$ over $q=q_{\mathrm s}(t)$
are exactly the points $(U,V,\vartheta)$ with $(U,T)\in[0,1]^2$ and
$\vartheta\in[\vartheta_-,\vartheta_+]$. With
$\Psi=\Psi_{\rm raw}-\Psi_{\rm raw}(\bar Y(q))$ and $\eta=-\tau Q/d_\kappa+\nu_\kappa e$,
the criterion \eqref{eq:master} at such a state, divided by $\nu_\kappa^3$, reads
\begin{equation}\label{eq:closed-criterion}
 \psi\bigl(\mathfrak b-e\bigr)-\mathfrak c-e(1-\mathfrak l)\le0,
\end{equation}
where $\mathfrak b,\mathfrak c$ are the expressions above at the given $\vartheta$. Both
are affine in $\vartheta$, so the maximum of the left side over
$\vartheta\in[\vartheta_-,\vartheta_+]$ is attained at $\vartheta_\pm$.
\item The function $F(q)=\Psi_{\rm raw}(\bar Y(q))$ belongs to $C^1[0,2]$ with
$F(0)=F(2)=F'(0)=F'(2)=0$; extended by zero it belongs to $C^1(\R)$, and
$\Psi\in C^1(\R^4)$ with $\Psi(Y_-)=\Psi(Y_+)=0$. At $q\in\{0,2\}$,
$\mathscr R_\kappa(Y,0)=0$ for all $Y\in\Omega_\kappa$.
\end{enumerate}
\end{lemma}

\begin{proof}\label{proof:lem:regular-criterion}
{The reference density.} At $q=t^3$, $Q=Q_t$, and dividing
$k\nu^3-3Q\nu-2b_\circ(Q)Q=0$ by $kt^3$ shows that $\nu_\circ(t^3)=tz_\circ(t)$. By
\Cref{lem:jets}, $3z_\circ^2-\frac{3th_t}k=\frac{6th_t}k+\frac{6b_\circ h_t}{kz_\circ}>0$ at the root,
and the coefficients are $C^\infty$ in $t\in[0,1]$, with
$2b_\circ(Q_t)h_t/k>0$; the implicit function theorem shows that $z_\circ$ is
$C^\infty$ and positive on $[0,1]$. Then
$d_\circ=k\nu_\circ^2-Q=t^2(kz_\circ^2-th_t)=t^2(2th_t+2b_\circ h_t/z_\circ)$ by the cubic, and
along $m_\circ$, $\frac{\dd x}{\dd t}=\frac{3t^2}{d_\circ}=\mathfrak p_\circ$. By
\Cref{thm:contact}, $m_\circ$ and $q_\circ$ are symmetric; in the coordinate $v$,
with $x$ the corresponding point of $[0,X_\circ(2)]$, $\frac{\dd x}{\dd v}=\mathfrak p_\circ$ on both
halves, because $\frac{\dd q}{\dd v}=3t^2$ there. The filter equation of $f_\circ$
in the variable $v$ is the equation for $\varphi_\circ$, with the initial value
$f_\circ(0)=0$; so $f_\circ=\varphi_\circ$, and $\check f_\circ$ at $v$ is $\varphi_\circ$ at $2-v$.
The coefficient $\mathfrak p_\circ(kt^2z_\circ^2-2\varphi-\varsigma\varphi^2)$ is $C^\infty$ on each closed
half, so $\varphi_\circ$ is $C^\infty$ on $[0,1]$ and on $[1,2]$. The function
$h_c=q_\circ-f_\circ$ satisfies, by $q_\circ'=km_\circ-Q$ and the filter equation,
$h_c'=(1+\varsigma)q_\circ^2-2(1+\varsigma q_\circ)h_c+\varsigma h_c^2$ (the computation of
\Cref{lem:filter-continuity} with $s/\kappa=\varsigma$), which in the variable $t$ is
the equation for $\epsilon_\circ$; so $h_c=\epsilon_\circ$. Write
$\epsilon_\circ'=\alpha(t)\epsilon_\circ+\mathfrak p_\circ(1+\varsigma)t^6$ with
$\alpha=\mathfrak p_\circ(-2-2\varsigma t^3+\varsigma\epsilon_\circ)$ bounded; variation of constants
from $\epsilon_\circ(0)=0$ gives $|\epsilon_\circ(t)|\le Ct^7$, so
$\epsilon_\circ(0)=\epsilon_\circ'(0)=\epsilon_\circ''(0)=0$ and $\epsilon_\circ/t^3$ is $C^\infty$ by
\Cref{lem:cubic-quotient}. By \eqref{eq:contact} and $f_\circ=q-h_c$,
\[
 \frac{Q(f_\circ+\check f_\circ)-2f_\circ\check f_\circ}{Q}
 =q+\check f_\circ-\frac{2\check f_\circ}{2-q}+\frac{h_c(2\check f_\circ-Q)}{Q}
 =\frac{q(2-q-\check f_\circ)+(h_c/q)(2\check f_\circ-Q)}{2-q},
\]
which divided by $2+\varsigma(f_\circ+\check f_\circ)$ is $\gamma$; so
$\Gamma_\varsigma(t^3)/Q_t=\gamma(t)$, a $C^\infty$ function on $[0,1]$, independent of
$\kappa$, with $\gamma\ge0$.

{The comparison density.} Dividing \eqref{eq:base-root} by
$\kappa t^3$ shows $\nu_\kappa(t^3)=t\hat\nu(t)$; as for $z_\circ$, $\hat\nu$ is $C^\infty$
and positive on $[0,1]$, and jointly continuous in $(t,\kappa)$ by the
argument in the proof of \Cref{lem:base-identity} applied to the cubic in
$z$, whose constant term $2(b+\kappa\omega\gamma)h_t/\kappa$ is positive. The cubic gives
$\kappa\hat\nu^2-th_t=2th_t+2(b+\kappa\omega\gamma)h_t/\hat\nu=\mathfrak d>0$ and
$d_\kappa=t^2\mathfrak d$. By \Cref{lem:local-profile}, $\bar m_\kappa$ is symmetric,
and $\frac{\dd x}{\dd v}=\mathfrak p$ on both halves as for $m_\circ$. The filter
equations in $v$ are those for $\varrho_s,\varrho_\circ$, so $\bar r_s=\varrho_s$ and
$\bar r_\circ=\varrho_\circ$, and these are $C^\infty$ on each closed half. The
functions $w_s=t^3-\frac\kappa s\varrho_s$ and $w_\circ=t^3-\varrho_\circ$ satisfy the equations
of (F) by the computation in the proof of \Cref{lem:filter-continuity}, and
vanish to order seven by the argument used for $\epsilon_\circ$; hence $w_s/t^3$ and
$w_\circ/t^3$ are $C^\infty$ on $[0,1]$. By \Cref{lem:filter-continuity},
$R_s(q)=\bar r_s(2-q)$ and $R_\circ(q)=\bar r_\circ(2-q)$; at $q=t^3$ the point $2-q$
corresponds to $v=2-t$. This gives (i) and the smoothness in (ii).

{Ranges and continuity.} By \Cref{lem:filter}, $f_\circ\ge0$ and
$h_c=q_\circ-f_\circ\ge0$, and $\gamma=\Gamma_\varsigma/Q\ge0$ by \Cref{lem:base-identity};
also $0\le\bar r_s\le\frac s\kappa q$
and $0\le\bar r_\circ\le q$, and by \Cref{lem:right-filters}, $0\le R_s\le\frac s\kappa(2-q)$ and
$0\le R_\circ\le2-q$; these give the ranges. Joint continuity in $(t,\kappa)$
follows from \Cref{lem:filter-continuity}: away from the end $t=0$ of a
side, the functions are continuous functions of $\bar r_s,\bar r_\circ$, $q$
and $1/q$, $1/(2-q)$; at $t=0$, the bounds $0\le w_s,w_\circ\le Ct^7$, uniform on
compact $\kappa$-sets, give the continuity of $1-w_s/t^3$ and $w_\circ/t^3$ with
the value $1$, respectively $0$, at $t=0$. Since $(2-q)^2$ divides $P$,
and $q$ divides $P$ when $l=1$, the polynomials $P'q$, $Pq$, $P'q^l$,
$Pq^{l-1}$, $Pq^l$, $Pq^{l+1}$ are divisible by $q$ and by $2-q$; so all
quotients $\mathscr D(\cdot)$ are polynomials. All denominators in
$\mathfrak b_\sigma,\mathfrak c_\sigma,\mathfrak l$ are positive and continuous: $\hat\nu$, $\mathfrak d$,
$2+s_iqU_i/\kappa+(2-q)\widehat R_i\ge2$, and those of $V,\vartheta_\pm$
(\Cref{lem:range}). This proves (ii).

{Coefficients of the potential.} For a term $\Psi_P$ of $\Psi_{\rm raw}$
write its derivative along a density as $\mathsf b_Pn^2-\mathsf d_P$ by
\eqref{eq:state-flow}, and put $b_P=-\mathsf b_P/\nu$ and
$c_P=(\mathsf b_P\nu^2-\mathsf d_P)/\nu^3$ with $\nu=\nu_\kappa(q)$. For $P(q)h$,
$(P(q)h)'=P'(\kappa n^2-Q)h+P(q^2-2h)$, so $\mathsf b_P=\kappa P'h$ and
$\mathsf b_P\nu^2-\mathsf d_P=P'hd_\kappa+P(q^2-2h)$. With $h=q\vartheta$, $\nu=t\hat\nu$,
$d_\kappa=t^2\mathfrak d$ and $q/t^3=\mathscr Dq$, this is \eqref{eq:regular-h}. For
$\Psi_P=P(q)\mathsf M$, $\mathsf M=u_0^iu_1^j$, with $u_l'=n^2-2u_l-s_lu_l^2$,
\[
 \mathsf b_P=\kappa P'\mathsf M+P(\partial_{u_0}\mathsf M+\partial_{u_1}\mathsf M),\qquad
 \mathsf d_P=P'Q\mathsf M+P\sum_{l'=0}^1\partial_{u_{l'}}\mathsf M\,(2u_{l'}+s_{l'}u_{l'}^2).
\]
With $u_{l'}=q\upsilon_{l'}$, $\mathsf M=q^l\mathsf m$, $\partial_{u_0}\mathsf M+\partial_{u_1}\mathsf M=q^{l-1}\mathsf m_1$
and, by Euler's identity $\sum\upsilon_{l'}\partial_{\upsilon_{l'}}\mathsf m=l\mathsf m$,
$\sum\partial_{u_{l'}}\mathsf M(2u_{l'}+s_{l'}u_{l'}^2)=2lq^l\mathsf m+q^{l+1}\mathsf m_2$. Hence
$\mathsf b_P\nu^2-\mathsf d_P=P'q^l\mathsf md_\kappa+Pq^{l-1}\mathsf m_1\nu^2-2lPq^l\mathsf m-Pq^{l+1}\mathsf m_2$,
and division by $\nu^3=t^3\hat\nu^3$ gives \eqref{eq:regular-u}.

{The spectral coefficients.} With $r_{s_i}=s_iqU_i/\kappa$,
$R_{s_i}=(2-q)\widehat R_i$, $r_\circ=q(1-\vartheta)$ and $R_\circ=(2-q)\widehat R_\circ$,
\[
 s_iQ+\kappa r_{s_i}R_{s_i}=s_iQ(1+U_i\widehat R_i),\qquad
 2r_\circ R_\circ-Q(r_\circ+R_\circ)=Q\bigl[2(1-\vartheta)\widehat R_\circ-q(1-\vartheta)-(2-q)\widehat R_\circ\bigr],
\]
and $Q/(d_\kappa\nu)=(2-t^3)/(\mathfrak d\hat\nu)$ on both sides, since $Q=t^3(2-t^3)$.
Hence, with $a=\xi_0a_{s_0}+\xi_1a_{s_1}+\chi a_\circ$ and
$\eta/\nu=-\tau(2-t^3)/(\mathfrak d\hat\nu)+e$, the quantity $(a-\eta)/\nu$ equals the
first three terms of $\mathfrak b_\sigma$ minus $e$.

{Centering.} The function $\bar Y(q)$ is absolutely continuous on
$[0,2]$, and $\frac{\dd\bar Y}{\dd q}=(\mathsf b\nu^2-\mathsf d)(\bar Y)/d_\kappa$ is
continuous on $(0,2)$. Hence $F\in C^1(0,2)$ with
\[
 F'(q)=\frac{\mathsf b_{\rm raw}(\bar Y)\nu^2-\mathsf d_{\rm raw}(\bar Y)}{d_\kappa}
 =\frac{\nu^3}{d_\kappa}\sum_P\bar c_P .
\]
By (ii), $\sum_P\bar c_P$ is bounded on each side, and
$\nu^3/d_\kappa=\nu\cdot\nu^2/d_\kappa\le\frac3{2\kappa}\nu\to0$ at $q\to0,2$; so $F'$
extends continuously by $0$, and $F\in C^1[0,2]$. At $Y_-=0$ every term of
$\Psi_{\rm raw}$ vanishes, and at $q=2$ the factor $(2-q)^2$ vanishes, so
$F(0)=F(2)=0$ and $\Psi(Y_\pm)=0$. For $\Psi=\Psi_{\rm raw}-F(q)$,
$\mathsf b_\Psi=\mathsf b_{\rm raw}-\kappa F'$ and $\mathsf d_\Psi=\mathsf d_{\rm raw}-QF'$, so
\[
 -\frac{\mathsf b_\Psi}\nu=\sum_P\Bigl(b_P+\frac{\kappa\hat\nu^2}{\mathfrak d}\bar c_P\Bigr),
 \qquad
 \frac{\mathsf b_\Psi\nu^2-\mathsf d_\Psi}{\nu^3}=\sum_Pc_P-\frac{F'd_\kappa}{\nu^3}=\sum_P(c_P-\bar c_P),
\]
using $\kappa\nu^2-Q=d_\kappa$ and $\nu^2/d_\kappa=\hat\nu^2/\mathfrak d$.

{The criterion.} By \Cref{lem:range}, the normalized states over
$q=q_{\mathrm s}(t)$ are the stated points. By the previous two steps,
$(a-\eta-\mathsf b_\Psi)/\nu=\mathfrak b-e$ and
\[
 \frac{\eta d_\kappa+\tau Q}{\kappa\nu^3}+\frac{\mathsf b_\Psi\nu^2-\mathsf d_\Psi}{\nu^3}
 =\frac{ed_\kappa}{\kappa\nu^2}+\mathfrak c=e\frac{\mathfrak d}{\kappa\hat\nu^2}+\mathfrak c=e(1-\mathfrak l)+\mathfrak c ,
\]
since $\mathfrak d=\kappa\hat\nu^2-t(2-t^3)$. For $0<q<2$, $\nu>0$, and
$\mathcal C(\nu,y)=\nu^3\psi(y/\nu)$ (\Cref{lem:conjugate}); dividing \eqref{eq:master}
by $\nu^3$ gives \eqref{eq:closed-criterion}. The functions $\vartheta\mapsto\mathfrak b$,
$\vartheta\mapsto\mathfrak c$ are affine, $\psi$ is convex, so the left side of
\eqref{eq:closed-criterion} is convex in $\vartheta$ and its maximum over an interval is
attained at an end point. This proves (iii).

{The ends.} At $q\in\{0,2\}$, $\nu_\kappa=Q=0$, $a=0$, and
$\mathscr R_\kappa(Y,0)=-\mathsf d_\Psi(Y)=-\nabla\Psi_{\rm raw}(Y)\cdot\mathsf d(Y)$. At $q=0$,
$Y=Y_-$ and $\mathsf d(Y_-)=0$. At $q=2$, the first component of $\mathsf d$ is
$Q(2)=0$, and the partial derivatives of $\Psi_{\rm raw}$ in $u_0,u_1,h$ contain
the factor $(2-q)^2$. Hence $\mathscr R_\kappa(Y,0)=0$. Finally $F'(0)=F'(2)=0$
and $F(0)=F(2)=0$, so the extension of $F$ by zero is $C^1$ on $\R$, and
$\Psi\in C^1(\R^4)$. This proves (iv).
\end{proof}

\begin{lemma}[Two cancellation identities]\label{lem:mobius}
For $q\in[0,2]$, $U\in[0,1]$, $\alpha>0$, $R\ge0$, $\delta_q=2-q$ and
$D=2+q\alpha U+\delta_qR$, one has $D\ge2$ and
\[
 \frac{1+UR}D=\frac12+\frac{U(R-\alpha)+\frac12\delta_q(\alpha U-R)}D ,
\]
and, for real $V,R_\circ$,
$2VR_\circ-qV-\delta_qR_\circ=2V(R_\circ-1)+\delta_q(V-R_\circ)$.
\end{lemma}

\begin{proof}\label{proof:lem:mobius}
$2(1+UR)-D=2UR-q\alpha U-\delta_qR$, and $q\alpha U=2\alpha U-\delta_q\alpha U$; this
gives the first identity. The second follows from $q=2-\delta_q$.
\end{proof}

\subsection{Enclosures}\label{sec:enclosures}

An {enclosure} of a real number $y$ is a closed interval
$[\alpha,\beta]\ni y$; an enclosure of a function $f$ on a set $S$ is an
interval containing $f(S)$. For $f\in C^N$ on an interval and $j\le N$
write $f_j=f^{(j)}/j!$ for the normalized derivatives; a {jet
enclosure} of order $N$ of $f$ on $S$ is a list of intervals $F_0,\dots,F_N$
with $f_j(t)\in F_j$ for all $t\in S$ and $j\le N$.

In interval arithmetic with outward rounding, each operation $+,-,\times,/$,
each monotone function ($\sqrt\cdot$, $\sqrt[3]\cdot$, $\exp$, $\log$) and the
functions $x^2$, $|x|$, $x_+$ are evaluated on intervals so that the result
contains the exact range \cite{Rump}; by induction on the number of
operations, evaluating an expression in enclosures of its arguments gives
an enclosure of its values, and the intersection of two enclosures of the
same quantity is an enclosure. All enclosures below are obtained in this
way, or by
\Cref{lem:cubic-root,lem:quadrature,lem:flow-step,lem:cubic-quotient,lem:box-bound}.

\begin{lemma}[Jet recurrences]\label{lem:jets}
Let $f,g,a,b,C,A,B\in C^N$ on an interval and $0\le j\le N$. Then
\[
 (fg)_j=\sum_{i=0}^jf_ig_{j-i},\qquad
 (1/f)_0=\frac1{f_0},\quad (1/f)_j=-\frac1{f_0}\sum_{i=1}^jf_i(1/f)_{j-i}\quad(f\ne0),
\]
\[
 (e^f)_0=e^{f_0},\quad j(e^f)_j=\sum_{i=1}^jif_i(e^f)_{j-i};\qquad
 z=\sqrt f>0:\quad z_0=\sqrt{f_0},\quad z_j=\frac{f_j-\sum_{i=1}^{j-1}z_iz_{j-i}}{2z_0}\ (j\ge1) .
\]
If $z>0$ solves $z^3-az-b=0$ with $a\ge0$, $b>0$, then $3z^2-a=2a+3b/z>0$,
$z\in C^N$ by the implicit function theorem, and at each point $t$, with
$s_l=\sum_{i=1}^{l-1}z_iz_{l-i}$ and the normalized derivatives $a_i,b_i,z_i$ at $t$,
\[
 z_j=\frac{b_j+\sum_{i=1}^ja_iz_{j-i}-3z_0s_j-\sum_{i=1}^{j-2}z_is_{j-i}}{2a_0+3b_0/z_0}
 \qquad(j\ge1).
\]
If $y'=C+Ay+By^2$, then
$(j+1)y_{j+1}=C_j+\sum_{i=0}^jA_iy_{j-i}+\sum_{i=0}^jB_i(y^2)_{j-i}$ for
$j<N$. Moreover, for every real $c$, with $\tilde C=C+Ac+Bc^2$,
$\tilde A=A+2Bc$, $\varrho_0=y_0-c$ and $\varrho_j=y_j$ ($j\ge1$),
$(j+1)y_{j+1}=\tilde C_j+\sum_{i=0}^j\tilde A_i\varrho_{j-i}+\sum_{i=0}^jB_i(\varrho^2)_{j-i}$,
where $(\varrho^2)_l=\sum_{i=0}^l\varrho_i\varrho_{l-i}$; in particular, for $c=y_0$, $\varrho_0=0$.
Evaluating these formulas in jet enclosures on a set $S$ gives jet
enclosures on $S$.
\end{lemma}

\begin{proof}\label{proof:lem:jets}
The formulas for products, reciprocals, exponentials and square roots
are the standard recurrences of Taylor arithmetic \cite[Chapter~13]{GriewankWalther}.
For the cubic, $z_0^3=a_0z_0+b_0$ gives $3z_0^2-a_0=2a_0+3b_0/z_0>0$, and
comparing the coefficients of order $j$ in $z^3=az+b$, with
$z^3=z_0^3+3z_0^2u+3z_0u^2+u^3$ for $u=z-z_0$, gives the formula, since the
coefficient of $u^2$ is $s_j$ and that of $u^3$ is $\sum_{i=1}^{j-2}z_is_{j-i}$. The
equation $y'=F$ gives $(j+1)y_{j+1}=F_j$, and $C+Ay+By^2=\tilde C+\tilde A(y-c)+B(y-c)^2$.
Each formula expresses the jet at $t$ through jets at the same $t$ (and $c$,
which may be taken equal to $y(t)$); so evaluation in enclosures valid on
$S$ gives enclosures valid on $S$.
\end{proof}

\begin{lemma}[The positive cubic root]\label{lem:cubic-root}
For $a\ge0$ and $b>0$ the equation $z^3-az-b=0$ has a unique positive root
$z(a,b)$, and $\max\{\sqrt a,\sqrt[3]b\}\le z(a,b)\le\sqrt a+\sqrt[3]b$. The
root increases in $a$ and in $b$. If $X$ is an interval containing
$z(a,b)$ with $3x^2-a>0$ on $X$ and $m\in X$, then
$z(a,b)\in m-(m^3-am-b)/(3X^2-a)$, the quotient being taken over $X$.
\end{lemma}

\begin{proof}\label{proof:lem:cubic-root}
The positive roots are the zeros of $z^2-a-b/z$, which increases strictly
from $-\infty$ to $+\infty$ on $(0,\infty)$. At the root, $z^3=az+b\ge b$ and
$z^2=a+b/z\ge a$. If $z>\sqrt a+\sqrt[3]b$, then
$z^2-a>2\sqrt a\sqrt[3]b+\sqrt[3]{b^2}\ge\sqrt[3]{b^2}$ and
$z(z^2-a)>\sqrt[3]b\cdot\sqrt[3]{b^2}=b$, which is impossible. Implicit
differentiation gives $\partial_az=z/(3z^2-a)>0$ and
$\partial_bz=1/(3z^2-a)>0$. Finally, by the mean value theorem,
$m^3-am-b=(3\xi^2-a)(m-z)$ for some $\xi$ between $m$ and $z$, so
$\xi\in X$ and $z=m-(m^3-am-b)/(3\xi^2-a)$.
\end{proof}

\begin{lemma}[Quadrature]\label{lem:quadrature}
Let $f\in C^N([c-h,c+h])$, $N\ge2$, and $M_N\ge\sup|f_N|$ on
$[c-h,c+h]$. Then
\[
 \Bigl|\int_{c-h}^{c+h}f(t)\dd t
 -\sum_{\substack{0\le j<N\\j\text{ even}}}\frac{2h^{j+1}}{j+1}f_j(c)\Bigr|
 \le\frac{2h^{N+1}}{N+1}M_N .
\]
\end{lemma}

\begin{proof}\label{proof:lem:quadrature}
By Taylor's formula with Lagrange remainder,
$f(t)=\sum_{j<N}f_j(c)(t-c)^j+f_N(\xi_t)(t-c)^N$ with $|f_N(\xi_t)|\le M_N$.
The integral of $(t-c)^j$ over $[c-h,c+h]$ is $2h^{j+1}/(j+1)$ for even
$j$ and zero for odd $j$, and $\int|t-c|^N=2h^{N+1}/(N+1)$.
\end{proof}

\begin{lemma}[A validated flow step]\label{lem:flow-step}
Let $F(t,y)=C(t)+A(t)y+B(t)y^2$ with $C,A,B\in C^{N+1}([a,a+h])$, $N\ge1$.
Let $Y_0$ be an interval, $m\in\R$, $r\ge\sup_{y\in Y_0}|y-m|$, and let
$Y$ be a closed interval such that
\begin{equation}\label{eq:picard}
 y_0+\theta F(t,y)\in Y\qquad\text{for all }y_0\in\operatorname{hull}(Y_0\cup\{m\}),\
 \theta\in[0,h],\ t\in[a,a+h],\ y\in Y .
\end{equation}
Then for every $y_0\in\operatorname{hull}(Y_0\cup\{m\})$ the solution of $y'=F(t,y)$,
$y(a)=y_0$, exists on $[a,a+h]$ and stays in $Y$. Let $z$ be the solution
with $z(a)=m$, $R\ge\sup_{[a,a+h]}|z_{N+1}|$ and
$L\ge\sup\{A(t)+2B(t)y:t\in[a,a+h],y\in Y\}$. Then every solution with
$y(a)\in Y_0$ satisfies, for $0\le v\le h$,
\[
 \Bigl|y(a+v)-\sum_{j=0}^Nz_j(a)v^j\Bigr|\le Rv^{N+1}+re^{Lv},\qquad y(a+v)\in Y .
\]
The number $R$ may be taken as an upper bound for the modulus of the
interval obtained from either recurrence of \Cref{lem:jets} with the jet
enclosures of $C,A,B$ on $[a,a+h]$ and with $y_0$ replaced by $Y$,
respectively with $\varrho_0=0$ and $c$ replaced by $Y$.
\end{lemma}

\begin{proof}\label{proof:lem:flow-step}
Let $\pi_Y$ be the nearest-point projection onto $Y$. The equation
$\tilde y'=F(t,\pi_Y(\tilde y))$ has a bounded globally Lipschitz
right-hand side and therefore a solution on $[a,a+h]$.
For fixed $t$, condition \eqref{eq:picard} places every value
$y_0+(t-a)F(\tau,\pi_Y(\tilde y(\tau)))$, $a\le\tau\le t$, in $Y$.
Their average is $\tilde y(t)$ and also belongs to the interval $Y$.
Thus the projection is inactive, and $\tilde y$ is the unique solution
of the original equation on the whole step.

Taylor's formula gives
$z(a+v)=\sum_{j\le N}z_j(a)v^j+z_{N+1}(\xi)v^{N+1}$ with $\xi\in[a,a+v]$. At every
$t\in[a,a+h]$ the normalized derivatives of $z$ satisfy the recurrence of
\Cref{lem:jets} with the normalized derivatives of $C,A,B$ at $t$ and
with $z(t)\in Y$ (in the shifted form with $c=z(t)\in Y$); by \Cref{lem:jets}
the interval recurrence with the enclosures on $[a,a+h]$ and $Y$ contains
$z_{N+1}(t)$, which justifies the
choice of $R$. For a solution $y$ with $y(a)\in Y_0$, the difference
$e=y-z$ satisfies $e'=F(t,y)-F(t,z)=[A+B(y+z)]e=[A+2B\tfrac{y+z}2]e$, and
$\frac{y+z}2\in Y$ by convexity, so $A+B(y+z)\le L$ and
$|e(a+v)|=|e(a)|\exp\int_a^{a+v}[A+B(y+z)]\le re^{Lv}$, for either sign of $L$.
\end{proof}

\begin{lemma}[Removing a cubic zero]\label{lem:cubic-quotient}
Let $f\in C^{N+4}([0,\beta])$ with $f(0)=f'(0)=f''(0)=0$, and let
$g(t)=f(t)/t^3$ for $t>0$, $g(0)=f_3(0)$. Then $g\in C^{N+1}([0,\beta])$,
\[
 g_j(t)=\frac{(j+3)!}{2\,j!}\int_0^1(1-v)^2v^jf_{j+3}(vt)\dd v\qquad(0\le j\le N+1),
\]
the weight $\frac{(j+3)!}{2j!}(1-v)^2v^j$ being nonnegative with integral one, and
for $0\le t\le\beta$ and $j\le N$
\[
 g_j(t)\in\sum_{l=j}^N\binom ljf_{l+3}(0)t^{l-j}
 +\binom{N+1}jt^{N+1-j}\,\mathrm{hull}\,f_{N+4}([0,\beta]) .
\]
\end{lemma}

\begin{proof}\label{proof:lem:cubic-quotient}
Taylor's formula with integral remainder at $0$ gives
$f(t)=\frac{t^3}2\int_0^1(1-v)^2f'''(vt)\dd v$, so
$g(t)=\frac12\int_0^1(1-v)^2f'''(vt)\dd v$, and differentiation under the
integral gives $g^{(j)}(t)=\frac12\int_0^1(1-v)^2v^jf^{(j+3)}(vt)\dd v$, which
is the formula, with $\int_0^1(1-v)^2v^j\dd v=2j!/(j+3)!$. Hence $g_l(0)=f_{l+3}(0)$
and $g_{N+1}(t)$ lies in the hull of the values of $f_{N+4}$ on $[0,t]$.
Taylor's formula for $g_j$ at $0$ of order $N-j$, whose remainder
coefficient is $\binom{N+1}jg_{N+1}(\xi)$, gives the last claim.
\end{proof}

\begin{lemma}[Polynomial translation]\label{lem:translation}
Let $p(q)=\sum_{j=0}^na_jq^j$ and $c\in\R$. Starting with $b_j=a_j$, for
$r=n-1,\dots,0$ replace $b_j$ by $b_j+cb_{j+1}$ for $j=r,\dots,n-1$ in
this order. Then $b_j=\sum_{i=j}^n\binom ija_ic^{i-j}$ and
$p(q)=\sum_jb_j(q-c)^j$.
\end{lemma}

\begin{proof}\label{proof:lem:translation}
After the pass with index $r$, the invariant is
$b_j=a_j$ for $j<r$ and
$b_j=\sum_{i=j}^n\binom{i-r}{j-r}a_ic^{i-j}$ for $j\ge r$.
It holds before the first pass with $r=n$. Processing $j$ in increasing
order uses the previous-pass value of $b_{j+1}$; Pascal's identity
therefore proves the invariant at the next pass, including $j=r$.
At $r=0$ it is the stated coefficient formula. The expansion of
$q^i=((q-c)+c)^i$ gives the polynomial identity.
\end{proof}

For the criterion, functions $\Phi$ of three variables $(U,T,t)$ on a box
are represented by enclosures of $\Phi$, its gradient and its Hessian on
the box. The rules
\[
 (\Phi\Upsilon)_i=\Phi_i\Upsilon+\Phi\Upsilon_i,\quad
 (\Phi\Upsilon)_{il}=\Phi_{il}\Upsilon+\Phi_i\Upsilon_l+\Phi_l\Upsilon_i+\Phi\Upsilon_{il},\quad
 (1/\Phi)_i=-\frac{\Phi_i}{\Phi^2},\quad
 (1/\Phi)_{il}=\frac{2\Phi_i\Phi_l}{\Phi^3}-\frac{\Phi_{il}}{\Phi^2},
\]
and, for the function $\psi$ of \Cref{lem:conjugate},
$(\psi\circ\Phi)_i=\psi'(\Phi)\Phi_i$ and
$(\psi\circ\Phi)_{il}=\psi''(\Phi)\Phi_i\Phi_l+\psi'(\Phi)\Phi_{il}$, with
$\psi'(x)=-1+\frac49(x+\frac32)_+^2$ and $\psi''(x)=\frac89(x+\frac32)_+$, are
the product, quotient and chain rules; evaluated in enclosures on a box,
they give enclosures on the box. A function of $t$ alone has
$\Phi_U=\Phi_T=0$, $\Phi_t=f_1$, $\Phi_{tt}=2f_2$ in terms of its jet.

\begin{lemma}[Upper bounds on a box]\label{lem:box-bound}
Let $\Phi\in C^2$ on $\mathsf B=I_U\times I_T\times I_t$, let $(c_U,c_T)\in I_U\times I_T$,
$t_c\in I_t$, and $d_U=I_U-c_U$, $d_T=I_T-c_T$, $d_t=I_t-t_c$ as intervals.
\begin{enumerate}
\item[(0)] $\max_{\mathsf B}\Phi$ is at most the upper end of any enclosure of $\Phi$
on $\mathsf B$.
\item If $\Phi_U\ge0$ on $\mathsf B$, then $\max_{\mathsf B}\Phi$ is attained on
$U=\sup I_U$; if $\Phi_U\le0$, on $U=\inf I_U$; likewise for $T$.
\item With $\beta_0=\Phi(c,t_c)+\Phi_t(c,t_c)d_t+\frac12\max\{0,\sup_{I_t}\Phi_{tt}(c,\cdot)\}d_t^2$,
$\Phi(c,t)\le\sup\beta_0$ for $t\in I_t$, where $c=(c_U,c_T)$; hence
$\beta=\min\{\sup\beta_0,\sup\Phi(c,I_t)\}$, with $\Phi(c,I_t)$ any enclosure of
$\Phi(c,\cdot)$ on $I_t$, satisfies $\beta\ge\Phi(c,t)$ for $t\in I_t$.
\item Let $\beta\ge\Phi(c,t)$ for $t\in I_t$, let $G_U,G_T$ enclose
$\Phi_U(c,t),\Phi_T(c,t)$ for $t\in I_t$, and $H$ enclose the Hessian in $(U,T)$
on $\mathsf B$. Then $\max_{\mathsf B}\Phi$ is at most each of
\[
\begin{gathered}
 \sup\Bigl(\beta+G_Ud_U+G_Td_T+\tfrac12\sup H_{UU}\,d_U^2+H_{UT}d_Ud_T+\tfrac12\sup H_{TT}\,d_T^2\Bigr),\\
 \sup\bigl(\beta+\Phi_U(\mathsf B)d_U+\Phi_T(\mathsf B)d_T\bigr).
\end{gathered}
\]
\item If moreover $\alpha_U=-\sup H_{UU}>0$, $\alpha_T=-\sup H_{TT}>0$,
$\gamma=\sup|H_{UT}|$, $\alpha_U\alpha_T>\gamma^2$ and $g_U=\sup|G_U|$,
$g_T=\sup|G_T|$, then
$\max_{\mathsf B}\Phi\le\beta+\frac{\alpha_Tg_U^2+\alpha_Ug_T^2+2\gamma g_Ug_T}{2(\alpha_U\alpha_T-\gamma^2)}$.
If $I_U$ is a point and $\alpha_T>0$, then $\max_{\mathsf B}\Phi\le\beta+g_T^2/(2\alpha_T)$, and
symmetrically; if both $I_U,I_T$ are points, $\max_{\mathsf B}\Phi\le\beta$.
\end{enumerate}
\end{lemma}

\begin{proof}\label{proof:lem:box-bound}
(0) is the definition of an enclosure. (i) For fixed $(T,t)$, $\Phi$ is
monotone in $U$. (ii) Taylor's formula in $t$ at fixed $c$ gives
$\Phi(c,t)=\Phi(c,t_c)+\Phi_t(c,t_c)(t-t_c)+\frac12\Phi_{tt}(c,\xi)(t-t_c)^2$ with $\xi\in I_t$,
which is at most
$\Phi(c,t_c)+\Phi_t(c,t_c)(t-t_c)+\frac12\max\{0,\sup_{I_t}\Phi_{tt}(c,\cdot)\}(t-t_c)^2$, and
$t-t_c\in d_t$.
(iii) For fixed $t$, Taylor's formula in $(U,T)$ at $c$ with the Hessian at
an intermediate point, and the bounds $\Phi_{UU}(U-c_U)^2\le\sup H_{UU}(U-c_U)^2$,
$\Phi_{TT}(T-c_T)^2\le\sup H_{TT}(T-c_T)^2$, give the first bound; the mean value
theorem gives the second. (iv) With $v_U=|U-c_U|$, $v_T=|T-c_T|$, the
first expansion gives
$\Phi-\beta\le g_Uv_U+g_Tv_T-\frac12(\alpha_Uv_U^2-2\gamma v_Uv_T+\alpha_Tv_T^2)$. The matrix
$M=\bigl(\begin{smallmatrix}\alpha_U&-\gamma\\-\gamma&\alpha_T\end{smallmatrix}\bigr)$ is
positive definite, and the maximum of $g\cdot v-\frac12v^TMv$ over $v\in\R^2$
is $\frac12g^TM^{-1}g$, which is the displayed bound. If $I_U$ is a point,
$v_U=0$ and $\max_{v_T\in\R}(g_Tv_T-\frac12\alpha_Tv_T^2)=g_T^2/(2\alpha_T)$; the case of a point
$I_T$ is symmetric; if both are points, $\Phi(U,T,t)=\Phi(c,t)\le\beta$.
\end{proof}

\subsection{The verified computations}\label{sec:verified}

\Cref{ca:transport,ca:sine,ca:primal-second} are proved by the programs of
\cite{Code}, which evaluate the defining formulas as follows.

{Arithmetic.} All quantities are enclosed in interval arithmetic with
outward rounding: in binary floating point of $128$ bits for
$\Theta_{\mathrm{cert}}$ and of $160$ bits for the primal certificate and the sine
energy, with the correctly rounded operations, roots, exponential,
logarithm and $\pi$ of MPFR \cite{MPFR}; and, for the bounds on rectangles,
in binary64 with every result rounded to nearest and then widened to its
floating-point predecessor and successor, which enclose the exact value
\cite{Rump}. Printed decimal and dyadic data are
exact rationals.

{Functions of $t$.} The cubic roots of \Cref{lem:regular-criterion} are
enclosed by \Cref{lem:cubic-root}, and their jets by \Cref{lem:jets}. Each
differential equation in items (R) and (F) before \Cref{lem:regular-criterion}
and each equation of the filters of the primal family (\Cref{prop:primal-jacobian}) has the form
$y'=C+Ay+By^2$ and is integrated by \Cref{lem:flow-step} on a partition of
its interval into cells, with Taylor order $12$ ($14$ for the primal
certificate) \cite{NJC}; on $[1,2]$ the coefficients are those of
\Cref{lem:regular-criterion} at $t=2-v$. Near $t=0$ the function $\epsilon_\circ$ is
taken from its equation and further on as $t^3-\varphi_\circ(t)$
(\Cref{lem:regular-criterion}(i)), and the quotients by $t^3$ are enclosed
by \Cref{lem:cubic-quotient}. Integrals are enclosed by \Cref{lem:quadrature}
with $N=12$ ($N=14$). The programs integrate an equation over a partition of an interval into
cells; on each cell, \Cref{lem:flow-step} is applied with $Y_0$ the
enclosure of the solution at the left end obtained from the previous
cell, and the enclosure at the right end is the intersection of $Y$ with
the bound of the lemma at $v=h$. The solution value at the initial point
of the first cell is exact. An enclosure of the solution on a subinterval
of a cell is obtained from the same bound with $v$ ranging over the
subinterval, and jet enclosures of the solution are then obtained from
\Cref{lem:jets}.

{Rectangles.} For each cell $\mathsf C$, end $\sigma$ and rectangle $\mathsf R$,
the left side of \eqref{eq:witness} is a function of $(U,T,t)$ on
$\mathsf R\times\mathsf C$. Enclosures of its value, gradient and Hessian are
computed by the rules of \Cref{sec:enclosures}, with $\mathfrak b_\sigma,\mathfrak c_\sigma$ in
the regrouped form. Each $t$-coefficient is enclosed both on $\mathsf C$ and
by its second-order Taylor expansion at a point $t_c\in\mathsf C$, the
polynomials of $q$ both directly and in the translated form of
\Cref{lem:translation}, and the determinant and second-trace coefficients
in both forms of \Cref{lem:mobius}; the enclosures are intersected with
each other and with the ranges of \Cref{lem:range,lem:regular-criterion}(ii).
\Cref{lem:box-bound} then gives an upper bound on $\mathsf R\times\mathsf C$. The
rectangles arise from $[0,1]^2$ by repeated bisection, and the verification
checks that, for each cell and each $\sigma$, the bisections form a complete
binary tree, so that the rectangles cover $[0,1]^2$. The divisibility
conditions of \Cref{lem:regular-criterion} on the polynomials of $\Psi_{\rm raw}$
are checked in exact rational arithmetic.

{Data and verification.} The cells, prices, rectangles and flow-step
witnesses are retained in \texttt{results/proof.tar.xz}. Its embedded
\texttt{claims.json} contains the exact defining inputs and interval
reports; \texttt{run.json} fixes the arithmetic settings; and a manifest
records the sizes and SHA-256 digests of the data and checker sources.
The checker validates these digests, the exact polynomial identities and
the complete cover geometry, recompiles the bundled C++ sources, and
re-establishes the flow predicates, every rectangle bound, the price
integrals, and the residual and Jacobian bounds on the entire primal cube.
Fresh enclosures must be contained in the saved intervals (or improve
the saved upper bounds for norm estimates). The exact audit in
Python assembles the lower bound, contraction inequalities and energy
comparisons, and independently re-evaluates the sine series. Its conclusions
are accepted only together with the successful native replay.
The native interval kernels are shared by construction and verification:
this is a separate replay of all acceptance predicates, not a second
implementation of the native mathematics. The trust basis includes these
kernels, the analytical lemmas, the compiler and runtimes, GMP/MPFR and
the checked IEEE binary64 arithmetic environment.

{Reproduction protocol.} The complete experiment starts with candidate
search, constructs all witnesses, packages them and performs full
verification:
\begin{verbatim}
python3 -B compute.py reproduce --jobs 8 --output results
\end{verbatim}
The output directory must be new. This command needs no pre-existing
\texttt{inputs/}: it invokes \texttt{programs\_search/generate.py} with
\texttt{programs\_search/settings.json}, uses NumPy and SciPy, and records
its generated exact inputs and provenance in \texttt{results/inputs/}.
A successful run retains these inputs, \texttt{proof.tar.xz},
\texttt{verification.json}, the short \texttt{results.txt} and the run
record \texttt{run.json}; intermediate construction files are removed.
An interrupted run can be continued with \texttt{--resume}, provided
its sources, settings and selected inputs are unchanged.

To construct and verify from the exact inputs of an existing run, without
candidate search, use a new output directory:
\begin{verbatim}
python3 -B compute.py certify --inputs results/inputs \
  --output results-fixed --jobs 8
\end{verbatim}
To verify only its proof archive, use
\begin{verbatim}
python3 -B compute.py verify --proof results/proof.tar.xz \
  --output verification-new.json --jobs 8
\end{verbatim}
Alternatively, after extracting the proof archive, its standalone
\texttt{verify.py} accepts \texttt{--jobs 8 --output verification-new.json}
under \texttt{python3 -I -B}, without the search or construction package.
Verification requires Python $\ge3.10$, C++17 and the GMP/MPFR development
files. The status \texttt{PASS\_FULL\_NUMERICS} means that all numerical
premises and their exact consequences passed; the checker reports achieved
values rather than testing hard-coded publication thresholds. The
outward endpoints in this article are taken from that report. Repeating
the floating-point search on another platform need not reproduce identical
candidates; the exact inputs saved with a run fix its certificate.

{Settings and recorded run.} The flow and quadrature tolerances are
$10^{-20}$ and $10^{-14}$ for the reference and transport calculations,
and $10^{-24}$ and $10^{-18}$ for the primal calculation. The transport
price accuracy parameter is $2^{-21}$, with an initial state work budget
of $4096$ splits and $16$ construction regions per spatial half; the
completed cover is rechecked in eight regions. The sine tail tolerance
is $10^{-18}$ and selects truncation index $K=3$, giving a polynomial
of degree $126$. These are work and refinement parameters; acceptance
uses the enclosed remainders and strict inequalities. The primal
quadrature partition has $2442$ cells.

The recorded full experiment, including a fresh candidate search, ran on
4 October 2026 on an arm64 Apple M4 Max with $14$ CPU cores
($10$ performance and $4$ efficiency cores) and $36$ GiB RAM.
The run records macOS $15.8.1$, Python $3.13.14$, NumPy $2.3.5$
and SciPy $1.17.0$, with eight requested worker processes and the
BLAS/OpenMP thread limits set to one. The native verifier used Apple Clang
$17.0.0$ (\texttt{clang-1700.6.3.2}), GMP $6.3.0$ and MPFR $4.2.2$,
with flags \texttt{-std=c++17 -O3 -ffp-contract=off -fno-fast-math}.
The recorded wall-clock times, in seconds, are
\begin{center}
\begin{tabular}{@{}lr@{}}
\toprule
phase&seconds\\\midrule
candidate search&$1291.413$\\
certificate construction&$509.245$\\
archive packaging&$60.276$\\
full verification&$356.891$\\\midrule
complete invocation&$2217.950$\\
\bottomrule
\end{tabular}
\end{center}
The total includes orchestration outside the timed phases.
The verification phase includes archive extraction ($2.113$ seconds);
the internal checker reports $354.640$ seconds, including manifest and
geometry checks, compilation ($6.418$ seconds), native replay and the
exact audit. These nested times are not additive. All four phases
completed, and the final status is \texttt{PASS\_FULL\_NUMERICS}.
The proof archive has SHA-256 digest
\begin{center}\small\ttfamily
efaf4681c909c9a3eac175e63687c9bb\\
3ce24101e9f32e11507f27081eafbbe8
\end{center}

\begin{proof}[Proof of \Cref{ca:transport}]\label{proof:ca:transport}
Verified computation (\Cref{sec:verified}). For (i), for every side, cell
$\mathsf C$ and end $\sigma$, the rectangles of the data cover $[0,1]^2$, and the
upper bound computed on each rectangle $\mathsf R$ for the left side of
\eqref{eq:witness} on $\mathsf C\times\mathsf R$ is negative; so the maximum over
$\mathsf C\times[0,1]^2$ is negative; the largest of these upper bounds is at most
$-2^{-43}$. For (ii), by \Cref{thm:contact}
$\Gamma_\varsigma$ is symmetric, so $\bar m_k$ is symmetric about $X_k(1)=0$
(\Cref{lem:local-profile}), and with $\dd x=\mathfrak p\dd t$ on the left half
(\Cref{lem:regular-criterion}(i)):
$\mathcal B_*$ is \eqref{eq:baseline-value} with
$\frac3{2k}\int_0^2\nu_k\dd q=\frac9k\int_0^1t^3\hat\nu\dd t$ and $\mathcal J_\varsigma(2)$ from
\Cref{lem:J-value}, in which $\dd x=\mathfrak p_\circ\dd t$, $f_\circ=\varphi_\circ$,
$\check f_\circ=\varphi_\circ(2-t)$, $\Gamma_\varsigma=Q_t\gamma$ and $f_{\rm end}=\varphi_\circ(2)$
(\Cref{lem:regular-criterion}(i)); $\bar M=2\int_0^1t^2\hat\nu^2\mathfrak p\dd t$;
$H_{\bar m_k}(s_i)=2\int_0^1\varrho_{s_i}\mathfrak p\dd t+\log(1+\varrho_{s_i}(1))$ and
$S(\bar m_k)=2\int_0^1\varrho_\circ^2\mathfrak p\dd t+\varrho_\circ(1)^2$ by \Cref{lem:symmetric-values};
$H_*(s)=\sqrt s+\log(1+e^{-2\sqrt s})-\log2$; and $\Pi$ is the sum of the
integrals of \eqref{eq:closed-price-main} over the cells.
\end{proof}

\begin{proof}[Proof of \Cref{cor:numerical}]\label{proof:cor:numerical}
By \Cref{tab:cert-data}, $\xi_0<0$, $\xi_1>0$, $\chi<0$ and $\tau<0$,
and $-\tau(1-\bar M)=\tau(\bar M-1)$. Each term of \eqref{eq:dual-target-value} is
therefore minimized and maximized at the displayed endpoints, and exact
rational arithmetic gives $c^\pm$. Applying the same assembly to the
unrounded rational endpoints in the verification report gives
\eqref{eq:archived-bound}; its displayed endpoints are rounded outward.
The remaining assertions follow from
\Cref{ca:transport} and \Cref{thm:dual}.
\end{proof}

\begin{proof}[Proof of \Cref{ca:sine}]\label{proof:ca:sine}
\Cref{lem:sine-sum} gives $|\cs-E_3|\le R_3$. The integers $C_n$ are the
coefficients of the displayed product, computed exactly, and $E_3$, $R_3$
are enclosed by evaluating the formulas of \Cref{lem:sine-sum}, which
involve only $\pi$ and the exponential function, in interval arithmetic
(\Cref{sec:verified}).
\end{proof}

\subsection{Proofs of the main theorems}\label{app:main-proofs}

\begin{proof}[Proof of \Cref{thm:energy}]\label{proof:thm:energy}
\Cref{cor:numerical} gives $c_{\mathrm{cert}}>1.14268353624$ and
$\underline{\mathsf e}\ge c_{\mathrm{cert}}$. \Cref{thm:sine-finite} gives the sine limit,
and \Cref{ca:sine} gives $\cs<1.14269218$.
\end{proof}

\begin{proof}[Proof of \Cref{thm:main-centered}]
\label{proof:thm:main-centered}
By \Cref{thm:dual}(ii) for $\Theta_{\mathrm{cert}}$, for every
$\varepsilon'>0$ there is $N_0$ such that
$N^{-3}\E_N(J)\ge c_{\mathrm{cert}}-\varepsilon'$ for $N\ge N_0$ and $J\in\overline{\mathcal Y}_N$;
this does not involve $\mu$ and $L$. By \Cref{thm:jacobi,thm:gain}, every
$\mathcal R\in\mathfrak E_N(\mu,L)$ is $\mathcal R_J$ for some $J\in\overline{\mathcal Y}_N$
and satisfies $A_N(\mathcal R)\ge2\epsN\Delta^{-1}\E_N(J)^{1/2}$, which
gives the first inequality in \eqref{eq:main-A-lower}. By
\Cref{thm:sine-finite},
$A_N(\mathcal R_N^\star)=2\epsN\Delta^{-1}\E_N(J_N^\star)^{1/2}$ for every
$0<\mu<L$ and $N^{-3}\E_N(J_N^\star)\to\cs$, which gives the second.
Dividing the two bounds and letting $\varepsilon\downarrow0$ gives
\eqref{eq:main-centered-ratio}, uniformly in $0<\mu<L$. Finally
$\cs<(1.000004)^2\cdot\num{certificate/bound/lo24}\le(1.000004)^2c_{\mathrm{cert}}$ by
\Cref{ca:sine,cor:numerical}.
\end{proof}

\subsection{Data of the certificate}\label{sec:certificate-data}

\begin{table}[ht]
\centering
\caption{Data of $\Theta_{\mathrm{cert}}$: reference profile
$b_\circ(Q)=p_0+\sum_{j=1}^5p_jQ/(Q+\zeta_j)$ and multipliers. The
scalars $\varsigma,b,\omega$ are in \eqref{eq:cert-scalars}.}
\label{tab:cert-data}
\small
\begin{tabular}{@{}rll@{}}
\toprule
$j$&$p_j$&$\zeta_j$\\\midrule
0&$0.11104619364409615$&\\
1&$0.036541761123835864$&$0.0002733662328804715$\\
2&$0.2725840545911301$&$0.002305100467697839$\\
3&$0.6766436180078876$&$0.012283728048983792$\\
4&$0.15964121932327846$&$0.062212277013292634$\\
5&$0.06784196413046893$&$0.5577071733518729$\\
\bottomrule
\end{tabular}\qquad
\begin{tabular}{@{}cr@{}}
\toprule
multiplier&value\\\midrule
$\xi_0$&$-0.15179495982154212$\\
$\xi_1$&$0.0599781840952077$\\
$\chi$&$-0.043940463594331326$\\
$\tau$&$-0.5723718152513751$\\
\bottomrule
\end{tabular}
\end{table}

\begin{table}[ht]
\centering
\caption{Coefficients $[q^j]P$ of the potential polynomials of
$\Theta_{\mathrm{cert}}$ in \eqref{eq:transport-potential}.}
\label{tab:transport-polynomials}
\footnotesize
\begin{tabular}{@{}crrrr@{}}
\toprule
$P$&$j=0$&$j=1$&$j=2$&$j=3$\\\midrule
$P_h$&$-0.00070089055213150747$&$0.0015674747425538555$&$-0.001144545815888650005$&$0.000269612693526983475$\\
$P_0$&$0.2710507810834969475$&$0.015419641463045405$&$-0.10137127981782688875$&$0.01665381639550413125$\\
$P_1$&$-0.021247530037117474$&$-0.325390035632602646$&$0.232867306470547665$&$-0.03585160759317273$\\
$P_{00}$&$-0.39310093331232166$&$-0.69174781211615728$&$-0.14281571278757247$&$-0.04048480783226955$\\
$P_{01}$&$0.09788231448377985$&$1.1313931929699583$&$0.508831908865809825$&$0.293940344319217375$\\
$P_{11}$&$-0.01558192959906196$&$-0.20447724830709404$&$-0.2712192281957868$&$-0.3614804839686326$\\
\bottomrule
\end{tabular}
\end{table}

\section{The primal certificate}\label{app:primal}

Prescribing the ground-coordinate speed fixes the ground parameter
exactly. The exponential speed family leaves a smooth finite-dimensional
system for the remaining observations. We derive its Jacobian and use a
preconditioned contraction to turn the numerical candidate into an exact
feasible density, with an explicit energy bound.

The candidate generator uses sequential quadratic programming over all
$17$ coefficients, with $192$ Gauss nodes on the left half, first without
the equation for $S$ and then with all six equations. At the resulting
numerical candidate, a column-pivoted QR factorization of the
$6\times17$ constraint Jacobian selects the six coordinates $(j_i)$;
$C$ is the inverse of the corresponding $6\times6$ submatrix, rounded to
exact dyadic rationals. A native interval calculation evaluates the
residual at the dyadic center and applies chord Newton corrections in
these coordinates when necessary. If $\eta$ is its upper bound for
$\|CF(0)\|_\infty$, the proposed radius is
$\varrho=\operatorname{round}_{\uparrow,160}(2\eta+2^{-160})$, where the
rounding is to $160$ binary significant bits. The positive extra term
also covers an exactly zero residual. The archived radius is
\[
 \varrho=27684846429260304880501991746477461659874142439\cdot2^{-207}
 =1.3459626741759871\ldots\cdot10^{-16}.
\]
The data are in \Cref{tab:primal-second-coefficients,tab:primal-second-matrix}.

\begin{proof}[Proof of \Cref{thm:primal-map}]\label{proof:thm:primal-map}
{The coordinate.} Put $B=\|D\|_\infty$. The function $X$ is
absolutely continuous and strictly increasing on $[0,2]$, with
$X(q_2)-X(q_1)\ge(q_2-q_1)/B$, and has finite values $X(0),X(2)$. Its
inverse $q$, extended by $0$ on $(-\infty,X(0)]$ and by $2$ on
$[X(2),\infty)$, is Lipschitz with constant $B$. Let $\mathsf N\subset(0,2)$ be
the null set where $X$ is not differentiable with $X'=1/D$; since $X$ is
absolutely continuous, $X(\mathsf N)$ is a null set. At every $x\in(X(0),X(2))$
outside $X(\mathsf N)$ at which $q$ is differentiable, differentiation of
$X(q(x))=x$ gives $q'(x)=D(q(x))$. By the definition of $m$,
$q'=D(q)=\kappa m-Q(q)$ almost everywhere on $(X(0),X(2))$, and outside this
interval $q'=0=\kappa m-Q(q)$. The change of variables for the absolutely
continuous increasing $X$ gives
$\int_{X(0)}^{X(2)}\phi(q(x))\dd x=\int_0^2\phi(p)D(p)^{-1}\dd p$ for measurable
$\phi\ge0$, which proves \eqref{eq:primal-masses}; also $m\le(B+1)/\kappa$.

{The ground state.} Put $u(x)=\exp\int_0^x(1-q)$. Then $u>0$,
$u'=(1-q)u$ is locally absolutely continuous, and
$u''=(-q'+(1-q)^2)u=(1-\kappa m)u$ almost everywhere, since
$(1-q)^2=1-Q(q)$. On $(-\infty,X(0)]$, $u=c_-e^x$, and on $[X(2),\infty)$,
$u=c_+e^{-x}$; so $u,u'\in L^2$ and $u\in H^1$. The computation in the proof
of \Cref{thm:ground}, which uses only $u>0$, $|u'/u|\le1$ and
$-u''+u=\kappa mu$, gives for every $f\in H^1$
\[
 \int(f'^2+f^2)-\kappa\int mf^2=\int\Bigl(f'-\frac{u'}uf\Bigr)^2\ge0,
\]
with equality for $f=u$. For the feature operator $\mathcal K_m$,
$\ip{\mathcal K_mf}f_\HH=\int m(x)\ip f{v_x}_\HH^2\dd x=2\int mf^2$ by
\Cref{lem:features}(i), so
$\lambda_0(K_m)=\|\mathcal K_m\|=\sup_{f\ne0}2\int mf^2/\|f\|_\HH^2$. The displayed
inequality gives $\lambda_0(K_m)\le2/\kappa$, and equality at $f=u$ gives
$\lambda_0(K_m)\ge2/\kappa$. Hence $\kappa(m)=\kappa$. The ground state $u_m$ of
\Cref{thm:ground} also gives equality in this identity, with the same
$\kappa$, so $u=cu_m$ and the ground coordinate of $m$ is $1-u'/u=q$.
\end{proof}

\begin{lemma}[Exact feasibility by contraction]\label{lem:contraction}
Let $\varrho>0$, let $F$ be continuously differentiable on the cube
$K=\{\|a\|_\infty\le\varrho\}$ in $\R^d$, and let $C$ be a $d\times d$ matrix with
\[
 \theta_F:=\sup_{a\in K}\|I-C\,DF(a)\|_\infty<1,
 \qquad \varepsilon_F:=\|CF(0)\|_\infty,\qquad \varepsilon_F+\theta_F\varrho\le\varrho ,
\]
where $\|\cdot\|_\infty$ is the maximum norm and the induced matrix norm. Then
$F$ has exactly one zero $a_*$ in $K$; the iterates $a^{(0)}=0$,
$a^{(j+1)}=a^{(j)}-CF(a^{(j)})$ lie in $K$, and
$\|a^{(j)}-a_*\|_\infty\le\theta_F^j\varepsilon_F/(1-\theta_F)$.
\end{lemma}

\begin{proof}\label{proof:lem:contraction}
For $G(a)=a-CF(a)$, the mean value inequality gives
$\|G(a)-G(b)\|_\infty\le\theta_F\|a-b\|_\infty$ and
$\|G(a)\|_\infty\le\varepsilon_F+\theta_F\varrho\le\varrho$ on $K$.
The contraction theorem \cite[Theorem~2.1]{TeschlODE} gives a unique
fixed point, the stated iterates, and the error bound, since
$\|G(0)\|_\infty=\varepsilon_F$.
Moreover, $\|I-CDF(0)\|_\infty<1$ makes $CDF(0)$ invertible by a
Neumann series; hence $C$ is invertible, and fixed points of $G$ are
exactly zeros of $F$.
\end{proof}

\begin{proposition}[Derivatives of the primal family]\label{prop:primal-jacobian}
Let $\phi_0,\dots,\phi_d$ be polynomials, $\phi_a=\phi_0+\sum_ia_i\phi_i$,
$\kappa>0$, and let $m_a$ be the density of \Cref{thm:primal-map} with
the symmetric speed $D_a(t^3)=t^2e^{\phi_a(t)}$, $D_a(2-q)=D_a(q)$. On the
left side write $Q=t^3(2-t^3)$, $X_a'(t)=3e^{-\phi_a(t)}$ for $\dd x/\dd t$, and
$n_a=t\hat n_a$, $\hat n_a=\sqrt{(e^{\phi_a}+t(2-t^3))/\kappa}$, for $\sqrt{m_a}$. Then the
masses, the energy, $H_{m_a}(s)$ for $s>0$, and $S(m_a)$ are continuously
differentiable in $a\in\R^d$, with
\begin{align}
 \partial_i\int m_a&=-\frac2\kappa\int_0^1X_a'Q\phi_i\dd t,
 &\partial_i\int\sqrt{m_a}&=-\int_0^1X_a'\frac{D_a+2Q}{\kappa n_a}\phi_i\dd t,
 \label{eq:primal-jacobian-mass}\\
 \partial_iH_{m_a}(s)&=-2\int_0^1X_a'\frac{sQ/\kappa+r_sR_s}{2+r_s+R_s}\phi_i\dd t,
 &\partial_iS(m_a)&=\int_0^1X_a'\bigl[2r_\circ R_\circ-Q(r_\circ+R_\circ)\bigr]\phi_i\dd t,
 \label{eq:primal-jacobian-spec}
\end{align}
where $r_s,r_\circ$ are the left filters of $m_a$ and $R_s,R_\circ$ its right
filters, as functions of $t$ on the left side, and $D_a=t^2e^{\phi_a}$;
also $\int m_a^{j/2}=2\int_0^1X_a'n_a^j\dd t$ for $j=1,2,3$. Moreover $E(m_a)\le E(m_0)\exp\bigl(\|a\|_\infty\sum_i\|\phi_i\|_\infty\bigr)$, the
norms being taken on $[0,1]$.
\end{proposition}

\begin{proof}\label{proof:prop:primal-jacobian}
{The family.} On a cube $\|a\|_\infty\le\varrho$, $c\le e^{\phi_a}\le C$ on
$[0,1]$ for constants $0<c\le C$. Hence $D_a\le C$,
$\int_0^2D_a^{-1}\dd q=2\int_0^13t^2(t^2e^{\phi_a})^{-1}\dd t=2\int_0^1X_a'\dd t<\infty$, and
\Cref{thm:primal-map} applies; $m_a$ is symmetric, $\sqrt{m_a}=n_a$ and
$\dd x=X_a'\dd t$ on the left half. Moreover $n_a$ is a comparison profile at
$\kappa$ (\Cref{def:comparison-profile}): with $\nu=n_a$ and $d=D_a$,
$\nu^2/d=\hat n_a^2e^{-\phi_a}$, $Q/(d\nu)=(2-t^3)/(e^{\phi_a}\hat n_a)$ and
$\nu^3/Q=\hat n_a^3/(2-t^3)$ are bounded on the cube.

{Masses and energy.} By \eqref{eq:primal-masses} and symmetry,
$\int m_a^{j/2}=2\int_0^1X_a'n_a^j\dd t$ for $j=1,2,3$. The integrands are
$C^\infty$ in $(t,a)$, and $\partial_iX_a'=-X_a'\phi_i$,
$\partial_i(n_a^2)=t^2e^{\phi_a}\phi_i/\kappa=D_a\phi_i/\kappa$, so
\[
 \partial_i(X_a'n_a^j)=X_a'n_a^j\Bigl[-1+\frac{jD_a}{2(D_a+Q)}\Bigr]\phi_i ,
\]
using $n_a^2=(D_a+Q)/\kappa$. For $j=2$ this is $-X_a'Q\phi_i/\kappa$; for $j=1$ it is
$-X_a'n_a\frac{D_a+2Q}{2(D_a+Q)}\phi_i=-X_a'\frac{D_a+2Q}{2\kappa n_a}\phi_i$, where
$(D_a+2Q)/n_a=t(e^{\phi_a}+2t(2-t^3))/\hat n_a$ is bounded. Differentiation under
the integral, justified by these bounded continuous derivatives, gives
\eqref{eq:primal-jacobian-mass}. For $j=3$ the bracket equals
$(D_a-2Q)/(2(D_a+Q))\in[-1,\frac12]$; along $\lambda\mapsto\lambda a$,
\[
 \Bigl|\frac{\dd}{\dd\lambda}\log(X_{\lambda a}'n_{\lambda a}^3)\Bigr|
 \le\Bigl|\sum_ia_i\phi_i\Bigr|\le\|a\|_\infty\sum_i\|\phi_i\|_\infty ,
\]
and integration over $\lambda\in[0,1]$ and then over $t$ gives the energy bound.

{Filters.} In the variable $v\in[0,2]$ of \Cref{lem:regular-criterion}
(with $t=v$ or $t=2-v$), the filter $r_s$ of $m_a$ solves
$\frac{\dd r}{\dd v}=X_a'(sn_a^2-2r-r^2)$, $r(0)=0$, with values in
$[0,\frac s\kappa\cdot2]$. On each closed half the right-hand side is $C^\infty$ in
$(v,r,a)$; by differentiable dependence on parameters
\cite[Theorem~2.11]{TeschlODE}, applied on $[0,1]$ and then on $[1,2]$ with
the initial value at $v=1$ obtained from the first half, $r_s$ is
continuously differentiable in $a$, uniformly in $v$. The same holds for
$r_\circ$. By symmetry, the right filters at $q$ are the left filters at
$2-q$, that is, at $2-v$.

{Spectral derivatives.} Fix $a$, $i$ and $\epsilon\ne0$ small, and apply
\Cref{lem:comparison-general} with the comparison profile $\nu=n_a$ (density
$m_a$) and $m=m_{a+\epsilon e_i}$, whose ground parameter is $\kappa$. In the
ground coordinate $q$ of $m$, $\dd x=\dd q/D_{a+\epsilon e_i}(q)$ on its support, and
$n^2-\nu(q)^2=(D_{a+\epsilon e_i}(q)-D_a(q))/\kappa$; outside the support both
vanish. Hence
\[
 \frac{H_{m_{a+\epsilon e_i}}(s)-H_{m_a}(s)}\epsilon
 =\int_0^2a_s^{(a)}\bigl(q,r_s^{(a+\epsilon e_i)}(q)\bigr)
 \frac{1-e^{-\epsilon\phi_i}}{\kappa\epsilon}\dd q ,
\]
where $a_s^{(a)}$ is the coefficient of the profile $n_a$, $r_s^{(a')}(q)$ is
the filter of $m_{a'}$ at its ground coordinate $q$, and $\phi_i$ is
evaluated at the $t$ corresponding to $q$; here
$(D_{a+\epsilon e_i}-D_a)/D_{a+\epsilon e_i}=1-e^{-\epsilon\phi_i}$. As $\epsilon\to0$,
$r_s^{(a+\epsilon e_i)}\to r_s^{(a)}$ uniformly, $(1-e^{-\epsilon\phi_i})/\epsilon\to\phi_i$
boundedly, $a_s^{(a)}(q,\cdot)$ is continuous, and
$|a_s^{(a)}|\le sQ/(2D_a)$ is integrable in $q$. Dominated convergence gives
\[
 \partial_iH_{m_a}(s)=\frac1\kappa\int_0^2a_s^{(a)}(q,r_s)\phi_i\dd q
 =-\int_0^2\frac{sQ/\kappa+r_sR_s}{D_a(2+r_s+R_s)}\phi_i\dd q ,
\]
and the substitution $\dd q/D_a=X_a'\dd t$ on each half, with the symmetry
that exchanges $r_s$ and $R_s$, gives \eqref{eq:primal-jacobian-spec}. The
same argument with $a_\circ$, for which $a_\circ/\kappa=[2r_\circ R_\circ-Q(r_\circ+R_\circ)]/(2D_a)$,
gives the formula for $\partial_iS$. The right-hand sides are continuous in $a$,
since the filters are, with the integrable majorants above; hence the
partial derivatives are continuous and the functions are continuously
differentiable.
\end{proof}

\begin{proof}[Proof of \Cref{thm:primal}]\label{proof:thm:primal}
By \Cref{prop:primal-jacobian}, the map $F$ is continuously
differentiable on $\R^6$, and
\[
 E(m_a)\le E(m_0)\exp\Bigl(\|a\|_\infty\sum_i\|\phi_i\|_\infty\Bigr),
\]
where $\|\phi_i\|_\infty\le1$ because $|T_j|\le1$ on $[-1,1]$. \Cref{lem:contraction}
gives the unique zero $a_*$ with $\|a_*\|_\infty\le\varrho$ and the convergence of the
iterates. By \Cref{thm:primal-map}, $m_{a_*}$ is bounded with compact
support and $\kappa(m_{a_*})=k$; the six components of $F(a_*)$ vanish, and
$S(m)=k^2\tr K_m^2/4$ for $\kappa(m)=k$. Hence
$m_{a_*}\in\mathcal A_{\mathrm{obs}}(\varsigma_{\mathrm{cal}})\cap\Mc$, and
$E(m_{a_*})\le E(m_0)e^{6\varrho}$.
\end{proof}

\begin{proof}[Proof of \Cref{ca:primal-second}]\label{proof:ca:primal-second}
Verified computation (\Cref{sec:verified}). The components of $F(0)$ are
enclosed with the formula $\int m_a^{j/2}=2\int_0^1X_a'n_a^j\dd t$ of
\Cref{prop:primal-jacobian} for the masses,
\Cref{lem:symmetric-values} for $H_{m_0}(s)$ and $S(m_0)$, and the formula for
$H_*$ in the proof of \Cref{ca:transport}; the filters of $m_0$ solve the
equations in the proof of \Cref{prop:primal-jacobian} with $a=0$. For the
derivatives, the coefficients $c_{j_i}$ of the directions are replaced by
the intervals $c_{j_i}+[-\varrho,\varrho]$; the resulting jet, flow and quadrature
enclosures hold for every $a$ in the cube simultaneously, and the formulas
\eqref{eq:primal-jacobian-mass}, \eqref{eq:primal-jacobian-spec} give
enclosures of $\partial_jF_i(a)$ valid for all $\|a\|_\infty\le\varrho$. Each of these
enclosures is contained in $[g_{ij}-10^{-11}|g_{ij}|,\,g_{ij}+10^{-11}|g_{ij}|]$, as
checked in exact rational arithmetic from its endpoints. The energy $E(m_0)=2\int_0^1X_0'n_0^3\dd t$ is
enclosed by \Cref{lem:quadrature}. Applying the exact matrix $C$ directly
to the unrounded residual and whole-cube Jacobian intervals gives the
sharper norm bounds in the proposition. The energy ceiling is evaluated
with outward rounding and checked again using a rational enclosure of
the exponential in the final audit.
\end{proof}

\begin{proof}[Proof of \Cref{cor:primal}]\label{proof:cor:primal}
Interval arithmetic with the exact entries of $C$ and the enclosures of
\Cref{tab:primal-values} gives
\[
 \max_i\Bigl|\sum_jC_{ij}F_j(0)\Bigr|<6.7299\cdot10^{-17},\qquad
 \max_i\sum_j\max\Bigl|\delta_{ij}-\sum_lC_{il}\,\partial_jF_l(a)\Bigr|
 <1.818\cdot10^{-7},
\]
where the inner maximum is over the enclosures of
$\partial_jF_l(a)$, which hold for all $\|a\|_\infty\le\varrho$. Then
$\theta_F<1$ and $\varepsilon_F+\theta_F\varrho<\varrho$, and
\Cref{thm:primal} applies. The outward enclosure of
$E(m_0)e^{6\varrho}$ in \Cref{ca:primal-second} gives $U_{\mathrm{feas}}$,
which is below the lower bound for $\cs$ in \Cref{ca:sine}. The left
inequality is \Cref{cor:numerical}.
\end{proof}

\begin{table}[ht]
\centering
\caption{Enclosures for the primal certificate (\Cref{ca:primal-second}),
with $i$ indexing the components of $F$ in the order
$\int m-1$, $\int\sqrt m-1$, $H(s_0)$, $H(s_1)$, $H(\varsigma_{\mathrm{cal}}k)$,
$S-k^2/6$, and $j$ the coordinates of $a$. Left: $F_i(0)$. Right: centers $g_{ij}$ with
$|\partial_jF_i(a)-g_{ij}|\le10^{-11}|g_{ij}|$ for $\|a\|_\infty\le\varrho$.}
\label{tab:primal-values}
\scriptsize
\begin{tabular}{@{}c@{\;\;}rr@{}}
\toprule
$i$&$10^{18}F_i(0)\ge$&$10^{18}F_i(0)\le$\\\midrule
1&$9.23750001$&$9.23750342$\\
2&$16.4295193$&$16.4297556$\\
3&$82.733744$&$82.733752$\\
4&$125.703781$&$125.703798$\\
5&$1349.64206$&$1349.64382$\\
6&$1.7673299$&$1.76733011$\\
\bottomrule
\end{tabular}\hfill
\begin{tabular}{@{}c@{\;\;}r@{\;\;}r@{\;\;}r@{\;\;}r@{\;\;}r@{\;\;}r@{}}
\toprule
$i$&$g_{i1}$&$g_{i2}$&$g_{i3}$&$g_{i4}$&$g_{i5}$&$g_{i6}$\\\midrule
1&$-0.189430530861$&$-0.105059758071$&$0.0256719247609$&$0.0610832934111$&$0.00996175138787$&$0.00700234895938$\\
2&$-0.588515286928$&$-0.199337084665$&$0.172896503558$&$0.10885975732$&$0.0207773327508$&$0.0152728567795$\\
3&$-1.97420562285$&$-0.994606449927$&$0.409204438536$&$0.620343347838$&$0.0930418433669$&$0.0666175212127$\\
4&$-3.2109737551$&$-1.54553571926$&$0.7526842381$&$0.974605886577$&$0.151664000634$&$0.102850651675$\\
5&$-45.9185397669$&$-16.9210847869$&$13.8807366346$&$9.19547375495$&$1.4112224423$&$1.32252174375$\\
6&$-0.0269825144198$&$-0.0181652591126$&$-0.00119920942717$&$0.00865479505692$&$0.00202926414165$&$0.0012220978259$\\
\bottomrule
\end{tabular}
\end{table}

\begin{table}[ht]
\centering
\caption{Coefficients $c_j$ of $\phi_0$ in the primal certificate
(\Cref{def:primal}).}
\label{tab:primal-second-coefficients}
\small
\begin{tabular}{@{}rl@{\qquad}rl@{\qquad}rl@{}}
\toprule
$j$&$c_j$&$j$&$c_j$&$j$&$c_j$\\\midrule
0&$\dyad{70572439499915}{-46}$&6&$\dyad{-7500767715826011}{-55}$&12&$\dyad{-1322856700011339}{-55}$\\
1&$\dyad{4053730959683531}{-52}$&7&$\dyad{5975462116023045}{-55}$&13&$\dyad{194617143379707}{-53}$\\
2&$\dyad{-912801861921127}{-50}$&8&$\dyad{-4809289143949099}{-55}$&14&$\dyad{-6357166955042095}{-59}$\\
3&$\dyad{8044123430117893}{-54}$&9&$\dyad{3781750787021907}{-55}$&15&$\dyad{1280750309650007}{-58}$\\
4&$\dyad{-1579507580532745}{-52}$&10&$\dyad{-5705414633228235}{-56}$&16&$\dyad{-5433390126452791}{-62}$\\
5&$\dyad{4698918251782431}{-54}$&11&$\dyad{8082726916495781}{-57}$&&\\
\bottomrule
\end{tabular}
\end{table}

\begin{table}[ht]
\centering
\caption{Entries $C_{ij}$ of the preconditioner of the primal certificate.}
\label{tab:primal-second-matrix}
\small
\begin{tabular}{@{}cl@{\qquad}cl@{\qquad}cl@{}}
\toprule
$ij$&$C_{ij}$&$ij$&$C_{ij}$&$ij$&$C_{ij}$\\\midrule
11&$\dyad{-1970922577736555}{-44}$&31&$\dyad{-4678803866072809}{-47}$&51&$\dyad{-8851563200498859}{-42}$\\
12&$\dyad{-6942178203035777}{-48}$&32&$\dyad{-5006373462643775}{-47}$&52&$\dyad{569832512970729}{-44}$\\
13&$\dyad{35795423859617}{-45}$&33&$\dyad{-6451692739882533}{-48}$&53&$\dyad{6338161405103213}{-45}$\\
14&$\dyad{682910787273371}{-47}$&34&$\dyad{8874711208097771}{-49}$&54&$\dyad{-8684969203656933}{-49}$\\
15&$\dyad{4337850365508349}{-54}$&35&$\dyad{7342059655627969}{-54}$&55&$\dyad{-1915318052260205}{-51}$\\
16&$\dyad{3971283601413753}{-44}$&36&$\dyad{8211305181331181}{-46}$&56&$\dyad{1938555127353725}{-39}$\\
21&$\dyad{-2443205397068073}{-41}$&41&$\dyad{-598599239723907}{-40}$&61&$\dyad{-6499908752602999}{-39}$\\
22&$\dyad{8180596473448049}{-47}$&42&$\dyad{2506693653473975}{-46}$&62&$\dyad{-378479574116267}{-44}$\\
23&$\dyad{6750685731695211}{-45}$&43&$\dyad{3884194600727785}{-45}$&63&$\dyad{3894327050650959}{-41}$\\
24&$\dyad{-4542844580359847}{-46}$&44&$\dyad{-662463102184701}{-44}$&64&$\dyad{-4791080514467945}{-43}$\\
25&$\dyad{-3178333096450425}{-52}$&45&$\dyad{-5472740375705655}{-53}$&65&$\dyad{8483029119226405}{-53}$\\
26&$\dyad{1514763362115155}{-40}$&46&$\dyad{2125021405106387}{-42}$&66&$\dyad{4705240194168321}{-38}$\\
\bottomrule
\end{tabular}
\end{table}

\section{Finitely many observations}\label{app:obstruction}

The obstruction comes from incompatible decay rates. The sine amplitude
$n_*$ decays faster than every exponential, whereas each observation
gradient is a constant plus a multiple of $e^{-2x}$, with remainder
$o(n_*)$. A locally sharp differentiable bound would force the energy
gradient into the span of these gradients and the two normalization
gradients. The asymptotic expansions rule out that identity.

Throughout, $m=m_*$ and $n=n_*$. By \Cref{prop:target}, $n$ is even,
continuous and positive, and $m$ is bounded and integrable. For $x\ge0$
put $\varpi=e^{2x}\ge1$; by the bounds for $\Sigma_1,\Sigma_2$ in the proof of
\Cref{prop:target} and with $M_1,M_2$ of \Cref{lem:sine-sum},
$1\le\Sigma_1(e^{-\pi\varpi})\le M_1$ and
$0<\beta_*:=1-2e^{-\pi}/(1-e^{-3\pi})\le\Sigma_2(e^{-\pi\varpi})\le M_2$. By
\eqref{eq:sine-theta}, with $c=8\beta_*^2$ and $C=8M_1^2M_2^2$,
\[
 c\,\varpi e^{-\pi\varpi/2}\le n(x)\le C\varpi e^{-\pi\varpi/2},\qquad
 m(x)\le C^2\varpi^2e^{-\pi\varpi}\qquad(x\ge0).
\]
In particular $\int e^{2|y|}m(y)\dd y<\infty$. With $\sigma=e^{2y}$,
$\dd y=\dd\sigma/(2\sigma)$, for $x\ge0$,
\[
 \int_x^\infty m\le\frac{C^2}2\int_\varpi^\infty\sigma e^{-\pi\sigma}\dd\sigma
 =\frac{C^2}2e^{-\pi\varpi}\Bigl(\frac\varpi\pi+\frac1{\pi^2}\Bigr),\qquad
 e^{-2x}\int_x^\infty e^{2y}m\le\frac{C^2}{2\varpi}\int_\varpi^\infty\sigma^2e^{-\pi\sigma}\dd\sigma ,
\]
and the last integral is $e^{-\pi\varpi}(\varpi^2/\pi+2\varpi/\pi^2+2/\pi^3)$. Since
$\varpi\ge1$, the sum of the two bounds is at most
$C^2e^{-\pi\varpi}(\varpi/\pi+3/(2\pi^2)+1/\pi^3)\le C^2(1+\varpi)e^{-\pi\varpi}$. Hence
\begin{equation}\label{eq:tail-J}
 \mathsf J(x):=\int_x^\infty m+e^{-2x}\int_x^\infty e^{2y}m(y)\dd y
 \le C^2(1+\varpi)e^{-\pi\varpi}=o(n(x)),\qquad n(x)e^{ax}\to0\ (a>0),
\end{equation}
as $x\to+\infty$, because
$\mathsf J/n\le C^2(1+\varpi)e^{-\pi\varpi/2}/(c\varpi)\le2C^2e^{-\pi\varpi/2}/c\to0$ and
$n(x)e^{ax}\le C\exp(2x+ax-\frac\pi2e^{2x})\to0$.

\begin{lemma}[Gradients of the observations]\label{lem:gradients}
For $s>0$ let $V_L$ be the function $V_s$ of \Cref{lem:filter}(ii) for $m$,
$V_R(x)=V_L(-x)$, $u_L=e^xV_L$, $u_R=e^{-x}V_R$, $W=u_L'u_R-u_Lu_R'$ and
$G_s(x,y)=u_L(\min(x,y))u_R(\max(x,y))/W$. Let $u$ be the ground state of
$m$ (\Cref{thm:ground}). For real $h\in L^1(\R)$ write $\mathcal K_h$ as in
\Cref{sec:embedding}, and put $W_2(x)=\int e^{-2|x-y|}m(y)\dd y$. Let
$v_1,\dots,v_d\in C_c^\infty(\R)$ and $m_\epsilon=m+\sum_l\epsilon_lv_l$ for
$\epsilon\in\R^d$.
\begin{enumerate}
\item There is $\epsilon_0>0$ such that $m_\epsilon\ge0$ for $|\epsilon|<\epsilon_0$, and
$(\epsilon,s)\mapsto H_{m_\epsilon}(s)$ is real analytic on
$\{|\epsilon|<\epsilon_0\}\times(0,\infty)$. For $d=1$, $v=v_1$ and $s>0$,
\[
 \partial_\epsilon H_{m_\epsilon}(s)\big|_{\epsilon=0}=s\int G_s(x,x)v,\qquad
 \partial_\epsilon\partial_sH_{m_\epsilon}(s)\big|_{\epsilon=0}
 =\int\bigl[G_s(x,x)+s\partial_sG_s(x,x)\bigr]v,
\]
with $\partial_sG_s(x,x)=-\int G_s(x,y)^2m(y)\dd y$ and
$0<G_s(x,y)\le\frac12e^{-|x-y|}$.
\item The function $\epsilon\mapsto\kappa(m_\epsilon)$ is continuously differentiable
near $0$, and for $d=1$,
$\partial_\epsilon\kappa(m_\epsilon)|_{\epsilon=0}=-\kappa\int u^2v/\int u^2m$.
\item For $d=1$, $\partial_\epsilon\tr K_{m_\epsilon}^2|_{\epsilon=0}=2\int W_2v$.
\item The functions $G_s(x,x)$, $\partial_sG_s(x,x)$, $u^2$ and $W_2$ are
continuous on $\R$, and as $x\to+\infty$, with constants depending on $s$,
\[
\begin{aligned}
 G_s(x,x)&=\tfrac12+C_se^{-2x}+o(n),&\qquad
 \partial_sG_s(x,x)&=C_s'e^{-2x}+o(n),\\
 u(x)^2&=C_ue^{-2x}+o(n),&\qquad
 W_2(x)&=C_2e^{-2x}+o(n).
\end{aligned}
\]
\end{enumerate}
\end{lemma}

\begin{proof}\label{proof:lem:gradients}
{Solutions.} Let $s>0$ and $r_s$ be the filter of $m$. By
\Cref{lem:filter}(ii), $V_L$ is $C^1$ with locally absolutely continuous
derivative and $V_L''+2V_L'=smV_L$, so $-u_L''+(1+sm)u_L=0$ almost
everywhere, and $u_L'/u_L=1+r_s$. Since $m$ is even, $u_R(x)=u_L(-x)$
solves the same equation, and $u_R'/u_R(x)=-(1+r_s(-x))$. Hence $W$ is
constant, $u_L'/u_L\ge1$ and $-u_R'/u_R\ge1$; so $u_L(y)\le u_L(x)e^{y-x}$
for $y\le x$, $u_R(y)\le u_R(x)e^{x-y}$ for $y\ge x$, and, evaluating $W$ at
$x$, $W=u_Lu_R(u_L'/u_L-u_R'/u_R)\ge2u_L(x)u_R(x)$. Hence for $y\le x$,
$0<G_s(x,y)=u_L(y)u_R(x)/W\le u_L(y)/(2u_L(x))\le\frac12e^{-(x-y)}$, and, since
$G_s(x,y)=G_s(y,x)$, the same bound holds for $y\ge x$. By the substitution
$y\mapsto-y$ in the equation of \Cref{lem:filter}(ii),
$V_R(x)=1+\frac s2\int_x^\infty(1-e^{-2(y-x)})m(y)V_R(y)\dd y$, so
$1\le V_R\le e^{s\int m/2}$ and $|V_R(x)-1|\le\frac s2e^{s\int m/2}\int_x^\infty m$.

{(i).} Let $\mathsf S=\bigcup_l\supp v_l$, a compact set, and
$c_{\mathsf S}=\min_{\mathsf S}m>0$ ($m$ is continuous and positive). For
$|\epsilon|\sum_l\|v_l\|_\infty<c_{\mathsf S}/2$, $m_\epsilon\ge c_{\mathsf S}/2$ on $\mathsf S$ and
$m_\epsilon=m$ elsewhere; this defines $\epsilon_0$. By
\Cref{lem:features}(ii),(iii), the operators $\mathcal K_m$ and $\mathcal K_{v_l}$ are
trace class, $A(\epsilon,s)=\frac s2(\mathcal K_m+\sum_l\epsilon_l\mathcal K_{v_l})=\frac s2\mathcal K_{m_\epsilon}$
is positive for $|\epsilon|<\epsilon_0$, $s\ge0$, and its nonzero eigenvalues are those
of $\frac s2K_{m_\epsilon}$, so $H_{m_\epsilon}(s)=\log\det(I+A(\epsilon,s))$. The trace-class determinant is analytic in its operator argument,
and its logarithmic derivative at an invertible $I+A$ is
$D\log\det(I+A)[B]=\tr((I+A)^{-1}B)$
\cite[Chapter~3]{SimonTrace}. Here $A(\epsilon,s)$ is a polynomial
family in trace norm and $\det(I+A)>0$ for real $|\epsilon|<\epsilon_0$
and $s\ge0$, so $H_{m_\epsilon}(s)$ is jointly real analytic.
With $d=1$ and
$B=\frac s2\mathcal K_v$,
\[
 \partial_\epsilon H|_{\epsilon=0}=\frac s2\tr\Bigl[\Bigl(I+\frac s2\mathcal K_m\Bigr)^{-1}\mathcal K_v\Bigr]
 =\frac s2\int v(x)\Bigl\langle\Bigl(I+\frac s2\mathcal K_m\Bigr)^{-1}v_x,v_x\Bigr\rangle_\HH\dd x ,
\]
the trace commuting with the Bochner integral (\Cref{lem:features}(ii)).

Fix $x$ and put $\mathsf g=\sqrt2G_s(\cdot,x)$. By the bounds above and
$|u_L'/u_L|,|u_R'/u_R|\le1+\|r_s\|_\infty$, $\mathsf g$ and $\mathsf g'$ are bounded by
multiples of $e^{-|t-x|}$, so $\mathsf g\in\HH$. On $(-\infty,x)$ and on
$(x,\infty)$, $\mathsf g$ is a multiple of $u_L$, respectively $u_R$, so
$-\mathsf g''+(1+sm)\mathsf g=0$ there; $\mathsf g$ is continuous at $x$, and
$\mathsf g'(x+)-\mathsf g'(x-)=\sqrt2(u_Lu_R'-u_L'u_R)(x)/W=-\sqrt2$. Consequently
$-\mathsf g''+(1+sm)\mathsf g=\sqrt2\delta_x$ distributionally.
Its weak form is
\[
 \ip{\mathsf g}h_\HH+s\int m\mathsf gh=\sqrt2h(x)=\ip{v_x}h_\HH
 \qquad(h\in\HH).
\]
Here the extension from smooth test functions follows by density and
the bound $|\int m\mathsf gh|\le\|m\|_1\|\mathsf g\|_\infty\|h\|_\infty$.
By \Cref{lem:features}(i),(ii), this is
$(I+\tfrac s2\mathcal K_m)\mathsf g=v_x$.
Therefore
$\ip{(I+\frac s2\mathcal K_m)^{-1}v_x}{v_x}_\HH=\ip{\mathsf g}{v_x}_\HH=\sqrt2\mathsf g(x)=2G_s(x,x)$,
which gives the first formula. Since
$\frac{\dd}{\dd s}(I+\frac s2\mathcal K_m)^{-1}=-(I+\frac s2\mathcal K_m)^{-1}\frac12\mathcal K_m(I+\frac s2\mathcal K_m)^{-1}$,
\[
 \partial_s\bigl(2G_s(x,x)\bigr)=-\tfrac12\ip{\mathcal K_m\mathsf g}{\mathsf g}_\HH
 =-\int m\mathsf g^2=-2\int G_s(\cdot,x)^2m .
\]
Hence $|\partial_sG_s(x,x)|\le\frac14\int m$, and by the mean value theorem the
difference quotients of $s\mapsto sG_s(x,x)$ are bounded uniformly in $x$
on compact $s$-intervals; since $v$ has compact support, dominated
convergence gives
$\partial_s\partial_\epsilon H|_{\epsilon=0}=\int[G_s(x,x)+s\partial_sG_s(x,x)]v$, and for the real
analytic $H$ this is $\partial_\epsilon\partial_sH|_{\epsilon=0}$.

{(ii).} On the fixed space $\HH$, the family
$\mathcal K_{m_\epsilon}=\mathcal K_m+\sum_l\epsilon_l\mathcal K_{v_l}$
is analytic in trace norm by \Cref{lem:features}(ii).
The eigenvalue $\lambda_0=2/\kappa$ is simple and isolated
(\Cref{thm:ground}). Analytic perturbation theory for an isolated simple
eigenvalue \cite[Chapter~VII, \S1]{Kato1995} gives a differentiable
eigenvalue and eigenprojection; the spectral gap ensures that this
branch remains the largest eigenvalue for small real $\epsilon$.
Its normalized eigenvector at $0$ is $e=u/\|u\|_\HH$, so
\[
 D\lambda_0(m)[v]=\ip{\mathcal K_ve}e_\HH
 =\frac{2\int vu^2}{\|u\|_\HH^2},\qquad
 \|u\|_\HH^2=\kappa\int mu^2.
\]
Differentiating $\kappa=2/\lambda_0$ yields
$D\kappa(m)[v]=-\kappa\int u^2v/\int mu^2$. The same argument applies
to all the finitely many directions, which proves (ii).

{(iii).} $\tr K_{m_\epsilon}^2=\iint m_\epsilon(x)m_\epsilon(y)e^{-2|x-y|}$
(\Cref{lem:filter}(i)) is a quadratic polynomial in $\epsilon$ with linear
coefficient $2\int vW_2$.

{(iv).} The function $G_s(x,x)=u_L(x)u_R(x)/W$ is continuous;
$G_s(x,y)$ is continuous in $x$ for fixed $y$ and bounded by $\frac12$, so
$\partial_sG_s(x,x)=-\int G_s(x,y)^2m(y)\dd y$ is continuous by dominated
convergence; $u^2$ is continuous; and $W_2$ is continuous by dominated
convergence. Let $x\to+\infty$. With
$\alpha=V_L(+\infty)=1+\frac s2\int mV_L=e^{H_m(s)}$ and
$\beta=-\frac s2\int e^{2y}m(y)V_L(y)\dd y$, the equation of
\Cref{lem:filter}(ii) gives
\[
 V_L(x)=\alpha+\beta e^{-2x}-\frac s2\int_x^\infty mV_L+\frac s2e^{-2x}\int_x^\infty e^{2y}mV_L
 =\alpha+\beta e^{-2x}+O(\mathsf J(x)),
\]
and $V_L'(x)=se^{-2x}\int_{-\infty}^xe^{2y}mV_L=-2\beta e^{-2x}+O(\mathsf J(x))$. Also
$V_R=1+O(\mathsf J)$ and $V_R'(x)=-s\int_x^\infty e^{-2(y-x)}mV_R=O(\mathsf J)$. Since
$u_L=e^xV_L$ and $u_R=e^{-x}V_R$,
$W=(V_L+V_L')V_R+V_L(V_R-V_R')\to2\alpha$, and $W$ is constant, so $W=2\alpha$.
Hence
\[
 G_s(x,x)=\frac{V_L(x)V_R(x)}{2\alpha}=\frac12+\frac\beta{2\alpha}e^{-2x}+O(\mathsf J(x)).
\]
Consider $\partial_sG_s(x,x)=-\int G_s(x,y)^2m(y)\dd y$. With $C''=e^{s\int m}$,
$u_L(x)^2\le C''e^{2x}$ and $u_R(y)^2\le C''e^{-2y}$, so the part $y>x$ is at most
\[
 u_L(x)^2W^{-2}\int_x^\infty u_R^2m\le(C'')^2W^{-2}\int_x^\infty m ;
\]
and the part $y<x$ equals
$-u_R(x)^2W^{-2}\bigl(\int_\R u_L^2m-\int_x^\infty u_L^2m\bigr)$, where
$\int_\R u_L^2m<\infty$, $\int_x^\infty u_L^2m\le C''\int_x^\infty e^{2y}m$ and
$u_R(x)^2=e^{-2x}(1+O(\mathsf J))$. Hence
$\partial_sG_s(x,x)=-W^{-2}\int u_L^2m\,e^{-2x}+O(\mathsf J(x))$. For the ground
state, the proof of \Cref{thm:ground} gives
$u(x)=e^{-x}\int_{-\infty}^xe^yj+e^x\int_x^\infty e^{-y}j$ with $j=\kappa mu/2$, $u$
bounded. Hence $e^xu(x)=c_u+\delta(x)$ with $c_u=\int_\R e^yj$ and
\[
 |\delta(x)|\le\int_x^\infty e^yj+e^{2x}\int_x^\infty e^{-y}j\le2\int_x^\infty e^yj
 \le\kappa\|u\|_\infty e^{-x}\int_x^\infty e^{2y}m\le\kappa\|u\|_\infty e^x\mathsf J(x),
\]
so $u(x)^2=e^{-2x}(c_u+\delta)^2=c_u^2e^{-2x}+O(\mathsf J(x))$. Finally, with
$C_2=\int e^{2y}m(y)\dd y$,
\[
 W_2(x)-C_2e^{-2x}=e^{2x}\int_x^\infty e^{-2y}m(y)\dd y-e^{-2x}\int_x^\infty e^{2y}m(y)\dd y,
\]
whose modulus is at most $\mathsf J(x)$. By \eqref{eq:tail-J}, every $O(\mathsf J)$
is $o(n)$.
\end{proof}

\begin{proof}[Proof of \Cref{thm:no-touching}]\label{proof:thm:no-touching}
Suppose, to the contrary, that for some $\varepsilon>0$ every $m'=m+v$ with
$v\in C_c^\infty$, $\|v\|_\infty<\varepsilon$, $m'\ge0$, $\int m'=\int\sqrt{m'}=1$ and
$\mathbf O(m')\in U$ satisfies $E(m')\ge\Phi(\mathbf O(m'))$.

{A feasible curve.} Let $\mathfrak m_0(v)=\int v$ and
$\mathfrak m_1(v)=\int v/(2n)$ on $C_c^\infty$. The function $1/(2n)$ is continuous
and unbounded, since $n$ is continuous, positive and tends to $0$ at
$\pm\infty$; so there are points $a,b$ with $1/(2n(a))\ne1/(2n(b))$. Let
$w_1,w_2\in C_c^\infty$ be nonnegative with integral one, supported in
intervals around $a$ and $b$ so short that the oscillation of $1/(2n)$ on
each is less than $\frac13|1/(2n(a))-1/(2n(b))|$; then $\mathfrak m_0(w_i)=1$ and
$\mathfrak m_1(w_1)\ne\mathfrak m_1(w_2)$. Let $v\in C_c^\infty$ with
$\mathfrak m_0(v)=\mathfrak m_1(v)=0$, let $\mathsf S$ be the union of the supports of
$v,w_1,w_2$, and $c_{\mathsf S}=\min_{\mathsf S}m>0$. For $|\epsilon|+|z|$ small,
$m_{\epsilon,z}=m+\epsilon v+z(w_1-w_2)\ge c_{\mathsf S}/2$ on $\mathsf S$, and
$F(\epsilon,z)=\int_{\mathsf S}(\sqrt{m_{\epsilon,z}}-\sqrt m)$ is $C^\infty$, with $F(0,0)=0$,
$\partial_\epsilon F(0,0)=\mathfrak m_1(v)=0$ and
$\partial_zF(0,0)=\mathfrak m_1(w_1)-\mathfrak m_1(w_2)\ne0$. The implicit function theorem
gives a $C^\infty$ function $z(\epsilon)$ with $z(0)=0$, $z'(0)=0$ and
$F(\epsilon,z(\epsilon))=0$. The densities $m_\epsilon=m_{\epsilon,z(\epsilon)}$ satisfy
$\int m_\epsilon=1+\epsilon\mathfrak m_0(v)+z\mathfrak m_0(w_1-w_2)=1$ and $\int\sqrt{m_\epsilon}=1$, and
$\|m_\epsilon-m\|_\infty\to0$ as $\epsilon\to0$.

{Differentiability.} By \Cref{lem:gradients}(i)--(iii), applied with
the two directions $v$ and $w_1-w_2$, each coordinate of
$\mathbf O(m_{\epsilon,z})$ is continuously differentiable in $(\epsilon,z)$ near $0$; if
the scales are functions $\sigma_j(\kappa)$, the same holds by the chain rule,
because $H$ and $\partial_sH$ are real analytic in $(\epsilon,z,s)$ and
$\sigma_j(\kappa(m))>0$. In particular $\mathbf O(m_\epsilon)\to\mathbf O(m)$, so for small $|\epsilon|$,
$\|m_\epsilon-m\|_\infty<\varepsilon$, $\mathbf O(m_\epsilon)\in U$, and
$g(\epsilon)=E(m_\epsilon)-\Phi(\mathbf O(m_\epsilon))\ge0$, while $g(0)=0$. Also
$E(m_{\epsilon,z})=\int(m_{\epsilon,z})^{3/2}$ is $C^1$ with $\partial_\epsilon E|_0=\frac32\int nv$, since
$m_{\epsilon,z}=m$ outside $\mathsf S$ and $m_{\epsilon,z}\ge c_{\mathsf S}/2$ on $\mathsf S$. Since
$z'(0)=0$, $g$ is differentiable at $0$ with $g'(0)=\ell(v)$, where
\[
 \ell(v)=\int\Bigl(\tfrac32n-\mathfrak g\Bigr)v,\qquad
 \mathfrak g=c_\kappa u^2+c_{\rm tr}W_2
 +\sum_j\bigl\{a_j\sigma_jG_{\sigma_j}(x,x)+b_j\bigl[G_{\sigma_j}(x,x)+\sigma_j\partial_sG_{\sigma_j}(x,x)\bigr]\bigr\},
\]
with constants $c_\kappa,c_{\rm tr},a_j,b_j$ determined by the partial derivatives of
$\Phi$ at $\mathbf O(m)$, by $\int u^2m$, and, for moving scales, by $\sigma_j'$ and the
derivatives of $H_m$ in $s$; these last terms are multiples of $u^2$, and
are included in $c_\kappa$. Omitted coordinates of $\mathbf O$ correspond to zero
constants. Since $g\ge0$ near $0$ and $g(0)=0$, $\ell(v)=g'(0)=0$.

{The pointwise identity.} The functionals $\mathfrak m_0,\mathfrak m_1$ are linearly
independent, since $\mathfrak m_0(w_1-w_2)=0\ne\mathfrak m_1(w_1-w_2)$ and $\mathfrak m_0(w_1)=1$.
Choose $\varphi_0,\varphi_1\in C_c^\infty$ with $\mathfrak m_i(\varphi_l)=\delta_{il}$. For every
$v\in C_c^\infty$, $v-\mathfrak m_0(v)\varphi_0-\mathfrak m_1(v)\varphi_1$ lies in
$\ker\mathfrak m_0\cap\ker\mathfrak m_1$, so
$\ell(v)=\alpha\mathfrak m_0(v)+\zeta\mathfrak m_1(v)$ with $\alpha=\ell(\varphi_0)$, $\zeta=\ell(\varphi_1)$. The
function $\frac32n-\mathfrak g-\alpha-\zeta/(2n)$ is continuous (\Cref{lem:gradients}(iv))
and integrates to zero against every $v\in C_c^\infty$; hence, for every $x$,
\[
 \tfrac32n(x)=\alpha+\frac\zeta{2n(x)}+\mathfrak g(x).
\]

{The contradiction.} By \Cref{lem:gradients}(iv), as $x\to+\infty$,
$\mathfrak g(x)=c_0+c_1e^{-2x}+o(n(x))$ with constants $c_0=\frac12\sum_j(a_j\sigma_j+b_j)$ and
$c_1$; in particular $\mathfrak g$ is bounded near $+\infty$. Multiplying the
identity by $n(x)\to0$ gives $\zeta/2=\lim_{x\to\infty}n(x)[\frac32n-\alpha-\mathfrak g]=0$.
Then $\frac32n=\alpha+c_0+c_1e^{-2x}+o(n)$, and $x\to\infty$ gives $\alpha+c_0=0$. Then
$\frac32ne^{2x}=c_1+o(ne^{2x})$, and $n(x)e^{2x}\to0$ by \eqref{eq:tail-J} gives
$c_1=0$. Finally $\frac32n=o(n)$ contradicts $n>0$.
\end{proof}
\end{document}